\documentclass[reqno,10pt]{amsart}
\usepackage{amscd,amssymb,amsmath,amsthm}
\usepackage{enumitem} 
\usepackage[dvips]{graphicx}
\usepackage[dvips]{color}
\usepackage[all]{xy}
\usepackage{refcount}
\theoremstyle{plain}
\newtheorem{proposition}{Proposition}
\newtheorem{theorem}[proposition]{Theorem}

\newtheorem*{conjecture*}{Conjecture}
\newtheorem{definition}[proposition]{Definition}
\newtheorem{corollary}[proposition]{Corollary}
\newtheorem{lemma}[proposition]{Lemma}
\newtheorem{remark}[proposition]{Remark}

\newtheorem{example}[proposition]{Example}

\newtheorem{proposition-definition}[proposition]{Proposition/Definition}
\newtheorem*{proposition*}{Proposition}
\newtheorem*{theorem*}{Theorem}
\newtheorem*{maintheorem*}{Main Theorem}
\newtheorem*{maincorollary*}{Main Corollary}
\newtheorem*{corollary*}{Corollary}
\newtheorem*{lemma*}{Lemma}
\newtheorem*{remark*}{Remark}
\newtheorem*{definition*}{Definition}
\newtheorem*{example*}{Example}
\newtheorem*{examples*}{Examples}
\newtheorem*{criterion*}{Generation Criterion}
\newtheorem*{explanation*}{Explanation}

\numberwithin{proposition}{section}
\numberwithin{equation}{section}

\newcommand{\N}{\mathbb{N}}
\newcommand{\Z}{\mathbb{Z}}

\newcommand{\R}{\mathbb{R}}
\newcommand{\C}{\mathbb{C}}
\newcommand{\K}{\mathbb{K}}
\newcommand{\D}{\mathbb{D}}
\newcommand{\scrC}{\mathcal{C}}
\newcommand{\scrW}{\mathcal{W}}
\newcommand{\CP}{\C\mathbb{P}}

\renewcommand{\P}{\mathbb{P}}
\newcommand{\bA}{\mathcal{A}}
\newcommand{\bB}{\mathcal{B}}
\newcommand{\bE}{\mathcal{E}}
\newcommand{\wE}{\mathcal{W}(E)}
\newcommand{\bM}{\mathcal{M}}
\newcommand{\bN}{\mathcal{N}}
\renewcommand{\hom}{\mathit{hom}}

\newcommand{\scrF}{\mathcal{F}}
\newcommand{\Mu}{\boldsymbol \mu}

\begin{document}

\title[The monotone wrapped Fukaya category]{The monotone wrapped Fukaya category and the open-closed string map}

\author{Alexander F. Ritter}
\address{A. F. Ritter, Mathematical Institute, University of Oxford, England.}
\email{ritter@maths.ox.ac.uk}
\author{Ivan Smith}
\address{I. Smith, Centre for Mathematical Sciences, University of Cambridge, England.}
\email{I.Smith@dpmms.cam.ac.uk}

\date{version: \today}

\begin{abstract}
We build the wrapped Fukaya category $\mathcal{W}(E)$ for any monotone symplectic manifold $E$, convex at infinity. We define the open-closed and closed open-string maps, $\mathrm{OC}:\mathrm{HH}_*(\mathcal{W}(E))\to SH^*(E)$ and $\mathrm{CO}: SH^*(E)\to \mathrm{HH}^*(\mathcal{W}(E))$. We study their algebraic properties and prove that the string maps are compatible with the $c_1(TE)$-eigenvalue splitting of $\mathcal{W}(E)$. We extend Abouzaid's generation criterion from the exact to the monotone setting.
We construct an acceleration functor $\mathcal{AF}: \mathcal{F}(E)\to \mathcal{W}(E)$ from the compact Fukaya category which on Hochschild (co)homology commutes with the string maps and the canonical map $c^*:QH^*(E)\to SH^*(E)$. We define the $SH^*(E)$-module structure on the Hochschild (co)homology of $\mathcal{W}(E)$ which is compatible with the string maps (this was proved independently for exact convex symplectic manifolds by Ganatra).
%
%
%
The module and unital algebra structures, and the generation criterion, also hold for the compact Fukaya category $\mathcal{F}(E)$, and also hold for closed monotone symplectic manifolds.

As an application, we show that the wrapped category of $\mathcal{O}(-k) \rightarrow \C\P^m$ is proper (cohomologically finite) for $1\leq k \leq m$. For any monotone negative line bundle $E$ over a closed monotone toric manifold $B$, we show that $SH^*(E)\neq 0$, $\mathcal{W}(E)$ is non-trivial, and $E$ contains a non-displaceable monotone Lagrangian torus $\mathcal{L}$ on which $\mathrm{OC}$ is non-zero.
\end{abstract}
\maketitle
\setcounter{tocdepth}{1}
\vspace{-10mm}
\tableofcontents
\vspace{-14mm}
\section{Introduction}
\subsection{Preamble}
The wrapped Fukaya category $\wE$ of a non-compact convex symplectic manifold $E$, introduced in \cite{FukayaSeidelSmith,Abouzaid-Seidel} in the exact setup, is of increasing importance both in symplectic topology (topology of Stein manifolds \cite{MaydanskiySeidel}, the nearby Lagrangian conjecture  \cite{AbouzaidMaslovZeroLagrangians}) and in Homological Mirror Symmetry \cite{AbouzaidAurouxKatzarkovOrlov,SeidelPOP}. The wrapped category is a version of the Fukaya category, which incorporates both closed Lagrangian submanifolds and non-compact Lagrangians modelled on a Legendrian cone at infinity, and which takes into account the dynamics of the Reeb flow at infinity.  Incorporating non-compact Lagrangians typically leads to a substantial increase in complexity; the wrapped Floer cohomology of a cotangent fibre is isomorphic to the homology of the based loop space \cite{AbbondandoloSchwarz}, in particular it is of infinite rank.

Despite its increasing prominence, there are surprisingly few cases in which the wrapped category has been computed. Essentially complete descriptions are known for cotangent bundles \cite{AbouzaidWrappedBasedLoops} and for punctured spheres \cite{AbouzaidAurouxKatzarkovOrlov}, whilst we have partial information for certain Stein manifolds (given as Lefschetz fibrations \cite{MaydanskiySeidel} or plumbings \cite{AbouzaidSmith}), and some general structural information concerning the behaviour of the category under Weinstein surgery \cite{BEE}.

\subsection{The generation criterion}
Following Abouzaid \cite{Abouzaid}, we study
\begin{equation}\label{EqOC}
\mathrm{OC}:\mathrm{HH}_*(\wE) \to SH^*(E),
\end{equation}
 the open-closed string map  from Hochschild homology to symplectic cohomology. Abouzaid's generation criterion states that for an exact convex symplectic manifold $E$ and a collection $L_1,\ldots,L_n \in \mathrm{Ob}(\wE)$ of exact Lagrangian submanifolds, if $\mathrm{OC}$ restricted to the subcategory generated by $L_1,\ldots,L_n$ hits the unit $1\in SH^*(E)$, then $L_1,\ldots,L_n$ split-generate the category $\wE$ (to \emph{split-generate} means to generate up to taking idempotent projections).

For monotone convex symplectic manifolds $E$, one works instead with monotone orientable Lagrangian submanifolds, and examples naturally arise in the context of (non-compact) Fano toric varieties. In this context, holomorphic discs bounding Lagrangians and holomorphic bubbles in $E$ (and therefore, the quantum cohomology $QH^*(E)$) play a prominent role, unlike the exact setup above where these cannot arise. This complicates the Floer theory, in fact we will explain in detail in \ref{Subsection the role of monotonicity} that the wrapped category is defined as a collection of mutually orthogonal $A_{\infty}$-categories $\mathcal{W}_{\lambda}(E)$ indexed by the spectrum of quantum multiplication by the first Chern class $c_1(TE)$ acting by quantum product on the quantum cohomology $QH^*(E)$ (where $\lambda\in \mathrm{Spec}(c_1(TE))$ lies in the Novikov field $\Lambda$ defined in \ref{Subsection Novikov ring}). This is necessary because if two Lagrangians $L_0,L_1$ involved different $m_0=\lambda$ eigenvalues, then the differential of the 
Lagrangian Floer ``chain complex'' $CF^*(L_0,L_1)$ would not square to zero.

Recall symplectic cohomology $SH^*(E)$ is a unital graded-commutative algebra that comes with a unital algebra homomorphism $c^*:QH^*(E)\to SH^*(E)$ which decomposes 
$$c^*: QH^*(E)=\oplus_{\lambda} QH^*(E)_{\lambda} \to \oplus_{\lambda} SH^*(E)_{\lambda} \subset SH^*(E)$$
into generalized $\lambda$-eigenspaces under multiplication by $c_1(TE)$.
We prove that these two ``eigensummand decompositions'' are related, imposing a restriction on the image of \eqref{EqOC}:
$$
\mathrm{OC}: \mathrm{HH}_*(\mathcal{W}_{\lambda}(E))\to SH^*(E)_{\lambda}\subset SH^*(E).
$$
Therefore the generation criterion, as stated above in terms of hitting the unit, is essentially never going to hold. We will extend Abouzaid's generation criterion \cite{Abouzaid} as follows.

\begin{theorem} \label{Theorem Introduction 2}
Let $\scrC \subset \mathcal{W}_{\lambda}(E)$ be a full subcategory. If $\mathrm{OC}|_{\scrC}$ hits an invertible element in $SH^*(E)_{\lambda}$, then $\scrC$ split-generates $\mathcal{W}_{\lambda}(E)$.\\[0.5mm]
This also holds for the compact Fukaya category $\mathcal{F}_{\lambda}(E)$ using $\mathrm{OC}|_{\scrC}: \scrF(E)_{\lambda} \to QH^*(E)_{\lambda}$, and it also holds if we replace $E$ by a closed monotone symplectic manifold $B$.
\end{theorem}
\begin{example}\label{Example Intro 1}
If $QH^*(E)_{\lambda}$ $($respectively $SH^*(E)_{\lambda})$ is $1$-dimensional, and thus a field, then it suffices to show that the image under $\mathrm{OC}$ of some Hochschild cycle is non-zero. For closed monotone $B$ we prove that $\mathrm{OC}^0[\mathrm{pt}]\neq 0$ in \ref{Subsection Generation results for closed toric Fano varieties}, and for $E$ convex and monotone, $\mathrm{OC}^0[\mathrm{pt}]\neq 0$ when $\lambda=m_0(L)\neq 0$ (see \ref{Subsection Generation results for convex symplectic manifolds}). However, one still needs to find an actual Lagrangian $L$ with $m_0(L)=\lambda$ for which $[\mathrm{pt}]\in CF^*(L,L)$ $($respectively $CW^*(L,L))$ is a cycle, since otherwise $[\mathrm{pt}]$ will not survive to cohomology. For toric manifolds we have a good supply of Lagrangian tori arising from critical points of the superpotential (see \ref{Section Brief survey on Landau Ginzburg}).
\end{example}

\begin{example}
Theorem \ref{Theorem Introduction 2} applies to the Clifford torus in $B=\C P^m$ taken with one of $1+m$ choices of  holonomy. It also applies to the lift of the Clifford torus to a sphere bundle of the complex line bundle $E=\mathcal{O}(-k)\to \C P^m$ where $1\leq k \leq m$ (to ensure monotonicity), taken with one of $1+m-k$ choices of holonomy. More generally, for any closed Fano variety $B$ for which the superpotential $W$ is Morse with distinct critical values, we show that $\mathcal{F}_{\lambda}(B)$ is split-generated by the monotone torus $L_z$ with holonomy data associated to $z\in \mathrm{Crit}(W)$.
\end{example}

The proof of Theorem \ref{Theorem Introduction 2} requires two ingredients. Firstly, we remark that Abouzaid's generation criterion is a fairly general consequence of homological algebra once the wrapped Fukaya category can be constructed together with the necessary Floer machinery to define string maps and a coproduct (we recall this in Section \ref{Subsection Outline of the proof of generation}). Thus our first contribution is a foundational one, namely we construct this machinery in the monotone setting. As discussed further in Section \ref{Subsection Linear vs quadratic Hams}, this is a fairly substantial undertaking. In particular, we need a new double-telescope complex for the coproduct, which has not appeared before (see Section \ref{Section The coproduct}). 

Secondly, we prove that $\mathrm{HH}_*(\mathcal{W}_{\lambda}(E))$ is a module over $SH^*(E)$, and that $c_1(TE)-\lambda\, \mathrm{Id}$ acts nilpotently on each element of $\mathrm{HH}_*(\mathcal{W}_{\lambda}(E))$ (Section 
\ref{Section OC and CO maps are module homs}). Moreover we show that $\mathrm{OC}$ is an $SH^*(E)$-module homomorphism (Section \ref{Section algebraic structures}). Therefore, $\mathrm{OC}$ lands in $SH^*(E)_{\lambda}$ and hitting any invertible in that eigensummand is as good as hitting the identity of $SH^*(E)_{\lambda}$. 

The module structure was constructed independently, in the exact case, by Ganatra \cite{Ganatra}, and the underlying $QH^*$-module structure is also used in forthcoming work of Abouzaid-Fukaya-Oh-Ohta-Ono \cite{AFOOO} on generation in the closed case.

\subsection{The acceleration functor}
\label{Subsection Intro The acceleration functor}

For convex symplectic manifolds $E$, one can also define the compact Fukaya category $\mathcal{F}(E)$ (Section \ref{Section Fukaya category}). In the monotone setup, this means the objects are \emph{closed} monotone orientable Lagrangian submanifolds (again one must in fact restrict these to all involve the same $m_0=\lambda$ eigenvalue), and the morphism spaces are the Lagrangian Floer complexes $CF^*(L_0,L_1)$. When $L_0,L_1$ are transverse, this complex is generated by the intersections $L_0\cap L_1$, and in the non-transverse case one introduces a \emph{compactly-supported} Hamiltonian $H$ so that the time one flow $\varphi_H^1(L_0)$ of $L_0$ becomes transverse to $L_1$ (we recall this in \ref{Subsection When Lagrangians do not intersect transversely}). 
Closed Lagrangians are also objects of the wrapped category, and the wrapped Floer complex $CW^*(L_0,L_1)$ can be loosely thought of (at least on cohomology) as a direct limit of Floer complexes $CF^*(\varphi_H^1(L_0),L_1)$ for a non-compactly supported Hamiltonian $H$ which is linear at infinity, where in the direct limit we increase the slope of $H$ via continuation maps. We remark that the substantial work required to build the wrapped category is precisely the chain-level construction of this direct limit procedure using a telescope construction that we adopt from Abouzaid-Seidel \cite{Abouzaid-Seidel}. In fact, we will also need a new double telescope chain model for the direct limit when we come to the construction of the coproduct in Section \ref{Section The coproduct}.

Therefore one might expect there to be an $A_{\infty}$-functor, called \emph{acceleration functor}
$$
\mathcal{AF}: \mathcal{F}(E)\to \mathcal{W}(E)
$$
which is the natural inclusion of objects, and the inclusion of the ``slope zero'' part $CF^*(L_0,L_1)$ into the above direct limit construction. However, a rigorous construction of this functor is again a non-trivial task, because ``slope zero'' Hamiltonians are not in general allowed in $\mathcal{W}(E)$ because for non-compact Lagrangians $L_0,L_1$ which have intersections at infinity, for ``slope zero'' $H$ the compactly supported flow $\varphi_H^1(L_0)$ of $L_0$ cannot become transverse to $L_1$.

\begin{theorem}\label{Theorem introduction acceleration diagram}
For any monotone convex $E$, there is an acceleration functor which fits into commutative diagrams, which we call \emph{acceleration diagrams},
$$
\xymatrix@C=50pt{ \mathrm{HH}_*(\scrF_{\lambda}(E)) \ar@{->}^{\mathrm{HH}_*(\mathcal{AF})}[r] \ar@{->}_{\mathrm{OC}}[d] & \mathrm{HH}_*(\mathcal{W}_{\lambda}(E))
\ar@{->}^{\mathrm{OC}}[d] & 
\mathrm{HH}^*(\scrF_{\lambda}(E)) \ar@{<-}^{\mathrm{HH}^*(\mathcal{AF})}[r] \ar@{<-}_{\mathrm{CO}}[d] & \mathrm{HH}^*(\mathcal{W}_{\lambda}(E))
\ar@{<-}^{\mathrm{CO}}[d] \\
 QH^*(E)_{\lambda}
\ar@{->}[r]^-{c^*} & SH^*(E)_{\lambda}
&
 QH^*(E)
\ar@{->}[r]^-{c^*} & SH^*(E)
}
$$
Moreover, the maps in the diagram are $QH^*(E)$-module homomorphisms, and the maps in the diagram on the right are unital algebra homomorphisms.
\end{theorem}

\noindent {\bf Remark.} \emph{$HH^*$ is not in general functorial under $A_{\infty}$-functors. We construct $HH^*(\mathcal{AF})$ at the cochain level as an inclusion of a non-full subcategory, for which $HH^*(\mathcal{AF})$ makes sense.}
%
\begin{corollary}
 Let $E$ be a convex toric Fano variety. Let $z\in \mathrm{Crit}(W)$ be a critical point of the superpotential with non-zero critical value $\lambda=W(z)\neq 0$. If the generalized eigensummand $QH^*(E)_{\lambda}$ is $1$-dimensional then the Lagrangian torus $L_z$ corresponding to $z$ split-generates $\scrF_{\lambda}(E)$ and also $\mathcal{W}_{\lambda}(E)$, in particular $\mathcal{W}_{\lambda}(E)$ is then proper (cohomologically finite).
\end{corollary}

We remark that just as for closed Fano toric varieties taken with the monotone toric symplectic form \cite[Theorem 7.11]{FOOOtoric} (compare also with the discussion in the non-compact case in \cite[Section 7.9]{Ritter5}) one can show that the Lagrangian tori $L_z$, with holonomy data, arising from critical points of the superpotential are in fact monotone. The proof of the Corollary uses the fact that the $0$-part of the $\mathrm{OC}$ map for the compact category,
$$
\mathrm{OC}^0:HF^*(L,L) \to QH^*(E)
$$
is in fact non-zero for $L=L_z$, and therefore hits an invertible in the field $QH^*(E)_{\lambda}$ (using the assumption), so via the acceleration diagram, using that $c^*$ is an algebra map, also 
$$\mathrm{OC}^0:HW^*(L,L) \to SH^*(E)$$
for $L=\mathcal{AF}(L_z)=L_z$ will hit an invertible element in $SH^*(E)_{\lambda}$. In particular, observe that in general $c^*:QH^*(E)_{\lambda}\to SH^*(E)_{\lambda}$ sends unit to unit, since $c^*:QH^*(E)\to SH^*(E)$ does.
%
%
%

\subsection{Linear versus quadratic Hamiltonians}\label{Subsection Linear vs quadratic Hams}
There are two known constructions of the wrapped category of exact convex symplectic manifolds, which take advantage of Hamiltonian functions which have either linear slope at infinity or which grow quadratically at infinity. In either set-up one must deal with the fact that $CF^*(L,L)$ actually means $CF^*(\varphi_{H}(L),L)$ (using a Hamiltonian function to displace $L$ slightly to achieve self-transversality) so products in Floer theory cannot be defined as $CF^*(L,L)^{\otimes k}\to CF^*(L,L)$. Rather, they are maps\footnote{we are implicitly using canonical identifications of type 
$CF^*(\varphi_{H_j}(L),L)\equiv CF^*(\varphi_{H+H_j}(L),\varphi_{H}(L))$.}
\begin{equation} \label{Eqn:basicproblem}
CF^*(\varphi_{H_1}(L),L) \otimes  CF^*(\varphi_{H_2}(L), L)\otimes \cdots  \otimes CF^*(\varphi_{H_k}(L),L) \longrightarrow CF^*(\varphi_{H_1+\cdots+H_k}(L),L)
\end{equation}
For non-compact $L$, the Floer complex on the right of \eqref{Eqn:basicproblem} may not be isomorphic, at the chain level, to the factors on the left, even if all the $H_i$ coincide. The Hamiltonians $H_i$ are not compactly supported, and ``large'' Hamiltonians $H_1+\cdots+H_k$ can cause $CF^*(\varphi_{H_1+\cdots+H_k}(L),L)$ to have many generators (the more we flow $L$, the more times it wraps around $L$, creating new intersections).

In the setting of linear slope Hamiltonians,  as in \cite{Abouzaid-Seidel}, one defines the morphism groups via a ``telescope" (homotopy direct limit) construction, $CW^*(L,L) = \oplus_k CF^*(\varphi_{kH}(L), L)[\mathbf{q}]$, involving arbitrarily  large multiples of a fixed Hamiltonian of linear slope; whilst in the setting of quadratic Hamiltonians, one uses a rescaling trick (conjugating by a suitable multiple of the Liouville flow) to identify the output chain complex in \eqref{Eqn:basicproblem} with the input complexes appearing on the left.
Unfortunately, the rescaling trick is only available in the exact case, with no analogue  for a general monotone convex symplectic manifold. Abouzaid \cite{Abouzaid} proved the generation criterion in the exact setup using quadratic Hamiltonians, and Ganatra uses the same setup in his discussion \cite{Ganatra} of the module structure of the string maps.

The proof of each of the module structure of the $\mathrm{OC}$-map and the generation criterion in the telescope model is surprisingly involved: the necessary modifications are numerous and non-trivial, and  constitute one of the main contributions of this paper.

Although the quadratic Hamiltonian formalism as in \cite{Abouzaid,Ganatra} bypasses the need in the exact setup for a telescope model, that formalism also has technical complications: a subtle perturbation scheme is needed to isolate the Hamiltonian orbits (see \cite[Definition 4.10]{Ganatra}) and the conformal rescaling scheme using the Liouville flow is rather involved. Our telescope model formalism also adapts to the exact setup, and avoids both these issues. Another benefit of our approach, is that it should adapt to general convex symplectic manifolds, without any monotonicity assumptions, at the cost of appealing to Kuranishi spaces or polyfolds.
\\
{\bf Remark.} \emph{In the non-monotone setting, the $m_0=\lambda$ values which index the categories $\mathcal{W}_{\lambda}(E)$ are not a priori tied to the eigenvalues of $c_1(TE)\in QH^*(E)$, see \ref{Subsection the role of monotonicity}.}
\subsection{Algebraic structure}
In \ref{Subsection psi structure map}, we construct a unital algebra homomorphism
$$
\psi:SH^*(E)\to \mathrm{End}(\wE), \; c\mapsto \psi_c
$$
where $\mathrm{End}(\wE)$ denotes the bimodule endomorphisms of the diagonal $\wE$-bimodule. Composing with the canonical map $\mathrm{End}(\wE) \to \mathrm{End}(\mathrm{HH}_*(\wE))$, the $\psi$-string map defines an $SH^*(E)$-module structure on $\mathrm{HH}_*(\wE)$.
\\
\noindent {\bf Remark.} \emph{Ganatra \cite{Ganatra} independently constructed the map $\psi$ for exact convex symplectic manifolds using the quadratic Hamiltonian formalism of
\cite{Abouzaid}: the construction is conceptually the same as ours, although technically 
quite different as explained in \ref{Subsection Linear vs quadratic Hams}.
 Ganatra calls $\psi$ the \emph{2-pointed closed-open string map} $_{2}\mathrm{CO}$, due to its similarities with $\mathrm{CO}$, and he calls the chain complex of bimodule pre-endomorphisms $\mathrm{end}(\wE)$ the \emph{$2$-pointed Hochschild chain complex}. The emphasis in \cite{Ganatra} is to view the module structure as the cap-product action of $\mathrm{CO}(c)\in \mathrm{HH}^*(\wE)$ on $\mathrm{HH}_*(\wE)$, whereas we have preferred to work directly with bimodule endomorphisms. The two approaches are equivalent due to the following diagram.}

\begin{theorem}
There is a commutative diagram of unital algebra homomorphisms
$$
\xymatrix@C=50pt{ QH^*(E) \ar@{->}^-{c^*}[r] \ar@{->}_-{\mathrm{CO}}[rrd]  & SH^*(E) \ar@{->}^-{\psi}[r]  \ar@{->}^-{\mathrm{CO}}[rd] & \mathrm{End}(\wE)
\ar@{->}^-{\mathrm{HH}_*(\cdot)}[r] & \mathrm{End}(\mathrm{HH}_*(\wE)) \\
& & \mathrm{HH}^*(\wE) \ar@{->}_-{\strut\textrm{cap product}}[ru] \ar@{->}_{\mu_{\wE}\circ_{\mathrm{left}}}[u] &
}
$$
using composition product for $\mathrm{End}(\wE)$ and $\mathrm{End}(\mathrm{HH}_*(\wE))$, cup product for $\mathrm{HH}^*(\wE)$.

The vertical map (defined in \ref{Subsection Hochschild cohomology as endomorphisms of the diagonal bimodule}) is an isomorphism.
The wrapped category is cohomologically unital, the cohomological units are $e_L=\mathrm{CO}^0([E])(1)\in CW^*(L,L)$ where $\mathrm{CO}^0(\cdot)(1): QC^*(E)\to CW^*(L,L)$. Moreover $\mathrm{CO}([E])\in \mathrm{HH}^*(\wE)$ is a unit for the cup product.
\end{theorem}

\begin{remark}
String maps were introduced by Seidel \cite{SeidelICM}.
In the closed monotone case, Albers \cite{Albers} defined maps $\tau: HF_*(H)\to HF^*(L,\varphi_H^1(L))$ and $\chi:HF_*(L,\varphi_H^1(L)) \to HF_*(H)$, identifiable with $\mathrm{CO}^0:QH^*(B)\to HF^*(L,L)$ and $\mathrm{OC}^0:HF^*(L,L) \to QH^*(B).$ Biran-Cornea \cite{BiranCorneaPearly} studied the $QH^*(B)$-module structure of these maps.
In the exact case, the $SH^*(E)$-module structure of $HW^*(L,L)$ was studied by Ritter \cite{Ritter3}.
\end{remark}

We will also prove that the string maps are compatible with the module structure:
$$
a\bullet \mathrm{OC}(x) = \mathrm{OC}(\psi_a(x))
\qquad \qquad
\mathrm{CO}(a\bullet b) = \psi_a\circ \mathrm{CO}(b) = \mathrm{CO}(a)*\mathrm{CO}(b)
$$
where $\bullet$ is the $\mu^2_{SH^*(E)}$-product,
$a,b\in SH^*(E)$, $x\in \mathrm{HH}_*(\wE)$, $\circ$ is defined in \ref{Subsection Conventient abbreviations: the symbols circ and Mu}, and the cup product $*$ on $\mathrm{HH}^*(\wE)$ is defined in \ref{Subsection Product on Hochschild cohomology}.

\subsection{Eigensummand decomposition of $\mathbf{OC}$ and $\mathbf{CO}$}

\begin{theorem}
The open-closed string map $\mathrm{OC}: \mathrm{HH}_*(\mathcal{W}_{\lambda}(E))\to SH^*(E)$ must land in the generalized $\lambda$-eigenspace $SH^*(E)_{\lambda}$ $($and analogously for $\mathcal{F}_{\lambda}(E)$ and $QH^*(E))$.\\
\indent
The closed-open string map $\mathrm{CO}: SH^*(E)_{\mu} \to \mathrm{HH}^*(\mathcal{W}_{\lambda}(E))$, restricted to the generalized $\mu$-eigenspace, will vanish if $\mu\neq \lambda$ $($and analogously for $\mathrm{CO}: QH^*(E)_{\mu} \to \mathrm{HH}^*(\mathcal{F}_{\lambda}(E)))$.
\end{theorem}

At the heart of this result is an observation due to Kontsevich, Seidel and Auroux \cite{Auroux}. Namely, $\mathrm{CO}^0: QH^*(E)\to HF^*(L,L)$ maps $c_1(TE)$ to the unit $\mathrm{CO}^0([E])$ rescaled by the $m_0(L)=\lambda$ value
 (see \ref{Subsection Relating m0 and c1 evals}). So the module action of $c_1(TE)-\lambda \, \mathrm{Id}$ on $HF^*(L,L)$ vanishes as $\mathrm{CO}^0(c_1(TE)-\lambda\, \mathrm{Id})=0\in HF^*(L,L)$. This implies that the bimodule endomorphism $\psi_{c_1(TE)-\lambda\, \mathrm{Id}}$ (after subtracting an exact term, since now we work at the chain level) has vanishing $0|0$-part. It follows that $\psi_{c_1(TE)-\lambda\, \mathrm{Id}}$ acts nilpotently on any given word in $\mathrm{CC}_*(\bM)$ (composing many times, eventually one of the maps in the composite must be taken with $0|0$-part, as the word has finite length, see \ref{Subsection Morphisms of bimodules induce chain maps on CC}). This implies that any given cycle in $\mathrm{CC}_*(\bM)$ is killed by a sufficiently large quantum power of $c_1(TE)-\lambda\, \mathrm{Id}$, therefore $\mathrm{OC}$ evaluated on the cycle must land in the $\lambda$-generalized eigensummand of $SH^*(E)$ (respectively $QH^*(E)$). One can dualize this argument (more precisely, this requires a subtle filtration argument) to show that $\mathrm{CO}$ must also respect eigensummands.

\subsection{Applications}
The total space of the line bundle $\mathcal{O}(-k) \rightarrow \C\P^m$ is monotone whenever $1\leq k \leq m$.  The ring $SH^*(\mathcal{O}(-k))$ was computed explicitly in \cite{Ritter4,Ritter5}, and we will show that
$$
SH^*(\mathcal{O}(-k))\cong \mathrm{Jac}(W)
$$
recovers the Jacobian ring of the superpotential for $\mathcal{O}_{\P^m}(-k)$. This is a proof of closed-string mirror symmetry in this example. Note that unlike the case of closed Fano toric manifolds, the Jacobian ring does not recover the quantum cohomology $QH^*(\mathcal{O}(-k))$ (see \ref{Subsection Generation for O of -k}).

Then, making heavy use of the module structure and compatibility with eigenvalue splittings, we compute the open-closed string map explicitly, and prove that the wrapped category is proper (cohomologically finite), split-generated by closed Lagrangian submanifolds.   By contrast, in all previous situations in which the wrapped Fukaya category has been computed (punctured spheres, certain plumbings, cotangent bundles), it either vanishes or is cohomologically infinite. The special case of $\mathcal{O}(-1) \rightarrow \P^1$ had been studied previously \cite[Sec.4.4]{SmithQuadrics}.

Local Calabi-Yau's (total spaces of canonical bundles over Fano varieties) are well-known testing grounds for many aspects of mirror symmetry, with rich symplectic topology \cite{SeidelSuspending,Bridgeland}. Monotone negative line bundles form another class of ``local symplectic manifolds'', which arise naturally in symplectic birational geometry, cf.\,\cite{SmithQuadrics}.  As remarked above, whilst in known computations in the Stein case the wrapped category $\wE$ either vanishes or is extremely complicated, the line bundles $\mathcal{O}(-k) \rightarrow \P^n$ show qualitatively different behaviour. 


Making use of the module structure, the compatibility with eigenvalue splittings and a Galois action on Floer cohomology, we prove:
\begin{theorem}
\strut

For any monotone negative line bundle $E \to B$ over a closed monotone toric manifold $B$,
\begin{enumerate}
\item the symplectic cohomology $SH^*(E)$ is non-zero;
\item the wrapped category $\mathcal{W}(E)$ is non-trivial; 
\item in particular, there is a non-displaceable Lagrangian torus $L\subset E$ with holonomy data such that $HF^*(L,L)\cong HW^*(L,L)\neq 0$, on which $\mathrm{OC}^0$ is non-zero.
\end{enumerate}
\end{theorem}
The proof of the theorem involves showing that critical points $p\in \mathrm{Crit}(W_B)$ of the superpotential of the base with non-zero critical value $W_B(p)\neq 0$ give rise to critical points of the superpotential $W_E$, which in turn give rise to Lagrangian tori $L\subset E$ with holonomy data such that $HF^*(L,L)\neq 0$. In particular, these exist since a recent result of Galkin \cite{Galkin} implies that $W_B$ has at least one critical point $p$ with $W_B(p)\neq 0$. The $QH^*(E)$-module structure on $HF^*(L,L)\neq 0$ implies that $c_1(TE)$ is non-nilpotent in $QH^*(E)$. By Ritter \cite{Ritter4}, 
$$SH^*(E)\cong QH^*(E)/(\textrm{generalized 0-eigenspace of quantum multiplication by }\pi^* c_1(E))$$
for any monotone negative line bundle $\pi: E\to B$.
As $c_1(TE)$ and $\pi^*c_1(E)$ are positively proportional, it follows that $\pi^*c_1(E)$ is also not nilpotent so
$SH^*(E)\neq 0$. By the argument explained at the end of \ref{Subsection Intro The acceleration functor}, $\mathrm{OC}^0:HF^*(L,L)\to QH^*(E)$ is non-zero and therefore by the acceleration diagram in Theorem \ref{Theorem introduction acceleration diagram} also $\mathrm{OC}^0:HW^*(L,L)\to SH^*(E)$ is non-zero. Therefore the category $\mathcal{W}_{\lambda}(E)$ containing $L$ is cohomologically non-trivial.

It is not hard to envisage other applications.  In a series of papers, Abouzaid \cite{AbouzaidFibreGenerates, AbouzaidWrappedBasedLoops, AbouzaidMaslovZeroLagrangians} proved that the cotangent fibre $T_q^*Q$ generates the wrapped category $\scrW(T^*Q)$ of the cotangent bundle of a closed manifold equipped with its canonical exact symplectic structure $d\theta$.  From a dynamical point of view, there is a great deal of interest in the convex non-exact manifolds defined by cotangent bundles equipped with magnetic forms, i.e. by $(T^*Q, d\theta + \pi^*\sigma)$ where $\sigma \in \Omega^2(Q)$ is some closed but not necessarily exact form pulled back from the base.  It seems likely that our results can be used to show a cotangent fibre still generates the wrapped category after turning on the magnetic potential.

 
\subsection{Remarks and Acknowledgements}\strut

\noindent\textbf{Remark about signs and characteristic.} Although our constructions rely on working in characteristic zero, for simplicity of exposition we will not discuss signs at all in this paper. A very careful treatment of signs in Floer theory is carried out in Seidel \cite{Seidel}, Abouzaid-Seidel \cite{Abouzaid-Seidel}, Abouzaid \cite{Abouzaid}, Ganatra \cite{Ganatra} and Ritter \cite{Ritter3}. In particular, in several constructions we added a remark about a ``sign correction factor''. These are not strictly meaningful in the absence of a detailed discussion of signs, as they are only defined relative to a certain convention for identifying orientation lines of the inputs with those of the outputs, for which a number of choices exist. So these remarks are understood to follow the conventions of \cite{Seidel,Abouzaid-Seidel,Abouzaid}.
\\[2mm]
\noindent\textbf{Acknowledgements.} \emph{We thank: Mohammed Abouzaid for his continuing interest in the project; Denis Auroux and Sheel Ganatra for helpful conversations; the anonymous referee for very extensive comments; Janko Latschev and the referee for pointing out the method in \ref{Subsection A comment about the existence of universal choices of auxiliary data} to implicitly construct the auxiliary data (previous arXiv versions make explicit constructions); Nick Sheridan and the referee for pointing out how to adapt the proof of the eigensummand decomposition for $\mathrm{OC}$ directly to $\mathrm{CO}$ in Theorem \ref{Theorem CO respects esummands}, bypassing duality issues.}\\[2mm]
{\textbf{Current version.} \emph{Since the first version \cite{Ritter-Smith-version1} of this paper appeared (Jan. 2012) two papers have appeared with significant conceptual overlap: (1) Ganatra's PhD and preprint \cite{Ganatra} studies the wrapped Fukaya category and the string maps in the exact setup, (2) Sheridan's preprint \cite{Sheridan} studies the Fukaya category of closed monotone symplectic manifolds. Still to appear, the work in progress by Abouzaid-Fukaya-Oh-Ohta-Ono \cite{AFOOO} will study the Fukaya category of closed toric manifolds without monotonicity assumptions. A large part of an earlier version \cite{Ritter-Smith-version1} constructed a commutative diagram relating the Fukaya category of a closed monotone symplectic manifold $B$ and of a convex (non-compact) monotone symplectic manifold $E$ related by a Lagrangian correspondence $\Gamma \subset \overline{B}\times E$. This part was removed to keep this paper of reasonable length; we hope this material will appear elsewhere.
}

\section{Hochschild (co)homology of $A_{\infty}$-categories}
\label{Section algebra}
We summarize the general algebraic constructions that we need in the paper.
For ease of reading, we accompany the $A_{\infty}$-relations of this Section with pictures of the corresponding natural bubbling phenomena arising in the context of Fukaya categories. For a thorough treatment of the $A_{\infty}$-module theory we refer to Seidel \cite{Seidel-Subalgebras}, Abouzaid \cite{Abouzaid}, and Ganatra \cite{Ganatra}.
\\
{\bf Remark.} \emph{ We restrict ourselves to non-curved $A_{\infty}$-categories.}%
%
\subsection{$A_{\infty}$-categories, and grading conventions}
\label{Subsection Grading conventions}

Recall that an $A_{\infty}$-category $\bA$ is: a collection of objects $\mathrm{Ob}(\bA)$; graded vector spaces $\mathrm{hom}_{\bA}(X_0,X_1)$ for all $X_0,X_1\in \mathrm{Ob}(\bA)$, working over a chosen ground field; composition maps 
$$
\mu_{\bA}^r: \bA(X_r,\ldots,X_0) \to \bA(X_r,X_0),
$$
where we abbreviate
$$
\bA(X_r,X_{r-1},\ldots,X_0)=\hom_{\bA}(X_{r-1},X_r)\otimes \hom_{\bA}(X_{r-2},X_{r-1}) \otimes \cdots \otimes \hom_{\bA}(X_0,X_1)
$$
for $r\geq 1$ (for $r=0$ define it to equal the ground field; and we will typically denote generators by $x_r \otimes \cdots \otimes x_1$); 
and finally the $\mu_{\bA}^r$ must satisfy the following $A_{\infty}$-relations.
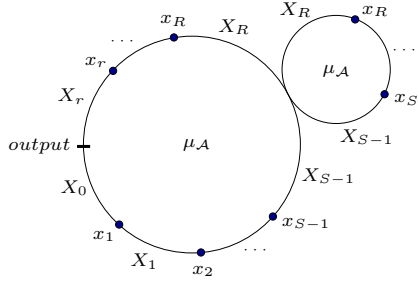
\begin{center}\input{Ainfinity.tex}\end{center}
%
%
$$
\sum (-1)^{\sigma_1^{S-1}}\mu_{\bA}^r(x_r \otimes \cdots \otimes x_{R+1} \otimes \mu_{\bA}^{R-S+1}(x_R,\ldots,x_S),x_{S-1},\ldots,x_1)=0 
$$
summing over all obvious choices of $R,S$, and we explain the sign and the gradings below. 
The $A_{\infty}$-relations imply that $\mu_{\bA}^1$ is a differential and the multiplication $\mu_{\bA}^2$ is a chain map.
%
\indent Let $|x|$ be the grading of a morphism in $\bA$, and define
 $\|x\|=|x|-1$ the reduced grading. Then grading signs are determined by the convention that the compositions $\mu_{\bA}^r$ are acting from the right with degree $1$. Abbreviate
$$
\sigma_i^j = \sigma(x)_i^j = \sum_{\ell=i}^j \|x_{\ell}\|,
$$
which is the reduced grading of $x_{j} \otimes x_{j-1} \otimes \cdots \otimes x_i$. The second expression emphasizes the letter $x$ which is used to denote the variables indexed by $i,i+1,\ldots, j$.
%
\subsection{Bimodules}
\label{Subsection Bimodules}
An $\bA$-bimodule $\bM$ is a collection of graded vector spaces $\bM(X,Y)$ for any objects $X,Y$ of $\bA$, together with multi-linear maps
%
%
$$
\mu_{\bM}^{r|s}: \bA(X_r,\ldots,X_0)\otimes \bM(X_0,Y_0) \otimes \bA(Y_0,\ldots,Y_s) \to \bM(X_r,Y_s)
$$
for $r,s\geq 0$, and any objects $X_i,Y_j$ of $\bA$, satisfying the following $A_{\infty}$-relations.
%
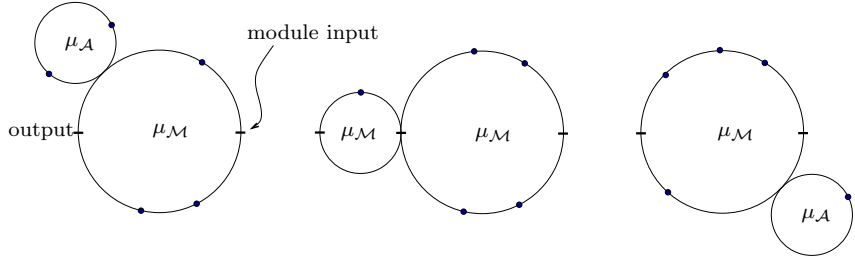
\begin{center}\input{bimodule.tex}\end{center}
%
%
$$
\begin{array}{ll} 
0= \sum (-1)^{\star} \mu_{\bM}^{r-R+S|s}(x_r, \ldots, x_{R+1}, {\scriptstyle\mu_{\bA}^{R-S+1}(x_R , \ldots , x_{S})} , x_{S-1} , \ldots , x_1, \underline{m} , y_1 , \ldots , y_s)+ \\
\;\;\,\, +
\sum (-1)^{\diamond} \mu_{\bM}^{r-R|s-S}(x_r , \ldots , x_{R+1} , {\scriptstyle\underline{\mu_{\bM}^{R|S}(x_{R},\ldots,x_1,\underline{m},y_1,\ldots,y_S)}} , y_{S+1} , \ldots , y_s) + \\
\;\;\,\, +
\sum (-1)^{\diamond} \mu_{\bM}^{r|s-R+S}(x_r , \ldots , x_1 , \underline{m} , y_1 , \ldots , y_{R-1} , {\scriptstyle\mu_{\bA}^{R-S+1}(y_R , \ldots , y_{S})} , y_{S+1} , \ldots , y_s)
\end{array}
$$
summing over all obvious choices of $R,S$, and where the signs are
\begin{equation}\label{Eqn signs diamond and star}
\diamond = \sigma(y)_{S+1}^{s} \qquad \qquad \star = \sigma(y)_1^s + \mathrm{deg}(\underline{m}) + \sigma(x)_1^{S-1}.
\end{equation}
Note: generators of the domain of $\mu_{\bM}^{r|s}$ are written $x_r\otimes \cdots \otimes x_1 \otimes \underline{m} \otimes y_1 \otimes \cdots \otimes y_s$, where we underline the bimodule element.

The above equations imply in particular that $\mu_{\bM}^{0|0}$ is a differential, and that the left multiplication $\mu_{\bM}^{1|0}$ and right multiplication $\mu_{\bM}^{0|1}$ are chain maps.

\begin{example*}
The diagonal bimodule $\bM=\bA$ has $\bM(X_0,Y_0)=\hom_{\bA}(Y_0,X_0)=\bA(X_0,Y_0)$ and composition maps $\mu_{\bM}^{r|s}(x_r,\ldots,x_1,\underline{m},y_1,\ldots,y_s)=(-1)^{1+\sigma(y)_1^s}\mu_{\bA}^{r+1+s}(x_r,\ldots,x_1,\underline{m},y_1,\ldots,y_s)$.
\end{example*}
%
%
\subsection{Convenient abbreviations: $\circ$ and $\Mu$}
\label{Subsection Conventient abbreviations: the symbols circ and Mu}
%
The abbreviation $A\circ B$ will mean: compose in all possible ways, respecting the Koszul sign convention (operators act from the right)
\begin{equation}\label{Eq Circ 1} A\circ B (x_n,\ldots,x_1)=\sum (-1)^{\| B\| \sigma_1^{s-1}} A^{n-r+s}(x_n,\ldots,x_{r+1},B^{r-s+1}(x_r,\ldots,x_s),x_{s-1},\ldots,x_1),
\end{equation}
where $\sigma_1^{s-1}$ uses reduced degrees unless one of the $x_j$ is a module input (then the unreduced $\mathrm{deg}(\underline{x}_j)$ is used) and $\|B\|$ is the degree of $B$ (to be specified). In particular, the sum includes insertions of $B^0$ with no inputs, which means $B^0:\K \to \bA(X,X)$ evaluated at $1\in \K$.

In the context of operators which by definition must receive a module input, it is understood that $A\circ B (x_r,\ldots,x_1,\underline{m},y_1,\ldots,y_s)$ means
\begin{equation}\label{Eq Circ 2}\sum (-1)^{\| B\| \sigma_{S+1}^s} A^{r-R|s-S}(x_r,\ldots,x_{R+1},\underline{B^{R|S}(x_R,\ldots,x_1,\underline{m},y_1,\ldots,y_S)},x_{S+1},\ldots,y_s).
\end{equation}
In particular, in this case the lowest order insertion of $B$ is $B^{0|0}(\underline{m})$ which receives one input.

The operator $B=\Mu$ (with $\|B\|=1$) will mean applying both $\mu^{r|s}_{\bM}$ and $\mu^r_{\bA}$ as appropriate: that is depending on whether or not a module input is present in the sub-word that $B$ receives.\\[1mm]
\begin{tabular}{lll} {\bf Example.} &\emph{The $A_{\infty}$-relations for an $A_{\infty}$-category $\bA$ are:} &
$
\mu_{\bA} \circ \mu_{\bA}=0.
$
\\ & \emph{The $A_{\infty}$-relations for an $\bA$-bimodule $\bM$ are:} &
$
\mu_{\bM} \circ \Mu = 0. \quad (\textrm{or\,: }\Mu \circ \Mu =0.)
$
\end{tabular}
%
%
\subsection{Pre-morphisms of bimodules $\mathbf{hom}(\bM,\bN)$}
\label{Subsection Bimodules pre-morphism}

A pre-morphism $f: \bM \to \bN$ of $\bA$-bimodules  is a collection of multi-linear maps (for $r,s\geq 0$):
$$
f^{r|s}: \bA(X_r,\ldots,X_0)\otimes \bM(X_0,Y_0) \otimes \bA(Y_0,\ldots,Y_s) \to \bN(X_r,Y_s).
$$
Denote $\mathrm{hom}(\bM,\bN)$ the collection of these; $\mathrm{end}(\bM)=\mathrm{hom}(\bM,\bM)$ the pre-endomorphisms.
%
%
\subsection{Yoneda composition product of pre-morphisms}
\label{Subsection Yoneda composition of pre-morphism}

The Yoneda product is the composition of pre-morphisms using \eqref{Eq Circ 2}, $$\circ: \mathrm{hom}(\bM_2,\bM_3)\otimes \mathrm{hom}(\bM_1,\bM_2) \to \mathrm{hom}(\bM_1,\bM_3).$$ 
\subsection{The unit pre-morphism}
\label{Subsection Unit pre-morphism}

The unit $1_{\bM}\in \mathrm{end}(\bM)$ for the ring $(\mathrm{end}(\bM),\circ)$ is
$$
1_{\bM}^{0|0}=\mathrm{identity}: \bM(X_0,Y_0)\to \bM(X_0,Y_0), \quad 1_{\bM}^{r|s}=0 \textrm{ if }r\geq 1\textrm{ or }s\geq 1.
$$
%
%
\subsection{Differential of a pre-morphism}
\label{Subsection differential of premorph}

For $f\in \mathrm{hom}(\bM,\bN)$, define
$$
\delta f = \mu_{\bN}\circ f - (-1)^{\|\Mu\|\, \| f\|}f \circ \Mu
$$
using the conventions in \ref{Subsection Conventient abbreviations: the symbols circ and Mu} where $\|\Mu\|=1$, $\|f\|=\mathrm{deg}(f)$. One can now verify that the category of bimodules is a dg-category.
%
\subsection{Morphisms of bimodules}
\label{Subsection Bimodules morphism}
A pre-morphism $f: \bM \to \bN$ is a morphism if $\delta f=0$, equivalently $f\in \ker \delta$.
Explicitly the $A_{\infty}$-relations $\delta f=0$ are written below.
%
\begin{center}\input{bimodulemorphism.tex}\end{center}
%
%
%
$$
\begin{array}{ll}
 \sum (-1)^{\star} f^{r-R+S|s}(x_r, \ldots , x_{R+1} ,  {\scriptstyle \mu_{\bA}^{R-S+1}(x_R, \ldots, x_S)}, x_{S-1} , \ldots x_1 , \underline{m} , y_1 , \ldots , y_s) +\\[1mm]
\,\, +\sum (-1)^{\diamond} f^{r-R|s-S}(x_r, \ldots , x_{R+1} ,  {\scriptstyle\underline{\mu_{\bM}^{R|S}(x_R,\ldots,x_1,\underline{m},y_1,\ldots,y_S)}} , y_{S+1} , \ldots , y_s)+\\[1mm]
\,\, +\sum (-1)^{\diamond} f^{r|s-S+R}(x_r, \ldots , x_1 , \underline{m} , y_1 , \ldots , y_{R-1} ,  {\scriptstyle\mu_{\bA}^{S-R+1}(y_R, \ldots, y_S)}, y_{S+1} , \ldots , y_s) \\[1mm]
 =(-1)^{\mathrm{deg}(f)}\sum (-1)^{\diamond\,\mathrm{deg}(f)} \mu_{\bN}^{r-R|s-S}(x_r, \ldots, x_{R+1},  {\scriptstyle\underline{f^{R|S}(x_R,\ldots,x_1,\underline{m},y_1,\ldots,y_S)}},y_{S+1},\ldots,y_s)
\end{array}
$$
summing over the obvious choices of $R,S$, and where the signs are as in \eqref{Eqn signs diamond and star}.
\begin{example*}
 For $\bM=\bN$, the unit pre-morphism $1_{\bM}$ from \ref{Subsection Yoneda composition of pre-morphism} satisfies $\delta 1_{\bM}=0$.
\end{example*}
The vector space of morphisms, modulo boundaries of pre-morphisms, will be denoted
$$
\mathrm{Hom}(\bM,\bN)=\ker \delta / \mathrm{im}\, \delta \qquad \qquad \mathrm{End}(\bM)=\mathrm{Hom}(\bM,\bM).
$$
%
%
\subsection{Shift of bimodules}
\label{Subsection Bimodules shift}
The shift operation $\bM[1]$ on a bimodule $\bM$ is defined by shifting the degrees down by $1$, so $\mathrm{deg}_{\bM[1]}(\underline{m})=\mathrm{deg}_{\bM}(\underline{m})-1$, but also changing signs: $$\mu_{\bM[1]}^{r|s}(x_r,\ldots,x_1,\underline{m},y_1,\ldots,y_s) = (-1)^{1+\sigma(y)_{1}^s}\mu_{\bM}^{r|s}(x_r,\ldots,x_1,\underline{m},y_1,\ldots,y_s)$$
That sign ensures that $\mathrm{Hom}^{n}(\bM,\bN[1])=\mathrm{Hom}^{n+1}(\bM,\bN)$ are identifiable chain complexes. On the other hand, for $\mathrm{Hom}^{n}(\bM[1],\bN)\cong \mathrm{Hom}^{n-1}(\bM,\bN), f \mapsto f[1]$ we require a sign change: $f[1]^{r|s}(x_r,\ldots,x_1,\underline{m},y_1,\ldots,y_s)=(-1)^{\sigma(y)_1^s}f^{r|s}(x_r,\ldots,x_1,\underline{m},y_1,\ldots,y_s)$.
%
%
\begin{example*}
 $\bM=\bA[1]$ has composition maps $\mu_{\bA[1]}^{r|s}=\mu_{\bA}^{r+1+s}$.
\end{example*}
%
%
\subsection{Cohomology, products and units}
\label{Subsection cohomological unit}

For an $A_{\infty}$-category $\bA$, the cohomology category $H^*(\bA)$ has morphism spaces equal to the cohomology of the chain complex $(\bA(X_0,X_1),\mu_{\bA}^{1})$ with (associative) composition maps equal to the $\mu^2$-product (up to sign): $$[x]\cdot [y]=(-1)^{|y|}\mu_{\bA}^2(x,y)=(-1)^{\|y\|+1}\mu_{\bA}^2(x,y).$$ 

The cohomology $H^*(\bM)$ of a bimodule $\bM$ is the cohomology of the chain complexes $(\bM(X,Y),\mu_{\bM}^{0|0})$.
A morphism is a quasi-isomorphism if the map on cohomology is an isomorphism. A feature of the $A_{\infty}$-theory is that quasi-isomorphisms are always invertible up to homotopy (this applies to $A_{\infty}$-categories and to bimodules).

An $A_{\infty}$-category $\bA$ is cohomologically unital (c-unital) if every object $X\in \mathrm{Ob}(\bA)$ has an endomorphism $e_X \in \hom^0_{\bA}(X,X)$ which is a unit for $H^*(\bA)$ with the above product (note the degree $|e_X|=0$, so $\|e_X\|=-1$). These are called cohomological units. This is weaker than strict-units $e \in \hom^0_{\bA}(X,X)$, which would require $\mu_{\bA}^1(e)=0$, $(-1)^{\|y\|+1}\mu_{\bA}^2(e,y)=y=\mu_{\bA}^2(y,e)$ for all $y$, and higher $\mu_{\bA}^k=0$ if we insert $e$ in any slot. Up to quasi-isomorphism, one can replace a cohomologically unital $A_{\infty}$-category by a strictly unital one and only consider strictly unital functors and strictly unital $A_{\infty}$-transformations.
A bimodule $\bM$ is cohomologically unital if $(-1)^{\mathrm{deg}(\cdot)+1}\mu_{\bM}^{1|0}(e_X,\cdot)$ and $\mu_{\bM}^{0|1}(\cdot,e_X)$ are identity maps on cohomology. This is weaker than strict-unitality, which requires
 $(-1)^{\mathrm{deg}(\underline{m})+1}\mu_{\bM}^{1|0}(e,\underline{m})=\underline{m}=\mu_{\bM}^{0|1}(\underline{m},e)$ and higher $\mu_{\bM}^{r|s}=0$ if we insert $e$ in any slot. 
For strictly unital $A_{\infty}$-algebras, the dg category of strictly unital bimodules is quasi-equivalent to that of all bimodules.
%
\subsection{Hochschild homology}
\label{Subsection cyclic bar complex}
%
Let $\bM$ be an $\bA$-bimodule. Its cyclic bar complex in dimension $n\geq 0$ is the vector space
$$
\mathrm{CC}_n(\bA,\bM) = \bigoplus \bM(X_0,X_n) \otimes {\bA}(X_n,\ldots,X_0)
$$
summing over all $X_j\in \mathrm{Ob}(\bA)$. A typical generator will be denoted $\underline{m}\otimes x_{n} \otimes \cdots \otimes x_1$. We underline the element belonging to the bimodule. 
The grading for $\mathrm{CC}_n(\bA,\bM)$ is: 
$$
\|\underline{m} \otimes x_{n} \otimes \cdots \otimes x_1\| = \mathrm{deg}(\underline{m})+\sigma_1^{n}.
$$
The bar differential $b: \mathrm{CC}_n(\bA,\bM) \to \mathrm{CC}_{n-1}(\bA,\bM)$ on a generator $\underline{m}\otimes x_{n}\otimes \cdots \otimes x_1$ is
%
%
$$
\begin{array}{ll}
\sum (-1)^{\sigma_1^{s-1}} \underline{m} \otimes x_{n} \otimes
\cdots \otimes x_{r+1} \otimes \mu_{\bA}^{r-s+1}(x_r,\ldots,x_s) \otimes x_{s-1} \otimes \cdots \otimes x_1\; +
\\
\sum (-1)^{\dagger} \underline{\mu_{\bM}^{s|n-r+1}(x_{s}, \ldots, x_1, \underline{m}, x_{n}, \ldots,x_{r})} \otimes x_{r-1} \otimes \cdots \otimes x_{s+1}
\end{array}
$$
summing over all obvious choices of $r,s$, where $\dagger=\sigma_1^{s}(\mathrm{deg}(\underline{m})+\sigma_{s+1}^{n})+\sigma_{s+1}^{r-1}$.

The Hochschild homology is the resulting homology:
$$
\mathrm{HH}_n(\bA,\bM) = H_n(\mathrm{CC}_*(\bA,\bM);b).
$$
\begin{example*}
 $\mathrm{HH}_n(\bA) = H_n(\mathrm{CC}_*(\bA,\bA);b)$ is the Hochschild homology of the diagonal (see \ref{Subsection Bimodules}).
\end{example*}
%
\subsection{Morphisms of bimodules induce chain maps on $\mathbf{CC}_*$}
\label{Subsection Morphisms of bimodules induce chain maps on CC}
\begin{lemma}\label{Lemma functoriality of HH}
 A morphism of $\bA$-bimodules $f: \bM \to \bN$ induces a chain map
$$
\mathrm{CC}_*(f): \mathrm{CC}_*(\bA,\bM) \to \mathrm{CC}_*(\bA,\bN).
$$
\end{lemma}
\begin{proof}
 Define 
$$\mathrm{CC}_n(f): \bM(X_0,X_n)\otimes \bA(X_n,\ldots,X_0) \to \mathrm{CC}_*(\bA,\bN)
$$
by sending $\underline{m}\otimes x_{n} \otimes \cdots \otimes x_1$ to 
$$\textstyle
\sum (-1)^{\diamond} \underline{f^{r|n-s+1}(x_{r}, \cdots, x_1, \underline{m}, x_{n}, \cdots , x_s)} \otimes  x_{s-1} \otimes \cdots \otimes x_{r+1}
$$
where $\diamond = \sigma_1^{r} (\deg(\underline{m})+\sigma_{r+1}^{n}) + \deg(f)\sigma_{r+1}^{s-1}$. It is a chain map by Lemma \ref{Lemma functoriality of HH image of exact elements}(1).
\end{proof}
\begin{lemma}\label{Lemma functoriality of HH image of exact elements}
\strut \begin{enumerate}
\item
$
\mathrm{CC}_*(\delta f) = b\circ \mathrm{CC}_*(f)- (-1)^{\mathrm{deg}(f)} \mathrm{CC}_*(f) \circ b.
$
\item $ \mathrm{CC}_*(f\circ g) =  \mathrm{CC}_*(f) \circ  \mathrm{CC}_*(g)$.
\end{enumerate}
\end{lemma}
\begin{proof}
$
\mathrm{CC}_*(\delta f) = 
\mathrm{CC}_*(\mu_{\bN}\circ f) - (-1)^{\| f\|} \mathrm{CC}_*(f \circ \Mu)=
b\circ \mathrm{CC}_*(f)- (-1)^{\| f\|} \mathrm{CC}_*(f) \circ b.
$ The second claim follows immediately from definitions.
\end{proof}
\begin{corollary}\label{Corollary HHomology module over bimodule endos}
The following is a unital algebra homomorphism, using composition products,
\begin{equation}\label{Eqn End M to End HHM}
\mathrm{End}(\bM) \to \mathrm{End}(\mathrm{HH}_*(\bM)), \; f\mapsto \mathrm{HH}_*(f).
\end{equation}
Thus $\mathrm{HH}_*(\bM)$ is a module over the unital ring of bimodule endomorphisms.
\end{corollary}
\begin{proof}
By Lemma \ref{Lemma functoriality of HH image of exact elements}, if $\delta f=0$ then $\mathrm{CC}_*(f)$ is a chain map (up to the sign $(-1)^{\mathrm{deg}(f)}$), so $\mathrm{HH}_*(f)\in \mathrm{End}(\mathrm{HH}_*(\bM))$ is well-defined for $f\in \ker \delta$. The map \eqref{Eqn End M to End HHM} is well-defined by Lemma \ref{Lemma functoriality of HH image of exact elements}(1): $\mathrm{HH}_*(\delta f)=0 \in \mathrm{End}(\mathrm{HH}_*(\bM))$.
Moreover, \eqref{Eqn End M to End HHM} preserves products by Lemma \ref{Lemma functoriality of HH image of exact elements}(2). Unitality holds already at the chain level: $1_{\bM}\mapsto \mathrm{CC}_*(1_{\bM})=\mathrm{id}_{\mathrm{CC}_*(\bM)}$.
\end{proof}
%

\begin{lemma}\label{Lemma trick if bimodule map has zero 0part}
If $\psi:\bM\to \bM$ is a bimodule endomorphism with $\psi^{0|0}=0$ (at the chain level), then $\mathrm{HH}_*(\psi)$ acts nilpotently on $\mathrm{HH}_*(\bM)$ $($meaning, given $c\in \mathrm{HH}_*(\bM)$, there is an integer $N$ possibly depending on $c$ such that $\mathrm{HH}_*(\psi)^N(c)=0)$.
\end{lemma}
\begin{proof}
An element of type $\underline{m}\otimes x_n \otimes \cdots \otimes x_1$ is called a length $n$ word. We claim that $\psi^N=\psi\circ \cdots \circ \psi$ (composing $N$ times at the chain level) must vanish on a word of length less than $N$. Indeed, consider a word $x_r \otimes \cdots x_1 \otimes \underline{m} \otimes y_1 \otimes \cdots \otimes y_s$ with $r+s<N$. If one of the $\psi$ factors in $\psi^N$ only receives a module input, then we get zero since $\psi^{0|0}=0$ by assumption. If all $\psi$ factors receive more than just a module input, then by definition of composition each $\psi$ factor eats at least one non-module input from the $r+s$ available inputs $x_r,\ldots,x_1,y_1,\ldots,y_s$. But by the pigeonhole principle, there are not enough available inputs since $N>r+s$.

Given a cycle $c\in \mathrm{CC}_*(\bM)$, let $N$ be larger than the length of any word occuring in the finite linear combination comprising $c$ to deduce 
$\mathrm{CC}_*(\psi^N)(c)=0$ at the chain level. Thus $\mathrm{HH}_*(\psi)^N[c]=\mathrm{HH}_*(\psi^N)[c]=0$ by Corollary \ref{Corollary HHomology module over bimodule endos}.
\end{proof}
%
\subsection{Hochschild cohomology}
\label{Subsection Hochschild cohomology}
The Hochschild cochain complex is
$$
\begin{array}{rcl}
\mathrm{CC}^d(\bA,\bM) &=& \displaystyle\prod_{n\geq 0}\prod \mathrm{Hom}({\bA}(X_n,\ldots,X_0),{\bM}(X_n,X_0)[d])\\[2mm]
 &=& \displaystyle\prod_{n\geq 0}\prod \mathrm{Hom}^d({\bA}(X_n,\ldots,X_0),{\bM}(X_n,X_0))
\end{array}
$$
where the unlabelled product is over all $X_j\in \mathrm{Ob}(\bA)$, and we used reduced gradings for $\bA$ but not for $\bM$ (for non-reduced gradings the shift is $[d-n]$ instead of $[d]$).

We write $\varphi=(\varphi^n)_{n\geq 0}$ for a typical element of $\mathrm{CC}^d(\bA,\bM)$, so $\varphi^n(x_n,\ldots,x_1)\in \bM(X_n,X_0)$ for $x_j\in \hom_{\bA}(X_{j-1},X_j)$. Using the notation \ref{Subsection Yoneda composition of pre-morphism}, the differential is defined by
$$
\delta \varphi = \mu_{\bM}\circ \varphi - (-1)^{\|\mu_{\bA}\|\, \| \varphi\|}\varphi \circ \mu_{\bA} 
$$
where $\|\mu_{\bA}\|=1$, $\| \varphi\|=d$. Explicitly $\delta\varphi(x_n,\ldots,x_1)\in \bM(X_n,X_0)$ is as follows.
%
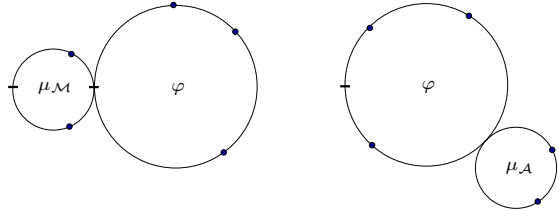
\begin{center}\input{Hcohomology.tex}\end{center}
%
%
%
%
$$\begin{array}{l}
\sum (-1)^{d\, \sigma_{1}^{s-1}}\mu_{\bM}^{n-r|s-1}(x_n,\ldots,x_{r+1},\underline{\varphi^{r-s+1}(x_r,\ldots,x_s)},x_{s-1},\ldots,x_1)-\\
\qquad -(-1)^{d}\sum (-1)^{\|\mu\|\cdot \sigma_1^{s-1}}\varphi^{n-r+s}(x_n,\ldots,x_{r+1},\mu_{\bA}^{r-s+1}(x_r,\ldots,x_s),x_{s-1},\ldots,x_1).
\end{array}
$$
\noindent {\bf Remark.} \emph{When $\bM=\bN=\bA$, replacing all $\mu$ above by $\psi\in \mathrm{CC}^*(\bA,\bA)$ defines the \emph{Gerstenhaber bracket} operation $[\varphi,\psi]$ on $\mathrm{CC}^*(\bA,\bA)$. More generally in that notation, $\delta=[\mu_{\bA},\cdot],$ and the $A_{\infty}$-relations $\mu_{\bA}\circ \mu_{\bA}=0$ correspond to $\delta\circ \delta=0$ via the Jacobi identity. Similarly, $\delta \circ \delta =0$ in general follows from the bimodule $A_{\infty}$-relations.}\\[1mm]
The Hochschild cohomology is the resulting cohomology:
$$
\mathrm{HH}^n(\bA,\bM) = H^n(\mathrm{CC}^*(\bA,\bM);\delta).
$$
\begin{example*}
 $\mathrm{HH}^n(\bA) = H^n(\mathrm{CC}^*(\bA,\bA);\delta)$ is the Hochschild cohomology of the diagonal.
\end{example*}
\begin{example*}
For cocycles $[\varphi]\in \mathrm{HH}^*(\bA,\bM)$, the $0$-part $\varphi^0:\K \to \bM(X,X)$ has $0=(\delta \varphi)^0(1)=\mu_{\bM}^{0|0}(\underline{\varphi^0(1)})$. So $\varphi^0(1)\in H^*(\bM(X,X);\mu_{\bM}^{0|0})$. So for $\bM=\bA$, $\varphi^0(1)\in H^*(\bA(X,X);\mu_{\bA}^1)$. 
\end{example*}
\subsection{Cup product on Hochschild cohomology}
\label{Subsection Product on Hochschild cohomology}
The cup product on $\mathrm{CC}^*(\bA,\bA)$, denoted $\varphi * \psi$, is defined as follows on $(x_n,\ldots,x_1)\in \bA(X_n,\ldots,X_0)$:
$$
\sum (-1)^{\sharp} \mu_{\bA}^{\ddag}(x_n,\ldots,x_{R+1},\varphi^{R-S+1}(x_{R},\ldots,x_S),x_{S-1},\ldots,x_{r+1},\psi^{r-s+1}(x_{r},\ldots,x_s),x_{s-1},\ldots,x_1)
$$
where $\sharp = \|\varphi\|\, \sigma_{1}^{S-1}+\|\psi\|\,\sigma_{1}^{s-1}$ (and of course $\ddag=n-R+S-r+s$).
One can check this descends to a product on the cohomology $\mathrm{HH}^*(\bA)$.

\noindent {\bf Remark.} \emph{Composition of maps, as in \ref{Subsection Yoneda composition of pre-morphism}, would not give a product on $\mathrm{HH}^*(\bA)$, indeed the differential of a composition is the Gerstenhaber bracket. The bracket terms do not occur in \ref{Subsection Yoneda composition of pre-morphism} because pre-morphisms of bimodules must receive module inputs.} 
\begin{example*}
The $0$-part $(\varphi * \psi)^0(1)\in \bA(X,X)$ equals $\mu_{\bA}^2(\varphi^0(1),\psi^0(1))$, the usual product. 
\end{example*}

\noindent We thank Nick Sheridan and the referee for explaining the following analogue of Lemma \ref{Lemma trick if bimodule map has zero 0part}.

\begin{lemma}\label{Lemma trick CO}
If $\varphi\in \mathrm{HH}^*(\bA)$ satisfies $\varphi^0=[e_X]\in \mathrm{Hom}(X,X)$ for all objects $X$, where $[e_X]$ is the cohomological unit, then cup product with $\varphi$ defines an isomorphism on $\mathrm{HH}^*(\mathcal{A})$.
\end{lemma}
\begin{proof}
Consider the length filtration on the Hochschild cochain complex
$$
F^{\geq p}\mathrm{CC}^* = \textrm{Hochschild cochains in } \mathrm{CC}^*(\mathcal{A})\textrm{ with at least }p\textrm{ inputs}.
$$
The cup product respects the filtration:
\begin{equation}\label{Eqn cup product on filtration}
*: F^{\geq p}\mathrm{CC}^* \otimes F^{\geq q}\mathrm{CC}^* \to F^{\geq p+q}\mathrm{CC}^*.
\end{equation}
The $E_1$-page for the filtration is the Hochschild cohomology of the cohomological category $H^*(\mathcal{A})$, which is just an ordinary linear category (compare Equation (1.13) in Seidel \cite{Seidel}).
The length filtration is \emph{exhaustive} (i.e. $\mathrm{CC}^*=\cup F^{\geq p}\mathrm{CC}^*$) and \emph{complete} (i.e. $\mathrm{CC}^*=\varprojlim F^{\geq p}\mathrm{CC}^*$ taken over the reverse inclusions $F^{\geq p}\mathrm{CC}^* \supset F^{\geq p+1}\mathrm{CC}^*$). So the Eilenberg-Moore comparison theorem \cite[Theorem 5.5.11]{Weibel} applies to the map of filtered complexes
$f=\varphi * \cdot : \mathrm{CC}^*(\mathcal{A}) \to \mathrm{CC}^*(\mathcal{A})$ (the map respects the filtration due to \eqref{Eqn cup product on filtration} taking $p=0$ and fixing the first input $\varphi\in F^{\geq 0}\mathrm{CC}^*$). 
Notice that on the $E_1$ page the map $f$ induces the map $f^1: H^*(\mathcal{A}) \to H^*(\mathcal{A})$ given by the identity, so it is an isomorphism. By the comparison theorem, the map $f$ on cohomology, namely $\varphi * \cdot: \mathrm{HH}^*(\mathcal{A}) \to \mathrm{HH}^*(\mathcal{A})$, must also be an isomorphism (this version of the comparison theorem does not require the spectral sequence to converge).
\end{proof}

\subsection{The unit in Hochschild cohomology}
\label{Subsection Unit in Hochschild cohomology}
The analogue of the unit map in \ref{Subsection Unit pre-morphism} would not work for the cup product. If the $A_{\infty}$-category $\bA$ has \emph{strict} units (see \ref{Subsection cohomological unit}), then define
$$1_{\mathrm{CC}^*}^0: \K \to \bA(X_0,X_0),\; 1_{\mathrm{CC}^*}^0(1)=1_{X_0}, \qquad
1_{\mathrm{CC}^*}^r=0 \textrm{ for }r>0.$$
Strict unitality then implies that $1_{\mathrm{CC}^*}$ is a unit for $(\mathrm{HH}^*(\bA),*)$.
%
%
\subsection{Hochschild cohomology as endomorphisms of the diagonal bimodule}
\label{Subsection Hochschild cohomology as endomorphisms of the diagonal bimodule}
\indent There is a chain map from the Hochschild complex to pre-morphisms
 $$\mu_{\bM}\circ_{\mathrm{left}}:\mathrm{CC}^*(\bA, \bM) \to \mathrm{hom}(\bA,\bM),\; \varphi\mapsto \mu_{\bM}\circ_{\mathrm{left}}\varphi.$$ 
  Explicitly, $\mu_{\bM}\circ_{\mathrm{left}}\varphi (x_r,\ldots,x_1,\underline{m},y_1,\ldots,y_s)$ equals the following:
$$
\sum (-1)^{\|\varphi\|\star} \mu_{\bM}^{r-R+S|s}(x_r,\ldots,x_{R-1},\varphi^{R-S+1}(x_R,\ldots,x_S),x_{S-1},\ldots, x_1,\underline{m},y_1,\ldots,y_s)).
$$
where $\star = \sigma(y)_1^s + \mathrm{deg}(\underline{m}) + \sigma(x)_1^{S-1}$.

\begin{lemma}\label{Lemma algebra hom from CC to hom}
The map $\mu_{\bM}\circ_{\mathrm{left}}$ is a chain map, inducing $$\mathrm{HH}^*(\bA,\bM)\to \mathrm{Hom}(\bA,\bM).$$ 
For the diagonal bimodule $\bM=\bA$, the map is the first term of an $A_{\infty}$-homomorphism,
 $$\mu_{\bM}\circ_{\mathrm{left}}:\mathrm{CC}^*(\bA) \to \mathrm{end}(\bA),\; \varphi \mapsto \mu_{\bM}\circ_{\mathrm{left}}\varphi,$$ 
(where the $A_{\infty}$-structure on $\mathrm{CC}^*(\bA)$ was constructed by Getzler \cite{Getzler}, and using the dg-algebra structure on $\mathrm{end}(\bA)$). This in turn induces an algebra homomorphism on cohomology
$$
\mu_{\bM}\circ_{\mathrm{left}}:\mathrm{HH}^*(\bA) \to \mathrm{End}(\bA).
$$
\end{lemma}
\begin{proof}[Sketch Proof]
This is a purely algebraic result. We will only summarize part of the argument, by showing the bubblings that arise in the context of Fukaya categories:
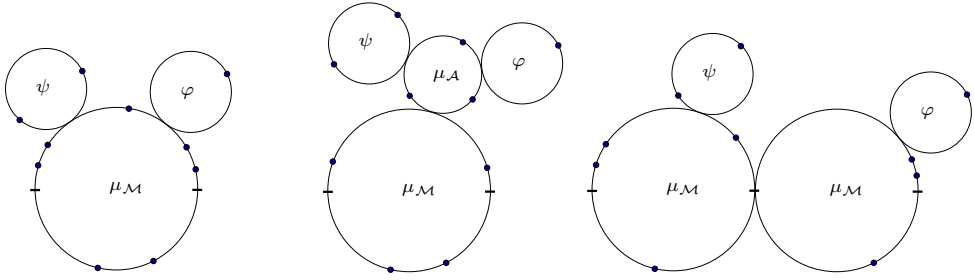
\begin{center}\input{CapYoneda.tex}\end{center}
Consider a $1$-family of Floer solutions defined on the domain in the first picture. Such a family can break as in the last two pictures (the other breakings involve terms coming from $\delta \varphi$, $\delta \psi$, $\mu_{\bA}$, $\mu_{\bM}$, but these can be ignored at the cohomology level). The three small bubbles in the middle picture are the cup product $\psi*\varphi$, so the middle picture is $\mu_{\bM}\circ_{\mathrm{left}}(\psi*\varphi)$, the picture on the right is the Yoneda composition $(\mu_{\bM}\circ_{\mathrm{left}}\psi)\circ (\mu_{\bM}\circ_{\mathrm{left}}\varphi)$.

For the second claim, we remark that the construction of the $A_{\infty}$-functor generalizes the formula for $F^1(\varphi)=\mu_{\bM}\circ_{\mathrm{left}}\varphi$ as follows: $F^n(\varphi_n,\ldots,\varphi_1)(x_r,\ldots,x_1,\underline{m},y_1,\ldots,y_s)$ equals
$$
\begin{array}{ll}
\sum (-1)^{\star\star}\!\!\!\! & \mu_{\bM}^{\dagger|s}(x_r,{\scriptstyle \ldots},x_{k_n+1},
\varphi_n^{k_n-\ell_n+1}(x_{k_n},{\scriptstyle \ldots},x_{\ell_n}),x_{\ell_n-1},{\scriptstyle \ldots},x_{k_2+1},\varphi_2^{k_2-\ell_2+1}(x_{k_2},{\scriptstyle \ldots},x_{\ell_2}),
\\
& \qquad x_{\ell_2-1},{\scriptstyle \ldots},x_{k_1+1},
\varphi_1^{k_1-\ell_1+1}(x_{k_1},{\scriptstyle \ldots},x_{\ell_1}),x_{\ell_1-1},{\scriptstyle \ldots},x_1,\underline{m},y_1,{\scriptstyle \ldots},y_s).
\end{array}
$$
where $\dagger= r-\sum k_j + \sum \ell_j$, and $\star\star = \sum \|\varphi_j\|(\sigma(y)_1^s + \mathrm{deg}(\underline{m}) + \sigma(x)_1^{\ell_j-1})$.
\end{proof}

\begin{lemma}\label{Lemma if c-unital then quasi-iso CC to hom}
If $\bA$ is cohomologically unital (see \ref{Subsection cohomological unit}), then $\mu_{\bM}\circ_{\mathrm{left}}$ is a quasi-isomorphism.
\end{lemma}

The result $\mathrm{HH}^*(\bA)\cong \mathrm{End}(\bA)$ is often stated in the literature; the only detailed proof in the literature that we are aware of is due to Ganatra \cite[(2.201)]{Ganatra} (Ganatra considers algebras, but the proof extends immediately to categories, and note that Ganatra places $\varphi$ on the right of the module input rather than the left).

\begin{corollary}\label{Corollary HCohom to Hhom}
If $\bA$ is c-unital,
 $\mathrm{HH}^*(\bA)\cong \mathrm{End}(\bA)$ is a unital algebra isomorphism.
\end{corollary}
\begin{proof}
This follows from the previous two Lemmas. In particular, a posteriori $(\mathrm{HH}^*(\bA),*)$ has a unit: the preimage of $1_{\bM}\in \mathrm{End}(\bA)$. If $\bA$ is strictly unital, one can see directly that $1_{\mathrm{CC}^*}\mapsto 1_{\bA}$ at the chain level (using the units from \ref{Subsection Unit pre-morphism} and \ref{Subsection Unit in Hochschild cohomology}).
\end{proof}

\noindent {\bf Remark.} \emph{There is an analogous isomorphism  $\mathrm{HH}_*(\bA,\bM)\cong H^*(\bM\otimes_{\bA^{\mathrm{opp}}\otimes \bA} \bA)$ of vector spaces (viewing $\bM$, $\bA$ respectively as right, left modules over $\bA^{\mathrm{opp}}\otimes \bA$, see \ref{Subsection Left modules and right modules}-\ref{Subsection tensor products}), in particular}
$
\mathrm{HH}_*(\bA) \cong H^*(\bA\otimes_{\bA^{\mathrm{opp}}\otimes \bA} \bA).
$
%
%
\subsection{Cap product action of $HH^*$ on $HH_*$}
\label{Subsection Cap product}
%
\strut

\begin{corollary}\label{Corollary from HCoh to HHom}
The map $\mu_{\bM}\circ_{\mathrm{left}}$ induces an algebra homomorphism
$$
\mathrm{HH}_*(\mu_{\bM}\circ_{\mathrm{left}}): \mathrm{HH}^*(\bA) \to \mathrm{End}(\bA) \to \mathrm{End}(\mathrm{HH}_*(\bA)),
$$
therefore $\mathrm{HH}_*(\bA)$ is an $\mathrm{HH}^*(\bA)$-module.
\end{corollary} 
\begin{proof}Combine Lemma \ref{Lemma algebra hom from CC to hom} (for the diagonal bimodule $\bM =\bA$) and Lemma \ref{Lemma functoriality of HH}.\end{proof}
More generally for any $\bA$-bimodule $\bM$, that composite 
$$
\mathrm{CC}_*(\mu_{\bM}\circ_{\mathrm{left}}):\mathrm{CC}^*(\bA,\bA)\to \mathrm{hom}(\bA,\bM) \to \mathrm{End}(\mathrm{CC}_*(\bA,\bM)),
$$
is the \emph{cap product}. Explicitly $\varphi\in \mathrm{CC}^*(\bA,\bA)$ acts on $c=\underline{m}\otimes x_{n}\otimes \cdots \otimes x_1 \in \mathrm{CC}_*(\bA,\bM)$ by:
$$
\sum (-1)^{\Box}{\mu_{\bM}^{r-R+S|n-s+1}(x_{r},{\scriptstyle\ldots},x_{R+1},{\scriptstyle\varphi^{R-S+1}(x_R,\ldots,x_S)},x_{S-1},{\scriptstyle\ldots},x_1,\underline{m},x_{n},{\scriptstyle\ldots}, x_s)}\otimes x_{s-1}\otimes {\scriptstyle\cdots} \otimes x_{r+1}
$$
where $\Box = 
\sigma_1^{r} (\deg(\underline{m})+\sigma_{r+1}^{n}) + 
\|\varphi\| (\sigma_{r+1}^{n}+\deg(\underline{m})+\sigma_{1}^{S-1})+
\sigma_{r+1}^{s-1}$.
One can check that cocycles $\varphi$ act as chain maps on $\mathrm{CC}_*(\bA,\bM)$.
%
\subsection{Functorial properties}
\label{Subsection Functorial properties}
%
The Hochschild homology $\mathrm{HH}_*(\bA,\bM)$ is covariantly functorial in $\bA$ and $\bM$. Indeed,
let $f:\bA \to \bB$ be an $A_{\infty}$-functor. Recall this means a map on objects $f: \mathrm{Ob}(\bA) \to \mathrm{Ob}(\bB)$ together with a 
grading-preserving multi-linear map
$$
f^n: {\bA}(X_n,\ldots,X_0) \to \bB(fX_n,fX_0)=\hom_{\bB}(f X_0, f X_n) 
$$
for each $n\geq 1$ (using reduced gradings), satisfying the $A_{\infty}$-relations:
$$
\begin{array}{l}
 0=\sum \mu_{\bB}^{m}(f^{k_{m}}(x_n, \ldots , x_{n-k_{m}+1}),\ldots,f^{k_1}(x_{k_1},\ldots,x_1)) - \\
\;\;\,\, - \sum (-1)^{\sigma_1^{s-1}} f^{n-r+s}(x_n, \ldots , x_{r+1} ,  \mu_{\bA}^{r-s+1}(x_r, \ldots, x_s) , x_{s-1} , \ldots x_1)
\end{array}
$$
summing over the obvious $r,s$, and over all partitions $k_1+\cdots+k_{m}=n$ and all $m$.

\begin{lemma}\label{Lemma change of rings}
 An $A_{\infty}$-functor $f: \bA \to \bB$ induces a \emph{change of rings} for bimodules:
$$
\bB\textrm{-Bimod} \to \bA\textrm{-Bimod}, \quad M \mapsto f^*M,
$$
which turns a $\bB$-bimodule $\bM$ into an $\bA$-bimodule $f^*\bM$. There is a tautological chain map
$$
\mathrm{taut}_f:\mathrm{CC}_*(\bA,f^*\bM) \to \mathrm{CC}_*(\bB,\bM).
$$
Moreover, $f$ induces a canonical morphism $\mathrm{diag}_f: \bA \to f^*\bB$ of $\bA$-bimodules (where $\bA$, $\bB$ are the diagonal bimodules), yielding a chain map
$$
\mathrm{taut}_f\circ
\mathrm{CC}_*(f_{\mathrm{diag}}):
\mathrm{CC}_*(\bA,\bA) \to \mathrm{CC}_*(\bA,f^*\bB) \to \mathrm{CC}_*(\bB,\bB).
$$
\end{lemma}
\begin{proof}
For any $X,Y \in \mathrm{Ob}(\bA)$, define
$f^*\bM(X,Y)=\bM(f X, f Y)$ and
let
$$\textstyle
\mu_{f^*\bM}^{r|s} = \sum \mu_{\bM}^{\tilde{r}|\tilde{s}}  \left( f^{k_{\tilde{r}}}\otimes \cdots \otimes f^{k_1} \otimes \underline{\mathrm{Id}} \otimes f^{\ell_{1}}\otimes \cdots \otimes f^{\ell_{\tilde{s}}} \right)
$$
summing over all partitions $k_1 + \cdots + k_{\tilde{r}}=r$ and $\ell_1 + \cdots + \ell_{\tilde{s}} = s$, and all $\tilde{r},\tilde{s}$, where $\underline{\mathrm{Id}}$ is the identity map on the module input.
%
The fact that $f^*\bM$ is an $A_{\infty}$-bimodule is a consequence of the $A_{\infty}$-relations satisfied by $f$. Define $\mathrm{taut}_f$ on $\mathrm{CC}_n(\bA,f^*\bM)$ by
$$\textstyle
\mathrm{taut}_f=\sum {\underline{\mathrm{Id}}} \otimes f^{\ell_1} \otimes \cdots \otimes f^{\ell_{\tilde{n}}}
$$
summing over all partitions $\ell_1+\cdots + \ell_{\tilde{n}}=n$ and all $\widetilde{n}$. The $A_{\infty}$-relations again imply that $\mathrm{taut}_f$ determines a chain map. The canonical bimodule map in the claim is:
$$\begin{array}{rcl}\mathrm{diag}_f : \bA(X_r,\ldots,X_0)\otimes \underline{\bA(X_0,Y_0)}\otimes \bA(Y_0,\ldots,Y_s)& \to & \underline{f^*\bB(X_r,Y_s)}= \underline{\bB(fX_r,fY_s)}\\[2mm]
\mathrm{diag}_f(x_r,\ldots,x_1,\underline{m},y_1,\ldots,y_s) &=&
f^{r+s+1}(x_r,\ldots,x_1,m,y_1,\ldots,y_s).
\end{array}
$$
So explicitly, the composite map $\mathrm{CC}_*(\bA,\bA) \to \mathrm{CC}_*(\bB,\bB)$ in the claim is:
$$
\textstyle
\sum (-1)^{\diamond} \underline{f^{r+n-s+1}(x_{r}, {\scriptstyle \ldots}, x_1, \underline{m}, x_{n}, {\scriptstyle \ldots} , x_s)} \otimes  f^{\ell_1}(x_{s-1},{\scriptstyle \ldots},x_{s-\ell_1}) \otimes \cdots \otimes f^{\ell_{\widetilde{s}}}(x_{r+\ell_{\widetilde{s}}},{\scriptstyle \ldots},x_{r+1})
$$
summing over $r,s$ and all partitions $\ell_1 + \cdots + \ell_{\tilde{s}} = s$,
where $\diamond=\sigma_1^r(\mathrm{deg}(\underline{m})+\sigma_{r+1}^n)$.
\end{proof}

The Hochschild cohomology $\mathrm{HH}^*(\bA,\bM)$ is covariantly functorial in $\bA$ and contravariantly functorial in $\bM$. However, unlike $\mathrm{HH}_*(\bA,\bA)$ above, the Hochschild cohomology $\mathrm{HH}^*(\bA,\bA)$ of the diagonal bimodule is not in general functorial in $\bA$ with respect to $A_{\infty}$-functors.

\begin{lemma}\label{Lemma HH under fully faithful functors}
The inclusion of a full subcategory $\bA \subset \bB$, induces a chain map
$$
\mathrm{restriction}:\mathrm{CC}^*(\bB,\bB) \to \mathrm{CC}^*(\bA,\bA). 
$$
\end{lemma}
\begin{proof}
For a cochain $\varphi\in \mathrm{CC}^*(\bB,\bB)$, we take the composite
$$
\bA(X_n,\ldots,X_0) = \bB(X_n,\ldots,X_0) \to \underline{\bB(X_n,X_0)} = \underline{\bA(X_n,X_0)}
$$
where the middle map is $\varphi^n$. 
\end{proof}

\begin{theorem}\label{Theorem Morita invariance}
If an $A_{\infty}$-functor $f: \bA \to \bB$ is a quasi-equivalence\footnote{meaning the underlying cohomology level functor is an equivalence. For example, a quasi-isomorphism.} then $\mathrm{HH}_*(\bA)\cong \mathrm{HH}_*(\bB)$ and $\mathrm{HH}^*(\bB)\cong \mathrm{HH}^*(\bA)$ for the diagonal bimodules.
\end{theorem}

The Theorem is essentially the $A_{\infty}$-analogue of Morita invariance (see Seidel \cite[Sec.(2d), Remark 2.7]{Seidel}), and we will not prove it. One approach to proving the Theorem, when $\bA$ is c-unital, is via Corollary \ref{Corollary HCohom to Hhom}: $\mathrm{HH}^*(\bA)\cong \mathrm{End}(\bA)$. Indeed, one checks that the map $\bB\textrm{-Bimod} \to \bA\textrm{-Bimod}$, induced by $f$ via Lemma \ref{Lemma change of rings}, is an equivalence of categories which sends the diagonal bimodule to the diagonal bimodule, and thus $\mathrm{End}(\bB)\cong \mathrm{End}(\bA)$. For $\mathrm{HH}_*$ one uses $\mathrm{HH}_*(\bA)\cong H^*(\bA\otimes_{\bA^{\mathrm{opp}}\otimes \bA}\bA)$ (see the Remark after Corollary \ref{Corollary HCohom to Hhom}).
%
%
%
\subsection{Left modules and right modules}
\label{Subsection Left modules and right modules}
We follow Seidel \cite[Sec.(1j)]{Seidel} and \cite[page 94]{Seidel-Subalgebras}.
Given two $A_{\infty}$-categories $\bA_1,\bA_2$, there is a generalization of the definition of $\bA$-bimodules $\mathcal{M}$ in \ref{Subsection Bimodules} to $(\bA_1,\bA_2)$-bimodules $\mathcal{M}$. Such a bimodule comprises a collection of graded vector spaces $\bM(X,Y)$ for any objects $X\in \bA_1$, $Y\in \bA_2$, together with multi-linear maps
$$
\mu_{\bM}^{r|s}: \bA_1(X_r,\ldots,X_0)\otimes \bM(X_0,Y_0) \otimes \bA_2(Y_0,\ldots,Y_s) \to \bM(X_r,Y_s)
$$
for $r,s\geq 0$, and any objects $X_i\in \bA_1,Y_j\in\bA_2$ satisfying the $A_{\infty}$-relations written in \ref{Subsection Bimodules} except we distinguish between using $\mu_{\bA_1}$ (on the $x$ elements) and $\mu_{\bA_2}$ (on the $y$ elements). 

Now consider the two $A_{\infty}$-categories $\mathcal{A}$, $\mathbb{K}$. Here the ground field $\mathbb{K}$ is viewed as an $A_{\infty}$-category with one object $X_{\mathbb{K}}$, with $\mathrm{hom}_{\mathbb{K}}(X_{\mathbb{K}},X_{\mathbb{K}})=\mathbb{K}\cdot \mathrm{id}_{X_{\mathbb{K}}}$ in degree $0$, and all $\mu^r_{\mathbb{K}}=0$.

A \emph{left $\mathcal{A}$-module} $\mathcal{L}$ is an $(\mathcal{A},\mathbb{K})$-bimodule $\mathcal{L}$ with $\mu_{\mathcal{L}}^{r|s}=0$ for $s\geq 1$. We abbreviate the graded vector spaces by $\mathcal{L}(X)=\mathcal{L}(X,X_{\mathbb{K}})$, and the multi-linear maps by
$$
\mu_{\mathcal{L}}^{r|}=\mu_{\mathcal{L}}^{r|0}: \mathcal{A}(X_r,\ldots,X_0) \otimes \mathcal{L}(X_0) \to \mathcal{L}(X_r).
$$
A \emph{right $\mathcal{A}$-module} $\mathcal{R}$ is a $(\mathbb{K},\mathcal{A})$-bimodule  $\mathcal{R}$ with $\mu_{\mathcal{R}}^{r|s}=0$ for $r\geq 1$. Let $\mathcal{R}(X)=\mathcal{R}(X_{\mathbb{K}},X)$,
$$
\mu_{\mathcal{R}}^{|s}=\mu_{\mathcal{R}}^{0|s}:\mathcal{R}(Y_0)\otimes \mathcal{A}(Y_0,\ldots,Y_s) \to \mathcal{R}(Y_s).
$$
%
%
\subsection{Tensor products}
\label{Subsection tensor products}
Let $\mathcal{A}$ be an $A_{\infty}$-category, and $\mathcal{L},\mathcal{R}$  a left and a right $\mathcal{A}$-module (see \ref{Subsection Left modules and right modules}). 
We will define a chain complex $\mathcal{R}\otimes_{\mathcal{A}} \mathcal{L}$ and an $\mathcal{A}$-bimodule $\mathcal{L}\otimes \mathcal{R}$ such that
$$
\mathrm{HH}_*(\mathcal{A},\mathcal{L}\otimes \mathcal{R}) \cong H^*(\mathcal{R}\otimes_{\mathcal{A}} \mathcal{L}).
$$
The chain complex is given by
$$
\mathcal{R}\otimes_{\mathcal{A}} \mathcal{L} = \bigoplus_{X_i\in \mathrm{Ob}(\mathcal{A})} \mathcal{R}(X_n)\otimes \mathcal{A}(X_n,\ldots,X_0) \otimes \mathcal{L}(X_0)
$$
graded by the sum of the degrees, using reduced degrees for $\mathcal{A}$, 
and the differential (of degree $+1$) on a generator $\underline{\mathfrak{r}}\otimes x_n\otimes \cdots \otimes x_1\otimes \underline{\mathfrak{l}}$ is defined by
$$
\begin{array}{l}
\textstyle
\quad\!\sum (-1)^{\mathrm{deg}(\underline{\mathfrak{l}})+\sigma_1^{R-1}}\; {\scriptstyle
\underline{\mu_{\mathcal{R}}^{|n-R+1}(\underline{\mathfrak{r}}, x_n\, \ldots, x_{R})}}\otimes x_{R-1}\otimes \cdots \otimes x_1 \otimes \underline{\mathfrak{l}} \;+ \\
+\sum (-1)^{\mathrm{deg}(\underline{\mathfrak{l}})+\sigma_1^{S-1}}\; \underline{\mathfrak{r}}\otimes x_n\otimes \cdots \otimes x_{R+1}\otimes {\scriptstyle 
\mu_{\mathcal{A}}^{R-S+1}(x_R, \ldots, x_S)} \otimes x_{S-1} \otimes \cdots \otimes x_1\otimes \underline{\mathfrak{l}} \;+ \\
 +\sum \underline{\mathfrak{r}}\otimes x_n\otimes \cdots \otimes x_{S+1} \otimes {\scriptstyle \underline{\mu_{\mathcal{L}}^{S|}(x_S,\ldots,x_1,\underline{\mathfrak{l}})}}.
\end{array}
$$
summing over the obvious $R,S$. The $\mathcal{A}$-bimodule is defined by
$$
(\mathcal{L}\otimes \mathcal{R})(X,Y) = \mathcal{L}(X)\otimes \mathcal{R}(Y)
$$
with composition maps
$$
\mu_{\mathcal{L}\otimes \mathcal{R}}^{r|s}: \mathcal{A}(X_r,\ldots,X_0)\otimes \mathcal{L}(X_0)\otimes \mathcal{R}(Y_0) \otimes \mathcal{A}(Y_0,\ldots,Y_s) \to \mathcal{L}(X_r)\otimes \mathcal{R}(Y_s)
$$
defined on generators by 
$$
\begin{array}{l}
 \mu_{\mathcal{L}\otimes \mathcal{R}}^{0|s}(\underline{\mathfrak{l}}, \underline{\mathfrak{r}}, y_1, \ldots, y_s) = \underline{\mathfrak{l}}\otimes \underline{\mu^{|s}_{\mathcal{R}}(\underline{\mathfrak{r}},y_1, \ldots, y_s)}\\
 \mu_{\mathcal{L}\otimes \mathcal{R}}^{r|0}(x_r,\ldots,x_1,\underline{\mathfrak{l}}, \underline{\mathfrak{r}}) = (-1)^{\mathrm{deg}(\underline{\mathfrak{r}})} \underline{\mu^{r|}_{\mathcal{L}}(x_r, \ldots, x_1,\underline{\mathfrak{l}})}\otimes \underline{\mathfrak{r}}\\
  \mu_{\mathcal{L}\otimes \mathcal{R}}^{0|0}(\underline{\mathfrak{l}}, \underline{\mathfrak{r}}) = 
  \underline{\mathfrak{l}}\otimes \underline{\mu^{|0}_{\mathcal{R}}(\underline{\mathfrak{r}}}) +
    (-1)^{\mathrm{deg}(\underline{\mathfrak{r}})} \underline{\mu^{0|}_{\mathcal{L}}(\underline{\mathfrak{l}})}\otimes \underline{\mathfrak{r}}
\end{array}
$$
and $\mu_{\mathcal{L}\otimes \mathcal{R}}^{r|s}$ is zero whenever $r,s$ are both non-zero.

Define $\mathcal{T}: \mathrm{CC}_*(\mathcal{L}\otimes \mathcal{R}) \to \mathcal{R}\otimes_{\mathcal{A}} \mathcal{L}$ by reordering the positions of $\underline{\mathfrak{r}}$ and $\underline{\mathfrak{l}}$:
$$
\mathcal{T}(\underline{\mathfrak{l}},\underline{\mathfrak{r}},y_1,\ldots,y_s)=
(-1)^{\mathrm{deg}(\mathfrak{l})(\mathrm{deg}(\mathfrak{r})+\sigma_1^s)} \underline{\mathfrak{r}}\otimes y_1 \otimes \cdots \otimes y_s \otimes \underline{\mathfrak{l}}.
$$
 After the reordering, the bar differential becomes precisely the differential defined for $\mathcal{R}\otimes_{\mathcal{A}} \mathcal{L}$, so $\mathcal{T}$ is a chain isomorphism inducing the isomorphism $\mathrm{HH}_*(\mathcal{A},\mathcal{L}\otimes \mathcal{R}) \cong H^*(\mathcal{R}\otimes_{\mathcal{A}} \mathcal{L})$.

\subsection{Grading conventions}
\label{Subsection A remark about gradings}
We summarize here all grading conventions in the paper, and we do not discuss these further (see \cite{Seidel,Abouzaid,Abouzaid-Seidel} for a detailed discussion of gradings).

\begin{enumerate}
\addtocounter{enumi}{-1}
 \item The grading group in general is $\R/2N\Z$ where $N$ is the minimal Chern number (so $N\Z=c_1(TE)(\pi_2(E))$), as in general it is not possible to assign a grading to a Hamiltonian chord or orbit without additional choices. When working with a single Lagrangian, one can refine the grading group to $\R$;
 \item The Novikov variable $t$ for $E$ has cohomological degree $|t|=2\lambda_E\in \R$;
 \item Unreduced gradings for Lagrangian Floer cohomology are as in Seidel \cite{Seidel}, although in general the grading group will only be $\R/2\Z$. This improves to an $\R$-grading in some situations, such as for the endomorphism algebra of a single Lagrangian, assuming monotonicity, and grading the Novikov variable as above;
 \item For Hamiltonian Floer cohomology, our grading conventions are determined by requiring $QH^*(M)\cong HF^*(H)$ to have degree $0$
 for closed monotone manifolds;
 \item and so $c^*: QH^*(E) \to SH^*(E)$ has degree $0$;
 \item Products preserve unreduced gradings, units lie in unreduced grading $0$;
 \item We never reduce gradings of module inputs/outputs or of $HF^*(H),\, SH^*(E)$;
 \item We reduce gradings of morphisms of $A_{\infty}$-categories, so all $\mu^d$ have degree $+1$;
 \item $\mathcal{AF}: \mathrm{HH}_*(\scrF(E)) \to \mathrm{HH}_*(\wE)$ has degree $0$;
 \item $\mathrm{CO}: SH^*(E) \to HW^*(K,K)$ has degree $0$ (unreduced gradings);
 \item $\mathrm{OC}: \mathrm{HH}_*(\wE) \to SH^*(E)$ has degree $\mathrm{dim}_{\C}E$;
 \item $\Delta^{r|s}: \mathcal{W}(L_r,\ldots,L_0)\otimes \underline{\mathcal{W}(L_0,L_0')} \otimes \mathcal{W}(L_0',\ldots,L_s')  \to \underline{\mathcal{W}(L_r,K,L_s')}$ has degree $\mathrm{dim}_{\C}E$ (the underlined terms are viewed as modules, so their degrees are not reduced).
\end{enumerate}
\section{The wrapped Fukaya category}
\label{Section wrapped fukaya cat}

\subsection{Symplectic manifolds conical at infinity}
\label{Subsection Symplectic manifolds conical at infinity}
$(E,\omega)$ will always denote a (non-compact) convex symplectic manifold. Here \emph{convex} means that $E$ is conical at infinity: outside of a compact domain $E^{\mathrm{in}}\subset E$ there is a symplectomorphism
$$
\psi: (E\setminus E^{\mathrm{in}},\omega|_{E\setminus E^{\mathrm{in}}}) \cong (\Sigma \times [1,\infty), d(R \alpha)).
$$
where $(\Sigma,\alpha)$ is a contact manifold, and $R$ is the coordinate on $[1,\infty)$. We call $E\setminus E^{\mathrm{in}}$ the \emph{conical end} of $E$.
The conical condition implies that outside of $E^{\mathrm{in}}$ the symplectic form becomes exact: $\omega=d\Theta$ where $\Theta=\psi^*(R\alpha)$. It also implies that the Liouville vector field $Z=\psi^*(R\partial_R)$ (defined by $\omega(Z,\cdot) = \Theta$) will point strictly outwards along $\partial E^{\mathrm{in}}$, and that $\psi$ is induced by the flow of $Z$ for time $\log R$, so we can write $\Sigma=\partial E^{\mathrm{in}}$, $\alpha=\Theta|_{\Sigma}$ (pull-back).

By a \emph{conical structure} $J$ we mean an $\omega$-compatible almost complex structure on $E$ (so $\omega(\cdot,J\cdot)$ is a $J$-invariant metric) satisfying the \emph{contact type condition} $
J^*\Theta=dR
$ for large $R$.
On $\Sigma$ this implies $J Z = Y$ where $Y$ is
the Reeb vector field for $(\Sigma,\alpha)$ defined by $\alpha(Y)=1$, $d\alpha(Y,\cdot)=0$. 
By choosing $\alpha$ or $\Sigma$ generically, one ensures that the periods of the closed orbits of $Y$ form a countable closed subset of $[0,\infty)$.
%
\subsection{Monotonicity assumptions, Maslov index}
\label{Subsection monotonicity assumptions}
We always assume $(E,\omega)$ is \emph{monotone}:
$$
c_1(TE) = \lambda_E \, [\omega]
$$
for some constant $\lambda_E >0$ in $\R$.
A submanifold $L\subset (E,\omega)$ is Lagrangian if $\mathrm{dim}_{\R} L = \dim_{\C} E$ and $\omega|_L = 0$. We will always assume that the Lagrangians are \emph{monotone}, meaning:
$$\omega[u] = \lambda_{L}\, \mu([u]) \quad \textrm{for } [u]\in \pi_2(E,L)$$
for some constant $\lambda_{L}>0$, where $\mu$ is the Maslov index and we abbreviated $\omega[u]=\int_{\D} u^*\omega$.
Recall the Maslov index $\mu(u)\in \Z$ is associated to any  
continuous disc $u: (\D ,\partial \D ) \to (E,L)$, and $\mu(u)$ is even for orientable $L$.
The Maslov index only depends on the $\pi_2(E,L)$ class of $u$.
``Gluing in a sphere'' $v:S^2 \to E$ causes
$\mu(v\# u) = 2c_1(TE)([v]) + \mu(u)$. 
This gluing relation forces $E$ to be monotone, and the monotonicity constants are therefore related by
$\lambda_{L} = \tfrac{1}{2\lambda_E}.$
%
\subsection{The Novikov field}
\label{Subsection Novikov ring}
We will always work over the field
$$
\Lambda = \{ \sum_{i=0}^{\infty} a_i t^{n_i}: a_i\in \K, n_i\in \R, \lim n_i = \infty \},
$$
called the \emph{Novikov field},
where $\K$ is some chosen ground field, and $t$ is a formal variable.
The choice of $\K$ will not matter, except in the toric examples at the end of the paper, where we will specialize to $\K=\C$. 
The Novikov field is graded by placing $t$ in (real) degree
$$
|t|=2\lambda_E.
$$
Floer/holomorphic solutions will be counted with Novikov weights, e.g. for the quantum product on $QH^*(E)$, holomorphic spheres $u$ are counted with weight $\pm t^{\omega[u]}$, lying in degree $$2\lambda_E \omega[u] = 2c_1(TE)[u].$$
Recall by \ref{Subsection A remark about gradings}, the grading is understood to lie in $\R/2N$, because there is an ambiguity in gradings in Floer theory (arising from attaching spheres which causes gradings to change by $c_1(TE)$ integrated over spheres). For simply-connected $E$ (and as usual $L,E$ monotone), a part of the Floer theory, e.g. $QH^*(E)$, $SH^*(E)$, $HF^*(L,L)$, $HW^*(L,L)$, is in fact $\R$-graded.

\subsection{The choice of Lagrangians in $\mathbf{\mathrm{Ob}(\mathcal{W}(E))}$}
\label{Subsection The objects of the wrapped Fukaya category}

The objects of $\mathcal{W}(E)$ are the closed monotone orientable Lagrangian submanifolds of $E$, together with all non-compact monotone orientable Lagrangian submanifolds $L\subset E$ which are \emph{conical at infinity}, meaning

\begin{enumerate}
 \item\label{Item transverse to bdry} $L$ intersects $\Sigma$ transversally;
 \item\label{Item conical on the end} the pull-back $\Theta|_L=0$ near $\Sigma$ and on the conical end.
\end{enumerate}

Conditions \eqref{Item transverse to bdry}-\eqref{Item conical on the end} ensure that $\ell=L\cap \Sigma$ is a Legendrian submanifold of the contact manifold $(\Sigma,\Theta|_{\Sigma})$, and that $L$ has the form $\ell\times [1-\epsilon,\infty)\subset \Sigma \times [1-\epsilon,\infty)$ near the conical end. Conversely, a monotone Lagrangian $L\subset E^{\mathrm{in}}$  with $\Theta|_L=0$ near $\Sigma$ can be extended to an $L$ of the above form. 
In fact, any monotone Lagrangian $L\subset E^{\mathrm{in}}$ which intersects $\Sigma$ in a Legendrian submanifold, can be Hamiltonianly isotoped in $E^{\mathrm{in}}$ relative to $\partial E^{\mathrm{in}}$ so that $\Theta|_L$ will vanish near $\Sigma$. This can be proved as in \cite[Lemma 3.1]{Abouzaid-Seidel} (it is a local argument, and the Liouville flow is defined near $\Sigma$).
Finally, as will be explained in detail in \ref{Subsection the role of monotonicity}, one must in fact restrict $\mathrm{Ob}(\wE)$ to Lagrangians having the same $m_0$-value.

\begin{remark}[Grading, characteristic]
For $\wE$ to be graded one would require $2c_1(E, L)=0 \in H^2 (E,L;\Z)$. For $\wE$ to be defined over a field of characteristic $\neq 2$ one would require $w_2(L)=0 \in H^2 (L; \Z/2)$ and then an $L$ should come with a choice of Pin structure. One also needs to make a choice of grading for the Lagrangians. For a discussion of this brane data, we refer to \cite{Seidel,Abouzaid-Seidel}: we will not discuss gradings or orientation signs for sake of brevity.
\end{remark}

\subsection{Reeb chords and Hamiltonian chords}
\label{Subsection Reeb chords and Hamiltonian chords}
Fix a Hamiltonian $H: E \to [0,\infty)$ which on the conical end has the form $H(\sigma,R)=R$ for $(\sigma,R)\in \Sigma\times [1,\infty)$. Later, in \ref{Subsection symplectic cohomology} (see the Convention), we will make a global rescaling of $H$, but the reader can ignore this for now.

Let $X$ be the Hamiltonian vector field: $\omega(\cdot,X)=dH$. The Reeb vector field on $\Sigma$ is $X|_{\Sigma}$.

Fix a countable collection $L_i\in \mathrm{Ob}(\wE)$, so they give rise to $\ell_i\subset \Sigma$.

A \emph{Reeb chord of period} $w$ is a trajectory $[0,1]\to \Sigma$ of $wX|_{\Sigma}$ with ends on $\ell_i,\ell_j$ (for some, possibly non-distinct, indices $i,j$). It is an \emph{integer Reeb chord} if $w$ is an integer. By rescaling $\omega$ (or by choosing $\Sigma$ generically), there are no integer Reeb chords.
An \emph{integer $X$-chord of weight $w$} is a flowline $[0,1]\to E$ of $wX$ with ends on $L_i,L_j$ for integer values $w\geq 1$. These lie in $E^{\mathrm{in}}$ (otherwise, there would be one in the conical end projecting to an integer Reeb chord in $\Sigma$).
For a generic choice of $H$ (see \cite{Abouzaid-Seidel}), integer $X$-chords are non-degenerate and any point of $L_i$ is not both an endpoint and a starting point of two (possibly non-distinct) integer $X$-chords (in particular, no critical point of $H$ may lie in $L_i\cap L_j$). So only finitely many integer $X$-chords have ends on given $L_i,L_j$, and they lie in $E^{\mathrm{in}}$.

\subsection{The Floer differential $\mathfrak{d}$}
\label{Subsection the floer differential}
Given $L_i,L_j\in \mathrm{Ob}(\mathcal{W}(E))$, let $CF^*(L_i,L_j;wH)$ denote the free $\Lambda$-module generated by the (finitely many) integer $X$-chords of weight $w$ with ends on $L_i,L_j$. Fix a conical structure $J$ on $E$ (we remark that the contact type condition in \ref{Subsection Symplectic manifolds conical at infinity} on the conical end is equivalent to $JZ=X$). The chain differential 
$${\mathfrak{d}}:CF^*(L_i,L_j;wH) \to CF^{*+1}(L_i,L_j;wH)$$
is defined as follows on a generator $y$. Let $u: \R \times [0,1] \to E$ be a non-constant isolated (modulo $\R$-translation) solution of Floer's equation $\partial_s u+J(\partial_t u - wX)=0$, satisfying: Lagrangian boundary conditions $u(\cdot,0)\in L_i$ and $u(\cdot,1)\in L_j$; and asymptotic conditions $u\to x,y$ as $s\to -\infty,+\infty$ respectively, where $x$ is an $X$-chord. Then $u$ contributes $$\pm t^{E_{\mathrm{top}}(u)}\, x$$ to $\mathfrak{d} y$, where $E_{\mathrm{top}}(u)\geq 0 \in \R$ is defined by:
$$E_{\mathrm{top}}(u) 
= \int_S u^*\omega - d(u^*(wH)\wedge dt) 
= \int_S u^*\omega +wH(x)- wH(y) 
= \frac{1}{2}\int_{S}\|du-wX\otimes dt\|^2\, ds\wedge dt.$$
We will always omit discussing orientation signs $\pm$, which is done carefully in \cite{Abouzaid-Seidel} by replacing each $\Lambda x$ by an orientation line.

The strip $\R\times [0,1]$ can also be viewed as the disc with two boundary punctures at $\pm 1$:
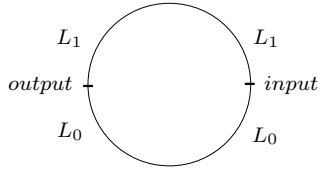
\begin{center}\input{mu1.tex}\end{center}
This domain is unstable: there is a residual $\R$-translation symmetry in $\mathrm{PSL}(2,\R)$ fixing $\pm 1$. Non-constant solutions arise in $\R$-families, and we quotient moduli spaces of non-constant solutions by this action, and only then count the $0$-dimensional part to define $\mathfrak{d}$.

If ${\mathfrak{d}}^2=0$, then the Floer cohomology of $L_i,L_j$ using $wH$ is defined as
$$
\boxed{HF^*(L_i,L_j;wH) = H^*(CF^*(L_i,L_j;wH); \mathfrak{d})}
$$
Since $E$ is monotone, ${\mathfrak{d}}^2=0$ precisely when $L_i,L_j$ have the same $m_0$-value (see \ref{Subsection the role of monotonicity}).

\subsection{The role of monotonicity}
\label{Subsection the role of monotonicity}

We keep the discussion quite general by supposing that we have a $1$-dimensional family of ``Floer solutions'' and we study the possible bubblings.

First observe that sphere bubbling is ruled out for dimension reasons, since by monotonicity a non-constant holomorphic sphere will increase the index of the family by at least $2$.

In the proof that ${\mathfrak{d}}^2=0$, a non-constant $J$-holomorphic sphere of Maslov index 2 (Chern number 1) may bubble off from an $s$-independent strip $u(s,t)=x(t)$, where $x$ is an integer $X$-chord of weight $w$. But this bubbling is ruled out for generic $H,J$, as follows. Recall that for generic $J$, the $J$-holomorphic spheres with $c_1=1$ come in a family of real dimension $2n - 4$, so they sweep out a real codimension $2$ set. In principle, this set could hit our moduli
space of strips. But for index reasons, the only bubbling-off allowed is that of a constant strip. Now, the image of a constant strip is a Hamiltonian chord, which is $1$-dimensional, so generically this does avoid that moduli space of spheres. 

Now consider disc bubbling.
By disc bubble, we mean a non-constant holomorphic disc $u:\D \to E$, bounding a single Lagrangian $u(\partial\D)\subset L$ say, with one boundary marked point $z_0\in \D$ landing at a prescribed generic point $p_L=u(z_0)\in L$ (so $z_0$ plays the role of the node when our family breaks). The index of families is additive, so only isolated discs will bubble off in our $1$-dimensional family and so, by monotonicity and a dimension count, the disc must have Maslov index $\mu(u)=2$.

\noindent {\bf Technical Remark about transversality.} \emph{Recall that transversality can be achieved for simple discs (i.e. somewhere injective discs). By Lazzarini \cite{Lazzarini}, for any holomorphic disc $u: (\D,\partial \D)\to (X,\mathcal{L})$, where $X$ is a smooth almost complex manifold and $\mathcal{L}$ is a totally real submanifold, there is an underlying simple holomorphic disc $v: (\D,\partial \D)\to (X,\mathcal{L})$ with image contained in $u(\D)$, and the relative homology class $[u]\in H_2(X,\mathcal{L})$ is a linear combination $\sum n_i [v_i]$ of such simple discs, where $n_i\in \Z_{>0}$. In our case, $X=E,$ $\mathcal{L}=L$ are monotone, and the Maslov index is additive, so Maslov index $2$ discs (which is the minimal Maslov index, as $L$ is orientable) must be simple: otherwise $u$ has strictly larger area than $v_i$ and thus strictly larger Maslov index. It follows that transversality can be achieved for Maslov 2 discs.
}

Now consider the moduli space $\mathcal{M}_1(\beta)$ of all such Maslov $2$ discs in the class $\beta\in \pi_2(E,L)$ but do not impose a condition on the value at the marked point. Let $\mathrm{ev}:\mathcal{M}_1(\beta) \to L$ be the evaluation at the marked point. Then $\mathrm{ev}_*[\mathcal{M}_1(\beta)]$ determines a locally finite cycle in $L$ (it is a cycle and not just a chain, because $\mathcal{M}_1(\beta)$ cannot have boundary degenerations since, by monotonicity, $2$ is the minimal Maslov index of non-constant disc bubbles). For dimension reasons this lf-cycle must be a scalar multiple of $[L]$ (viewed as a locally finite cycle since $L$ may be non-compact - we explain this in \ref{Subsection construction of m0}). Then define
$$
\mathfrak{m}_0(L) = \sum t^{\omega[\beta]}\, \mathrm{ev}_*[\mathcal{M}_1(\beta)]= m_0(L)\, [L] \in C_{\mathrm{dim}(L)}^{\mathrm{lf}}(L;\Lambda), \textrm{ where }m_0(L)\in \Lambda.
$$

When there are no such $\mu=2$ discs, then $\mathfrak{m}_0(L)=0$, and if this holds for all objects $L$ then the $A_{\infty}$-category is called \emph{unobstructed}, and no disc bubbling discussion is necessary. In the applications we have in mind, we typically expect $\mu=2$ discs. So we briefly explain why in the monotone case all bubbling contributions will in fact cancel out when we restrict the category to Lagrangians having the same $m_0$-value. In other words, monotonicity allows us to pretend that our $A_{\infty}$-categories $\mathcal{W}(E,\lambda)$ are all unobstructed.

Firstly, a standard Gromov compactness argument, looking at the possible degenerations of Maslov 2 discs with two boundary punctures, with Lagrangian boundary conditions $L_i,L_j$ between the punctures, shows that:
$$
\mathfrak{d}\circ \mathfrak{d} (x) = (m_0(L_i)-m_0(L_j))\, x.
$$
For this argument, we refer the reader to Oh \cite{Oh1,Oh2}.
Thus, one cannot even define the cohomology $HF^*(L_i,L_j;wH)$ when $m_0(L_i)\neq m_0(L_j)$, let alone the category.
If we restrict to only Lagrangians with equal $m_0(L)$ values, then $\mathfrak{d}\circ \mathfrak{d}=0$. For higher $A_{\infty}$-operations, the bubbling of such Maslov 2 discs is ruled out for dimension reasons (the index of the main component would have negative virtual dimension).
We summarise this in the following Lemma (we will explain the claim about eigenvalues of $c_1$ in \ref{Subsection Relating m0 and c1 evals}).

\begin{lemma}\label{Lemma fukaya cats split according to evalues}
$\wE$ is a disjoint union of $A_{\infty}$-categories $\mathcal{W}_{\lambda}(E)$ indexed by $\lambda\in \Lambda$, where all $L\in \mathrm{Ob}(\mathcal{W}_{\lambda}(E))$ satisfy $m_0(L)=\lambda$. 
 Since $E$ is monotone, and assuming $\mathrm{char}\, \K=0$, the non-triviality of $\mathcal{W}_{\lambda}(E)$ implies that $\lambda$ is an eigenvalue of quantum cup product by $c_1$: $$c_1(TE):QH^*(E) \to QH^*(E).$$ The analogous statement holds for the compact category $\scrF(E)$ or the Fukaya category $\scrF(B)$ of a closed (rather than convex) monotone manifold $B$. \end{lemma}
\noindent\textbf{Convention.} We always treat $\wE$ as a non-curved $A_{\infty}$-category which splits into non-interacting pieces, not as a curved category. When we write $\wE$ we actually always mean $\mathcal{W}_{\lambda}(E)$ for some fixed $m_0$-value $\lambda$. The same comment pertains to $\scrF(B)$ and $\scrF(E)$.
%
%
\subsection{The construction of $\mathfrak{m}_0$ in the non-compact setting}
\label{Subsection construction of m0}
\begin{lemma}
Let $E$ be conical at infinity.
 Let $L\subset E$ be any Lagrangian submanifold conical at infinity (\ref{Subsection The objects of the wrapped Fukaya category}). Then all pseudo-holomorphic discs $u: (\D,\partial\D) \to (E,L)$ have $u(\partial \D)$ lying either in $E^{\mathrm{in}}$ or in a contact hypersurface $\Sigma \times \{R\}$ of the conical end.
\end{lemma}
\begin{proof}
  For general reasons of convexity \cite[Sec.7.3]{Abouzaid-Seidel}, which essentially reduce to Green's formula, any pseudo-holomorphic map $v: S \to E$ defined on a Riemann surface $S$ (which can have corners) with $v(S)\subset\Sigma\times [R,\infty)$ landing in the conical end of $E$, with $v(\partial S) \subset L\cup (\Sigma \times \{R\})$, must in fact lie entirely in $\Sigma \times \{R\}$. This maximum principle uses the fact that $\Sigma \times \{R\}$ is a contact hypersurface in $E$ on which the Liouville flow is outward-pointing. The claim follows immediately by taking $v$ to be the restriction of $u$ to $u^{-1}(\Sigma \times [R,\infty))$ (for generic $R$ so that $\Sigma \times \{R\}$ intersects the image of $u$ transversely).
 \end{proof}

By the Lemma, given a point $p_L$ of $L$, all holomorphic discs bounding $L$ satisfying $p_L\in u(\partial \D)$ must lie entirely in a compact region of $E$ determined by $p_L$. Thus all the constructions relating to the moduli space of holomorphic discs and in particular the construction of $\mathfrak{m}_0(L)$ reduce to the constructions in the closed monotone case, provided we replace cycles/chains by locally finite cycles and locally finite chains. Indeed, for $L\subset E$ a monotone Lagrangian in $\mathrm{Ob}(\mathcal{W}_{\lambda}(E))$, on each compact subset $C\subset E$ the chain $\mathfrak{m}_0$ represents (up to the scalar multiple $\lambda=m_0(L)$) the fundamental cycle for the pair $(C\cap L,\partial C\cap L)$. Exhausting $E$ by a sequence of larger and larger compact subsets $C$ proves that $\mathfrak{m}_0(L)$ is, up to the scalar multiple $\lambda=m_0(L)$, an infinite sum of chains which is precisely the locally finite cycle $[L]$.
%
%

\subsection{The homotopy $\mathfrak{K}$}
\label{Subsection kappa map}
Denote $\mathfrak{K}$ the Floer continuation map 
$$
\boxed{\mathfrak{K}: CF^*(L_i,L_j;wH) \to CF^*(L_i,L_j;(w+1)H)}
$$
determined by a monotone homotopy from $(w+1)H$ to $wH$. This involves counting isolated solutions similar to those in the definition of $\mathfrak{d}$, except now $J,H$ depend on the domain coordinate $s$ and there is no $\R$-reparametrization. The choice of homotopy $J_s,H_s$ interpolating the given data on the ends is chosen once and for all. Section \ref{Subsection the auxiliary data for wrapped category}
describes the auxiliary data carefully (for $\mathfrak{K} $ one uses: $d = 1$, $F = \{1\}$, $w_0=w + 1$, $w_1 = w$, and notice that the $\R$-translational freedom of the domain is fixed by imposing that the preferred point lies at $0$). 

The above maps $\mathfrak{K}$ are the chain maps which at the level of cohomology are used to define the \emph{cohomology direct limit} as $w\to \infty$:
$$\boxed{
HW^*(L_i,L_j)=\varinjlim HF^*(L_i,L_j;wH)}
$$
%
\subsection{Telescope construction of the wrapped complex}
\label{Subsection the wrapped floer complex}
Following \cite[Sec.3.7, Sec.5.1]{Abouzaid-Seidel},

\begin{definition}\label{Definition wrapped floer complex}
 The wrapped complex is $$\boxed{CW^*(L_i,L_j)=\bigoplus_{w=1}^{\infty} CF^*(L_i,L_j;wH)[\mathbf{q}]}$$ where $\mathbf{q}$ is a formal variable of degree $-1$ satisfying $\mathbf{q}^2=0$. This is a module over the exterior algebra $\C[\partial_{\mathbf{q}}]$, where the differentiation operator $\partial_{\mathbf{q}}$ has degree $+1$. The chain differential is chosen to respect this module structure: 
$$
\boxed{\mu^1(x+\mathbf{q}y) = (-1)^{|x|}\mathfrak{d} x + (-1)^{|y|}(\mathbf{q}\mathfrak{d} y + \mathfrak{K}  y -y)} \; (\textrm{using unreduced degrees})
$$
\end{definition}

Observe that for a $\mathfrak{d}$-cocycle $y$, the role of $\mathbf{q}y$ is to identify $y$ and $\mathfrak{K} y$ at the cohomology level, as expected in the cohomology direct limit; whereas the role of the subcomplex $(\partial_{\mathbf{q}}=0)$ is to yield representatives of $\oplus HF^*(L_i,L_j;wH)$. We refer to \cite[Sec.3.7]{Abouzaid-Seidel} for the proof that $H^*(CW^*(L_i,L_j);\mu^1) \cong HW^*(L_i,L_j)$ canonically.
%

\subsection{The auxiliary data}
\label{Subsection the auxiliary data for wrapped category}
By \cite{Abouzaid-Seidel}, there is a universal family $\mathcal{R}^{d+1,\mathbf{p}}\to \mathcal{R}^{d+1}$ over the moduli space of stable discs, with elements $(S,F,\mathbf{p},\phi)$ where (Figure \ref{Figure popsicle}):
\begin{figure}
\input{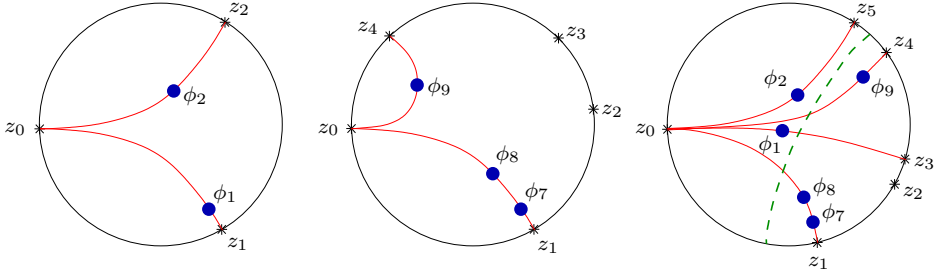}
\caption{Left: $S=$ pair-of-pants surface, $d=2$, $F=\{1,2\}$, $p_1=1$, $p_2=2$. Middle: $d=4$, $F=\{ 7,8,9 \}$, $p_7=p_8=1$, $p_9=4$. Right: the dotted line is a cut, which defines a possible breaking.} 
\label{Figure popsicle}
\end{figure}
\begin{enumerate}
  \item $F$ is a finite set of indices $f$ (with $d\geq 1$, $d+|F|\geq 2$).\\
  {\bf Remark.} \emph{Later, when we define $\mu^d$ in terms of a sum of contributions $\mu^{d,\mathbf{p},\mathbf{w}}$, we will require $F$ to be a subset of $\{ 1,\ldots,d \}$};

  \item A map $\mathbf{p}: F \to \{1,\ldots,d\}$ defining labels $p_f$ (possibly not distinct).\\
  {\bf Remark.} \emph{When we define $\mu^d$, we will require $\mathbf{p}$ to be an inclusion $F\subset \{1,\ldots,d\}$};
 
  \item \emph{{\bf Disc with $d+1$ ends}.} $S$ is a Riemann surface isomorphic to $\D$ with $d+1$ distinct boundary punctures $z_0,z_1,\ldots,z_d$ (of which $z_0$ is distinguished, and the others are ordered according to the orientation of $\partial S$); 

  \item\label{Item preferred points} \emph{{\bf Strip-like parametrizations}.} $\mathbf{\phi}=(\phi_f)_{f\in F}$ are holomorphic  $\phi_f: S \to Z$ extending to smooth isomorphisms on the compactifications $\overline{S} \to \overline{Z}$, with $\phi_f(z_0)= -\infty$, $\phi_f(z_{p_f})= +\infty$ (where $Z$ is the strip $\R\times [0,1]$ viewed as a disc with $2$ ends). This is the same as a choice of a {\bf \emph{preferred point}} on the hyperbolic geodesic joining $z_0$, $z_{p_f}$ (corresponding via $\phi_f$ to the point $(0,\frac{1}{2})$ and to the geodesic $(t=\frac{1}{2})\subset Z$);  

\end{enumerate}

 We will omit the discussion of finite symmetries \cite[Sec.2.3, 2.6, Lemma 3.7]{Abouzaid-Seidel}, but we mention that all auxiliary data constructions are done so as to preserve the symmetries which exchange preferred points lying on the same geodesic.

The above family is compactified by boundary strata described by the natural bubbling configurations one expects when a subset of the preferred points of (\ref{Item preferred points}) converges towards a boundary puncture \cite[Fig.2, Fig.3]{Abouzaid-Seidel}. We omit discussing gluing \cite[Sec.2.4]{Abouzaid-Seidel}, but we mention one important technical aspect:

\noindent {\bf Remark 1.} \emph{One actually allows the geodesics in (\ref{Item preferred points}) to be deformed (except near the ends).
We explain why in an example. Suppose we want to glue the first two discs in Figure \ref{Figure popsicle} at one puncture. Call the discs $S^-,S^+$ and we want to glue $z_1\in \overline{S^-}$ with $z_0\in \overline{S^+}$. Gluing gives a $6$-punctured disc with a geodesic and 2 curves which are not geodesics (they overlap along a  common curve carrying the preferred point $\phi_1$) but they can be deformed into two geodesics. The ``glued'' $F$ is $F=F^-\sqcup F^+=\{1,2,7,8,9\}$. There are four choices for the ``glued'' $\mathbf{p}$: $z_{p_{1}}$ can be any of the $z_1,z_2,z_3,z_4$ coming from $S^+$, and in the cases $z_2$ and $z_3$ we would need to draw a new geodesic before gluing, so that we obtain glued data as in (\ref{Item preferred points}) (up to deforming).
Now observe the third disc in Figure \ref{Figure popsicle}. The dotted line is a ``cut'': we consider a $1$-parameter family of such discs in which the auxiliary data contained on one side of the cut converges towards a common boundary puncture. In the limit, the disc breaks into two, namely $S^-$ glued with $S^+$ with ``glued'' $\mathbf{p}$ determined by $\mathbf{p}(1)=3$, and the new geodesic we draw on $S^+$ is remembering that $\phi_1$ came from the geodesic joining $z_0$ to $z_3$ before breaking.
}

Now assign further auxiliary data:
\begin{enumerate}
 \item \emph{{\bf Weights}.} $\mathbf{w}=(w_0,\ldots, w_d)$ are positive integers with $w_0=\sum_{k\geq 1} w_k+|F|$;

 \item \emph{{\bf Strip-like ends}.} Holomorphic embeddings $\epsilon_{k}: Z_+ \to S$, for $k=1,\ldots,d$, of $Z_+=[0,\infty)\times [0,1]\subset Z$ mapping the boundaries $t=0,1$ to $\partial S$ and converging to $z_k$ as $s\to +\infty$. Similarly $\epsilon_0: Z_- \to S$ for the negative strip $Z_-=(-\infty,0]\times [0,1]$;

 \item \emph{{\bf Closed $1$-forms}.} Closed $1$-forms $\alpha_k$ on $S$ for $k=1,\ldots,d$ pulling back to: $0$ on $\partial S$; $dt$ via $\epsilon_k,\epsilon_0$ on $(|s|\gg 0) \subset Z_{\pm}$; $0$  via the other $\epsilon_j$ on $(s\gg 0) \subset Z_+$.

 \item  \emph{{\bf Subclosed $1$-forms}.} $1$-forms $\beta_f$ on $S$ with $d\beta_f\leq 0$ (for the complex orientation of $S$), $d\beta_f=0$ near $\partial S$, and $\beta_f$ pulls back to: $0$ on $\partial S$; $dt$ via $\epsilon_0$ on $s\ll 0$; $0$ via the other $\epsilon_k$ on $s\gg 0$. One can choose $\beta_f$ so that $d\beta_f=0$ except on a small neighbourhood of the preferred point $\phi_f$, however this choice is not necessary (see \ref{Subsection A comment about the existence of universal choices of auxiliary data}).

\end{enumerate}

The above data determines $\gamma = \sum_{k\geq 1} w_k \alpha_k + \sum_{f\in F} \beta_f$, a $1$-form on $S$ with $d\gamma\leq 0$, $d\gamma=0$ near $\partial S$. Also $\gamma$ pulls-back to: $0$ on $\partial S$; $w_j dt$ via $\epsilon_j$ on $(|s|\gg 0) \subset Z_{\pm}$. Observe that this is consistent with Stokes's theorem, $0\leq - \int d\gamma = w_0 - \sum_{k\geq 1} w_k=|F|$.

The choices of $\alpha,\beta$ can be made to vary smoothly over the moduli space \cite[Sec.2.5,2.6]{Abouzaid-Seidel} consistently with the boundary strata, by an inductive argument (we discuss this further in \ref{Subsection A comment about the existence of universal choices of auxiliary data}). We make one important technical remark about gluing:

\noindent {\bf Remark 2.} \emph{When gluing a disc $S^-$ at $z_k^-$ with a disc $S^+$ at $z_0^+$, we require the weights to agree: $w_k^-=w_0^+$. The $\alpha$ forms glue naturally. For the $\beta$ forms: let $\beta_{f^+}=\alpha_k^-$ on $S^-$, $\beta_{f^-}=0$ on $S^+$, these glue naturally to yield $\beta_f$'s on the gluing indexed by $f\in F=F^-\sqcup F^+$.
}

The last piece of auxiliary data is the almost complex structure: 
\begin{enumerate}
 \item \emph{{\bf Almost complex structures}.} A family $(I_z)_{z\in S}$ of conical structures which pull-back via $\epsilon_j$ on $|s|\gg 0$ to almost complex structures $I^w_t$ depending only on $t,w$ (which are chosen before-hand for each integer $w$). This choice can be made smoothly over the moduli space, and compatibly with gluing;

 \item \emph{{\bf Infinitesimal deformations}.} Certain deformation data $K$ for $I$ \cite[Sec.3.2]{Abouzaid-Seidel}, namely: $K$ is an endomorphism of $TE$ which anti-commutes with $I$, such that $\omega(\cdot,K\cdot)$ is symmetric (together with a condition at infinity which ensures that $J=I \exp(-IK)$ is of contact type at infinity);

 \item \emph{{\bf Regular almost complex structures.}} The $I,K$ give rise to a family of perturbed almost complex structures $J=I \exp(-IK)$, agreeing with $I$ asymptotically, and which vary smoothly over the moduli space.%
\end{enumerate}
%
\subsection{Domain reparametrizations and moduli}
\label{Subsection A comment about domain reparametrizations and moduli}
There is a (real $3$-dimensional) $PSL(2,\R)$ reparametrization action on discs by biholomorphisms which moves the boundary punctures $z_0,z_1,\ldots,z_d$. So we may for example fix the choice of parametrization by choosing the location of three of the boundary marked points, say $z_0=-1$, $z_1=e^{-\pi i/3}$, $z_2=e^{+\pi i/3}$. Indeed, the position of the remaining boundary points $z_3,\ldots,z_d$ on the anti-clockwise boundary arc from $z_2$ to $z_0$ determine a choice of $d-3$ real local coordinates for $\mathcal{R}^{d+1}$. All the auxiliary data above is chosen for a given moduli element $[S]\in \mathcal{R}^{d+1}$, rather than a parametrized choice of disc $S$, meaning: the data is constructed consistently with reparametrizations of the underlying domain $S$ (the use of strip-like parametrizations $\phi_f$ ensure that the data is constructed independently of the choice of parametrization of the domain). Given a parametrized disc $S$, there are new parameters in the auxiliary data involved in the choice of the location of the preferred points, but we emphasize that a biholomorphism of the parametrized domain with data (so pulling back all auxiliary data) gives rise to the same moduli element. So later, when we count Floer solutions for a given parametrized domain with data, the two solutions which are related by a biholomorphism of the domain with data are considered to be the same solution. In particular, for $\mathfrak{d}$ and $\mathfrak{K}$ we emphasized that there is only one moduli element, so we fixed the domain with data once and for all (input at $+1$, output at $-1$, and for $\mathfrak{K}$ we put the preferred point at $0$), and we counted solutions (for $\mathfrak{d}$ we quotiented by the residual $\R$-translational freedom on non-constant solutions and we ignored constant solutions).
%
\subsection{Existence of universal choices of auxiliary data}
\label{Subsection A comment about the existence of universal choices of auxiliary data}
%
We mentioned that the auxiliary data from \ref{Subsection the auxiliary data for wrapped category} needs to be constructed consistently over the universal family $\mathcal{R}^{d+1,\mathbf{p}}$. If one were to attempt this construction explicitly, it would mean: given a Riemann surface $S$ with strip-like ends and some additional data $\phi_f$ (the preferred points), one needs to stipulate explicit rules for constructing the one-forms $\alpha_k$ and $\beta_f$ consistently as we vary $S$ and $\phi_f$. For example, one typically picks the support of $d\beta_f$ to be located near the preferred point $\phi_f$, although this is not necessary in general. 
This construction needs to be done compatibly with gluing/breaking, i.e. over the compactification $\overline{\mathcal{R}}^{d+1,\mathbf{p}}$, so that when disc bubbling occurs the auxiliary data converges to the correct auxiliary data on the various components. 
This amounts to constructing a section of the bundle 
\begin{equation}\label{Eqn aux bundle}
\mathcal{A}ux\to \overline{\mathcal{R}}^{d+1,\mathbf{p}}
\end{equation}
 described below. It is sufficient to prove the existence of the section, without giving an explicit construction. We keep the discussion quite general, since the same argument will apply to the construction of consistent auxiliary data when we define operations, etc. We want the fibre bundle \eqref{Eqn aux bundle}
to satisfy the following properties:
\begin{enumerate}
\item \label{Item fibre 1} the fibre over $(S,\phi_f)$ is the space of all choices of auxiliary data (so $\alpha_k,\beta_f,\,$etc.),
\item \label{Item fibre 2} the bundle is locally trivial,
\item \label{Item fibre 3} the fibre is non-empty,
\item \label{Item fibre 4} the fibre is contractible.
\end{enumerate}
If this holds, then it follows that there exists a section $\overline{\mathcal{R}}^{d+1,\mathbf{p}} \to \mathcal{A}ux$, using a standard local-to-global argument \cite[Sec.12.2]{Steenrod} (the obstruction to extending a section of a fibration lives in the relative cohomology of the base but with coefficients in the homotopy groups of the fibre \cite[Sec.6.4]{Griffiths}, so we have no obstructions by \eqref{Item fibre 4}). The section is precisely a choice of universally consistent auxiliary data. The local-to-global argument gets applied inductively over the strata of $\overline{\mathcal{R}}^{d+1,\mathbf{p}}$, so that the section constructed on the lower strata gets extended when one defines the section on the higher strata, i.e. compatibly with gluing/breaking.

The problem of finding universal data therefore reduces in practice to checking conditions \eqref{Item fibre 3} and \eqref{Item fibre 4}. For the data $\alpha_k,\beta_f$: \eqref{Item fibre 3} follows easily from de Rham theory (e.g. compare \cite[Appendix A]{Ritter3}), and \eqref{Item fibre 4} follows because the space of choices of $\alpha_k,\beta_f$ is convex: linear interpolations of two sets of choices will continue to satisfy all the requirements. Two such choices of universal data can be linearly interpolated, so a standard homotopy argument implies that the constructions we make in Floer theory for those two choices are chain homotopic (and thus homological constructions will be independent of the choice of auxiliary data).

The local-to-global argument \cite[Sec.12.2]{Steenrod} is very general, and includes the above infinite dimensional fibre bundle.  If one wanted to be more explicit, one can appeal to explicit coordinates. The space $\overline{\mathcal{R}}^{d+1,\mathbf{p}}$ carries a canonical structure of a smooth manifold with corners \cite[Cor 6.5]{Abouzaid-Seidel}. In particular the marked points and strip-like ends can be trivialized, so one can locally identify the fibres of the fibre bundle near the boundary strata so that it is easy to extend smooth sections from a boundary stratum to a higher stratum by ``averaging''.

\subsection{Moduli space of pseudo-holomorphic maps}
\label{Subsection moduli space of pseudoholo maps in wrapped case}
Let $x_k$ be $X$-chords of weight $w_k$ with ends on $(L_0,L_d)$, $(L_0,L_1)$, $(L_1,L_2)$, $\ldots$, $(L_{d-1},L_d)$, respectively for $k=0,\ldots,d$, and where $L_i\in \mathrm{Ob}(\wE)$. Denote by $\mathcal{R}^{d+1,\mathbf{p}}(\mathbf{x})$ the moduli space of solutions $u:S\to E$ of
$$
(du-X\otimes \gamma)^{0,1}\equiv \tfrac{1}{2}[(du-X\otimes \gamma) + J \circ (du-X\otimes \gamma)\circ j] = 0
$$
with Lagrangian boundary conditions $u(\partial_i S)\subset L_i$ where $\partial_i S$ is the component of $\partial S$ between $z_i$ and $z_{i+1}$; and asymptotic conditions $u\circ \epsilon_k\to x_k$ as $|s|\to \infty$.
For a generic choice of $K$ the $\mathcal{R}^{d+1,\mathbf{p}}(\mathbf{x})$ are smooth manifolds of the expected dimension \cite[Thm 3.5]{Abouzaid-Seidel}. 

In the exact setup of \cite{Abouzaid-Seidel}, there is a natural compactification of $\mathcal{R}^{d+1,\mathbf{p}}(\mathbf{x})$ modeled on the compactification of $\mathcal{R}^{d+1,\mathbf{p}}$ \cite[Sec.3.5]{Abouzaid-Seidel}. The compactness relied on two ingredients: (1) a maximum principle forcing solutions to lie in $E^{\mathrm{in}}$ \cite[Sec.7.4]{Abouzaid-Seidel}; (2) an a priori energy estimate:
$$
E(u) = \int_S \tfrac{1}{2}\|du-X\otimes \gamma\|^2 \equiv \int_S u^*\omega - u^*dH\wedge \gamma \leq \int_S u^*\omega - d(u^*H\wedge \gamma) \equiv E_{\mathrm{top}}(u)
$$
using $H\geq 0$, $d\gamma\leq 0$, so by Stokes's theorem the topological energy $E_{\mathrm{top}}$ is bounded (in the exact setup $\omega=d\Theta$) in terms of the asymptotics $\mathbf{x}$ using action functionals \cite[Sec.7.2]{Abouzaid-Seidel}.

In our non-exact setup, (1) still holds since the argument only relied on exactness of $\omega$ on the conical end. We no longer have an a priori estimate for $E_{\mathrm{top}}$ since the action functionals are multi-valued. However, we will count solutions with weight $t^{\mathrm{E}_{\mathrm{top}}(u)}$ (which is a homotopy invariant of $u$ relative to the ends), so we only need compactness for $\mathcal{R}^{d+1,\mathbf{p}}(\mathbf{x};C)\subset \mathcal{R}^{d+1,\mathbf{p}}(\mathbf{x})$, which is the union of the components which have $E_{\mathrm{top}}\leq C$. This follows as in the exact case (2) as we artificially imposed an a priori bound on $E(u)$ via $E(u)\leq E_{\mathrm{top}}(u) \leq C$.

\noindent {\bf Remark.}\,\emph{By \ref{Subsection the role of monotonicity}, monotonicity ensures that moduli spaces of Floer solutions have well-behaved compactifications and eliminates the issues with multiple covers of the general theory \cite{FOOO}.}

\subsection{The wrapped $A_{\infty}$-structure}
\label{Subsection wrapped A infinity structure}
We now define, for $d\geq 2$, the maps $\mu^{d,\mathbf{p},\mathbf{w}}$:
$$
CF^*(L_{d-1},L_d;w_dH)[\mathbf{q}]\otimes \cdots \otimes CF^*(L_0,L_1;w_1H)[\mathbf{q}] \to CF^*(L_0,L_d;w_0H)[\mathbf{q}]
$$
First we determine the $\mathbf{q}^0$ coefficient of the image. An isolated $u\in \mathcal{R}^{d+1,\mathbf{p}}(x_0,x_1,\ldots,x_d)$ will contribute $\pm t^{E_{\mathrm{top}}(u)} x_0$ to  $\mu^{d,\mathbf{p},\mathbf{w}}(\mathbf{q}^{i_d}x_d\otimes \cdots \otimes \mathbf{q}^{i_1}x_1)$ where recall
$$
\boxed{E_{\mathrm{top}}  (u) = \int_S u^*\omega - d(u^*H\wedge \gamma)}
$$
and where $i_k=1$ for $k\in \mathbf{p}(F)$, $i_k=0$ for $k\notin \mathbf{p}(F)$, and it will contribute $0$ otherwise (for orientation signs, see \cite[Sec.3.8]{Abouzaid-Seidel} and the Remark below). The conditions on $\mathbf{p}(F)$ ensure that when $i_k=1$ the formal variable $\mathbf{q}^1$ plays the algebraic role of remembering that there is a preferred point on the geodesic in $S$ connecting $z_0,z_k$. For symmetry reasons \cite[Lemma 3.7]{Abouzaid-Seidel}, there is a cancellation of the counts of pseudo-holomorphic maps involving auxiliary data where more than one preferred point lies on a geodesic, so we only need to keep track of whether there is  or there isn't a preferred point on each geodesic (encoded by $\mathbf{q}^1$, $\mathbf{q}^0$). Encoding this information in the algebra is necessary, so that the breaking analysis for the family $\mathcal{R}^{d+1,\mathbf{p}}\to \mathcal{R}^{d+1}$ will indeed yield the $A_{\infty}$-relations for the $\mu^d$ defined below (see Remarks 3 and 4 below, or see \cite[Equation (61)]{Abouzaid-Seidel}).

The $\mathbf{q}^1$ coefficient of the image is determined by the requirement that the map $\mu^{d,\mathbf{p},\mathbf{w}}$ respects the $\partial_{\mathbf{q}}$ operator as follows, for $d\geq 2$: 
$$
\partial_{\mathbf{q}}\mu^{d,\mathbf{p},\mathbf{w}}(c_d\otimes \cdots \otimes c_1)=\sum_{k=1}^d (-1)^{\sigma(c)_{k+1}^{d}} \mu^{d,\mathbf{p},\mathbf{w}}(c_d\otimes\cdots\otimes\partial_{\mathbf{q}}c_k\otimes\cdots\otimes c_1),
$$
where $\sigma(\cdot)$ is defined in \ref{Subsection Grading conventions}.
From this equation, it follows that also the $\mathbf{q}^1$ coefficient of $\mu^{d,\mathbf{p},\mathbf{w}}(\mathbf{q}^{i_d}x_d\otimes \cdots \otimes \mathbf{q}^{i_1}x_1)$ is now determined (for an example, see \cite[Ex.3.14]{Abouzaid-Seidel}). The geometrical meaning of $\partial_{\mathbf{q}}$-equivariance is explained in Remark 3 below. 
Finally, recall that $\mathbf{q}^2=0$, so we have determined the image.

Summing up all $\mu^{d,\mathbf{p},\mathbf{w}}$ as $\mathbf{p},\mathbf{w}$ vary, we obtain the $A_{\infty}$-operations for $d\geq 2$:
$$
\mu^{d}: CW^*(L_{d-1},L_d)\otimes \cdots \otimes CW^*(L_0,L_1) \to CW^*(L_0,L_d).
$$
\noindent {\bf Remark.} \emph{The sum is weighted by signs \cite[Equation (75)]{Abouzaid-Seidel}: $\mu^{d,\mathbf{p},\mathbf{w}}(\mathbf{q}^{i_d}x_d\otimes \cdots \otimes \mathbf{q}^{i_1}x_1)$ is weighted $(-1)^{\dagger}$ where $\dagger=\sum_{j=1}^d j (\|x_j\|+1) + \sum_{j\in F} \sigma(x)_{j+1}^d$ (the second sum is the Koszul sign expected from viewing $\mathbf{q}$ as an operator of degree $-1$ acting from the left, the first sum is as in \cite[(4.8)]{Abouzaid}).}\\[1mm]
\noindent {\bf Remark about symmetry.} \emph{In view of the symmetry argument mentioned above, it is always understood that, when ``summing up'', we only let $\{\mathbf{p},F\}$ vary among
%
%
inclusions of subsets $$\mathbf{p}:F\subset \{1,\ldots,d\}.$$ For example, the first picture in Figure \ref{Figure popsicle} contributes to the sum, whereas if we swapped the labels $\phi_1,\phi_2$ then the picture would not be part of the sum (by symmetry, the count of those solutions would contribute the same except for a reversal in orientation sign \cite[Lemma 9.1]{Abouzaid-Seidel}).
At a more technical level, consider the associahedra in \cite[Figure 2 and 3]{Abouzaid-Seidel}: the bottom edge, which involves a bubbling with two preferred points on one geodesic, seems like a legitimate contribution to the $A_{\infty}$-equations. The key is that the preferred points are labelled, and they can converge at different speeds to the puncture where breaking occurs. But the auxiliary data (in particular the $\beta_f$) are constructed to be independent of the choice of labelling. So the broken Floer solutions contributing to the bottom edge of that associahedron will arise in pairs, depending on the order in which we place the labels $\phi_f$. The symmetry argument implies that these pairs of contributions will cancel for orientation reasons (essentially because we are reordering the tensor of two orientation lines associated to the two preferred points). A genericity argument \cite[Theorem 3.6]{Abouzaid-Seidel} rules out that the two preferred points coincide, so the solutions come in genuine pairs.
}\\[1mm]
We conclude with two technical remarks relating to gluing:\\[1mm]
\noindent {\bf Remark 3.} \emph{In Remark 1 of \ref{Subsection the auxiliary data for wrapped category} we described the gluings of the first two discs $S^-,S^+$ in Figure \ref{Figure popsicle} for four choices of ``glued'' data $\mathbf{p}$. Let us pretend $\phi_7$ is not present. The discs $S^-,S^+$ define two operators $\varphi^-,\varphi^+$,
 whose $\mathbf{q}^0$-output is non-zero only if we input terms of the form $(\mathbf{q}y_2, \mathbf{q}y_1)$, $(\mathbf{q} x_4,x_3,x_2,\mathbf{q} x_1)$ respectively. The $\mathbf{q}^0$-outputs $x^-_0,x_0^+$ of the two operators in these respective cases are then a count of isolated solutions for the relevant data. Gluing means composing operators, but $S^-$ requests a $\mathbf{q}^1$-input at the gluing puncture (namely $\mathbf{q}y_1$). The $\mathbf{q}^1$-output of $\varphi^+$ is determined by $\partial_{\mathbf{q}}$-equivariance: 
 for example $\varphi^+(\mathbf{q}x_4,\mathbf{q}x_3,x_2,\mathbf{q} x_1)=\pm\mathbf{q} x_0^+$ (ignore signs for simplicity).
So in the composition, we take $\mathbf{q}y_1=\pm\varphi^+(\mathbf{q}x_4,\mathbf{q}x_3,x_2,\mathbf{q} x_1) \pm \varphi^+(\mathbf{q}x_4,x_3,\mathbf{q}x_2,\mathbf{q}x_1)$: those correspond to two of the four gluing choices (the other two involve $\mathbf{q}^2=0$). 
For example, $\varphi^+(\mathbf{q}x_4,\mathbf{q}x_3,x_2,\mathbf{q} x_1)$ corresponds to the choice $\mathbf{p}(1)=3$, and it arises from the breaking of Figure \ref{Figure popsicle} (Right). In particular, at the level of composing operators, each term of the composition corresponds to precisely one type of breaking. So the choices of possible ``glued'' data arising from the auxiliary data are consistent with the possible breakings of a family of solutions. So distinct operator composites do correspond to distinct breakings.  
}\\[2mm]
\begin{figure}
\input{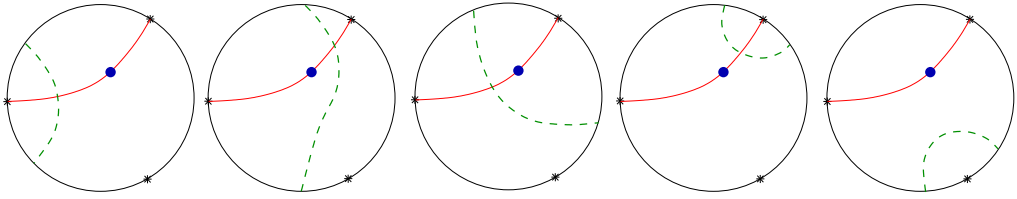}
\caption{$\mu^1\mu^2(\mathbf{q}x_2,x_1) \! +(-1)^{\|x_1\|} \mu^2(\mu^1\mathbf{q}x_2,x_1) \!+\!  \mu^2(\mathbf{q}x_2,\mu^1x_1)\!=\!0$.} 
\label{Figure cutting}
\end{figure}
\noindent {\bf Remark 4.} \emph{We now explain the $A_{\infty}$-equations for the $\mu^d$ in an example, ignoring signs. Recall $\mu^1(x) = \mathfrak{d} x$ and $\mu^1(\mathbf{q}x)=(\partial_{\mathbf{q}}+\mathfrak{K} \partial_{\mathbf{q}})(\mathbf{q}x)  + (\mathbf{q}\mathfrak{d}\partial_{\mathbf{q}}) (\mathbf{q}x)$.
Consider the $A_{\infty}$-equation $\mu^1\circ \mu^2  + \mu^2(\mu^1,\cdot) +  \mu^2(\cdot,\mu^1)=0$. Evaluate this equation on $(\mathbf{q}x_2,x_1)$. First, the $\mathbf{q}^1$-coefficient is:
 $\mathfrak{d}\mu^2(x_2,x_1)+\mu^2(\mathfrak{d}x_2,x_1)+\mu^2(x_2,\mathfrak{d}x_1)$, which is the usual $A_{\infty}$-equation without $\mathbf{q}$'s coming from the 3 possible cuts of a $3$-punctured discs with no geodesics. Secondly, the $\mathbf{q}^0$-coefficient: this has 7 terms of which two cancel (the cancelling terms are $\partial_{\mathbf{q}}\mu^2(\mathbf{q}x_2,x_1)$, $\mu^2(x_2,x_1)$ arising respectively from $\mu^1\mu^2(\mathbf{q}x_2,x_1)$, $\mu^2(\mu^1(\mathbf{q}x_2),x_1)$, ignoring signs). This leaves the five terms $\mathfrak{d}\mu^2(\mathbf{q}x_2,x_1)+\mathfrak{K} \mu^2(x_2,x_1)+\mu^2((\mathfrak{K} \partial_{\mathbf{q}}) \mathbf{q} x_2,x_1)+\mu^2((\mathbf{q}\mathfrak{d}\partial_{\mathbf{q}}) \mathbf{q} x_2,x_1) + \mu^2(\mathbf{q}x_2,\mathfrak{d}x_1)$. 
These are the breakings coming from the cuts shown
in Figure \ref{Figure cutting}. In general,  $\mathfrak{K} \partial_{\mathbf{q}}$ corresponds to the disc with 2 punctures, 1 geodesic, 1 preferred point;  $\mathbf{q}\mathfrak{d}\partial_{\mathbf{q}}$ arises as an incoming disc with $2$ punctures, 1 geodesic, no preferred point (i.e. the $\mathfrak{d}$ disc with a new geodesic drawn).}
%
\subsection{The diagonal bimodule composition maps $\mathbf{\mu^{r|s}}$}
\label{Subsection wrapped diagonal bimodule}
%
Abbreviate by $\mathcal{A}=\wE$ the $A_{\infty}$-category, and by $\bM=\mathcal{W}(E)$ the diagonal bimodule (see \ref{Subsection Bimodules}). The composition maps 
$$\mu^{r|s}: \wE(L_r,\ldots,L_0)\otimes \bM(L_0,L_0') \otimes \wE(L_0',\ldots,L_s') \to \bM(L_r,L_s'),$$
for $r,s\geq 0$ are determined by \ref{Subsection Bimodules}. For example: $\mu^{0|0}=-\mu^1$ is the Floer differential. So the moduli spaces of domains with auxiliary data are the same as those for the maps $\mu^d$. We do have an artistic preference in how we represent the parametrized domains pictorially: 
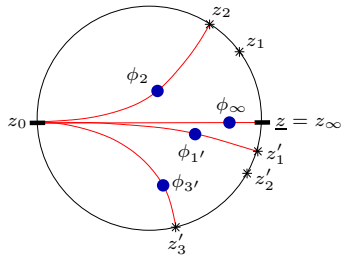
\begin{center}
\input{Bimod-comp.tex}\\
\end{center}
for $\mu^d$ we fixed the output at $-1$ leaving a $PSL(2,\R)/\R$-reparametrization freedom, now we prefer to fix two of the boundary punctures, drawn as dashes: the module output at $-1$ and the module input at $+1$ leaving a residual $\R$-translational freedom (fixing $\pm 1$).
In the picture, the map $\mathbf{p}: F \subset \{ 3',2',1',\infty,1,2\}$ is defined by the inclusion of the subset $\{3',1',\infty,2\}$.
%
\subsection{Wrapped category with local systems}
\label{Section Fukaya with local systems}
\label{Subsection Lagrangians with local systems}
\label{Subsection Novikov rings for local systems}
Let $\Lambda_0 \subset \Lambda$ be the subring of series with only non-negative powers of $t$ (see \ref{Subsection Novikov ring}). Let $\Lambda_0^{\times} \subset \Lambda_0$ be the multiplicative group of units: the series with non-zero $t^0$-coefficient. $\Lambda_0$ is a local ring with maximal ideal $\Lambda_+=\ker($evaluation $\Lambda_0 \to \K$ of the $t^0$-coefficient) consisting of the series with only strictly positive powers of $t$.

%
%
%
%
The $A_{\infty}$-category $\underline{\mathcal{W}}(E)$ equips objects $L\in \mathrm{Ob}(\wE)$ with additional data: a local system of coefficients $\underline{\Lambda}^L$ with fibre $\Lambda$ and structure group $\Lambda_0^{\times}$. So:
\begin{enumerate}
 \item $\underline{\Lambda}_x^L \cong \Lambda$ over each $x\in L$;
 \item for any path $\gamma$ in $L$ from $x$ to $y$, there is an isomorphism $P^L_{\gamma}: \underline{\Lambda}_x^L \to \underline{\Lambda}_y^L$ given by multiplication by an element in $\Lambda_0^{\times}$ depending only on $[\gamma]\in \pi_1(L,x,y)$;
 \item 
concatenating paths yields composed $P^L$'s.
\end{enumerate}

For such $L_0,L_1$ and Hamiltonian $wH$ the Floer complex is now
$$
\underline{CF}^*(L_0,L_1;wH) = \bigoplus \mathrm{Hom}_{\Lambda}(\underline{\Lambda}_{x(0)}^{L_0},\underline{\Lambda}_{x(1)}^{L_1})
$$
summing over all integer $X$-chords $x$ of weight $w$,
and a Floer solution $u: (\D\setminus \{\pm 1\},\partial\D) \to (E,L_0\cup L_1)$ asymptotic to Hamiltonian chords $x,y$ at $-1,+1$, which contributed $\pm t^{\omega[u]} x$ to the old differential $\mathfrak{d}y$, now contributes 
$$
\mathrm{Hom}_{\Lambda}(\underline{\Lambda}_{y(0)}^{L_0},\underline{\Lambda}_{y(1)}^{L_1}) \ni \phi_y %
\mapsto 
\pm t^{\omega[u]} P_{\gamma_1}^{L_1} \circ \phi_y \circ P_{\gamma_0}^{L_0} \in \mathrm{Hom}_{\Lambda}(\underline{\Lambda}_{x(0)}^{L_0},\underline{\Lambda}_{x(1)}^{L_1}),
$$
where $\gamma_0,\gamma_1$ are the paths in $L_0,L_1$ swept by $u|_{\partial\D}$ along the oriented arcs in $\partial\D$ connecting $-1$ to $+1$, and $+1$ to $-1$. This defines the new differential $\underline{\mathfrak{d}}$.

Equivalently, one considers flat $\Lambda_0^{\times}$-connections on the trivial line bundle $L \times \Lambda$ over $L$. So the holonomy is a homomorphism $\pi_1(L) \to \Lambda_0^{\times}$, which determines a class $[b_L]\in H^1(L;\Lambda_0^{\times})$, and the parallel transport $P_{\gamma}^L: \{x\}\times \Lambda \to \{y\}\times \Lambda$ for a path $\gamma$ from $x$ to $y$ is given by multiplication by $b_L[\gamma]\in \Lambda_0^{\times}$ (evaluation of cocycles on chains). Independence of the homotopy class $\gamma$ is recorded by the fact that $b_L$ is a cocycle.

The $\Lambda$-module  $\underline{CF}^*(L_0,L_1;wH) = CF^*(L_0,L_1;wH)=\oplus \Lambda x$ is generated by $X$-chords of weight $w$, except now the above $u$ contributes $\pm t^{\omega[u]}b_{L_0}[\gamma_0]b_{L_1}[\gamma_1]\, x$
to the new differential $\underline{\frak{d}} y$. For zero $b_{L_0},b_{L_1}$, which is a trivial local system, this recovers the usual $CF^*$.
If we consider a $1$-family of discs $u$ as above with fixed asymptotics $x,y$, then the homotopy class of the arcs $\gamma_0,\gamma_1$ is constant in the family, so $t^{\omega[u]}b_{L_0}[\gamma_0]b_{L_1}[\gamma_1]$ is constant. Thus $\underline{\mathfrak{d}}^2=0$ if and only if the old $\mathfrak{d}^2=0$.
For the analogue $\underline{\mathfrak{K}}$ of \ref{Subsection kappa map}: $\pm t^{\omega[u]}b_{L_0}[\gamma_0]b_{L_1}[\gamma_1]\, x$ is a contribution to 
$\underline{\mathfrak{K}} y$ precisely when $\pm t^{\omega[u]} x$ is a contribution to 
$\mathfrak{K} y$. Together with $\underline{\mathfrak{d}}$ this defines $\mu_{\underline{\mathcal{W}}(E)}^1$. 

Similarly, $\mu_{\underline{\mathcal{W}}(E)}^n$ now counts solutions $u$ with weight $\pm t^{\omega[u]}b_{L_0}[\gamma_0]\cdot b_{L_1}[\gamma_1]\cdots b_{L_n}[\gamma_n]$ where $\gamma_i$ are the images via $u$ of the oriented arcs in $\partial \D$ which land in $L_i$.


Mimicking \ref{Subsection the role of monotonicity}, the new $\underline{\mathfrak{m}}_0$'s now depend on $[b_L]\in H^1(L;\Lambda_0^{\times})$: 
$$
\underline{\mathfrak{m}}_0(L,b_L) = \sum t^{\omega[\beta]} b_L[\partial\beta]\, \mathrm{ev}_*[\mathcal{M}_1(\beta)]
\in C_{\mathrm{dim}\, L}^{\mathrm{lf}}(L;\underline{\Lambda}^L)
$$
summing over the homotopy classes $\beta\in \pi_2(E,L)$ of Maslov index $2$, where $\mathcal{M}_1(\beta)\subset \mathcal{M}_1$ corresponds to the disc bubbles $u$ in class $[u]=\beta$ (see \ref{Subsection the role of monotonicity}). As argued in \ref{Subsection the role of monotonicity}, $\mathcal{M}_1(\beta)$ is a $\mathrm{dim}(L)$-cycle so $\mathrm{ev}_*[\mathcal{M}_1(\beta)]$ is a multiple of $[L]$, and finally we define $\underline{m}_0(L,b_L)\in \Lambda_0$ by
$$
\underline{\mathfrak{m}}_0(L,b_L) = \underline{m}_0(L,b_L)\, [L].
$$
Just like in the absence of local systems, $\underline{W}(E)$ splits up according to $m_0$-values (see \ref{Subsection the role of monotonicity}).
Even though typically the $m_0$-values have changed by introducing local systems (i.e. $m_0(L) \neq \underline{m}_0(L,b_L)$), the possible $m_0$-values are still constrained to lie in the spectrum of eigenvalues of $c_1(TE)*\cdot$ because Lemma \ref{Lemma fukaya cats split according to evalues} holds also in the presence of local systems. The advantage of local systems is that a wider range of eigenvalues can now arise geometrically as $m_0$-values.

Local systems will only appear in the toric applications in Section \ref{Section Applications}, and we suppress them from the notation elsewhere.
%
%
\section{Symplectic cohomology}
\label{Section symplectic cohomology}
%
\subsection{Hamiltonian Floer cohomology}
\label{Subsection Hamiltonian Floer cohomology}

We refer to Salamon \cite{Salamon} for a precise definition of the Floer cohomology $HF^*(H^B)$ of a closed monotone symplectic manifold $(B,\omega_B)$, for a Hamiltonian $H^B: B \to \R$. \cite[Lecture 3]{Salamon} discusses the role of monotonicity. We recall that at the chain level, $CF^*(H^B)$ is generated, over the Novikov field $\Lambda$, by all $1$-periodic Hamiltonian orbits of $H^B$. This is not entirely accurate, as for time-independent Hamiltonians the $1$-orbits are not isolated: they arise in $S^1$-families, due to the freedom in the choice of basepoint. So one must make a time-dependent perturbation of $H^B$, discussed in \ref{Subsection SC Floer datum}. 
The differential $\mathfrak{d}: CF^*(H^B) \to CF^{*+1}(H^B)$ is defined as follows on generators: there is a contribution 
$$\pm t^{E_{\mathrm{top}}(u)} x_-$$
to $\mathfrak{d}(x_+)$ for each isolated (up to $\R$-translation) map $u: \R \times S^1 \to B$ solving Floer's equation
$$
\partial_s u + J^B (\partial_t u - X_{H^B})=0,
$$
with asymptotic conditions $u(s,t) \to x_{\pm}(t)$ as $s\to \pm \infty$, where $J^B$ is a fixed almost complex structure on $B$ compatible with $\omega_B$.
Above, $E_{\mathrm{top}}(u)\geq 0 \in \R$ is defined by:
$$E_{\mathrm{top}}(u) 
= \tfrac{1}{2}{\textstyle \int_{S}}\|du-X_{H^B}\otimes dt\|^2\, ds\wedge dt
= \int_S u^*\omega - d(u^*(H^B)\wedge dt) 
= {\textstyle \int_S} u^*\omega +H^B(x_-)- H^B(x_+).$$
 Up to isomorphism, the $HF^*(H^B)$ do not depend on the choice of $H^B,J^B$. Indeed, a homotopy $H^B_s,J^B_s$ which is constant for $|s|\gg 0$, induces a \emph{continuation map} $HF^*(H^B_{+\infty},J^B_{+\infty}) \to HF^*(H^B_{-\infty},J^B_{-\infty})$ by counting solutions $u$ of Floer's equation above where now $J^B_s,H^B_s$ depend on the coordinate $s\in \R$ (and we no longer have an $\R$-reparametrization action). This continuation map is always an isomorphism (the inverse is defined using the ``reversed homotopy'' $H^B_{-s},J^B_{-s}$). When $H^B$ is $C^2$-small and Morse, and $H^B,J^B$ are time-independent, then $CF^*(H^B)$ reduces to the Morse complex for $H^B$, in particular all $1$-periodic Hamiltonian orbits are constant, and so by invariance we have an isomorphism
$$
QH^*(B) \cong HF^*(H^B),
$$
which by PSS \cite{PSS} (see \ref{Subsection PSS}) intertwines the quantum product and the pair-of-pants product.

\subsection{Symplectic cohomology}
\label{Subsection symplectic cohomology}
For symplectic manifolds $E$ conical at infinity (see \ref{Subsection Symplectic manifolds conical at infinity}) $HF^*(H^E)$ depends on the choice of $H^E: E \to \R$ because the larger $H^E$ is on the conical end the likelier it is that there are many $1$-periodic Hamiltonian orbits there. The continuation maps are no longer isomorphisms, because they can no longer be reversed: only certain homotopies $H^E_s$ with $\partial_s H^E_s \leq 0$ are guaranteed to be well-defined. To construct an invariant one needs to choose a class of Hamiltonians which grow in a controlled way at infinity. We consider $H^E$
which at infinity are linear of positive slope in the radial coordinate $R$ (see \ref{Subsection Symplectic manifolds conical at infinity}). One then takes a direct limit over the continuation maps which increase the slope of $H^E$ at infinity: $HF^*(H^E_+) \to HF^*(H^E_-)$ where $\partial_s H^E_s\leq 0$. This defines the \emph{symplectic cohomology}
$$
SH^*(E) = \varinjlim HF^*(H^E),
$$
first constructed in the non-exact setup by Ritter \cite{Ritter2,Ritter4}. We refer the reader also to the original construction in the exact setup due to Viterbo \cite{Viterbo}; the survey by Seidel \cite{Seidel2}; and the construction of algebraic structures such as the pair-of-pants product by Ritter \cite{Ritter3}. We follow the conventions of \cite{Ritter4}, except we use the simpler Novikov field defined in \ref{Subsection Novikov ring} as we work in the monotone setting.

When $H^E$ is $C^2$-small on $E^{\mathrm{in}}$, Morse, time-independent, and of small enough slope on the conical end, the chain complex $CF^*(H^E)$ reduces to the Morse complex for $E$. Since this is part of the direct limit, there is a canonical map (which need not be injective/surjective)
$$
c^*:QH^*(E) \to SH^*(E).
$$
\begin{remark}
$c^*$ is a ring homomorphism, so it splits into eigensummands $c^*:QH^*(E)_{\lambda} \to SH^*(E)_{\lambda}$ indexed by the eigenvalues of quantum product by $c_1(TE)$ acting on $QH^*(E)$.
\end{remark}
\noindent \textbf{Warning (Grading).} $SH^*(E)$ is typically not $\Z$-graded. Floer cohomology is graded by the Conley-Zehnder index, which is a $\Z/2N$-grading where $N\Z = c_1(TE)(\pi_2(E))$.\\[2mm]
\indent
For the purpose of defining the open-closed string map in Section \ref{Section open-closed string map}, it is actually convenient to view also $SH^*(E)$ as a homotopy direct limit (recall \ref{Subsection the wrapped floer complex}). Namely, define 
$$\boxed{
SC^*(E)=\bigoplus_{w=1}^{\infty} CF^*(wH)[\mathbf{q}]
}
$$
with differential $\mu^1=\mu_{SC^*(E)}^1$ defined by the same formula as in Definition \ref{Definition wrapped floer complex}, 
\begin{equation}\label{Eqn differential for big SCE}
\boxed{\mu^1(x+\mathbf{q}y) = (-1)^{|x|}\mathfrak{d} x + (-1)^{|y|}(\mathbf{q}\mathfrak{d} y + \mathfrak{K}  y -y)}
\end{equation}
where $\mathfrak{d}:CF^*(wH)\to CF^{*+1}(wH)$ is the differential defining Hamiltonian Floer cohomology and $\mathfrak{K}:CF^*(wH)\to CF^*((w+1)H)$ is the Floer continuation map analogous to \ref{Subsection kappa map}. Explicitly, the auxiliary $\beta_f$ form that defines these maps can be chosen to be of the form
\begin{equation}\label{Eqn betaf for SH continuation}
\beta_f = c(s)\, dt \textrm{ where }c: \R \to [0,1], c'(s)\leq 0, \textrm{ with } c=1\textrm{ for }s\ll 0 \textrm{ and } c=0 \textrm{ for }s\gg 0
\end{equation}
and the natural choice would be to pick $c(s)\neq 0,1$ only near the circular $s$-slice of $\R \times S^1$ which contains the preferred point $\phi_f$ (although this is not strictly necessary).
Recall in general we do not necessarily need the support of $d\beta_f$ to lie near $\phi_f$. Notice the above $\beta_f$ are constructed $S^1$-invariantly, and we will use these also for the $A_{\infty}$-structure in \ref{Subsection Ainfinity product on SC}.\\ \indent
We define the product $\mu^2$ on $SC^*(E)$ and the $A_{\infty}$-operations $\mu^n$ later, in \ref{Subsection product on SH}.\\[2mm]
\textbf{Convention.} In \ref{Subsection Reeb chords and Hamiltonian chords}, we chose $H$ to have slope $1$ at infinity, but we could have rescaled $H$ by a constant $\epsilon>0$ (once and for all). From now on, we assume $H$ is Morse and $C^2$-small on $E^{\mathrm{in}}$, and linear on the conical end with slope $\epsilon>0$ smaller than the minimal Reeb period of the closed Reeb orbits in $\Sigma$. Then the only $1$-periodic orbits of $H$ lie in $E^{\mathrm{in}}$ and are constant (whilst $wH$ typically will have non-constant $1$-periodic orbits on the conical end for large $w$), and the complex $CF^*(H)$ becomes the Morse complex for $H$, so $HF^*(H)\cong QH^*(E)$. 
\begin{lemma}\label{Lemma SH is H of SC and c is inclusion}
 $SH^*(E)\cong H^*(SC^*(E);\mu^1)$ and the canonical map $c^*: HF^*(H) \to SH^*(E)$ arises as the inclusion of the subcomplex $CF^*(H)\to SC^*(E)$, where $CF^*(H)$ is the $\mathbf{q}^0$-part of the weight $w=1$ summand of $SC^*(E)=\oplus CF^*(wH)[\mathbf{q}]$.
\end{lemma}
\begin{proof}
 The first claim follows by the same arguments as in \ref{Subsection the wrapped floer complex}, the second claim follows by definition (and required the rescaling in the Convention).
\end{proof}
%
\subsection{Contractible vs Non-contractible orbits}
\label{Subsection Contractible vs noncontractible}
%
The Floer differential preserves the free homotopy class of the orbits, so one can restrict the chain complexes $CF^*(H^B)$ and $SC^*(E)$ to only contractible orbits. Denote their cohomologies by $HF^*_0(H^B)$ and $SH_0^*(E)$. 
By invariance, and considering small $H^B$ as mentioned in \ref{Subsection Hamiltonian Floer cohomology}, we deduce $HF^*_0(H^B)=HF^*(H^B)$ for any $H^B$. Similarly when $H^E$ has small slope, such as the $H$ in \ref{Subsection symplectic cohomology}, we have $HF^*(H^E)=HF^*_0(H^E)$, however this may not hold for general $H^E$. Moreover $SH_0^*(E) \subset SH^*(E)$ is a subalgebra containing $c^*(QH^*(E))$ and the unit $1=c^*(1)$, and there is a projection map $SH^*(E) \to SH^*_0(E)$ (which arises from a homomorphism of chain complexes, but it is not an algebra homomorphism). When $\pi_1(E)=1$, $SH^*(E)=SH^*_0(E)$.

\subsection{The PSS-isomorphisms}
\label{Subsection PSS}

There are mutually inverse PSS-isomorphisms \cite{PSS}
$$
\psi^+: QH^*(B) \to HF_0^*(H^B) \qquad \psi^-: HF_0^*(H^B) \to QH^*(B)
$$
defined by counting spiked planes as follows. Parametrize $\C^*=\{ z\in \C: z\neq 0 \}$ by $z=e^{-2\pi(s+it)}$ for $(s,t)\in \R \times S^1$. Consider solutions $u: \C^* \to B$ of Floer's equation
$$
\partial_s u + J^B(\partial_t u - c(s) X_{H^B}) =0
$$
where $c:\R \to [0,1]$ is a decreasing cut-off function which equals $1$ for $s\leq -2$ and equals $0$ for $s\geq -1$. Notice $u$ is $J^B$-holomorphic near the puncture $z=0$.

The \emph{geometric energy} of $u$ is $E_{\mathrm{geom}}(u)=\int_{\C^*} |\partial_s u|_{J^B}^2 \, ds\wedge dt$. If it is finite, then $u$ converges to a Hamiltonian orbit of $H^B$ as $s\to -\infty$, and $u$ extends $J^B$-holomorphically over $z=0$.

For a generic cycle $\alpha \in C_{2\dim_{\C}B-*}(B)$, define $\psi^+(\alpha)$ by counting $0$-dimensional moduli spaces of such finite energy solutions $u$ with $u(0)\in \alpha$.

The definition of $\psi^-$ is similar, except parametrize $\C^*$ by $z=e^{+2\pi(s+it)}$. In this case, solutions $u$ converge to a Hamiltonian orbit as $s\to +\infty$, and $u$ is $J^B$-holomorphic for $s\ll 0$ and extends $J^B$-holomorphically over the puncture $z=0$.
 
In \cite[Theorem 37]{Ritter4}, it is proved that there are mutually inverse PSS-isomorphisms for $E$,
$$
QH^*(E) \to HF_0^*(H^E) \qquad HF_0^*(H^E) \to QH^*(E),
$$
where $H^E: E \to \R$ is radial at infinity, $H^E=h^E(R)$, such that $h^E$ is $C^2$-small and concave (in particular $H^E$ is globally bounded). Homotoping $H$ to $H^E$, where $H$ is as in \ref{Subsection symplectic cohomology}, defines a continuation map $HF_0^*(H^E) \to HF^*(H)$ which is an isomorphism and the composite $QH^*(E) \to HF_0^*(H^E) \to HF^*(H) \to SH^*(E)$ equals the $c^*$ of \ref{Subsection symplectic cohomology} (for details, see \cite{Ritter3}). 

\subsection{Floer datum to break the $S^1$-symmetry}
\label{Subsection SC Floer datum}

Choosing $H$ generically, we may assume the $1$-orbits of $H$ are transversally non-degenerate, but we need a time-dependent perturbation to make them non-degenerate in the $S^1$-direction.
The Floer equation is $(du-X_H\otimes dt)^{0,1}=0$, and the time-dependent perturbation corresponds to replacing this equation by
$$
(du-X_H\otimes dt-X_{H_{\textrm{pert}}}\otimes c_-(s)dt-X_{H_{\textrm{pert}}}\otimes c_+(s)dt)^{0,1}=0
$$
where the \emph{Floer datum} $H_{\textrm{pert}}$ is a small time-dependent Hamiltonian supported near the $1$-orbits for $H$, ensuring that the $1$-orbit breaks into two isolated $1$-orbits \cite[Lemma 2.1]{Cieliebak-Floer}. The cut-off functions $c_-,c_+: \R \to [0,1]$ are supported respectively in the region $s\ll 0$ and $s\gg 0$. By standard arguments in Floer theory, one can show that these additional choices do not affect $HF^*(H)$ up to isomorphism.
In the non-compact setup of \ref{Subsection symplectic cohomology}, we do not need to impose that $c_{\pm}'\leq 0$ (the analogue of the condition $d\gamma\leq 0$).
This is because $H_{\mathrm{pert}}$ is compactly supported, so the maximum principle still applies at infinity since we have not changed $H$ there. So we avoid the complications \cite[Definition 4.10]{Ganatra} arising when Hamiltonians grow quadratically. We now justify the analogue of the a priori energy estimate in \ref{Subsection moduli space of pseudoholo maps in wrapped case}. 
Suppose we are counting Floer solutions $u$ of $(du-X_H\otimes \gamma)^{0,1}=0$ defined on a punctured Riemann surface $S$. Consider an end of $S$ where this equation becomes $(du - X_H\otimes w\, dt)^{0,1} = 0$, then we need to modify $X_H\otimes w\, dt=X_{wH} \otimes dt$ so that it becomes $X_{w(H+H_{\mathrm{pert},w})}\otimes dt$, where $H_{\mathrm{pert},w}:S^1 \times E \to \R$ is a small time-dependent function depending on the weight $w$ and supported near the $1$-orbits for $wH$. We can extend this to a function $H_{\mathrm{pert}}: S \times E \to \R$ supported near the ends of $S$, where it equals $c_{\pm}(s)\, H_{\mathrm{pert},w}$ if the end has weight $w$ and sign $\pm$. 
Finally, the perturbed Floer equation becomes: $(du-X_{\widetilde{H}}\otimes \gamma)^{0,1}=0$ where $\widetilde{H}=H+H_{\mathrm{pert}}$ (which now depends on $S$). 
Then the a priori energy estimate in \ref{Subsection moduli space of pseudoholo maps in wrapped case} becomes
$$
\begin{array}{rcl}
\textstyle E(u) &\equiv& \int_S u^*\omega- u^*d\widetilde{H}\wedge \gamma
\\[1mm]
& \leq & E_{\mathrm{top}}(u) + \displaystyle{\sum_{\textrm{ends}} \textrm{measure(support(}c'_{\pm}))
\cdot \sup |c'_{\pm}|\cdot  \|H_{\mathrm{pert,w}}\|_{\infty}}
\end{array}
$$
where $E_{\mathrm{top}}(u) =\int_Su^*\omega - d(u^*\widetilde{H}\wedge \gamma)$, and the support of $c_{\pm}'$ is a compact subset of $S$ whose measure is finite and independent of $u$. Above, we used that $\widetilde{H}\geq 0$ and $d\gamma\leq 0$ and that on an end, for very large $|s|$, $\gamma=w\, dt$ and $\widetilde{H}$ is $t$-dependent but not $s$-dependent.\\
\indent When we define chain level operations on $SC^*(E)=\oplus CF^*(wH)[\mathbf{q}]$, it is understood that:
\begin{itemize}
\item $CF^*(wH)$ in fact means $CF^*(w(H + H_{w,\textrm{pert}}))$ where $H_{w,\textrm{pert}}$ is the Floer datum chosen for the weight $w\geq 1$, as described above.
\item Whenever we define a chain level map (e.g. the $A_{\infty}$-structure on $SC^*(E)$, or operations involving $SC^*(E)$), we will construct an auxiliary one-form $\gamma$ determined by auxiliary forms $\alpha_k,\beta_f$.
It is always understood that the Floer equation $(du-X_H\otimes \gamma)^{0,1}=0$ defining the counts of solutions $u: S\to E$, in fact needs to be perturbed, so $X_H\otimes \gamma$ is replaced by $X_{H + H_{\textrm{pert}}}\otimes \gamma$, which equals $X_{wH} \otimes dt + X_{H_{w,\textrm{pert}}}\otimes c_{\pm}(s)dt$ at the ends.
\item The data $H_{\mathrm{pert}}$ depends on $S$, so this new auxiliary data also needs to be constructed universally and consistently, following the recipe of \ref{Subsection A comment about the existence of universal choices of auxiliary data}.
\end{itemize}

We will not mention these perturbations again, so that the notation stays under control. 
%
%
\subsection{The $A_{\infty}$-structure on $SC^*(E)$.}
\label{Subsection product on SH}
\label{Subsection Ainfinity product on SC}

Consider the moduli space of decorated domains as in the first picture below: cylinders $\R\times S^1$ with finitely many ordered positive punctures $z_n,\ldots,z_2$ along the $t=0$ line.

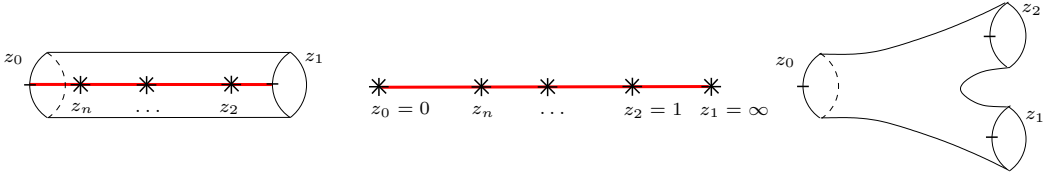
\begin{center}\input{SHpert2.tex}\end{center}

One can equivalently view this as the moduli space of punctured spheres, as in the second picture, where there is exactly one negative puncture $z_0$ and one distinguished positive puncture $z_1$, and ordered positive punctures $z_n,\ldots,z_2$ along the geodesic connecting $z_0$ to $z_1$, where it is understood that the cylindrical ends at the punctures must have $t=0$ lining up with the geodesic (each interior puncture carries an \emph{asymptotic marker} which determines the $t=0$ direction in the end parametrization).
 If there are three or more punctures, we can remove all parametrization freedom by considering only $\C P^1$ and fixing $z_0=0, z_2=1, z_1=\infty$ (letting $z_n,\ldots,z_2$ be ordered punctures on the real line segment $(0,1)\subset \C P^1=\C\cup \{\infty\}$).
 
 If there are only two punctures, then there is a residual $\R$-translation reparametrization freedom. The third picture above is an equivalent way to view the 3-punctured case, $n=1$, and is called a pair-of-pants. Roughly speaking, given Hamiltonian $1$-orbits $x_j$, the $x_0$-coefficient of $\mu^n(x_n,\ldots,x_1)$ is a count of Floer solutions on the domain above, taking $x_j$ to be the asymptotic condition at the puncture $z_j$.
As usual, we allow the presence of preferred points $\phi_f$ lying on the geodesic line $t=0$ (cylinder model) joining $z_0$ to a positive puncture $z_{p_f}$. It is always understood that the point $\phi_f$ belongs to the compactified cylinder where all punctures except $z_0,z_{p_f}$ have been filled in. When constructing auxiliary data, it is always understood that the $\beta_f$ forms are required to vanish at all positive punctures. 
\\[1mm]
{\bf Rotational symmetry trick.} \emph{We insist that the $\beta_f$ forms are constructed $S^1$-equivariantly in the cylinder model, universally and consistently as in \ref{Subsection A comment about the existence of universal choices of auxiliary data}. (Note this means that the support of $d\beta_f$ is not located near the preferred point $\phi_f$, but rather at best near a ``preferred circle''.) So we do not need to constrain the $\phi_f$ to the $t=0$ line, provided we work with moduli of domains: we only care about location of the $\phi_f$ up to rotating each $\phi_f$ individually.}
\\[1mm]
\indent The differential $\mu^1$ counts solutions $u: \R \times S^1 \to E$ of Floer's equation $(du-X\otimes \gamma)^{0,1}=0$ where $\gamma$ depends on the auxiliary data analogous to \ref{Subsection the auxiliary data for wrapped category}, except the ends are cylindrical rather than strip-like (as we replaced the time-interval $[0,1]$ by $S^1$). So we have a closed form $\alpha$ which pulls back to $dt$ in the cylindrical ends near the punctures; and a preferred point $\phi_f$ (if present) on the geodesic $t=0$ corresponds to the translation $\R \times S^1 \to \R \times S^1$, $\phi_f\mapsto (0,0)$. These translations allow one to consistently build $1$-forms $\beta_f$ with $d\beta_f\leq 0$, and by the above convention we build these $S^1$-invariantly, for example as in \eqref{Eqn betaf for SH continuation}.

We now discuss the construction of the pair-of-pants product $\mu^2$, so the case $n=2$, but the same argument with additional punctures defines more generally the $A_{\infty}$-structure maps $\mu^n: SC^*(E)^{\otimes n} \to SC^*(E)$ for $n\geq 2$. 
In analogy with $\mu^2_{\wE}$, one constructs
$$\mu^2=\mu^2_{SC^*(E)}:SC^*(E)\otimes SC^*(E) \to SC^*(E),$$
except the domain is now the $3$-punctured sphere (equivalently the cylinder with a positive puncture at $(s,t)=(0,0)$), rather than a $3$-punctured disc.
%

The two positive punctures $1$ and $\infty$ receive inputs from $CF^*(w_1H)[\mathbf{q}], CF^*(w_{\infty} H)[\mathbf{q}]$, whereas the negative puncture $0$ receives the output from $CF^*(w_0H)[\mathbf{q}]$. The count of solutions on these domains, with auxiliary data, determine the $\mathbf{q}^0$-contribution of 
$$\mu^2:CF^*(w_1 H)[\mathbf{q}] \otimes CF^*(w_{\infty} H)[\mathbf{q}] \to CF^*(w_0H)[\mathbf{q}].$$
 Finally, requiring $\mu^2$ to be $\partial_{\mathbf{q}}$-equivariant determines the $\mathbf{q}^1$ contributions.
 
 As usual, the auxiliary form $\gamma= \sum w_k \alpha_k + \sum \beta_f$ which determines the Floer equation pulls back to $w_1\, dt$, $w_{\infty}\, dt$, $w_0\, dt$ in the parametrizations near $1,\infty,0$ (where $w_0=w_1+w_{\infty}+|F|$). There are two closed forms $\alpha_1$, $\alpha_{\infty}$ associated to $1,\infty$ satisfying $\alpha_{\infty}=0$ near $1$ and $\alpha_{\infty}=dt$ near $0$ and $\infty$; $\alpha_1=0$ near $\infty$, and $\alpha_1=dt$ near $0$ and $1$. Along the $t=0$ line in the cylinder model, there are two geodesics: one connecting $0$ to $1$, and the other connecting $0$ to $\infty$. These geodesics may carry preferred points $\phi_f$, to which one associates $S^1$-invariant $\beta_f$ forms with $d\beta_f\leq 0$ which satisfy: $\beta_f=0$ near $1,\infty$ and $\beta_f=dt$ near $0$.
 
 As usual, the auxiliary data is built universally and consistently for all $\mu^n$, $n\geq 1$, following the argument in \ref{Subsection A comment about the existence of universal choices of auxiliary data}. The bundle $\mathcal{A}ux$ of auxiliary data is over the compactification of the moduli space of all pairs-of-pants. The additional strata of the compactification arise because of the bubbling off of a twice-punctured sphere (i.e. a cylinder) at some end. As usual one constructs the auxiliary data inductively over the strata of the compactification so that the universal auxiliary data that we already constructed for $\mu^1$ will be extended to define the universal auxiliary data for $\mu^2$. This is necessary to ensure that $\mu^2$ will be a chain map (i.e. the cylinders which break off in a $1$-family of pairs-of-pants, carry the correct universal auxiliary data chosen for $\mu^1$). Just as in \ref{Subsection A comment about the existence of universal choices of auxiliary data}, the fibre of $\mathcal{A}ux$ is non-empty by de Rham theory (a proof can be found in \cite[Appendix A]{Ritter3}) and it is contractible, indeed convex, because one can linearly interpolate the data whilst preserving all the conditions. Linear interpolation of two sets of universal auxiliary data and a homotopy argument implies that the $\mu^2$ structure on the cohomology $SH^*(E)$ does not depend on the auxiliary choices.

That $\mu^2$ is a chain map follows by considering the possible breaking of a $1$-dimensional moduli space with the above data. Namely, one cylinder can break off at one of the three punctures, carrying an $\R$-symmetry (this is an index 1 cylinder, and for index reasons there cannot be more than one cylinder breaking off). If the cylinder carries no preferred points then this count contributes to the $\mathfrak{d}$-part of $\mu^1$, if it carries one preferred point then this contributes to the $\mathfrak{K}$-part of $\mu^1$. In the definition of $\mu^2$ and in this breaking analysis, one can always forget configurations where there is more than one preferred point on a geodesic as the count of such configurations will cancel for the same reason as in \cite[Lemma 3.7]{Abouzaid-Seidel}.

\noindent {\bf Technical Remark.} \emph{The symmetry argument \cite[Lemma 3.7]{Abouzaid-Seidel} relies on showing that the transposition of two preferred points lying on the same geodesic switches the orientation sign with which the isolated solution is counted \cite[Lemma 9.1]{Abouzaid-Seidel}. This works generally both for moduli spaces of punctured discs and for moduli spaces of punctured genus zero Riemann surfaces, using that the auxiliary data has been constructed universally and consistently.}

That $\mu^2$ is graded-commutative on cohomology follows by considering a suitable family of domains in which one makes the punctures $1,\infty$ move in $\CP^1 \setminus \{0\}$ until their positions are interchanged (we need to break the $S^1$-invariance of the $\beta_f$ forms in this argument, and recall by the homotopy argument above that on cohomology we do not care whether the universal auxiliary data before/after the movement coincides).
%
\subsection{Associativity and unitality}
\label{Subsection product on SH is associative}
%
For general reasons, as in \ref{Subsection the wrapped floer complex}, the telescope construction for $SC^*(E)$ corresponds on cohomology to the direct limit construction $SH^*(E)=\varinjlim HF^*(H)$. In particular, $\mathfrak{d}$-cocycles $y\in HF^*(wH)$ generate $SH^*(E)$ (allowing $w$ to vary), and cocycles $y,\mathfrak{K}y$ are identified on cohomology via the boundary of $\mathbf{q}y$, which encodes the fact that generators of the direct limit are identified if they are related by applying continuation maps $\mathfrak{K}$. Thus on cohomology $\mu^2$ is determined by the usual pair-of-pants product $HF^*(w_1 H)\otimes HF^*(w_{\infty} H)\to HF^*(w_0 H)$ constructed as in \cite{Ritter3}.
Thus cohomological properties of $\mu^2$ follow from \cite{Ritter3} by general TQFT arguments, in particular: the product is associative and unital, where the cohomological unit for $SH^*(E)$ is the image $c^*(1)$ of $1$ under $c^*:QH^*(E)\to SH^*(E)$. Indeed, $c^*:QH^*(E)\to SH^*(E)$ is a unital ring homomorphism (this was proved in \cite{Ritter3,Ritter4}). A clean way to prove the latter is carried out in \ref{Subsection Category W0E and S0CE}, where we enlarge the chain complex $SC^*(E)\subset SC_{\Diamond}^*(E)$ to also include a $QC^*(E)$-summand, so that $c^*$ becomes  the inclusion of this summand at the chain level.

Associativity can also be proved directly, by verifying the $A_{\infty}$-relation relating $\mu^1,\mu^2,\mu^3$, obtained from the bubbling analysis for the domains in \ref{Subsection Ainfinity product on SC} for $n=3$.

We sketch the argument from \cite{Ritter3} that the product $\mu^2$ is unital. The unit is represented by the image of $1$ under the PSS isomorphism $QH^*(E) \to HF^*(H)$. This corresponds to a count of Floer solutions $\C\to E$ with a tautological intersection condition $[E]$ at the origin, using a form $\gamma$ such as $\gamma=c(s)\, dt$ where $c:\R\to [0,1]$ is an increasing cut-off function with $c=0$ near $s=-\infty$ (so the Floer solution is $J$-holomorphic near the origin), $c=1$ near $s=\infty$, and where $e^{2\pi(s+it)}$ parametrizes $\C\setminus \{0\}$.  The particular choice of $\gamma$ only affects the representative by a chain homotopy, so choices do not matter on cohomology. We call these solutions the \emph{$E$-spiked discs}. Gluing the $E$-spiked discs onto a pair-of-pants at a $CF^*(H)$ input puncture, yields the Floer solutions defined on a cylinder which on cohomology determine the continuation map $HF^*(w_{\infty} H) \to HF^*(w_0 H)$. But this map represents the identity map in the direct limit $SH^*(E)$ (generators get identified via continuation maps), which shows that $\mu^2(1,\cdot)$ represents the identity map on $SH^*(E)$, and similarly for $\mu^2(\cdot,1)$.
%
\subsection{$SH^*(E)$ is a module}
\label{Subsection product on SH is a module}
$SH^*(E)$ is an $SH^*(E)$-module by acting by $\mu^2$, and a $QH^*(E)$-module by using the ring homomorphism $c^*:QH^*(E)\to SH^*(E)$ and then acting by $\mu^2$.
\section{The compact Fukaya category}
\label{Section Fukaya category}
%
\subsection{The compact Fukaya category $\bE$}
\label{Subsection Fukaya categories B and E}
\label{Subsection Fukaya category of E}
Let $E$ be a monotone symplectic manifold conical at infinity. Denote the (compact) Fukaya category by $$\bA=\scrF(E).$$ 
The objects form the subset of $\mathrm{Ob}(\wE)$ consisting of the \emph{closed} monotone Lagrangians (and recall from \ref{Subsection the role of monotonicity} that we fix a common $m_0$-value) . Since the Lagrangians are compact, the maximum principle from \ref{Subsection moduli space of pseudoholo maps in wrapped case} forces all Floer solutions to lie in a compact subset determined by the Lagrangians, so the non-compactness of $E$ is unproblematic and the constructions for $\scrF(E)$ also carry over to $\scrF(B)$ for closed monotone symplectic manifolds $B$.

\begin{lemma}\label{Lemma HF=HW for closed Lags}
$HF^*(L_0,L_1)\cong HW^*(L_0,L_1)$ for any closed Lagrangians $L_0,L_1\in \mathrm{Ob}(\mathcal{F}(E))$.
\end{lemma}
\begin{proof}
$HW^*(L_0,L_1)\cong \varinjlim HF^*(L_0,L_1;H)$ taking the direct limit over continuation maps $HF^*(L_0,L_1;H) \to HF^*(L_0,L_1;H')$ for linear Hamiltonians, where $H'$ has larger slope at infinity than $H$. For $L_0,L_1$ closed, we can choose a cofinal sequence of Hamiltonians, so that each $H'$ is obtained from the previous $H$ by increasing the slope at infinity outside of a large compact set $C$ containing $L_0,L_1$. By the maximum principle, the Floer continuation solutions are trapped in $C$. But in $C$, the homotopy $H_s$ from $H'$ to $H$ is $s$-independent, so Floer solutions would have an $\R$-symmetry in $s$ if they were not constant. Since they are rigid they must be constant. So these continuation maps are identity maps at the chain level.
\end{proof}

Suppose first that $L_0,\ldots,L_n \in \mathrm{Ob}(\bA)$ are pairwise transverse. Then no Hamiltonians arise in the definition of 
$$\mu^n_{\bA}: \bA(L_n,\ldots,L_0) \to \bA(L_n,L_0).$$ 
Namely, $\mu^n_{\bA}(x_n,\ldots,x_1)$ counts $J$-holomorphic discs $u: S \to E$ modulo reparametrization (weighing counts by $t^{E_{\mathrm{top}}(u)}$ as usual, with signs as in \ref{Subsection wrapped A infinity structure}), where $S$ is illustrated below. 
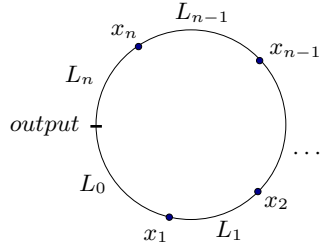
\begin{center}\input{AinfinityCompact.tex}\end{center}
On the boundary there are marked points (not punctures) which are free to move subject to being distinct and ordered anticlockwise; they receive inputs $x_j\in CF^*(L_{j-1},L_{j})$. We chose to fix the position of the output at $-1$ (which gives the output in $CF^*(L_0,L_n)$), and following \ref{Subsection A comment about domain reparametrizations and moduli} we count moduli of solutions, under the residual (real $2$-dimensional) $PSL(2,\R)/\R$-reparametrization freedom which moves the input boundary marked points whilst fixing $-1$.

Compared with the wrapped case, formally this would correspond to taking $w=0$ and allowing only $\mathbf{q}^0$ terms (in the wrapped case we only allow weights $w\geq 1$). So moduli spaces are defined using the equation $du^{0,1}=0$ rather than $(du-X\otimes \gamma)^{0,1}=0$, and $CF^*(L_i,L_j)$ is generated by the intersections $L_i\cap L_j$. So there are no auxiliary $\alpha,\beta,\gamma$ forms, nor do we have strip-like end parametrizations. Just as for $\mathfrak{d}$ in \ref{Subsection the floer differential}, the domain defining $\mu^1_{\bA}$ has two boundary marked points which we fix at $\pm 1$ leaving a residual $\R$-symmetry, and we count only non-constant solutions modulo this symmetry. The $A_{\infty}$-equations follow by a standard bubbling argument depicted in the picture of \ref{Subsection Grading conventions}.
%
\subsection{The Floer datum: non-transverse Lagrangians}
\label{Subsection When Lagrangians do not intersect transversely}
The ordinary Lagrangian Floer complexes are related to the wrapped versions via
$$
CF^*(L_i,L_j;H) \equiv CF^*(\varphi_H^1(L_i),L_j)
$$
where $\varphi_H^1$ is the time $1$ flow of $X_H$, and where one must choose suitable almost complex structures (for example \cite{Ritter3} explains this). When $L_i,L_j$ are not transverse, then $CF^*(L_i,L_j)$ is in fact \emph{defined} as above using a compactly supported $H$ so that $\varphi_H^1(L_i),L_j$ are transverse. Setting up the $A_{\infty}$-category to keep track of the Hamiltonian perturbation data requires substantial work, carried out in detail by Seidel \cite[Sec.(8e),(8f),(8j)]{Seidel}.
 So we only recall the key ideas. This perturbation was not required in the wrapped case as the weights $w\neq 0$.\\[1mm]
\noindent {\bf Remark.} \emph{For $L=L_i=L_j$ one can use a Morse-Bott model using $C_*(L)=CF^*(L,L)$ as carried out by Auroux \cite{Auroux}. However it would be a non-trivial task to implement such a Morse-Bott theory approach more generally for the whole $A_{\infty}$-category $\mathcal{F}(E)$. }\\[1mm]
\indent
In the Floer datum approach \cite{Seidel}, loosely speaking, the $A_{\infty}$-operations $\mu^k_{\bA}$ ``carry'' the Hamiltonian perturbation data which is used to perturb the $J$-holomorphic curve equation which defines the moduli spaces, to achieve transversality. The equation is defined on the associahedron and a point on the associahedron determines the following data: boundary marked points; strip-like ends; Hamiltonians on each strip-like end such that the flow moves $L$ to $\varphi_{H}^1(L)$ so that transversality is achieved. The Floer chain complexes $CF^*(L_0,L_1)\equiv CF^*(L_0,L_1;H_{L_0,L_1},J_{L_0,L_1})$ are generated by the Hamiltonian chords using auxiliary data
$$H_{L_0,L_1}=H_{L_1,L_0} \qquad\qquad J_{L_0,L_1}=J_{L_1,L_0}$$
called \emph{Floer datum}. This is chosen generically once and for all for each pair $L_0,L_1$ (the generators then correspond bijectively to the transverse intersections $\varphi_H^1(L_0)\cap L_1$). Outside of a compact set containing $L_0,L_1$ we choose $H_{L_0,L_1}=0$ and $J_{L_0,L_1}=J$ of contact type, so the maximum principle applies at infinity (this would not work in the wrapped setup for non-compact $L_0,L_1$ because a non-compactly supported $H_{L_0,L_1}$ is usually needed to make $\varphi_H^1(L_0)$, $L_1$ transverse). For transverse $L_0,L_1$ one can pick $H_{L_0,L_1}=0$ and $J_{L_0,L_1}=J$. Even if one restricted the category to a subcollection $L_i$ of pairwise transverse Lagrangians, one would still need a Floer datum for pairs of type $L,L$.

The \emph{perturbation datum} is a universally consistent choice (over the associahedron) of interpolating data $H,J$ for each operation $\mu^k_{\bA}$ with given Lagrangians $L_0,\ldots,L_k$, which agrees with the Floer datum on the strip-like ends. So in fact we are still working, in general, with boundary punctures rather than boundary marked points; choices of strip-like ends; and the equation $(du-X\otimes \gamma)^{0,1}=0$ rather than $du^{0,1}=0$. However, now we can always choose $\gamma=dt$ near the strip-like ends of the domain of $u$, we do not require $d\gamma\leq 0$ and we can choose $\gamma$ to be zero away from the strip-like ends (this would have been impossible in the wrapped setup where $d\gamma\leq 0$ is required by the maximum principle). On the strip-like ends we can ensure the equation is $(du-X_{H_{L_0,L_1}}\otimes c(s)dt)^{0,1}=0$ for the given Floer datum, for a cut-off function $c:\R \to [0,1]$ which is $1$ for $|s|\gg 0$ and equals $0$ for $|s|$ close to $0$, so explicitly:
$$\partial_s u + J_s (\partial_t u - c(s)X_{H_{L_0,L_1}})=0,$$
where $J_s=J_{L_0,L_1}$ for $|s|\gg 0$, $J_s=J$ for $|s|$ close to $0$. An a priori energy estimate follows by the same argument as in \ref{Subsection SC Floer datum}. Unlike the wrapped setup, we do not need a telescope construction for the Floer chain complex (compare Lemma \ref{Lemma HF=HW for closed Lags}).\\
\noindent {\bf Remark.}\,\emph{In the wrapped setup, for non-compact $L_0,L_1$ the maximum principle requires $H_{L_0,L_1}$ to have non-zero slope $\epsilon$ at infinity. Stokes's theorem and $d\gamma\leq 0$ force the output of $\mu^n$ to involve a Hamiltonian of slope at least $n\epsilon$. But for large $n$, $CF^*(L_0,L_1;n H_{L_0,L_1})$ is typically not quasi-isomorphic to $CF^*(L_0,L_1;H_{L_0,L_1})$. So the telescope construction is necessary.}
%
%
\subsection{The bimodule structure}
\label{Subsection The bimodule structure for curly B}

Just as in \ref{Subsection wrapped diagonal bimodule}, $\mu^d$ determines via \ref{Subsection Bimodules} the composition maps
$$\mu_{\bM}^{r|s}: \bA(L_r,\ldots,L_0)\otimes \bM(L_0,L_0') \otimes \bA(L_0',\ldots,L_s') \to \bM(L_r,L_s').$$
 for the diagonal $\bA$-bimodule $\bM=\bA$. Below, the first picture is $\mu_{\bM}^{3|2}:\bA(L_3,L_2,L_1,L_0)\otimes \bM(L_0,L_0') \otimes \bA(L_0',L_1',L_2') \to \bM(L_3,L_2')$, and the bubbling pictures show the $A_{\infty}$-relations.
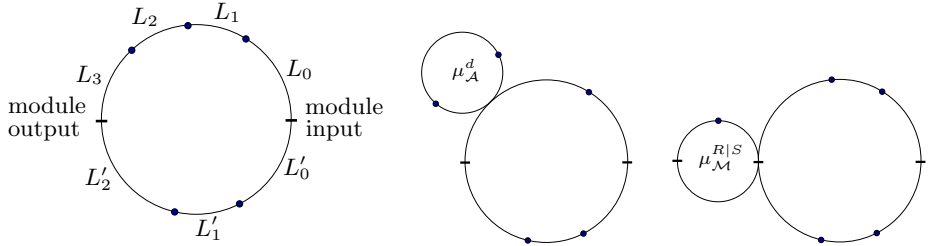
\begin{center}
\input{Bimod_Comp_B1b.tex} \qquad \qquad \input{Ainfinity_B1b.tex}\\
\end{center}
%
\subsection{The open-closed string map}
\label{Subsection open closed map OC_B}
To define $\mathrm{OC}:\mathrm{CC}_*(\bA,\bM)\to QC^*(E)$, we let
$$
\mathrm{OC}^n: \bM(L_0,L_n) \otimes \bA(L_n,\ldots,L_0) \to QC^*(E)
$$
be the count of maps as in the first picture (weighing counts by $t^{E_{\mathrm{top}}(u)}$, with signs as in \ref{Subsection open-closed with q coefficients}):
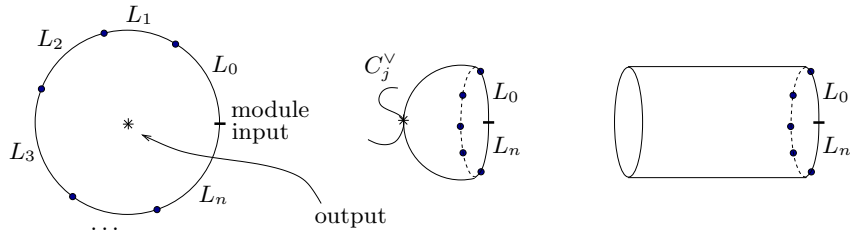
\begin{center}
\input{OC1b.tex}
\end{center}
 The underlying parametrized domains are the same as those used to define $\mu_{\bA}^{n}$ except we have an additional marked point $q_c$ in the interior, which determines the output in $QH^*(E)$.
\begin{remark}\label{Remark cannot pull back data via forgetful map}
We cannot simply pull back the $PSL(2,\R)$-invariant auxiliary data used for $\mu_{\bA}$ via the forgetful map, which forgets the new marked point $q_c$.
Consider the moduli spaces of domains (modulo the $PSL(2,\R)$-reparametrization action, whereas later we will fix this choice by placing $q_c$ at $0$ and the module input at $+1$). If $q_c$ converges to a boundary point which is not a puncture, then a $J$-holomorphic disc bubble appears which is attached to the main component at a boundary node rather than a boundary puncture. This is not consistent with the required breaking, which is a bubble carrying a module input puncture (counted by $\mathrm{OC}^0$).
So given the auxiliary data obtained via the forgetful map, we must deform the data when $q_c$ gets close to a boundary point which is not a puncture (equivalently, if we fix the domain parametrization so that $q_c=0$, this occurs when all boundary punctures converge to the module input puncture). As in \ref{Subsection A comment about the existence of universal choices of auxiliary data}, the universal and consistent construction of auxiliary data then requires the data on those bubbles to be the data inductively chosen for $\mathrm{OC}^0$.
\end{remark}
We will fix the parametrization of the domain, removing all reparametrization freedom, by fixing at $0$ the interior marked point, and at $+1$ the positive puncture which receives the module input from $\bM(L_0,L_n)$. The other $n$ positive boundary punctures can move freely subject to being distinct and ordered; they receive the $\bA(L_n,\ldots,L_0)$ inputs of $\mathrm{OC}^n$.
Thus $\mathrm{OC}^n$ counts isolated maps $u: \D \setminus\{\textrm{punctures}\}\to E$ solving Floer's equation, subject to an intersection condition at $0$.
If all $L_i,L_{i+1}$ were transverse, Floer's equation is simply the $J$-holomorphic equation $du^{0,1}=0$, and we could use boundary marked points rather than boundary punctures. In general, however, as discussed in \ref{Subsection When Lagrangians do not intersect transversely} on each strip-like end of a boundary puncture the equation is of type $(du-X\otimes \gamma)^{0,1}=0$ using the relevant Floer datum, with $\gamma=dt$ near infinity and $\gamma=0$ away from the end.

The intersection condition at $0$ follows the usual routine: we fix a generic choice of intersection-dual cycles $C_j^{\vee}$ in $C_*(E)$ for the basis of lf-cycles $C_j$ that we chose to define $QC^*(E)=QH^*(E)$; then $\mathrm{OC}(\textrm{input})=\sum n_jC_j$ where $n_j$ is the weighted count of solutions above (with boundary conditions determined by the input) which pass through $C_j^{\vee}$ at $0$. The second picture above shows what the image of $u$ looks like in the target manifold $E$.\\[2mm]
\noindent {\bf Technical remark about transversality.} \emph{The condition that the solution ``passes'' through $C_j^{\vee}$ is  an intersection condition between $C_j^{\vee}$ and the evaluation map at the centre of the disk. However, some care is needed because the evaluation does not sweep a pseudo-cycle (the
moduli space has a boundary of real codimension one). If a disk or a sphere bubbles off (simple or not), then by monotonicity the index of the main component drops at least by 2, whilst remaining generically regular. Hence generically that main component will not hit $C_j^{\vee}$.
}\\[2mm]
To define
$$
\mathrm{OC}^n: \bM(L_0,L_n) \otimes \bA(L_n,\ldots,L_0) \to HF^*(H^E)
$$
we modify the above setup by making $0$ a negative puncture (here negative means: with a choice of cylindrical parametrization $(-\infty,0]\times S^1$ near the puncture, where the $S^1$ is parametrized by the argument of $z$ and the puncture corresponds to the limit at $-\infty$). 
The equation on the cylindrical end near the interior puncture is also of type $(du-X\otimes \gamma)^{0,1}=0$, indeed
$(du-X_{H^E}\otimes c(s)dt)^{0,1}=0$, where recall the Hamiltonian $H^E$ at infinity is linear of positive slope. Explicitly this is the Floer continuation equation 
$$\partial_s u + J (\partial_t u - c(s)X_{H^E})=0$$
in this cylindrical parametrization, where $c:\R \to [0,1]$ is a decreasing cut-off function which equals $1$ for $s\ll 0$ and equals $0$ for $s$ close to $0$ (so the equation near $s=0$ is $\partial_s u + J \partial_t u=0$ and thus agrees with the condition that $u$ is $J$-holomorphic away from the punctures). The third picture above is an equivalent way of viewing the map $\mathrm{OC}:\mathrm{CC}_*(\bA,\bM)\to CF^*(H^E)$: the domain is a half-infinite cylinder $(-\infty,0]\times S^1$ with $n+1$ punctures on the circle $\{0\}\times S^1$, one of which is a distinguished puncture at $(0,0)$, and we require the maps to satisfy the above Floer equation, so the left asymptotic circle $\{-\infty\}\times S^1$ plays the role of the original puncture at $0$ from the first picture. 

That $\mathrm{OC}$ is a chain map follows by considering the possible degenerations of the $1$-dimensional moduli spaces: (1) if boundary punctures move together, we obtain $\mu_{\bA}^d$ or $\mu_{\bM}^{r|s}$ contributions; and (2) if the Floer trajectory breaks (on the region where $c(s) \equiv 1$), we obtain contributions described by the Floer differential $\mathfrak{d}: CF^*(H^E) \to CF^{*+1}(H^E)$ (in the above case, when we map $\mathrm{OC}$ into $QC^*(E)$, type (2) breaking would have corresponded to a limiting solution $u$ which intersects the boundary chain $\partial C_j^{\vee}$, but in our setup $C_j^{\vee}$ are cycles so this does not contribute).
Thus we obtain the following two maps on homology:
\begin{equation}\label{Equation two OC maps}
\mathrm{OC}: \mathrm{HH}_*(\bA) \to QH^*(E) \qquad\quad \mathrm{OC}: \mathrm{HH}_*(\bA) \to HF^*(H^E).
\end{equation}
\begin{lemma}\label{Lemma OC compatible with c}
The two maps in \eqref{Equation two OC maps} commute via the PSS map $c^*: QH^*(E) \to HF^*(H^E)$.
\end{lemma}
\begin{proof}
This is a standard gluing argument. At the chain level, a chain homotopy is involved because the glued data from the composite needs to be homotoped to the universal data which defines the second map. A clean way to implement this argument, is to place a new preferred point $\phi_f\in (0,\frac{1}{2}]$ on the geodesic $(0,+1)$ with an associated $\beta_f$ form (as in the wrapped constructions), so that the new $\gamma$ satisfies $\gamma=\beta_f=dt$ near $0$ in the coordinates of the cylindrical end.
As $\phi_f\to 0$ we can ensure that the support of $\beta_f$ itself also converges to $0$ (so $\beta_f=0$ except near $0$), and we can ensure that $\beta_f$ becomes the $S^1$-invariant form that we use to define the PSS map (compare the unitality argument at the end of \ref{Subsection product on SH is associative}).
Consider a $1$-family of Floer solutions on those domains. Suppose $\phi_f\to 0$. Then a Floer continuation cylinder breaks off at $0$ involving Hamiltonians $H^E$ at one end and $0$ at the other ($\beta_f=0$ on the main component). Since the solution has finite energy, a standard removal of singularities argument implies that the Floer continuation solution extends holomorphically over the puncture at the $0$ end, and thus these cylinders are counted by a PSS map $QC^*(E)\to CF^*(H^E)$. The main component in this degeneration has $\beta_f=0$, so these define $\mathrm{OC}:\mathrm{CC}_*(\bA,\bM)\to QC^*(E)$. Other possible degenerations of the $1$-family are: a disc bubble not carrying $\phi_f$ which is counted by the bar differential on $\mathrm{CC}_*(\bA,\bM)$, or a Floer cylinder not carrying $\phi_f$  which breaks off at $0$ which is counted by the Floer differential. In both these cases the main component (which carries $\phi_f$) is a rigid solution on a domain as above, and the count of these rigid solutions is the chain homotopy. If $\phi_f\to \frac{1}{2}$ and no bubbling or breaking occurs, we obtain solutions for the above domains with $\phi_f=\frac{1}{2}$ and this is (one of the possible definitions of) the map $\mathrm{OC}:\mathrm{CC}_*(\bA,\bM)\to CF^*(H^E)$.
\end{proof}
\subsection{The closed-open string map}
\label{Subsection closed open map CO_B}
In the notation of \ref{Subsection Hochschild cohomology},
$$
\mathrm{CO}^d: QC^*(E) \to \prod_{n\geq 0} \prod \mathrm{Hom}(\bA(L_n,\ldots,L_0),\bM(L_n,L_0)[d])
$$
counts Floer solutions (with weight $t^{E_{\mathrm{top}}(u)}$ and sign as in \ref{Subsection closed-open with q coefficients}) on the following domain.

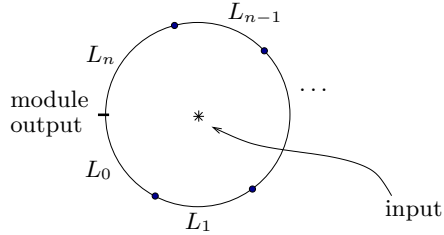
\begin{center}
\input{CO1b.tex}
\end{center}
Except for the following two modifications, this is the same setup as for $\mathrm{OC}$ (in particular, as in Remark \ref{Remark cannot pull back data via forgetful map}, one can pull back the auxiliary data from $\mu_{\bA}$ and then make a perturbation of the data when all boundary punctures converge to the module output).
\begin{enumerate}

\item to define $\mathrm{CO}^d(C_j)$, the intersection condition at the marked point $0$ is now with the lf-cycle $C_j$ (as opposed to the dual cycle $C_j^{\vee}$ used for $\mathrm{OC}$);

\item there is a negative boundary puncture at $-1$ for the module output in $\bM(L_n,L_0)$ (instead of the module input puncture at $+1$  used for $\mathrm{OC}$).
%

\end{enumerate}

The proof that $\mathrm{CO}$ is a chain map is analogous to the $\mathrm{OC}$ case. So
$$
\mathrm{CO}: QH^*(E) \to \mathrm{HH}^*(\bA).
$$
%
%
%
\section{The open-closed and closed-open string maps}
\label{Section open-closed string map}
%
\subsection{The open-closed string map}
\label{Subsection open-closed with q coefficients}
\begin{figure}[h]
\input{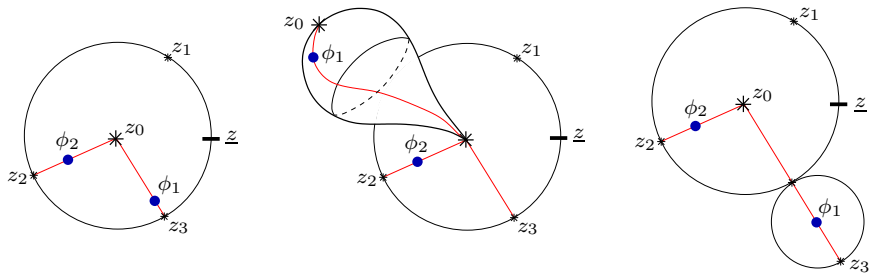}
\caption{Left: $n=3$, $F=\{1,2\}$, $p_1=3$, $p_2=2$. Right: two broken configurations caused because the preferred point $\phi_1$ converged to a puncture.} 
\label{Figure popsicle 2}
\end{figure}
We will use the notation from \ref{Subsection Grading conventions}-\ref{Subsection Bimodules}: we denote the wrapped category and its diagonal bimodule by 
\begin{equation}\label{Eqn curlyA and curlyM}
\bA=\mathcal{W}(E) \qquad \qquad \mathcal{M}=\bA=\mathcal{W}(E).
\end{equation}

As always, $E$ is monotone or exact, and in the monotone setting we always mean $\mathcal{W}_{\lambda}(M)$ for a fixed $m_0$-value $\lambda$. We now define the analogue of \ref{Subsection open closed map OC_B} for $\wE$, $$\mathrm{OC}: \mathrm{CC}_*(\mathcal{A},\mathcal{M}) \to SC^*(E).$$ We first construct
$$
\begin{array}{rl}
\mathrm{OC}^{n,\mathbf{p},\mathbf{w}}: & CF^*(L_n,L_0;w_{n+1}H)[\mathbf{q}] \otimes CF^*(L_{n-1},L_n;w_n H)[\mathbf{q}]\otimes \cdots \otimes CF^*(L_0,L_1;w_1H)[\mathbf{q}] \\
 &
 \longrightarrow CF^*(w_0H)[\mathbf{q}].
\end{array}
$$
%
Observe Figure \ref{Figure popsicle 2}. Consider discs $S$ as in \ref{Subsection the auxiliary data for wrapped category} for $d+1=n+2$ except now, mimicking \ref{Subsection open closed map OC_B}:
\begin{enumerate}[labelsep=*,leftmargin=1.5pc]
\setlength\itemsep{0em}
\item\label{ItemFixParamtr}  We fix the parametrization of the domain by taking $S=\D\setminus\{\textrm{punctures}\}$, $\underline{z}=z_{n+1}=+1$ and $z_0=0$ (compare  \ref{Subsection open closed map OC_B} and \ref{Subsection A comment about domain reparametrizations and moduli}).

\item The last positive boundary puncture $z_{n+1}=\underline{z}$ is distinguished (it receives the module input from $\mathrm{CC}_*(\mathcal{A},\mathcal{M})$). We denote it by a dash in Figure \ref{Figure popsicle 2}.

\item The interior puncture $z_0$ is \emph{negative} (it receives the output from $CF^*(w_0H)[\mathbf{q}]\subset SC^*(E)$), meaning we make a choice $\epsilon_0: (-\infty,0]\times S^1\to S$ of a cylindrical end parametrization, with $s\to -\infty$ corresponding to $z_0$. We fix $\epsilon_0(s,t)=\frac{1}{2}e^{2\pi (s+it)}$ to be the parametrization by polar coordinates of $\{z\in \C: 0<|z|\leq 1/2\}\subset \D\setminus \{\textrm{punctures}\}$. This fixes the $S^1$ parametrization at the asymptote (generators of $CF^*(w_0 H)$ are \emph{parametrized} $1$-orbits).

\item We no longer interpret preferred points $\phi_f$ in terms of maps (compare
 \ref{Subsection the auxiliary data for wrapped category}). The map formalism is not useful when geodesics connect to an interior puncture because the location of a boundary puncture together with an interior puncture fix the disc's reparametrization freedom, leaving no freedom to move the preferred point.

\item The preferred point $\phi_f$ can move along the geodesic connecting $z_0$ to $z_{p_f}$. By the Rotational symmetry trick in \ref{Subsection Ainfinity product on SC} it does not matter at which angle the $\phi_f$ come into $z_0$ (the $\beta_f$ are $S^1$-invariant near $0$). When compactifying the moduli spaces, a cylinder can break off at $0$ carrying preferred points $\phi_f$. By \ref{Subsection Ainfinity product on SC}, the location of each $\phi_f$ on the cylinder is only defined up to rotation, but once we glue the cylinder back onto the main component we fix these rotational-freedoms by placing $\phi_f$ on the geodesic connecting $z_0$, $z_{p_f}$.

\item Just as in Remark \ref{Remark cannot pull back data via forgetful map}, if one uses the auxiliary data pulled back from the data for $\mu_{\bA}$, one must deform this data when all punctures converge to the module input.
Consistent universal auxiliary data is constructed as in \ref{Subsection A comment about the existence of universal choices of auxiliary data}, inductively over strata of the compactification of the above moduli space of decorated domains. On disc bubbles which do not carry the interior puncture $z_0$, we require that the auxiliary data agrees with the data defining $\mu_{\bA}$ and $\mu_{\mathcal{M}}$. On sphere bubbles at $z_0$ we require that the data agrees with that defining $\mu_{SC^*(E)}^1$. (These requirements will ensure that $\mathrm{OC}$ is a chain map.)

\end{enumerate}

Analogously to the $\mathcal{R}^{d+1,\mathbf{p}}(\mathbf{x})$ defined in \ref{Subsection moduli space of pseudoholo maps in wrapped case}, except for the aforementioned modifications, we obtain a moduli space of maps solving $(du-X\otimes \gamma)^{0,1}=0$ which we denote $\mathcal{OC}^{n,\mathbf{p}}(\mathbf{x})$. In particular, note that $(du-X\otimes \gamma)^{0,1}=0$ turns into $\partial_s u + J(\partial_t u - w_0 X) = 0$ via $\epsilon_0$ near $z_0$.

To define $\mathrm{OC}^{n,\mathbf{p},\mathbf{w}}$ we mimic the definition of the $\mu^{n,\mathbf{p},\mathbf{w}}$ in \ref{Subsection wrapped A infinity structure}.
For the same symmetry reasons \cite[Lemma 3.7]{Abouzaid-Seidel}, there is a cancellation of the counts of solutions when more than one preferred point lies on a geodesic, so we only need to keep track of whether there is  or there isn't a preferred point on each geodesic (encoded by $\mathbf{q}^1$, $\mathbf{q}^0$).
 First we determine the ${\mathbf{q}}^0$ coefficient of the image. 
An isolated $u\in \mathcal{OC}^{n,\mathbf{p}}(x_0,x_1,\ldots,x_n,\underline{x})$ will contribute $\pm t^{E_{\mathrm{top}}(u)} x_0$ to  $\mathrm{OC}^{n,\mathbf{p},\mathbf{w}}({\mathbf{q}}^{i_{n+1}}\underline{x} \otimes {\mathbf{q}}^{i_{n}}x_n \otimes \cdots \otimes {\mathbf{q}}^{i_1}x_1)$ where $i_k=1$ for $k\in \mathbf{p}(F)$, $i_k=0$ for $k\notin \mathbf{p}(F)$, and it will contribute $0$ otherwise (for signs mimic \cite[Sec.3.8]{Abouzaid-Seidel}, see below). The ${\mathbf{q}}^1$ coefficient of the image is determined by the requirement that the map $\mu^{d,\mathbf{p},\mathbf{w}}$ respects the $\partial_{\mathbf{q}}$ operator.\\[1mm]
\indent Finally, sum up the $\mathrm{OC}^{n,\mathbf{p},\mathbf{w}}$ as $\mathbf{p},\mathbf{w}$ vary to obtain:
$$
\mathrm{OC}^n : \mathcal{M}(L_0,L_n)\otimes \bA(L_n,\ldots,L_0) \to  SC^*(E)
$$
\noindent\textbf{Example.}\;In the first picture of Figure \ref{Figure popsicle 2}, after relabeling $\phi_1,\phi_2$ to $\phi_3,\phi_2$ (as we only sum over inclusions $\mathbf{p}:F\subset \{1,\ldots,n+1\}$), an isolated solution $u$ with asymptotic conditions $x_j$ at $z_j$ will contribute $\pm t^{E(u)} x_0\in CF^*(w_0H)\subset SC^*(E)$ to  $\mathrm{OC}^3(\underline{x}\otimes \mathbf{q}x_3\otimes \mathbf{q}x_2 \otimes x_1)$.\\[1mm]
\noindent {\bf Remark.} \emph{The sum is weighted by signs \cite[(5.24)]{Abouzaid} like in \ref{Subsection wrapped A infinity structure}: $\mathrm{OC}^{n,\mathbf{p},\mathbf{w}}({\mathbf{q}}^{i_{n+1}}\underline{x} \otimes {\mathbf{q}}^{i_{n}}x_n \otimes \cdots \otimes {\mathbf{q}}^{i_1}x_1)$ is weighted $(-1)^{\dagger\dagger}$ where $$\textstyle \dagger\dagger=(n+1)(\mathrm{deg}(\underline{x})+1)+\sum_{j=1}^{n} j (\|x_j\|+1) + \sum_{j\in F} (\sigma(x)_{j+1}^{n} +\mathrm{deg}(\underline{x}))$$ (the second sum is the Koszul sign expected from viewing $\mathbf{q}$ as an operator of degree $-1$ acting from the left, the first sum is as in \cite[(5.24)]{Abouzaid}, and $\mathrm{deg}(\underline{x})=\|\underline{x}\|+1$ is unreduced).}

\begin{lemma}
 $\mathrm{OC}$ is a chain map, so it induces $\mathrm{OC}: \mathrm{HH}_*(\wE) \to SH^*(E)$.
\end{lemma}
 \begin{proof}[Sketch Proof]
  The proof that $\mathrm{OC}$ is a chain map is similar to the proof of the $A_{\infty}$-relations for $\mu_{\bA}$ (see Remarks 3 and 4 in 
\ref{Subsection wrapped A infinity structure}). 
We illustrate the argument in a concrete example. Consider picture 1 in Figure \ref{Figure popsicle 2}. If $\phi_1$ moves towards $z_3$, the limit (picture 3) contributes to $\mathrm{OC}^3 (\underline{x}\otimes (\mathfrak{K}\partial_{\mathbf{q}})(\mathbf{q}x_3) \otimes \mathbf{q}x_2 \otimes x_1)$ which is a part of $\mathrm{OC}^3\circ (\underline{\mathrm{id}}\otimes \mu^1_{\bA} \otimes \mathrm{id}^{\otimes 2})$. If $\phi_1$ moves towards $z_0$, in the limit (picture 2) a $2$-punctured sphere bubbles off with $\phi_1$ lying on the geodesic joining the two punctures. This breaking contributes to $\mu^1_{\mathrm{SC^*(E)}}(\mathrm{OC}^3(\underline{x}\otimes x_3 \otimes \mathbf{q}x_2 \otimes x_1))$. Now consider the $1$-dimensional components of the compactified moduli space $\mathcal{OC}^{n,\mathbf{p}}(\mathbf{x})$. The oriented sum of the possible boundary degenerations sums up to zero; this corresponds to the equation $\mathrm{OC}\circ b = \mu^1_{SC^*(E)} \circ \mathrm{OC}$ where $b$ is the bar differential. 
 \end{proof}
 
\subsection{The closed-open string map}
\label{Subsection closed-open with q coefficients}
\begin{figure}[h]
\input{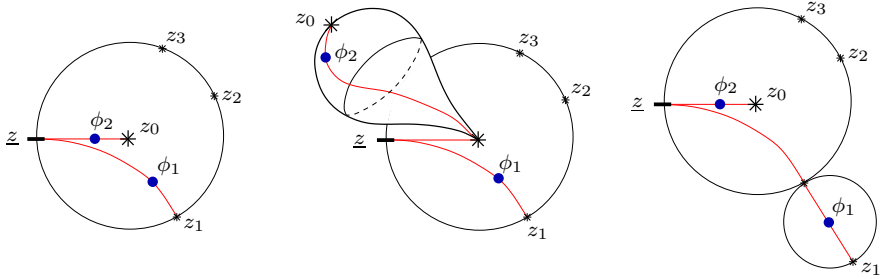}
\caption{Left: $n=3$, $F=\{1,2\}$, $p_1=1$, $p_2=0$. Right: two broken configurations in which a preferred point converged to a puncture.} 
\label{Figure popsicle 2CO}
\end{figure}
We now define the analogue of \ref{Subsection closed open map CO_B} for $\wE$,
\begin{equation}\label{Eqn CO from SC to CC bA M}
\mathrm{CO}: SC^*(E) \to \mathrm{CC}^*(\bA,\mathcal{M}).
\end{equation}
 We first construct
$$
\begin{array}{l}
\mathrm{CO}^{n,\mathbf{p},\mathbf{w}}:  CF^*(w_0H)[\mathbf{q}] \longrightarrow\\ \;
\qquad
\mathrm{Hom}\,(\,
CF^*(L_{n-1},L_n;w_n H)[\mathbf{q}]\otimes \cdots \otimes CF^*(L_0,L_1;w_1H)[\mathbf{q}],
\;\;
CF^*(L_n,L_0;w_{n+1}H)[\mathbf{q}]\,)
\end{array}
$$
This involves discs $S$ as in \ref{Subsection the auxiliary data for wrapped category} for $d+1=n+2$ except now mimicking \ref{Subsection closed open map CO_B}:

\begin{enumerate}[labelsep=*,leftmargin=1.5pc]
\setlength\itemsep{0em}

\item\label{ItemFixParamtr2}  We fix the parametrization of the domain by taking $S=\D\setminus\{\textrm{punctures}\}$, $\underline{z}=z_{n+1}=-1$, $z_0=0$ (compare  \ref{Subsection closed open map CO_B} and \ref{Subsection A comment about domain reparametrizations and moduli}).

\item $\underline{z}=-1$ is a \emph{negative} boundary puncture, and it is distinguished (it receives the module output from $CF^*(L_n,L_0;w_{n+1}H)[\mathbf{q}]$). We denote it by a dash in Figure
\ref{Figure popsicle 2CO}. The other boundary punctures, $z_1,\ldots,z_n$, are positive.

\item The interior puncture $z_0$ is \emph{positive} (it receives the $1$-orbit input from $CF^*(w_0H)[\mathbf{q}]\subset SC^*(E)$), so we make a choice $\epsilon_0: [0,\infty)\times S^1\to S$ of a cylindrical end parametrization, with $s\to +\infty$ corresponding to $z_0$. We fix $\epsilon_0(s,t)=\frac{1}{2}e^{2\pi (-s+it)}$ to be the parametrization by polar coordinates of $\{z\in \C: 0<|z|\leq 1/2\}\subset \D\setminus \{\textrm{punctures}\}$. This fixes the $S^1$ parametrization at the asymptote (generators of $CF^*(w_0 H)$ are \emph{parametrized} $1$-orbits).
To $z_0$ we associate a closed form $\alpha_0$, and the subclosed forms $\beta_f$ satisfying $p_f=0$.

\item By the Rotational symmetry trick (\ref{Subsection Ainfinity product on SC}) it does not matter that the preferred points on the geodesic $(-1,0)$ connecting $\underline{z},z_0$ do not line up with the $t=0$ direction (see \ref{Subsection open-closed with q coefficients}). 

\item  Just as in Remark \ref{Remark cannot pull back data via forgetful map}, if one uses the auxiliary data pulled back from the data for $\mu_{\bA}$, one must deform this data when all punctures converge to the module output. Analogously to \ref{Subsection open-closed with q coefficients}, one constructs consistent universal auxiliary data following \ref{Subsection A comment about the existence of universal choices of auxiliary data}, by extending the auxiliary data prescribed on bubbles counted by $\mu_{\bA}$, $\mu_{\mathcal{M}}$, $\mu_{SC^*(E)}^1$.

\item For $p_f\neq 0$, so preferred points on the geodesics connecting $\underline{z}$ to $z_{p_f}$, the corresponding map (compare \ref{Subsection open closed map OC_B}) is $\phi_f: S\to \D\setminus \{-1,1\}$, $\phi_f(\underline{z})=-1$, $\phi_f(z_{p_f})=1$, and we identify the map with the associated preferred point $\phi_f(0)$. 
The $\phi_f\in \D$ belong to the open unit disc with the puncture $z_0=0$ filled in.
That a collision between $\phi_f$ and $z_0$ is not a problem is implicitly concealed in the construction of the universal data, as $\beta_f$ is required to vanish at $z_0$ (recall $\gamma=w_0\, dt$ near $z_0$), so the support of $d\beta_f$ will not touch $z_0$.

\end{enumerate}

The analogous discussion as in \ref{Subsection open-closed with q coefficients} holds for $\mathrm{CO}$. Namely it involves taking a signed sum over the choices of $\mathbf{p},\mathbf{w}$, then extending linearly (since $SC^*(E)=\oplus CF^*(w_0H)[\mathbf{q}]$), and then taking the product of all $\mathrm{Hom}$ spaces as we let $n$ and the $L_j$ vary. This yields \eqref{Eqn CO from SC to CC bA M}, and thus
\begin{equation}\label{Eqn CO on cohomology}
\mathrm{CO}: SH^*(E) \to \mathrm{HH}^*(\wE).
\end{equation}
\noindent {\bf Remark.} \emph{As in \ref{Subsection open-closed with q coefficients}, the sum over $\mathbf{p},\mathbf{w}$ involves signs: $\mathrm{CO}^{n,\mathbf{p},\mathbf{w}}({\mathbf{q}}^{i_{0}}x_0)({\mathbf{q}}^{i_{n}}x_n \otimes \cdots \otimes {\mathbf{q}}^{i_1}x_1)$ is weighted $(-1)^{\star}$ where $\star=\sum_{j=1}^{n} j (\|x_j\|+1) + \sum_{j\in F} \sigma(x)_{j+1}^{n}$.}\\[1mm]
\indent The $0$-part of \eqref{Eqn CO from SC to CC bA M}, for $L\in \mathrm{Ob}(\wE)$
is
$
\mathrm{CO}^0:SC^*(E)\to \mathrm{Hom}(\K, CW^*(L,L)).
$
The $\mathrm{CC}^*$-differential $\delta$ on those homomorphisms is just the differential $\mu_{\mathcal{M}}^{0|0}=-\mu^1_{\bA}$ acting on $CW^*(L,L)$. Thus, evaluating at $1\in \K$, the map $\mathrm{CO}^0(\cdot)(1)$ defines
\begin{equation}\label{Eqn COzero from SH to HW}
\mathrm{CO}^0: SH^*(E) \to HW^*(L,L).
\end{equation}
\begin{remark}\label{Remark CO from QH to wrapped} One can also define a map
\begin{equation}\label{Eqn QH to HH WE}
\mathrm{CO}: QH^*(E)\to \mathrm{HH}^*(\bA) 
\end{equation}
by composing $c^*:QH^*(E)\to SH^*(E)$ with $\mathrm{CO}:SH^*(E)\to \mathrm{HH}^*(\bA)$. In particular, the $0$-part $\mathrm{CO}^0=\mathrm{CO}^0(\cdot)(1)$ of \eqref{Eqn QH to HH WE} yields
\begin{equation}\label{Eqn CO0 from QH to HW}
\mathrm{CO}^0: QH^*(E) \to HW^*(L,L).
\end{equation}
 Arguing as in the proof of Lemma \ref{Lemma OC compatible with c}, \eqref{Eqn QH to HH WE} can be defined at the chain level as before, except:
\begin{enumerate}
\item we replace the interior puncture at $z_0=0$ by an interior marked point, with lf-cycle intersection condition given by the $QC^*(E)$ input (compare \ref{Subsection closed open map CO_B}),
\item we place a \textrm{fixed} preferred point $\phi_{\textrm{fixed}}$ on the geodesic $(-1,0)$, say at $\phi_{\textrm{fixed}}=-1/2$ (the choice of location of $\phi_{\textrm{fixed}}$ and the choice of associated subclosed form $\beta_{\textrm{fixed}}$, will not matter up to chain homotopy, by a standard homotopy argument).
\end{enumerate}
The second condition is necessary, otherwise \eqref{Eqn CO0 from QH to HW} would have output weight zero, which is not allowed.
The role of $\phi_{\textrm{fixed}}$ is essentially analogous to the Floer datum for the compact category $\scrF(E)$, which is necessary when defining $\mathrm{CO}^0: QH^*(E)\to HF^*(L,L)$ for $L\in \mathrm{Ob}(\scrF(E))$.
\end{remark}
%
%
\section{The acceleration functor}
\label{Section acceleration functor}
\subsection{The $A_{\infty}$-categories $\mathbf{\mathcal{F}_{\Diamond}(E)},\mathbf{\mathcal{W}_{\Diamond}(E)}$ and the cochain complex $\mathbf{SC_{\Diamond}^*(E)}$}
\label{Subsection Category W0E and S0CE}

We will now construct an $A_{\infty}$-functor $\mathcal{AF}:\scrF(E) \to \wE$ from the compact category to the wrapped category. On objects, $\mathcal{AF}(L)=L$ is the inclusion of the compact Lagrangians $L\in \mathrm{Ob}(\scrF(E))\subset \mathrm{Ob}(\wE)$.
On morphism spaces there is no obvious inclusion: we cannot in general allow weight $w=0$ in Definition \ref{Definition wrapped floer complex}, because the Floer datum approach of \ref{Subsection When Lagrangians do not intersect transversely} does not work for non-compact Lagrangians.
The key idea will be to allow weight $w=0$ only when $L_0,L_1$ are compact Lagrangians, which is all we need to define $\mathcal{AF}$.

The $A_{\infty}$-category $\mathcal{W}_{\Diamond}(E)$ has the same objects as $\mathrm{Ob}(\wE)$ and the same morphism spaces $CW^*(L_0,L_1)$ as $\wE$ except when $L_0,L_1$ are both compact, in which case 
  $$\boxed{\begin{array}{rcl}\mathrm{hom}_{\mathcal{W}_{\Diamond}(E)}(L_0,L_1)&=&CF^*(L_0,L_1)[\mathbf{q}] \,\oplus\, CW^*(L_0,L_1)\\[1mm]
&=&   CF^*(L_0,L_1)[\mathbf{q}] \,\oplus\, \bigoplus_{w=1}^{\infty} CF^*(L_0,L_1;wH)[\mathbf{q}].\end{array}}$$ 
  The $A_{\infty}$-operations are defined by the same disc counts as in the construction of $\wE$, with the novelty that if one of the boundary punctures carries an input from $CF^*(L_0,L_1)[\mathbf{q}]$ then no closed $1$-form $\alpha$ is associated to that puncture (see \ref{Subsection the auxiliary data for wrapped category}), so $(du-X\otimes \gamma)^{0,1}=0$ turns into the $J$-holomorphic equation $du^{0,1}=0$ near that puncture when $L_0,L_1$ are transverse (since the solution has finite energy it extends continuously over the puncture, so that puncture can be replaced by a marked point receiving as input an intersection point from $L_0\cap L_1$). When $L_0,L_1$ are not transverse then the equation becomes $(du-X_{H_{L_0,L_1}}\otimes c(s)dt)^{0,1}=0$ on the strip-like end of that puncture, using the Floer datum machinery of \ref{Subsection When Lagrangians do not intersect transversely}.

To extend the differential $\mu^1$ of Definition \ref{Definition wrapped floer complex} to $CF^*(L_0,L_1)[\mathbf{q}]$, we need to define
$$
\mathfrak{d}: CF^*(L_0,L_1) \to CF^{*+1}(L_0,L_1) \quad\textrm{ and }\quad
\mathfrak{K}: CF^*(L_0,L_1) \to CF^*(L_0,L_1;H).
$$
The first (following \ref{Subsection the floer differential}) is precisely the Floer differential as defined for the compact Fukaya category (in particular, using a Floer datum when $L_0,L_1$ are not transverse, see \ref{Subsection When Lagrangians do not intersect transversely}). The second (following \ref{Subsection kappa map}) 
counts finite-energy Floer continuation solutions on the strip $\R \times [0,1]$, where the homotopy $H_s$ interpolates $H_s=H$ for $s\ll 0$ 
with $H_s=0$ for $s\gg 0$ (this becomes $H_s=H_{L_0,L_1}$ for $s\gg 0$ using the Floer datum, when $L_0,L_1$ are not transverse). In the transverse setup, the solutions for the $\mathfrak{K}$ map extend continuously over the puncture at $s=+\infty$. We remark that on cohomology, the map $$\mathfrak{K}:HF^*(L_0,L_1) \to HF^*(L_0,L_1;H)$$ is the analogue of the PSS map for open strings (compare \ref{Subsection PSS}).

Let $\mathcal{F}_{\Diamond}(E)$ be the full subcategory of $\mathcal{W}_{\Diamond}(E)$ whose objects are the compact Lagrangians. So $\mathcal{F}_{\Diamond}(E)$ enlarges the morphism spaces of $\scrF(E)$ from $CF^*(L_0,L_1)$ to
$\mathrm{hom}_{\mathcal{W}_{\Diamond}}(L_0,L_1)$.

Similarly, enlarge $SC^*(E)$ from \ref{Subsection symplectic cohomology} to
$$\boxed{
\begin{array}{rcl}
SC_{\Diamond}^*(E)&=&QC^*(E)[\mathbf{q}] \,\oplus\, SC^*(E)\\[1mm] &=&
QC^*(E)[\mathbf{q}]\,\oplus\, \bigoplus_{w=1}^{\infty} CF^*(wH)[\mathbf{q}].
\end{array}}
$$
We can extend the product $\mu^2$, by allowing an insertion from $QC^*(E)$ at a puncture: this means that no $\alpha$ form is associated to that puncture, so the Floer equation becomes holomorphic near that puncture, and as we only count finite-energy solutions this implies that the solution extends holomorphically over that puncture. In particular, if both of the two positive punctures defining $\mu^2$ involve $QC^*(E)$-inputs, and there are no preferred geodesics, then $\gamma=0$ so $\mu^2$ counts holomorphic spheres and coincides with the quantum product of the $QC^*(E)$-inputs.

To extend $\mu^1=\mu^1_{SC^*(E)}$ defined in \eqref{Eqn differential for big SCE} to $QC^*(E)[\mathbf{q}]$, we again need to define
$$
\mathfrak{d}: QC^*(E) \to QC^{*+1}(E) \quad\textrm{ and }\quad
\mathfrak{K}: QC^*(E) \to CF^*(H).
$$
Let $\mathfrak{d}$ be the differential for the chain complex $QC^*(E)$: in our convention, we pick locally finite cycles representing a basis for $H_*^{lf}(E)\cong H^{2n-*}(E)$, so $\mathfrak{d}=0$. Let $\mathfrak{K}$ be the Floer continuation map for the cylinder $\R \times S^1$, with $H_s=H$ for $s\ll 0$ and $H_s=0$ for $s\gg 0$, with the lf-cycle intersection condition at $s=+\infty$ (finite-energy implies the solution holomorphically extends over the puncture at $s=+\infty$). On cohomology, $\mathfrak{K}$ is the PSS map from \ref{Subsection PSS},
$$
\mathfrak{K}: QH^*(E) \to HF^*(H).
$$
The natural inclusion on subcomplexes $SC^*(E)\to SC_{\Diamond}^*(E)$ yields an isomorphism
\begin{equation}\label{Eqn iso from SH to SH Diamond}
SH^*(E)\to SH_{\Diamond}^*(E)
\end{equation}
on cohomology. This is because the telescope construction corresponds to a direct limit on cohomology, so we can compare cocycles for large slopes (indeed slope $w\geq 1$ suffices), but in that case the two complexes coincide.
There is a natural map
\begin{equation}\label{Eqn c from QH to SH Diamond}
c^*: QH^*(E) \to SH_{\Diamond}^*(E)
\end{equation}
induced by the inclusion of the subcomplex $QC^*(E)\to SC_{\Diamond}^*(E)$, and our previous comments about the $\mu^2$ map imply that $c^*$ is an algebra homomorphism. On cohomology $c^*$ is unital since $c^*[E]$ is the unit for $SH_{\Diamond}^*(E)$ 
(by the same proof as in \ref{Subsection product on SH is associative}). Also, \eqref{Eqn c from QH to SH Diamond} factorizes as
$
c^*: QH^*(E) \to SH^*(E) \to SH_{\Diamond}^*(E)
$
 through the usual $c^*$ map and the isomorphism \eqref{Eqn iso from SH to SH Diamond} (at the chain level, \eqref{Eqn c from QH to SH Diamond} and that factorization are homotopic via the homotopy $x\mapsto x\, \mathbf{q}$).

Finally, one can extend the string maps compatibly with the above enlargements,
$$
\mathrm{OC}:\mathrm{HH}_*(\mathcal{W}_{\Diamond}(E))\to SH_{\Diamond}^*(E) \quad \textrm{ and } \quad
\mathrm{CO}: QH^*(E)\to SH_{\Diamond}^*(E) \to \mathrm{HH}^*(\mathcal{W}_{\Diamond}(E)),
$$
by allowing zero-weight insertions for pairs of compact Lagrangians $L_i,L_{i+1}$, using the Floer datum \ref{Subsection When Lagrangians do not intersect transversely} when $L_i,L_{i+1}$ are not transverse. To define $\mathrm{CO}:SH_{\Diamond}^*(E) \to \mathrm{HH}^*(\mathcal{W}_{\Diamond}(E))$ we use  Remark \ref{Remark CO from QH to wrapped}: we always place a fixed preferred point $\phi_{\textrm{fixed}}$ at $-1/2$, whenever $z_0=0$ receives an lf-cycle intersection condition from $QC^*(E)$. We still allow a (free) preferred point to lie on $(-1,0)$; this does not cause symmetry cancellations because $\phi_{\textrm{fixed}}$ is fixed. 
%
\subsection{The acceleration functor $\mathbf{\mathcal{AF}}$}
\label{Subsection Acceleration functor}

Define the $A_{\infty}$-functor
$$f:\wE \to \mathcal{W}_{\Diamond}(E)$$
to be the identity on objects, the inclusion on morphisms, and zero on higher order tensors. 

\begin{lemma}\label{Lemma the inclusion of W0}
 $f$ is a quasi-isomorphism (analogously $\mathcal{F}(E) \to \mathcal{F}_{\Diamond}(E)$ is a quasi-isomorphism).
\end{lemma}
\begin{proof}
$f$ is an isomorphism at cohomology because the wrapped cohomology is a direct limit over the weights $w\to \infty$ (see \ref{Subsection the wrapped floer complex}), so it does not matter if the weights start at $w=0$ or $w=1$. The analogous claim for the inclusion $\mathcal{F}(E) \to \mathcal{F}_{\Diamond}(E)$ follows by Lemma \ref{Lemma HF=HW for closed Lags}.
\end{proof}
\indent
Since $A_{\infty}$-functors which are quasi-isomorphisms can be inverted up to homotopy 
\cite[Sec.(1i)]{Seidel}, we can pick an inverse $f^{-1}:\mathcal{W}_{\Diamond}(E) \to \wE$. Composing the natural inclusion $\scrF(E)\to \mathcal{F}_{\Diamond}(E) \to \mathcal{W}_{\Diamond}(E)$ (inclusion on objects and morphisms, zero on higher order tensors) with $f^{-1}$ yields an $A_{\infty}$-functor, which we call the \emph{acceleration functor},
$$
\mathcal{AF}: \scrF(E) \to \mathcal{F}_{\Diamond}(E) \to \mathcal{W}_{\Diamond}(E) \to \wE.
$$
By Lemma \ref{Lemma change of rings}, this induces chain maps on Hochschild homology, so:
$$
\mathrm{HH}_*(\mathcal{AF}): \mathrm{HH}_*(\scrF(E))\to \mathrm{HH}_*(\mathcal{F}_{\Diamond}(E)) \to \mathrm{HH}_*(\mathcal{W}_{\Diamond}(E)) \to \mathrm{HH}_*(\wE)
$$
where the first and last maps are in fact isomorphisms by Theorem \ref{Theorem Morita invariance} and
Lemma \ref{Lemma the inclusion of W0}.

Hochschild cohomology is not in general covariantly functorial  (see \ref{Subsection Functorial properties}). Nevertheless, $\mathcal{F}_{\Diamond}(E) \to \mathcal{W}_{\Diamond}(E)$ is the inclusion of a full subcategory, so by Lemma \ref{Lemma HH under fully faithful functors} it induces a restriction map $\mathrm{HH}^*(\mathcal{W}_{\Diamond}(E)) \to \mathrm{HH}^*(\mathcal{F}_{\Diamond}(E))$. By Theorem \ref{Theorem Morita invariance} and
Lemma \ref{Lemma the inclusion of W0}, we obtain:
$$
\mathrm{HH}^*(\mathcal{AF}):
\mathrm{HH}^*(\wE) \cong \mathrm{HH}^*(\mathcal{W}_{\Diamond}(E)) \to \mathrm{HH}^*(\mathcal{F}_{\Diamond}(E)) \cong \mathrm{HH}^*(\scrF(E)).
$$
%
%
\subsection{The acceleration diagram}
\label{Subsection the acceleration diagram}
%
\begin{theorem}\label{Theorem acceleration commutative diagram}
There are commutative diagrams of linear maps
$$
\xymatrix@C=50pt{ \mathrm{HH}_*(\scrF(E)) \ar@{->}^{\mathrm{HH}_*(\mathcal{AF})}[r] \ar@{->}_{\mathrm{OC}}[d] & \mathrm{HH}_*(\wE)
\ar@{->}^{\mathrm{OC}}[d] 
%
%
&
\mathrm{HH}^*(\mathcal{F}_{\Diamond}(E)) \ar@{<-}^{\mathrm{HH}^*(\mathcal{AF})}[r] \ar@{<-}_{\mathrm{CO}}[d] & \mathrm{HH}^*(\mathcal{W}_{\Diamond}(E))
\ar@{<-}^{\mathrm{CO}}[d]
\\
 QH^*(E)
\ar@{->}[r]^-{c^*} & SH^*(E)
%
%
&
 QH^*(E)
\ar@{->}[r]^-{c^*} \ar@{->}[ur]^-{\mathrm{CO}}& SH_{\Diamond}^*(E)
}
$$
In \ref{Subsection Acceleration diagram respects products} we prove that all maps are $QH^*(E)$-module homomorphisms, and that in the diagram on the right all maps are unital algebra homomorphisms.
\end{theorem}
\begin{proof}
For the first diagram, we prove that the following diagram commutes,
$$
\xymatrix@C=50pt{ \mathrm{HH}_*(\scrF(E)) \ar@{->}^{\mathrm{HH}_*(\mathrm{incl})}[r] \ar@{->}_{\mathrm{OC}}[d] & \mathrm{HH}_*(\mathcal{W}_{\Diamond}(E))  \ar@{->}_{\mathrm{OC}}[d] \ar@{<-}^{\mathrm{HH}_*(f)}[r] & \mathrm{HH}_*(\wE)
\ar@{->}_{\mathrm{OC}}[d] \\
 QH^*(E)
\ar@{->}[r]^-{c^*} & SH_{\Diamond}^*(E) \ar@{<-}[r]^-{\mathrm{incl}} & SH^*(E)
}
$$
The square on the right tautologically commutes since $f$ is the inclusion. The left square commutes at the chain level because the extension of the $\mathrm{OC}$ map to $\mathcal{W}_{\Diamond}(E)$ has been precisely defined so as to coincide with the definition of $\mathrm{OC}$ for the compact category.
Indeed $\mathrm{CC}_*(\mathcal{F}(E))\subset \mathrm{CC}_*(\mathcal{W}_{\Diamond}(E))$ is a subcomplex, and since no $\mathbf{q}$-terms appear in this subcomplex, the domains involved in the $\mathrm{OC}$ map do not carry preferred points. Therefore there are no $\beta_f$ forms. By construction, when inputs come from $CF^*(L_i,L_{i+1})$ there are also no $\alpha_i$ forms. So the total form used is $\gamma=0$ (or a perturbation datum, when non-transverse $L_i,L_{i+1}$ arise, in which case $X\otimes\gamma$ vanishes outside of a compact subset of $E$). Thus the output weight is $w_0=0$, so the solutions land in the subcomplex $QC^*(E)\subset SC_{\Diamond}^*(E)$.

The second diagram commutes by construction. Note that  $\mathrm{HH}^*(\mathcal{AF})\circ \mathrm{CO}\circ c^*$ takes values in $CF^*(L_0,L_n;H)\subset \mathrm{hom}_{\mathcal{F}_{\Diamond}(E)}(L_0,L_n)$ rather than $CF^*(L_0,L_n)\subset \mathrm{hom}_{\mathcal{F}(E)}(L_0,L_n)$, due to the fixed preferred point, as explained at the end of \ref{Subsection Category W0E and S0CE}; these values agree with the values of $\mathrm{CO}:QC^*(E)\to \mathrm{CC}^*(\mathcal{F}_{\Diamond}(E))$ up to the quasi-isomorphism $x\mapsto \mathfrak{K}(x)$.
\end{proof}
\begin{remark}
The second diagram can be rephrased in terms of the $\psi$-string map of
\ref{Subsection Acceleration diagram respects products}. The following is a commutative diagram of unital algebra homomorphisms.
$$
\xymatrix@C=50pt{ 
&
\mathrm{End}(\mathcal{F}_{\Diamond}(E)) \ar@{<-}^{\mathrm{restriction}}[r] \ar@{<-}_{\psi}[d] & \mathrm{End}(\mathcal{W}_{\Diamond}(E))
\ar@{<-}^{\psi}[d]
\\
&
 QH^*(E)
\ar@{->}[r]^-{c^*} \ar@{->}[ur]^-{\psi}& SH_{\Diamond}^*(E)
}
$$
\end{remark}
%
\section{Algebraic structures}
\label{Section algebraic structures}
%
%
%
\subsection{The $\psi$-string map}
\label{Subsection psi structure map}
\label{Subsection Comparison of Ganatra}
We use the notation $\bA=\mathcal{W}(E)$, $\mathcal{M}=\mathcal{W}(E)$ from \eqref{Eqn curlyA and curlyM}. Our discussion applies equally well to the compact category, with obvious simplifications (\ref{The psi structure for the compact Fukaya category}).
%

\begin{theorem}\label{Theorem unital alg hom QH to End M}
There is a unital algebra homomorphism
$$
SH^*(E) \to \mathrm{End}(\bM) \to \mathrm{End}(\mathrm{HH}_*(\bM)), \;\; c \mapsto \mathrm{HH}_*(\psi_c),
$$
defining a module action of $SH^*(E)$ on $\mathrm{HH}_*(\bM)$.
\end{theorem}

By Corollary \ref{Corollary HHomology module over bimodule endos}, we just need to construct a linear map
$
SH^*(E)\to \mathrm{End}(\bM)$, $c\mapsto \psi_c,
$
which is a unital algebra homomorphism. We also claim:
\begin{theorem}\label{Theorem CO is unital alg hom}\label{Theorem commutative diagram compared to Ganatra}
There is a commutative diagram of unital algebra homomorphisms
$$
\xymatrix@C=50pt{ QH^*(E) \ar@{->}^-{c^*}[r] \ar@{->}_-{\mathrm{CO}}[rrd]  & SH^*(E) \ar@{->}^-{\psi}[r]  \ar@{->}^-{\mathrm{CO}}[rd] & \mathrm{End}(\mathcal{M})
\ar@{->}^-{\mathrm{HH}_*(\cdot)}[r] & \mathrm{End}(\mathrm{HH}_*(\mathcal{M})) \\
& & \mathrm{HH}^*(\mathcal{M}) \ar@{->}_-{\strut\textrm{cap product}}[ru] \ar@{->}_{\mu_{\mathcal{M}}\circ_{\mathrm{left}}}[u] &
}
$$
using composition product for $\mathrm{End}(\bM)$ and $\mathrm{End}(\mathrm{HH}_*(\bM))$, cup product for $\mathrm{HH}^*(\bM)$ as in \ref{Subsection Product on Hochschild cohomology}, the vertical map is the algebra homomorphism in \ref{Subsection Hochschild cohomology as endomorphisms of the diagonal bimodule}, and see \ref{Subsection Cap product} for the cap product.

The $A_{\infty}$-category $\mathcal{A}$ is cohomologically unital: the cohomological units are $$e_L=\mathrm{CO}^0(c^*[E])\in HW^*(L,L).$$
By Corollary \ref{Corollary HCohom to Hhom}, it follows that
 $\mathrm{CO}(c^*[E])\in \mathrm{HH}^*(\mathcal{A})$ is a unit for the cup product and the vertical map above is an isomorphism. In particular, the $0$-part $
\mathrm{CO}^0: SH^*(E)\to HW^*(L,L)
$ (see \ref{Subsection closed-open with q coefficients})
is a unital algebra homomorphism, using $\mu^2_{\wE}$ on $HW^*(L,L)$.
\end{theorem}
%
\subsection{The construction of $\psi$}
\label{Subsection The construction of psi}
Consider the following domain $S$,
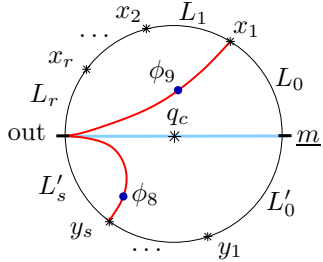
\begin{center}
\input{psi_structure2.tex}
\end{center}
which is the same domain and auxiliary data as for $\mu_{\mathcal{M}}^{*|*}$ (see \ref{Subsection wrapped diagonal bimodule}),
except we also have a positive interior puncture at $q_c=0$ (compare \ref{Subsection closed-open with q coefficients}). Associated to $c$ there is a closed form $\alpha_c$, and any $\beta_f$ form with $p_f=q_c$ (i.e. if the preferred point $\phi_f$ lies on the new geodesic $(-1,0)$).\\
\noindent {\bf Remark.} \emph{We fixed the $PSL(2,\R)$-parametrization of the disc by fixing punctures at $-1,0,+1$. If one wanted to work with moduli of domains, one would allow the module input/output boundary punctures to move as well, and $q_c$ would be allowed to move on the hyperbolic geodesic connecting those two boundary punctures.}\\[1mm]
Given a $1$-orbit $\widetilde{c}$ for $w_c H$, for some integer weight $w_c\geq 1$, consider
$$
c=\widetilde{c}\,\mathbf{q}^{i_c}
\in
CF^*(w_cH)[\mathbf{q}]\subset SC^*(E)=\bigoplus_{w=1}^{\infty} CF^*(wH)[\mathbf{q}].
$$
We may abusively refer to ``the asymptotic condition $c$ at $q_c$'' when we mean $\widetilde{c}$.
We now define
$$\psi_c^{r|s}: \bA(L_r,\ldots,L_0)\otimes \bM(L_0,L_0') \otimes \bA(L_0',\ldots,L_s') \to \bM(L_r,L_s'),$$
for $r,s\geq 0$. We follow the routine of \ref{Subsection wrapped A infinity structure}: we first define the map on a vector space generator,
$$
\psi_{c}^{r|s}(\mathbf{q}^{i_r}x_r,\ldots,\mathbf{q}^{i_1}x_1,\mathbf{q}^{i_0}\underline{m},\mathbf{q}^{j_1}y_1,\ldots,\mathbf{q}^{j_s}y_s) \in CF^*(L_s',L_r;wH)[\mathbf{q}]
$$
where $\underline{m}\in CF^*(L_0',L_0;w_0 H)$, $x_i\in CF^*(L_{i-1},L_i;w_i H)$, $y_i\in CF^*(L_{i}',L_{i-1}';w_i' H)$.
We define the $\mathbf{q}^0$-coefficient of the output to be the 
 count (as usual, with weights $t^{E_{\mathrm{top}}(u)}$ and with orientation signs) of the isolated solutions $u: S \to E$ of Floer's equation $(du-X\otimes \gamma)^{0,1}=0$, on the above decorated domains (compare \ref{Subsection closed-open with q coefficients}).
In particular 
$$\textstyle\gamma=w_c\alpha_c + w_0\alpha_0 + \sum w_i \alpha_i + \sum w_j'\alpha_j' + \sum \beta_f \;\;\textrm{ and }\;\;  w=w_c + w_0 + \sum w_i + \sum w_j' + |F|$$
 where $|F|$ is the number of preferred points.
 The exponents $i_r,\ldots,i_1,i_0,j_1,\ldots,j_s\in \{0,1\}$ determine whether or not a preferred point arises on the geodesic connecting $-1$ to the relevant boundary puncture, and $i_c\in \{0,1\}$ is associated to the geodesic $(-1,0)$. As usual, for symmetry reasons \cite[Lemma 3.7]{Abouzaid-Seidel}, the counts of Floer solutions which involve more than one preferred point on a geodesic cancel out.
  As always, we extend $\psi_c$ so that it is $\partial_{\mathbf{q}}$-equivariant and $\Lambda$-multilinear (compare \ref{Subsection wrapped A infinity structure}).
 \\[1mm]
\noindent {\bf Remark about signs.}\,\emph{As in \ref{Subsection wrapped A infinity structure} and \ref{Subsection closed-open with q coefficients}, solutions are counted with a sign 
$(-1)^{k+\sum j\, \mathrm{Deg}(z_j)}$ where: $(z_{r+s+1},\ldots,z_1)=(x_r,\ldots,x_1,\underline{m},y_1,\ldots,y_s)$ in that order, $\mathrm{Deg}(z_j)=\|z_j\|+1$ for non-module inputs, $\mathrm{Deg}(z_j)=\mathrm{deg}(\underline{m})+1$ for the module input (when $j=s+1$). The $k$ is the Koszul sign obtained from viewing $\mathbf{q}$ as an operator of degree $-1$ acting from the left (for example if $j_2=1$, then $\mathbf{q}^{j_2}y_2$ above contributes $\|x_r\| + \cdots + \|x_1\|+ \mathrm{deg}(\underline{m}) + \|y_1\|$ to $k$).}
\begin{remark} Geometrically, $\psi_c$ is similar to $\mathrm{CO}(c)$, with the following differences. For $\psi_c$ there is a fixed input boundary puncture at $+1$ receiving the module input, whereas for $\mathrm{CO}$ all input boundary punctures are free to move ($\mathrm{CO}$ does not receive a distinguished module input). By Remark \ref{Remark cannot pull back data via forgetful map}, for $\mathrm{CO}$ one cannot simply pull back the auxiliary data used for $\mu_{\bA}$, whereas for $\psi$ we can pull back the data used for $\mu_{\bM}$ (which is the same as for $\mu_{\bA}$) because $q_c$ can only bubble off at a module input/output puncture. 
In the bubbling analysis for $\mathrm{CO}$, there can arise a disc bubble carrying $q_c$ with just one output puncture, counted by $\mathrm{CO}^0(c)$; for $\psi_c$ a disc bubble carrying $q_c$ will always have both an input and an output puncture.\end{remark}
\subsection{Bubbling analysis: the possible degenerations}
\label{Subsection The possible degenerations}
\indent We now consider the degenerations for a $1$-family of Floer solutions on such domains.
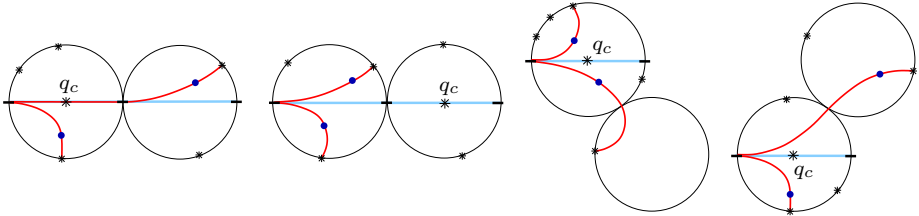
\begin{center}
\input{psi_structure4b.tex}
\end{center}
\indent As an illustration, in the case $R=3$, $S=2$, the above pictures (1)-(4) are four possible degenerations. In (1), a subset of the boundary punctures converges to $+1$. In (2) a subset of the boundary punctures converges to $-1$. In (3) a subset of the boundary punctures on the lower arc connecting $-1,+1$ converges to a puncture on the lower arc different from $\pm 1$. In (4) the analogue of (3) happens for the upper arc. In all four cases, a subset of the preferred points may converge to the same puncture, so these will reappear on the new disc bubble.

The main component of the above broken solutions is the one carrying $q_c$, and it is counted by the $\psi$ map. The new disc bubbles that appear, which do not carry $q_c$, are counted by $\mu_{\mathcal{M}}^{*|*}$ in pictures (1)-(2) and by $\mu_{\mathcal{A}}^*$ in pictures (3)-(4). We recall that the auxiliary data has been constructed universally and consistently as in \ref{Subsection A comment about the existence of universal choices of auxiliary data}, so that those new disc bubbles (which do not carry $q_c$) will be endowed with the universal data constructed for $\mu_{\mathcal{M}}$ and $\mu_{\mathcal{A}}$.

\indent The only other remaining degeneration, occurs when a Floer cylinder breaks off at $q_c=0$, possibly carrying a subset of the preferred points. These cylinders are counted by the differential $\mu^1_{SC^*(E)}(c)$ on $SC^*(E)$.
%
\subsection{Proof that $\psi_c$ is a morphism of bimodules for $\mathbf{[c]\in SH^*(E)}$}
\label{Subsection Proof that psi is a morph of bimods}
Extending $\psi_c$ linearly in $c$, we may now replace $c$ by a cycle in $SC^*(E)$ (i.e. a representative of a class $[c]\in SH^*(E)$). The Floer cylinders breaking off at $q_c$ all cancel out, since $\mu^1_{SC^*(E)}(c)=0$, so the above degenerations (1)-(4) prove that $\psi_c$ is a morphism of $\bA$-bimodules (as defined in \ref{Subsection Bimodules morphism})
$$\psi_c: \mathcal{M}\to \mathcal{M}.$$
As usual, it is the condition of $\partial_{\mathbf{q}}$-equivariance that ensures that the algebra accurately reflects the bubbling analysis (compare the Remarks in \ref{Subsection wrapped A infinity structure}). By Lemma \ref{Lemma functoriality of HH}, we obtain a chain map associated with the bimodule map $\psi_c$, 
 $$\mathrm{CC}_*(\psi_c): \mathrm{CC}_*(\bA,\mathcal{M})\to \mathrm{CC}_*(\bA,\mathcal{M}).$$
We may abusively write $\psi_c$ when we mean $\mathrm{CC}_*(\psi_c)$.  
Explicitly, the chain map $\mathrm{CC}_*(\psi_c)$ on a vector space generator $\underline{m}\otimes x_{n}\otimes \cdots \otimes x_1 \in \mathrm{CC}_n(\bA,\mathcal{M})$ equals
$$
\sum (-1)^{\sigma_1^{r}(\mathrm{deg}(\underline{m})+\sigma_{r+1}^{n})}\, \psi_c^{r|n-s+1}(x_{r} \otimes \cdots \otimes x_1 \otimes \underline{m} \otimes x_{n}\otimes \cdots \otimes x_s)\otimes x_{s-1}\otimes \cdots \otimes x_{r+1}
$$
summing over all obvious choices of $r,s\geq 0$. 
Thus, on homology, we obtain:
$$
\mathrm{HH}_*(\psi_c): \mathrm{HH}_*(\bA,\mathcal{M}) \to  \mathrm{HH}_*(\bA,\mathcal{M}).
$$
A standard argument shows that $\mathrm{HH}_*(\psi_c)$ does not depend on the choice of representative $c$ for $[c]\in SH^*(E)$ (one checks that changing $c$ by a boundary chain corresponds, by the above degeneration analysis, to changing $\psi_c$ by the boundary of a pre-morphism, as in \ref{Subsection Bimodules morphism}).

\subsection{Proof that $\mathbf{\psi}$ is an $\mathbf{SH^*(E)}$-module action}
\label{Subsection Proof that psi is a SH module action}
We consider a ``higher-order'' $\psi$-string map: in addition to $q_c=0$, we introduce another positive interior puncture $q_{c'}\in (0,1)$ which is allowed to move along the geodesic joining $q_c=0$ and $+1$.
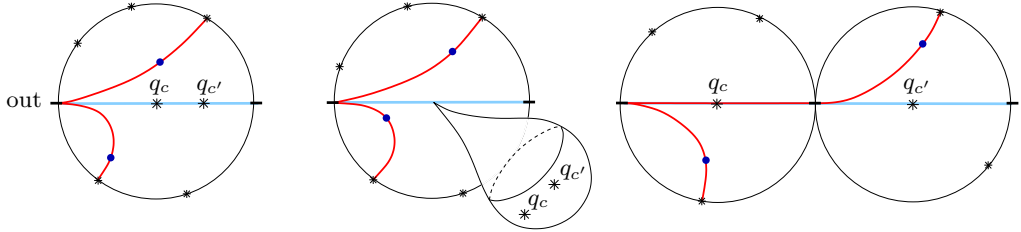
\begin{center}
\input{psi_structure3b.tex}
\end{center}
The possible degenerations in a $1$-family are: 
\begin{enumerate}
\item If $q_{c'}$ converges to $q_c=0$, then a $3$-punctured sphere counted by the $\mu^2_{SC^*(E)}$ product appears as in picture 2. These breakings are counted by $\psi_{c*c'}$.

\item If $q_{c'}$ converges to $+1$ (possibly together with a subset of the boundary punctures and a subset of the preferred points), then a disc bubble carrying $q_{c'}$ appears which is counted by $\psi_{c'}$ as in picture 3. These degenerations are counted by $\psi_{c}\circ \psi_{c'}$.

\item If $q_{c'}$ converges to a point in the interval $(0,1)$, then a subset of the boundary punctures (possibly together with a subset of the preferred points) may converge to $+1$, or $-1$, or a boundary puncture different from $\pm 1$. These disc bubbles are counted respectively by the composition maps $\mu_{\mathcal{M}}$, $\mu_{\mathcal{M}}$, $\mu_{\bA}$. The main component of the broken configuration still carries $q_c,q_{c'}$, and the count of these (now rigid) main components defines a pre-morphism $K:\mathcal{M}\to \mathcal{M}$ as in \ref{Subsection Bimodules morphism}. The count of the broken configurations is therefore $\delta K$, where $\delta$ is the differential on pre-morphisms.
\end{enumerate}

This proves that $SH^*(E) \to \mathrm{End}(\bM)$ is an algebra homomorphism. By Lemma \ref{Lemma functoriality of HH image of exact elements}(1), 
$$\mathrm{CC}_*(\psi_{c*c'})-\mathrm{CC}_*(\psi_{c}) \circ \mathrm{CC}_*(\psi_{c'}) = \mathrm{CC}_*(\delta K)=\mathrm{CC}_*(K)\circ b + b \circ \mathrm{CC}_*(K),$$
so $\mathrm{CC}_*(K)$ is a chain homotopy which shows 
$\mathrm{HH}_*(\psi_{c*c'})=\mathrm{HH}_*(\psi_{c}) \circ \mathrm{HH}_*(\psi_{c'}),$
as required.
%
%
\subsection{The $\mathbf{\psi}$-string map for $\mathbf{c\in QH^*(E)}$}
\label{Subsection The psi structure for QH}
%
Analogously, for any locally finite cycle $c\in QH^*(E)$, we can construct an algebra homomorphism
$$
\psi: QH^*(E)\to \mathrm{End}(\bM), \; c\mapsto \psi_c.
$$
In the construction in 
\ref{Subsection The construction of psi} we now require $q_{c}$ to be an interior marked point, rather than an interior puncture, and we have no auxiliary forms associated to $q_c$. In \ref{Subsection The construction of psi}, all auxiliary forms vanished near $q_{c}$ except for the $\alpha$ form associated to $q_c$. So in this new setup, $\gamma=0$ near $q_c$, therefore the Floer solutions $u: S\to E$ are $J$-holomorphic near the marked point $q_c$. So it remains to discuss the intersection condition at $q_c$.

As the solutions $u$ are $J$-holomorphic near $q_c$, we may require that they intersect the given lf-cycle $c$ at $q_c$, so $u(q_c)\in c$. More precisely, consider the moduli space $\mathcal{S}$ of all solutions $u$ but without imposing the condition $u(q_c)\in c$. We can ensure that the lf-cycles chosen for the definition of $QH^*(E)$ are transverse to the evaluation map $\mathrm{ev}:\mathcal{S}\to E$ at $q_c=0\in \D$. Then $\mathrm{ev}^{-1}(c)\subset \mathcal{S}$ is the space of all solutions $u$ described above satisfying the condition $u(q_c)\in c$. Then $\psi_c$ is the count of the isolated points (i.e. the zero-dimensional components) of $\mathrm{ev}^{-1}(c)\subset \mathcal{S}$ with appropriate orientation signs and with weight $t^{E_{\mathrm{top}}(u)}$ (as in \ref{Subsection moduli space of pseudoholo maps in wrapped case}).

Following the same arguments as in \ref{Subsection Proof that psi is a morph of bimods}, \ref{Subsection The possible degenerations} and \ref{Subsection Proof that psi is a SH module action}, we obtain a $QH^*(E)$-module structure on $\mathrm{HH}_*(\mathcal{A},\mathcal{M})$ (in particular in \ref{Subsection Proof that psi is a SH module action} the $3$-punctured sphere bubbles which carry lf-cycle intersection conditions at $q_c,q_{c'}$ are precisely those which count the quantum product $c*c'$).

\begin{lemma}\label{Lemma SH and QH action on HH agree}
There is a commutative diagram
$$
\xymatrix@C=30pt@R=17pt{ SH^*(E)  \ar@{->}^-{\psi}[r] & \mathrm{End}(\bM) \ar@{->}^{}[r] & \mathrm{End}(\mathrm{HH}_*(\bM))\\
 QH^*(E) \ar@{->}^{c^*}[u] \ar@{->}_{\psi}[ur] &
}
$$
\end{lemma}
\begin{proof}
This can be proved just like Lemma \ref{Lemma OC compatible with c}. A cleaner proof, is to first define
$$
\psi: SH_{\Diamond}^*(E) \to \mathrm{End}(\mathcal{W}_{\Diamond}(E))
$$
using notation from \ref{Subsection Category W0E and S0CE}, by combining the two constructions from \ref{Subsection Proof that psi is a SH module action} and \ref{Subsection The psi structure for QH} (see \ref{Subsection Acceleration diagram respects products}). Now observe that by definition this is compatible with the inclusion $c^*:QC^*(E)\to SC_{\Diamond}^*(E)$.
\end{proof}
%
%
%
\subsection{Proof that $\psi$ is unital}
%
\begin{lemma}\label{Lemma psi E is unit}
For $\psi: QH^*(E) \to \mathrm{End}(\bM)$, 
$$\psi_{[E]} = (-1)^{1+\mathrm{deg}(\cdot)}\, 1_{\bM} \in \mathrm{End}(\bM).$$
\end{lemma}
\begin{proof}
$\psi_{[E]}:\mathcal{M}(L_0,L_0')\to \mathcal{M}(L_0,L_0')$ is a count of $J$-holomorphic discs as in the picture,
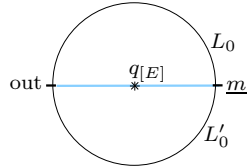
\begin{center}
\input{psi_structure5.tex}
\end{center}
involving a tautological intersection condition $[E]$ at $q_{[E]}=0$.
We claim that $\psi_{[E]}$ is the identity bimodule map $\bM \to \bM$ from \ref{Subsection Unit pre-morphism} (up to sign). In the case when there are $r+s\geq 1$ non-module inputs, if the moduli space of solutions for $\psi_{[E]}^{r|s}$ were non-empty, then it would contain a non-constant solution (as $r+s\geq 1$, there are at least three Lagrangians involved, and these do not have a triple intersection, for generic Hamiltonian perturbation data). But this non-constant solution $u$ gives rise to a $1$-family of solutions $u\circ g$ where $g\in SL(2,\R)$ belongs to the $\R$-family of hyperbolic isometries which fix $\pm 1$ (the module output/input). These isometries move the $r+s\geq 1$ other boundary markers and the marker $q_{\mathrm{[E]}}=0$, but we can simply reposition the marker $q_{[E]}$ to be at the origin of the disc since the intersection condition with $[E]$ is automatically satisfied also for $u\circ g$. These solutions are therefore not rigid, which is not allowed for dimension reasons, so the moduli space must in fact be empty. 
\\
\noindent {\bf Clarification.} \emph{The auxiliary data in the construction of $\psi_{[E]}$ is, by construction, $PSL(2,\R)$-equivariant so $u\circ g$ is a legitimate solution counted by $\psi_{[E]}$ (it automatically hits $[E]$ at $0$). Although for $\mu_{\bM}$ the solutions $u$, $u\circ g$ are considered to be the same (since we must quotient out by the residual $\R$-freedom of the $PSL(2,\R)$ maps fixing $\pm 1$), for $\psi_{[E]}$ these solutions are considered to be distinct because the $r+s\geq 1$ boundary markers have moved. This reparametrization trick is legitimate: transversality for solutions of $\psi_{[E]}$ is automatically implied by the transversality for solutions of $\mu_{\bM}$.} 
\\
\indent When $r+s=0$, there are only the module output/input markers, and only two Lagrangians $L_0,L_1$ are involved. The previous argument still holds, so solutions are not rigid unless they are constant, except that now constant solutions are allowed: an intersection point in $L_0\cap L_1$ when $L_0,L_1$ are transverse (if they are not transverse, then we identify $\D\setminus \{\pm 1\}\cong \R \times S^1$ and use the perturbation datum $H_{L_0,L_1},J_{L_0,L_1}$ which is $s$-independent, see \ref{Subsection When Lagrangians do not intersect transversely}, so for any Hamiltonian chord $x\in CF^*(L_0,L_1)\equiv CF^*(L_0,L_1;H_{L_0,L_1})$ the analogue of the ``constant solutions'' are the $s$-independent solutions $u(s,t)=x(t)$). So $\psi_{[E]}^{0|0}$ is the identity.\\
\noindent {\bf Remark.} \emph{The sign $\psi_{[E]}^{0|0}(\underline{m})=(-1)^{\mathrm{deg}(\underline{m})+1}\underline{m}$ is caused by the Remark about signs in \ref{Subsection The construction of psi}, consistently with the sign for cohomological units $\mu_{\bM}^2(e_X,\underline{m})=(-1)^{\mathrm{deg}(\underline{m})+1}\underline{m}$ (see \ref{Subsection cohomological unit}).}
\end{proof}
\begin{corollary}
$\psi\!:\!SH^*(E)\!  \to\! \mathrm{End}(\bM) \!\to\! \mathrm{End}(\mathrm{HH}_*(\bM))$ is a unital algebra homomorphism.
\end{corollary}
\begin{proof}
We already proved that it is an algebra homomorphism in \ref{Subsection Proof that psi is a SH module action}. Unitality follows from Lemmas \ref{Lemma SH and QH action on HH agree} and \ref{Lemma psi E is unit} because $c^*$ is a unital ring homomorphism.
\end{proof}
%
%
\subsection{Proof of Theorem \ref{Theorem CO is unital alg hom}}
\label{Subsection proof of theorem CO is unital alg hom}
We already proved in Lemma \ref{Lemma algebra hom from CC to hom}
that the right half of the diagram commutes, we now prove the left half commutes.
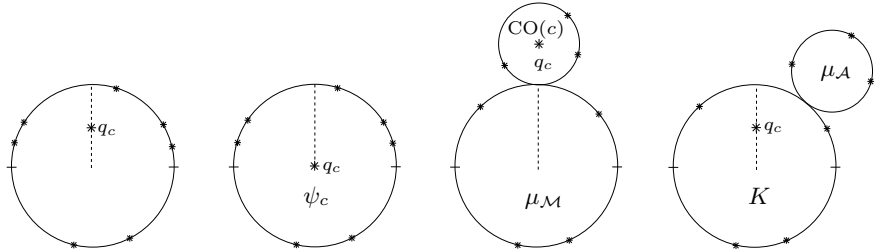
\begin{center}\input{COvsPsi.tex}\end{center}
Consider the domains from \ref{Subsection The construction of psi}, except now we allow the positive interior puncture $q_c$ to move along the imaginary line segment $[0,1)i$, rather than fixing it at $0$, as in the first picture (in the picture we have omitted preferred points for simplicity). Consider the possible degenerations of a $1$-family of solutions for such data. The two interesting degenerations occur if $q_c\in [0,1)i$ converges to one of the end-points: if $q_c\to 0$ we get the second picture above which is just the $\psi_c$ map; if $q_c\to i$ then a bubble appears which carries $q_c$ and these bubbles are counted by $\mathrm{CO}(c)$, and the total broken configuration (in the third picture) is counted by $\mu_{\bM}\circ_{\mathrm{left}}\mathrm{CO}(c)$.\\ 
\noindent {\bf Technical Remark.} \emph{We used the residual $\R$-symmetry for the domains used to define $\mu_{\bM}$, to fix at $+i$ the input coming from $\mathrm{CO}(c)$.}\\
\indent
There other degenerations involve a $\mu_{\bA}$-bubble or a $\mu_{\bM}$-bubble, with  $q_c$ converging to a point in $[0,1)i$. The main component is a rigid solution as in the first picture above: call $K_c:\bM \to \bM$ the pre-morphism of bimodules which counts such rigid solutions. Thus, these degenerations are counted by $\delta K_c$. The total count of the boundaries of those $1$-families prove that $\psi_c = \mu_{\bM}\circ_{\mathrm{left}}\mathrm{CO}(c)$ up to a $\delta$-boundary term. Thus, on homology, $\mathrm{HH}_*(\psi)=\mu_{\bM}\circ_{\mathrm{left}}\mathrm{CO}$, as required. In \ref{Subsection the unit for HF and HHcohomology of FE} and \ref{Subsection CO is a unital algebra hom} we prove that $\mathrm{CO}$ is a unital algebra homomorphism.
%
\subsection{The cohomological unit}
\label{Subsection the unit for HF and HHcohomology of FE}
%
%
%
\begin{lemma}\label{Lemma units for HF and HH}
The maps in the diagram in Theorem \ref{Theorem commutative diagram compared to Ganatra} are unital, namely
\begin{enumerate}

\item $\bA$ is cohomologically unital: $e_L=\mathrm{CO}^0(1)\in HW^*(L,L)$ are cohomological units.

\item $\mathrm{CO}(1)\in \mathrm{HH}^*(\bA)$ is a unit for the cup product defined in \ref{Subsection Product on Hochschild cohomology}.
\end{enumerate}
\end{lemma}
\begin{proof}
Let $m\in \bA(L_1,L_0)= \bM(L_1,L_0)$ be a $\mu_{\bA}^1$-cycle, and consider $\mathrm{CO}^0(1) \in \bA(L_1,L_1)$. By the commutativity of the left half of the diagram in Theorem \ref{Theorem commutative diagram compared to Ganatra}, 
$$
\begin{array}{rcl}
\mu_{\bA}^2(\mathrm{CO}^0(1),m)
&=&
-\mu_{\bM}^2(\mathrm{CO}^0(1),\underline{m})
\\
&=&
-
(\mu_{\bM}\circ_{\mathrm{left}} \mathrm{CO}(1))(\underline{m})
\\
&=&
-(\psi_E(\underline{m}) + \delta K_c(\underline{m}))
\\
&=&
-(-1)^{\mathrm{deg}(\underline{m})+1} m -\mu_{\bM}^{0|0}(K_c(\underline{m}))
\\
&=&
(-1)^{\|m\|+1} m +\mu_{\bA}^{1}(K_c(\underline{m})),
\end{array}
$$
where $K_c$ was defined in \ref{Subsection proof of theorem CO is unital alg hom}.
Thus $\mu^2_{\bA}(e_L,m)=(-1)^{\|m\|+1} m$ on cohomology, as required by \ref{Subsection cohomological unit}.
A similar argument as in \ref{Subsection proof of theorem CO is unital alg hom}, using the obvious right-composition $\mu_{\bM}\circ_{\mathrm{right}}$ making $q_c\to -i$ instead of $+i$ (we are not using compatibility with product structures in this argument), shows that $\mu^2_{\bA}(m,e_L)=m$ up to a $\mu^1_{\bA}$-boundary.
So $\bA$ is cohomologically unital. The second claim follows by combining the first claim with Corollary \ref{Corollary HCohom to Hhom}.
\end{proof}
\begin{remark}\label{Remark cohomological unit, PSS, MorseBott} In general, constructing the cohomological unit for $\bA$ is not straightforward, even in the exact setup (see the element $\Phi_H$ in \cite[Sec.(8c) and (8k)]{Seidel}).
For $L,E$ exact, one can construct a relative PSS-map $c^*:H^*(L) \to HF^*(L,L)$ and show that it is a unital algebra isomorphism using TQFT arguments (see $c^*:H^*(L) \to HW^*(L,L)$ in \cite{Ritter3}).
Another approach (see the final comment of \cite[Sec.(8l)]{Seidel}), is to use a Morse-Bott chain level model for $HF^*(L,L)$: in the monotone setup, Auroux \cite{Auroux} takes $CF^*(L,L)=C_{n-*}(L)$ and shows that $[L]$ is a cohomological unit. However, it is not straightforward to make the Morse-Bott model compatible with the $A_{\infty}$-structure (indeed, we used the Floer datum approach, see \ref{Subsection When Lagrangians do not intersect transversely}).
\end{remark}
%
\subsection{The $\mathbf{CO}$ map is a unital algebra homomorphism}
\label{Subsection CO is a unital algebra hom}
%
\begin{corollary}\label{Corollary CO is unital}
The vertical map in Theorem \ref{Theorem commutative diagram compared to Ganatra} is an isomorphism, and the left half of the diagram consists of unital algebra homomorphisms.
\end{corollary}
\begin{proof}
By Lemma \ref{Lemma units for HF and HH}(1) and Corollary \ref{Corollary HCohom to Hhom}, the first claim follows. Since $\psi$ is a unital algebra homomorphism, the second claim follows by the left half of the commutative diagram.
\end{proof}
One can also show directly that $\mathrm{CO}:SH^*(E)\to \mathrm{HH}^*(\bM)$ is an algebra homomorphism by considering the degenerations of a $1$-family of Floer solutions shown in the first picture below: we place two positive interior punctures $q_c=ri$, $q_{c'}=-ri$ on the imaginary axis, where $r\in (0,1)$.
Apart from the obvious bubblings of type $\mu_{\bA},\mu_{\bM}$ which give $\delta$-boundary terms, there are degenerations when $r$ converges to $0$ or $1$. If $r\to 0$, the bubbling at $0$ is counted by the quantum product; see picture 3. If $r\to 1$, we obtain the broken configuration in picture 2, which is counted by the cup product $\mathrm{CO}(c)*\mathrm{CO}(c')$; see \ref{Subsection Product on Hochschild cohomology}.\\
\noindent {\bf Technical Remark.}\,\emph{We used the residual $\mathrm{PSL}(2,\R)/\R$-symmetry for the domains that define $\mu_{\bA}$, to fix at $\pm i$ the two inputs coming from $\mathrm{CO}(c),\mathrm{CO}(c')$.}\\
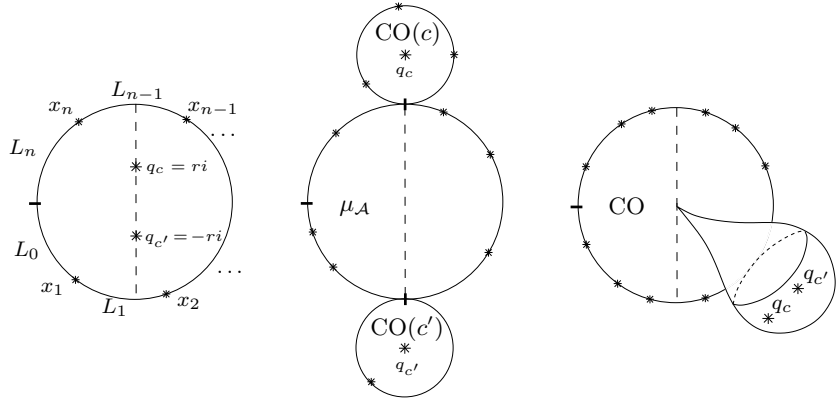
\begin{center}\input{CO_alg_hom.tex}\end{center}

 The above Corollary also shows that the $0$-part $\mathrm{CO}^0: SH^*(E)\to HW^*(L,L)
$ (see \ref{Subsection closed-open with q coefficients}), for any $L\in \mathrm{Ob}(\bA)$, is a unital algebra homomorphism, because the cup product on $\mathrm{CC}^*$ restricted to the $0$-parts is the usual $\mu^2_{\bA}$-product on $HW^*(L,L)$. Alternatively, the same direct proof as above applies, ignoring all boundary punctures except $\underline{z}=-1$.
%
%
\subsection{The $\mathbf{CO}$ and $\mathbf{OC}$ string maps respect the module structure}
\label{Subsection The open-closed string map respects the module structure}

That the $\mathrm{CO}$ map respects the module structure is immediate from the fact that it is a ring homomorphism (the module action by $c$ is defined by the cup product action of $\mathrm{CO}(c)$ on $\mathrm{HH}^*(\bM)$).

\begin{theorem}
$\mathrm{OC}:\mathrm{HH}_*(\wE) \to SH^*(E)$ respects the $SH^*(E)$-module structure:
$$
\boxed{c\bullet \mathrm{OC}(x) = \mathrm{OC}(\psi_c(x))}
$$
where $\bullet$ is the $\mu^2_{SH^*(E)}$-product,
$c\in SH^*(E)$, $x\in \mathrm{HH}_*(\wE)$. Similarly, $\mathrm{OC}:\mathrm{HH}_*(\scrF(E)) \to QH^*(E)$ and $\mathrm{OC}:\mathrm{HH}_*(\wE) \to SH^*(E)$ both respect the $QH^*(E)$-module structure.
\end{theorem}
\begin{proof}
We will just prove the first claim, since the other claims are analogous (indeed simpler). 
We consider a ``higher-order'' $\mathrm{OC}$-map: in addition to the negative interior puncture $z_0=0$, we introduce a positive interior puncture $q_{c}\in (0,1)$ which is allowed to move along the geodesic joining $z_0=0$ and the module input puncture $+1$, as in the first picture below.
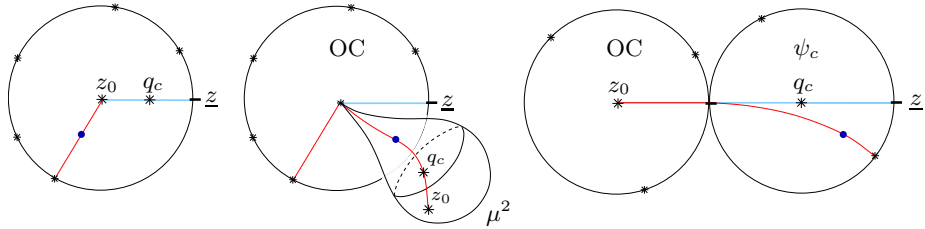
\begin{center}
\input{OCisModMap.tex}
\end{center}
The domains are the same as those used to define $\mathrm{OC}$ except we introduce a new positive interior puncture $q_c\in (0,1)$ which is allowed to move on the geodesic connecting $z_0=0,\underline{z}=+1$ (this is an interior marked point when we work with the $QH^*(E)$-action). In particular, there may be preferred points on the geodesics connecting any puncture to $z_0=0$ (when working with the compact category these do not arise), and there may also be preferred points on the geodesic connecting $z_0=0$ to $q_c$ (these do not arise when working with the compact category, or when working with the $QH^*(E)$-action for the wrapped category). 
As usual, all $\beta_f=0$ near the puncture $q_c$, and the construction of universal auxiliary data follows the routine of \ref{Subsection A comment about the existence of universal choices of auxiliary data} (compare \ref{Subsection Ainfinity product on SC}, \ref{Subsection open-closed with q coefficients}, \ref{Subsection closed-open with q coefficients}). In particular, the auxiliary data is chosen to extend the data already chosen for $\mu_{\bA}^*$, $\mu_{\bM}^{*|*}$, $\mu_{SC^*(E)}^1$, $\mu_{SC^*(E)}^2$, $\mathrm{OC}$, $\psi_c$ (bubbles counted by these maps arise in the compactification of the above moduli spaces).

Let us consider the possible degenerations for a $1$-family of Floer solutions for these decorated domains. The uninteresting degenerations are: disc bubbling at the module input $\underline{z}$ counted by $\mu_{\bM}$ or disc bubbling at another boundary puncture counted by $\mu_{\bA}$ (these two contribute to the bar differential acting on $\mathrm{CC}_*(\bA,\bM)$); or a Floer cylinder breaks off at $z_0$ possibly carrying preferred points (these contribute to $\mu^1_{SC^*(E)}$); or a Floer cylinder breaks off at $q_c$ (possibly carrying preferred points). In all three cases, the main component is a rigid solution whose domain is as in the first picture above. The count $K_c$ of these rigid solutions is the chain homotopy. These degenerations are therefore counted by $K_{\mu^1_{SC^*(E)}(c)} +
K_c \circ b_{\mathrm{CC}_*(\wE)} + \mu^1_{SC^*(E)} \circ K_c$, and therefore disappear on cohomology.

The interesting degenerations occur when $q_c \to z_0=0$ or $q_c\to \underline{z}=+1$, in the second and third pictures above. When $q_c \to 0$, a bubble appears carrying two punctures $z_0,q_c$, and these contribute to $\mu^2_{SC^*(E)}$. When $q_c \to +1$, a disc bubbles off carrying $q_c$ (which due to the residual $\R$-translation freedom can be fixed to be located at $q_c=0$ of the new disc domain) which are counted by $\psi_c$, whilst the main component is counted by $\mathrm{OC}$. Thus, these degenerations are counted by $\mathrm{OC}\circ \psi_c - \mu^2_{SC^*(E)}(c,\cdot) \circ \mathrm{OC}$, as required.
\end{proof}
The analogous ``higher order'' $\mathrm{CO}$-map, with an additional positive interior puncture $q_c\in (-1,0)$, and the analogous two degenerations as in the previous picture are:
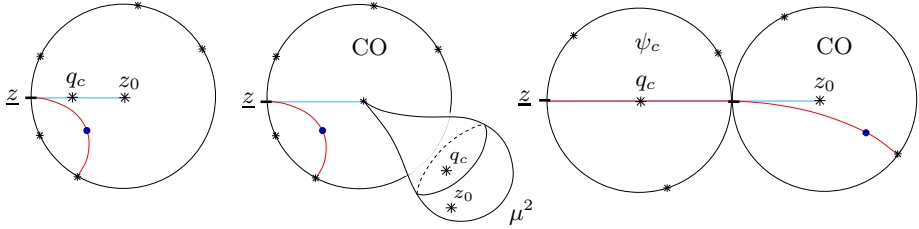
\begin{center}
\input{OCisModMap2.tex}
\end{center}
This implies that on cohomology,
$$
\boxed{\mathrm{CO}(c\bullet x) = \psi_c\circ \mathrm{CO}(x)}
$$
where $\bullet$ is the $\mu^2_{SH^*(E)}$-product, $c,x\in SH^*(E)$, and $\circ$ is the composition in the sense of \ref{Subsection Conventient abbreviations: the symbols circ and Mu}.
%
\subsection{The $\psi$-structure for the compact Fukaya category}
\label{The psi structure for the compact Fukaya category}
The construction of $\psi_c$ for the category $\mathcal{F}(E)$, and a given lf-cycle $c\in QH^*(E)$, is a simplification of \ref{Subsection The construction of psi}. There are no auxiliary $\alpha_k,\beta_f,\gamma$ forms, no preferred points / geodesics, so the Floer equation is replaced by the $J$-holomorphic equation $du^{0,1}=0$ when each pair of Lagrangians is transverse, whereas in the non-transverse case we apply the Floer datum machinery of \ref{Subsection When Lagrangians do not intersect transversely}. By the same algebra formalism as in \ref{Subsection The construction of psi}, $\psi$ is determined by the counts
$$
\psi_{c}^{r|s}(x_r,\ldots,x_1,\underline{m},y_1,\ldots,y_s) \in CF^*(L_s',L_r)
$$
where $\underline{m}\in CF^*(L_0',L_0)$, $x_i\in CF^*(L_{i-1},L_i)$, $y_i\in CF^*(L_{i}',L_{i-1}')$. Notice no Hamiltonians are present (except those coming from the Floer datum machinery, which are suppressed from the notation). These maps count $J$-holomorphic curves $u: S \to E$ (as usual, with weight $t^{E_{\mathrm{top}}(u)}$ and orientation signs) on the domains $S$ as in the picture in \ref{Subsection The construction of psi} except:\\
$\bullet$ the geodesic/preferred point decorations are omitted,\\
$\bullet$ $q_c=0$ is now just an interior marked point, rather than a puncture.\\
Therefore, the auxiliary data in the construction of $\psi_{c}^{r|s}$ is the same as the universal data which defines $\mu_{\bM}^{r|s}$ for the compact category in \ref{Subsection The bimodule structure for curly B} with the exception that we have additionally an interior marked point $q_{c}=0$. Recall that the data for $\mu_{\bM}^{r|s}$ is, by construction, $PSL(2,\R)$-equivariant (recall \ref{Subsection A comment about domain reparametrizations and moduli}), but unlike \ref{Subsection The bimodule structure for curly B} for the domains $S$ above we no longer have any residual $\R$-translation freedom because we fixed a new marker at $q_c=0$.

The discussion of the intersection condition at $q_c$ is the same as in \ref{Subsection The psi structure for QH}. In the solution counts, the weights $t^{E_{\mathrm{top}}(u)}$ are still defined as in \ref{Subsection moduli space of pseudoholo maps in wrapped case}. In the case of transverse Lagrangians, when no Floer datum is needed, $E_{\mathrm{top}}(u)=\int_{\D} u^*\omega$ is the usual energy of the $J$-holomorphic curve, but in general the Floer datum Hamiltonians will appear in the formula for $E_{\mathrm{top}}$). 
%
%
\subsection{The acceleration functor respects the algebraic structures}
\label{Subsection Acceleration diagram respects products}

Recall the definition of $\mathcal{W}_{\Diamond}(E)$ in \ref{Subsection Category W0E and S0CE}. Just as we extended the $\mathrm{OC},\mathrm{CO}$ maps to $\mathcal{W}_{\Diamond}(E)$, we can extend the module action. This reduces to extending the definition of the bimodule endomorphism $\psi_c$ to $\psi_c\in \mathrm{end}(\mathcal{W}_{\Diamond}(E))$
(see the proof of Lemma \ref{Lemma SH and QH action on HH agree}), so geometrically we need to explain what we do when we feed $CF^*(L_i,L_{i+1})$ inputs. As usual, this means we do not associate an $\alpha$-form to the input (but there may be a preferred geodesic and preferred points, as usual, determining $\beta_f$ forms). In the case when $L_i,L_{i+1}$ are transverse, the Floer solution is holomorphic near the boundary puncture, so as usual finite-energy implies that the Floer solution extends continuously over the puncture. When $L_i,L_{i+1}$ are not transverse, we use the Floer datum machinery from \ref{Subsection When Lagrangians do not intersect transversely}. The properties of the module structure for $\mathrm{HH}_*(\mathcal{W}_{\Diamond}(E))$ and $\mathrm{HH}^*(\mathcal{W}_{\Diamond}(E))$ now carry over just as for $\mathcal{W}(E)$.
To conclude the proof of the final claim in Theorem \ref{Theorem acceleration commutative diagram}, in light of \ref{Subsection The open-closed string map respects the module structure}, it remains to show the following result. 

\begin{theorem}\label{Theorem Acceleration functor is algebra hom}
For the acceleration functor $\mathcal{AF}$ from \ref{Subsection Acceleration functor}, $\mathrm{HH}_*(\mathcal{AF})$ and $\mathrm{HH}^*(\mathcal{AF})$ are $QH^*(E)$-module maps, and $\mathrm{HH}^*(\mathcal{AF})$ is a unital algebra homomorphism.
\end{theorem}
\begin{proof}
Let $c\in QC^*(E)$. Consider what happens when $\psi_c\in \mathrm{end}(\mathcal{W}_{\Diamond}(E))$ receives all its inputs from $CF^*(L_i,L_{i+1})$ summands. By the construction described above, in this case the Floer solutions counted by $\psi_c\in \mathrm{end}(\mathcal{W}_{\Diamond}(E))$ are precisely those counted by the map $\psi_c\in \mathrm{end}(\scrF(E))$ (in particular, the total output weight is zero, since no preferred points and no $\alpha$ forms are present, so $\psi_c$ lands again in a summand of the type $CF^*(L,L')$). Therefore 
$$\mathrm{incl}\circ \psi_c = \psi_c \circ \mathrm{incl}$$
at the level of bimodule maps, and hence also on Hochschild homology, where $\mathrm{incl}: \scrF(E)\to \mathcal{W}_{\Diamond}(E)$ is the natural inclusion. 
Recall the natural inclusion $f: \wE \to \mathcal{W}_{\Diamond}(E)$ from \ref{Subsection Acceleration functor}. Since the action on $\mathcal{W}_{\Diamond}(E)$ is defined by extending the action defined on $\mathcal{W}(E)$,
$$
f\circ \psi_c = \psi_c \circ f
$$
at the level of bimodule maps, and hence on Hochschild homology. Finally $\mathrm{HH}_*(\mathcal{AF})=\mathrm{HH}_*(\mathrm{incl})\circ \mathrm{HH}_*(f)^{-1}$ is a composite of maps which preserve the module structure, so the composite also preserves the module structure.

Similarly, the cup product action of $\mathrm{CO}(c)$ on $\mathrm{CC}^*(\mathcal{W}_{\Diamond}(E))$ with all inputs coming from $CF^*(L_i,L_{i+1})$ agrees with $\mathrm{CO}(c)$ acting on $\mathrm{CC}^*(\scrF_{\Diamond}(E))$ (up to a chain homotopy, due to the fixed preferred point; compare the proof of Theorem \ref{Theorem acceleration commutative diagram}). Thus
$$
\mathrm{HH}^*(\mathrm{incl})\circ \mathrm{CO}(c) = \mathrm{CO}(c) \circ \mathrm{HH}^*(\mathrm{incl})
$$
where $\mathrm{CC}^*(\mathrm{incl}):\mathrm{CC}^*(\mathcal{W}_{\Diamond}(E))\to \mathrm{CC}^*(\scrF_{\Diamond}(E))$ is the restriction map from Lemma \ref{Lemma HH under fully faithful functors}. So $\mathrm{HH}^*(\mathcal{AF})$ preserves the $QH^*(E)$-module structure.
That the inclusion $\mathcal{F}_{\Diamond}(E)\to \mathcal{W}_{\Diamond}(E)$ respects the
cup product action on $\mathrm{CC}^*$ follows by homological algebra (since the $\mu_{\bA},\mu_{\bM}$ maps agree on the subcategory). Thus $\mathrm{HH}^*(\mathcal{AF})$ preserves the cup product, so it is an algebra homomorphism. The fact that it is unital follows from the second diagram in Theorem \ref{Theorem acceleration commutative diagram}: $$\mathrm{HH}^*(\mathcal{AF})(\mathrm{CO}(c^*[E])) = (\mathrm{HH}^*(\mathcal{AF})\circ \mathrm{CO}\circ c^*)[E] = \mathrm{CO}([E]) \in \mathrm{HH}^*(\scrF(E)). \qedhere$$
\end{proof}
%
\section{The string maps respect eigensummands}
\label{Section OC and CO maps are module homs}
%
\subsection{Eigensummands}
\label{Subsection Eigensummands}

Let $a=c_1(TE)-\lambda 1\in QH^*(E)$, and let $a^N=a*a*\cdots *a$ be the quantum product $N$ times. We call $\lambda$-eigensummand $QH^*(E)_{\lambda}$ the subalgebra of $QH^*(E)$ given by the generalized eigenspace for $c_1(TE)$:
$$
QH^*(E)_{\lambda} = \{ q\in QH^*(E): a^N*q=0 \textrm{ for some }N\in \N\}.
$$
These are of course only non-zero when $\lambda$ is an eigenvalue of $c_1(TE)*\in \mathrm{End}(QH^*(E))$.

Similarly, since $SH^*(E)$ is a $QH^*(E)$-module via $c^*$ (see \ref{Subsection product on SH is a module}) we get a subalgebra $SH^*(E)_{\lambda}$ of $SH^*(E)$ given by the generalized $\lambda$-eigensummand for multiplication by $c^*(c_1(TE))$.\\
{\bf Remark.} \emph{$SH^*(E)$ is typically infinite-dimensional, so it is not clear whether it fully decomposes into eigensummands or whether there is a well-defined projection $SH^*(E) \to SH^*(E)_{\lambda}$.}
%
\subsection{Relating $m_0$ and $c_1$-eigenvalues}
\label{Subsection Relating m0 and c1 evals}
Recall $m_0(L)$ is the count of $J$-holomorphic discs, of Maslov index 2, with a boundary marked point constraint $[\mathrm{pt}]\in C_*(L)$, modulo $PSL(2,\R)$-reparametrization. Following \cite[Lemma 6.7]{Auroux}, each disc hits $\mathrm{PD}[c_1(TE)]$ once algebraically \cite[Lemma 3.1]{Auroux} so we fix the parametrization of the domain so that the intersection with $\mathrm{PD}[c_1(TE)]$ occurs at $0$ and the intersection with $[\mathrm{pt}]$ occurs at $-1$. This defines $\mathrm{CO}^0(c_1(TE))$  in the Morse-Bott model $CF_{Bott}^*(L,L)=C_{n-*}(L)$
 (compare Remark \ref{Remark cohomological unit, PSS, MorseBott}). This argument implies the following result, due to Kontsevich, Seidel and Auroux \cite{Auroux},
 \begin{equation}\label{CO of c1}
\boxed{\mathrm{CO}^0(c_1(TE)) = m_0(L)\,\mathrm{CO}^0([E])=m_0(L)\,[L]\in CF_{Bott}^*(L,L)}
\end{equation}
 where $\mathrm{CO}^0([E])=[L]$ since only constant discs contribute, as other solutions are not isolated (the intersection condition $[E]$ is automatic).
Recall \cite[Lemma 6.7]{Auroux} requires that the monotone Lagrangian $L$ is disjoint from a representative lf-cycle $D$ for $c_1(TE)$, to rule out a contribution of constant discs in the calculation of $\mathrm{CO}(D)$, implying \eqref{CO of c1}. We use:

\begin{lemma}[{\cite[Lemma 6.4]{Ritter5}}]\label{Lemma Maslov cycle disjoint from c1}
For any orientable Lagrangian $L$ in a monotone symplectic manifold $E$, there is a representative lf-cycle $D$ of $\mathrm{PD}[c_1(TE)]$ that does not intersect $L$.
\end{lemma}

The idea of the proof, is that a representative for $\mathrm{PD}[c_1(TE)]$ can be obtained from the vanishing set of a generic smooth section a complex line bundle on $E$ with Chern class $c_1(TE)$. Then, using that $L$ is orientable, we can deform this section near $L$ so that it does not vanish at points of $L$, because $c_1(TE)|_L=\lambda_E \omega|_L=0$ (using monotonicity of $E$, and $\omega|_L=0$).\\
{\bf Remark 1.}\,\emph{For non-orientable $L$ this may fail: $\R P^2\subset \C P^2$ always intersects $\mathrm{PD}[c_1(T\C P^2)]$.}\\
{\bf Remark 2.}\,\emph{If one imposes the condition $\mathrm{char}\, \K\neq 2$, then one can drop the assumptions that $L$ is orientable and that $E$ is monotone. Following \cite[Prop 3.1]{SmithQuadrics}, $H^2(E,L)\cong H_{\dim E-2}(E\setminus L)$ so the Maslov class\footnote{The Maslov index
$
\mu_L: H_2(E,L)\to \Z
$
defines a class $\mu_L\in H^2(E,L)$, and $\mu_L\mapsto 2c_1(TE)$ via $H^2(E,L)\to H^2(E)$. The Maslov index of a disc bounding $L$ equals the homological intersection number with $\mathrm{PD}[\mu_L]$.} $\mu_L$ can be chosen to be disjoint from $L$. Finally $\mathrm{CO}(c_1(TE))=\frac{1}{2}\mathrm{CO}(\mu_L)$.}

\begin{lemma}\label{Lemma COc1 and m0} Let 
$L\in \mathrm{Ob}(\mathcal{W}_{\lambda}(E))$, so $m_0(L)=\lambda$.\\ The map
$\mathrm{CO}^0:QC^*(E)\to CW^*(L,L)$, defined in Remark \ref{Remark CO from QH to wrapped}, satisfies
$$\qquad\qquad\qquad\boxed{\mathrm{CO}^0(c_1(TE)) = \lambda \, e_L + \mu^1(K) \in CW^*(L,L)} \qquad \textrm{for some }K\in CW^*(L,L),$$
where $e_L=\mathrm{CO}^0([E])\in CW^*(L,L)$ is the cohomological unit (Lemma \ref{Lemma units for HF and HH}).
\end{lemma}
\begin{proof}
Let $c=D$ be the lf-cycle from Lemma \ref{Lemma Maslov cycle disjoint from c1}. We run an argument similar to that in \ref{Subsection proof of theorem CO is unital alg hom}.
Consider the first picture below: the auxiliary data is obtained via the forgetful map from the data defining $\mathrm{CO}^0([E])$, as described in Remark  \ref{Remark CO from QH to wrapped}, forgetting the new interior marked point $q_c\in [0,1)i$.
Explicitly, the negative puncture at $-1$ will define the output in $CF^*(L,L;H)\subset CW^*(L,L)$; the interior marked point $z_0=0$ receives the automatic lf-cycle intersection condition $[E]$; the new interior marked point $q_c\in [0,1)i$ receives the lf-cycle intersection condition $c$; and the fixed preferred point is at $-1/2$.
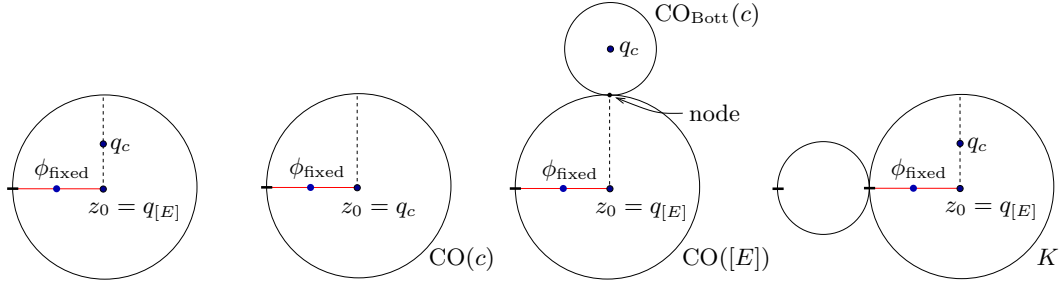
\begin{center}
\input{cone_mzero.tex}
\end{center}
Call $K$ the count of rigid solutions as in picture 1.
Now consider a $1$-family of solutions. If $q_c\to 0$ a bubble appears, counting the quantum product $[E]*c=c$, and as it only involves constant solutions we get picture 2, counted by $\mathrm{CO}(c)$.
If $q_c\to i$, in picture 3, a disc bubble appears which is connected to the main component along a node, not a puncture (compare Remark \ref{Remark cannot pull back data via forgetful map}), in particular the bubble will carry a $PSL(2,\R)$-invariant almost complex structure. This bubble is counted by \eqref{CO of c1} (note it is not constant because $c$ is disjoint from $L$). The bubble imposes the automatic cycle intersection condition $[L]$ at the node of the main component, so the total count is $m_0(L)$ times the output of the main component, which is $\mathrm{CO}^0([E])$. If $q_c$ converges to a point in $[0,1)i$, in picture 4, bubbling can occur at $-1$, counted by $\mu_{CW^*(L,L)}^1$ of the output of the main component, which is $K$.
\end{proof}
\noindent {\bf Remark.}\,\emph{Lemma \ref{Lemma COc1 and m0} also holds for $\mathcal{F}_{\lambda}(E)$ taking $\mathrm{CO}^0:QC^*(E)\to CF^*(L,L)$ (using the Floer datum model of \ref{Subsection When Lagrangians do not intersect transversely}). The Floer datum then plays the role of $\phi_{\mathrm{fixed}}$.}
\subsection{Producing a nilpotent action of $c_1(TE)-\lambda\, I$ at the chain level}
\label{Subsection Producing a nilpotent action on the chain level}

Let $\bA$ be the $A_{\infty}$-category $\scrF_{\lambda}(E)$ or $\mathcal{W}_{\lambda}(E)$, and let $\bM=\bA$ be the diagonal bimodule.

\begin{theorem}\label{Theorem Eigensummand decomposition}
$c=c_1(TE)-\lambda 1\in QH^*(E)$ acts nilpotently on $\mathrm{HH}_*(\bM)$ $($for any cycle $w\in \mathrm{HH}_*(\bM)$, there is an integer $N$ possibly depending on $w$, such that $\mathrm{HH}_*(\psi_{c^N})(w)=0).$
\end{theorem}
\begin{proof}
We will construct a  pre-morphism $\mathcal{K}\in \mathrm{end}(\bM)$ such that $(\psi_c+\delta \mathcal{K})^{0|0}=0$ (at the chain level). Since on homology 
$$\mathrm{HH}_*(\psi_c)=\mathrm{HH}_*(\psi_c+\delta \mathcal{K}) \in \mathrm{End}(\mathrm{HH}_*(\bM)),$$ 
the Theorem then follows by Lemma \ref{Lemma trick if bimodule map has zero 0part} applied to $\psi_c+\delta \mathcal{K}\in \mathrm{end}(\bM)$.

Following the left half of the diagram in Theorem \ref{Theorem commutative diagram compared to Ganatra} (which commutes on homology) we obtain a bimodule pre-morphism $\mathcal{K}_1\in \mathrm{end}(\bM)$ such that
$$
\mu_{\bM}\circ_{\mathrm{left}}\mathrm{CO}(c) = \psi_c + \delta\mathcal{K}_1.
$$
Define $\mathcal{K}_2\in \mathrm{CC}^*(\bM)$ taking $\mathcal{K}_2^0: \K \to \bM(L,L)$ linear with $\mathcal{K}_2^0(1)=K\in CW^*(L,L)$ as in Lemma \ref{Lemma COc1 and m0}, and $\mathcal{K}_2^r=0$ for $r\geq 1$. 
By construction, $(\delta \mathcal{K}_2)^0(1)=\mu_{\bM}^{0|0}(\underline{\mathcal{K}_2^0(1)})=-\mu_{\bA}^1(K),$ so 
\begin{equation}\label{Equation CO(c) is K2} 
(\mathrm{CO}(c)+\delta \mathcal{K}_2)^0(1)=0
\end{equation}
by Lemma \ref{Lemma COc1 and m0}, at the chain level. 
Finally, define $\mathcal{K}= \mathcal{K}_1 + \mu_{\bM}\circ_{\mathrm{left}} \mathcal{K}_2 \in \mathrm{end}(\bM)$. Since $\mu_{\bM}\circ_{\mathrm{left}}$ is a chain map, we have
$$
\psi_c + \delta \mathcal{K} = 
\mu_{\bM}\circ_{\mathrm{left}}\mathrm{CO}(c) + \delta (\mu_{\bM}\circ_{\mathrm{left}} \mathcal{K}_2) = 
\mu_{\bM}\circ_{\mathrm{left}}(\mathrm{CO}(c) +\delta\mathcal{K}_2).
$$
The $0|0$ part evaluated on any $\underline{m}\in \bM(L,L)$, using Equation \eqref{Equation CO(c) is K2}, is zero as required:
$$
(-1)^{\mathrm{deg}(\underline{m})}\mu_{\bM}^{0|0}((\mathrm{CO}(c) +\delta\mathcal{K}_2)^0(1),\underline{m})=0.\qedhere
$$
\end{proof}

\begin{remark} One can equivalently define the module action by the lower half of Theorem \ref{Theorem commutative diagram compared to Ganatra}, so $\mathrm{CO}(c)$ acts by cap product on $\mathrm{HH}_*(\bM)$. To see that Equation \ref{Equation CO(c) is K2} ensures that the action of $\mathrm{CO}(c)+\delta \mathcal{K}_2$ is nilpotent at the chain level it then suffices to notice that the analogue of Lemma \ref{Lemma trick if bimodule map has zero 0part} holds:  each time we apply cap product by $\mathrm{CO}(c)+\delta \mathcal{K}_2$, the cap product from \ref{Subsection Cap product} must feed at least one of the non-module inputs to $\mathrm{CO}(c)+\delta \mathcal{K}_2$ (otherwise we get zero as the $0$ part of the map vanishes), thus shortening the word in $\mathrm{CC}_*(\bM)$.
\end{remark}

\subsection{The OC map respects eigensummand decompositions}
\label{Subsection CO and OC maps respect eigensummand decompositions}
\begin{theorem}\label{Theorem OC is a module map}\strut\begin{enumerate}
\item $\mathrm{OC}: \mathrm{HH}_*(\scrF_{\lambda}(E)) \to QH^*(E)$ lands in the eigensummand $QH^*(E)_{\lambda}$,
$$
\mathrm{OC}: \mathrm{HH}_*(\scrF_{\lambda}(E)) \to QH^*(E)_{\lambda}.
$$
\emph{(This also holds if we replace $E$ by a \emph{closed} monotone symplectic manifold $B$.)}\\[-2mm]
\item $\mathrm{OC}: \mathrm{HH}_*(\wE) \to SH^*(E)$ lands in the eigensummand $SH^*(E)_{\lambda}$,
$$
\mathrm{OC}: \mathrm{HH}_*(\mathcal{W}_{\lambda}(E)) \to SH^*(E)_{\lambda}. 
$$
\item The first acceleration diagram respects these structures:
$$
\xymatrix@C=50pt{ \mathrm{HH}_*(\scrF_{\lambda}(E)) \ar@{->}^{\mathrm{HH}_*(\mathcal{AF})}[r] \ar@{->}_{\mathrm{OC}}[d] & \mathrm{HH}_*(\mathcal{W}_{\lambda}(E))
\ar@{->}^{\mathrm{OC}}[d] \\
 QH^*(E)_{\lambda}
\ar@{->}[r]^-{c^*} & SH^*(E)_{\lambda}
}
$$
\end{enumerate}
\end{theorem}
\begin{proof}
This follows immediately from Theorem \ref{Theorem Eigensummand decomposition} (and the fact that $\mathrm{OC}$ and the acceleration diagram respect the module structures, see \ref{Subsection The open-closed string map respects the module structure} and \ref{Subsection Acceleration diagram respects products}).
%
\end{proof}
%
\subsection{The CO map respects eigensummand decompositions}
\label{Subsection CO maps respect eigensummand decompositions}
The $0$-parts 
$$\begin{array}{l}
\mathrm{CO}^0: QH^*(E)\to HF^*(L,L) \qquad \qquad
\mathrm{CO}^0:SH^*(E)\to HW^*(L,L)
\end{array}$$
restricted to an eigensummand $QH^*(E)_{\lambda}$, respectively $SH^*(E)_{\lambda}$, vanish unless $m_0(L)=\lambda$. This follows by Lemma \ref{Lemma COc1 and m0} via Theorem \ref{Theorem CO is unital alg hom} as $c_1(TE)-\mu\,\mathrm{Id}$ acts by multiplication by the scalar $\lambda-\mu \neq 0\in \Lambda$ on $HF^*(L,L)$ and $HW^*(L,L)$, for $\mu \neq \lambda$. The acceleration diagram respects the above decomposition: for any $L \in \mathrm{Ob}(\mathcal{F}_{\lambda}(E))$,
$$
\xymatrix@C=50pt{ HF^*(L,L) \ar@{<-}^{\mathrm{HH}^0(\mathcal{AF})}[r] \ar@{<-}_{\mathrm{CO}^0}[d] & HW^*(L,L)
\ar@{<-}^{\mathrm{CO}^0}[d] \\
 QH^*(E)_{\lambda}
\ar@{->}[r]^-{c^*} & SH^*(E)_{\lambda}
}
$$
We claim more generally that $\mathrm{CO}:QH^*(E)\to \mathrm{HH}^*(\mathcal{F}_{\lambda}(E))$ factorizes:
\begin{equation}\label{Eqn CO factorizes}
\mathrm{CO}=\bigoplus_{\lambda} \mathrm{CO}_{\lambda}: QH^*(E)=\bigoplus_{\lambda} QH^*(E)_{\lambda} \to \bigoplus_{\lambda} \mathrm{HH}^*(\mathcal{F}_{\lambda}(E)).
\end{equation}
If we replace $E$ by a closed monotone symplectic manifold then \eqref{Eqn CO factorizes} follows directly from Theorem \ref{Theorem OC is a module map} and a duality argument relating $\mathrm{OC}$ and $\mathrm{CO}$ maps (see Sheridan \cite[Prop 2.6, Cor 2.10]{Sheridan}).
In the non-compact setup, this duality no longer holds; $\mathrm{OC}$ is defined using ordinary cycles but $\mathrm{CO}$ is defined using locally-finite cycles.
Nick Sheridan and the anonymous referee pointed out to us a way to adapt Theorem \ref{Theorem OC is a module map} directly to $\mathrm{CO}$ bypassing this dualization.

\begin{theorem}\label{Theorem CO respects esummands}
If $\mu \neq \lambda$, then the following maps vanish,
$$
\begin{array}{rcrcl}
\mathrm{CO}_{\lambda,\mu}&:& QH^*(E)_{\mu} &\to & \mathrm{HH}^*(\mathcal{F}_{\lambda}(E)),\\
\mathrm{CO}_{\lambda,\mu}&:& SH^*(E)_{\mu} &\to & \mathrm{HH}^*(\mathcal{W}_{\lambda}(E)).
\end{array}
$$
\end{theorem}
\begin{proof}
Let $c=c_1(TE)$.
Given $q\in QH^*(E)_{\mu}$ (respectively $SH^*(E)_{\mu}$), we know $(c-\mu)^N*q=0$ for some $N\in \N$. Since $\mathrm{CO}$ is an algebra homomorphism (see \ref{Subsection CO is a unital algebra hom}),
\begin{equation}\label{Eqn CO(c) is zero}
\hspace{20mm}
\mathrm{CO}(c-\mu)^N * \mathrm{CO}(q)=0 \in \mathrm{HH}^*(\mathcal{F}_{\lambda}(E)) \quad\textrm{ (respectively }\mathrm{HH}^*(\mathcal{W}_{\lambda}(E))).
\end{equation}
By Lemma \ref{Lemma COc1 and m0}, $\mathrm{CO}^0(c-\mu)^N = (\lambda-\mu)^N\cdot e_L \in CF^*(L,L)$ (respectively $CW^*(L,L)$) where $e_L$ is the cohomological unit and $m_0(L)=\lambda$. If $\mu\neq \lambda$ then Lemma \ref{Lemma trick CO} applied to $\varphi=(\lambda-\mu)^{-N}\mathrm{CO}(c-\mu)^N$ forces $\mathrm{CO}(c-\mu)^N*\cdot$ to be an isomorphism, so $\mathrm{CO}(q)=0$ by \eqref{Eqn CO(c) is zero}.
\end{proof}
\begin{corollary}\label{Corollary decomposition into esummands and diagram}
Equation \eqref{Eqn CO factorizes} holds. Also the restriction of $\mathrm{CO}$ to $\oplus SH^*(E)_{\lambda}$ decomposes:
$$
\mathrm{CO}=\bigoplus_{\lambda} \mathrm{CO}_{\lambda}: \bigoplus_{\lambda} SH^*(E)_{\lambda} \to \bigoplus_{\lambda} \mathrm{HH}^*(\mathcal{W}_{\lambda}(E))
$$
Moreover, the acceleration diagram respects the eigensummand decomposition: the following is a commutative diagram of unital algebra homomorphisms (via Theorems \ref{Theorem CO is unital alg hom} and \ref{Theorem Acceleration functor is algebra hom}).
$$
\xymatrix@C=50pt{ \mathrm{HH}^*(\mathcal{F}_{\lambda}(X)) \ar@{<-}^{\mathrm{HH}^*(\mathcal{AF})}[r] \ar@{<-}_{\mathrm{CO}}[d] & \mathrm{HH}^*(\mathcal{W}_{\lambda}(X))
\ar@{<-}^{\mathrm{CO}}[d] \\
 QH^*(E)_{\lambda}
\ar@{->}[r]^-{c^*} & SH^*(E)_{\lambda}
}
$$
 In particular, $\mathrm{CO}^0: SH^*(E)_{\lambda}\to \mathrm{HW}^*(L,L)$ is unital, where $L\in \mathrm{Ob}(\mathcal{W}_{\lambda}(E))$.
\end{corollary}
%
\section{The coproduct}
\label{Section The coproduct}
%
\subsection{The coproduct as a bimodule map}
%
%
%
Continue with the notation $\bA=\mathcal{W}(E)$ and $\mathcal{M}=\bA=\mathcal{W}(E)$ from \ref{Subsection psi structure map}. Fix an object $K$ in $\mathrm{Ob}(\mathcal{A})$. 
Let $\mathcal{L}$, $\mathcal{R}$ denote the Yoneda $\mathcal{A}$-modules
$$
\begin{array}{lll}
\mathcal{L}(X) = \mathrm{hom}_{\mathcal{A}}(K,X) = CW^*(K,X),\\ 
\mathcal{R}(X) = \mathrm{hom}_{\mathcal{A}}(X,K) = CW^*(X,K),
\end{array}
$$
where $X\in \mathrm{Ob}(\mathcal{A})$. Recall $\mathcal{R}$ represents the object $K$ (see \cite[Sec.(2g)]{Seidel}).
By \ref{Subsection tensor products},
$$
\mathrm{HH}_*(\mathcal{A},\mathcal{L}\otimes \mathcal{R}) \cong 
H^*(\mathcal{R}\otimes_{\mathcal{A}} \mathcal{L}).
$$
We want to construct a \emph{coproduct} $\Delta$ as an $\mathcal{A}$-bimodule map 
$$\Delta:\mathcal{M} \to \mathcal{L}\otimes \mathcal{R},$$
in particular $H^*(\Delta^{0|0}): HW^*(L_0',L_0) \to HW^*(K,L_0) \otimes HW^*(L_0',K)$ is the usual coproduct. 

\begin{figure}
\input{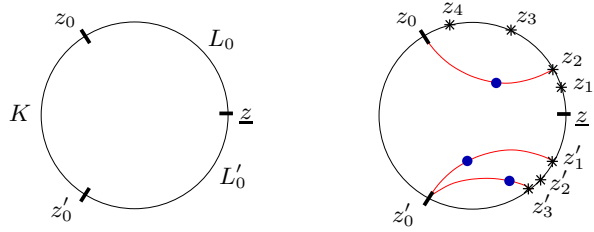}
\caption{$\Delta^{0|0}\!:\!CF^*(L_0',L_0;\underline{w}H)\!\to\!  CF^*(K,L_0;w_0H)\!\otimes\!
CF^*(L_0',K;w_0'H)$ and $\Delta^{4|3}(x_4,x_3,\mathbf{q}x_2,x_1,\underline{x},\mathbf{q}x_1',x_2',\mathbf{q}x_3')$.}  
\label{Figure coproduct}
\end{figure}
%
\subsection{The auxiliary data for the coproduct}
\label{Subsection First steps in a preliminary definition of coproduct}
Consider domains $S$ as in the second picture in Figure \ref{Figure coproduct}.
We consider a universal family $\mathcal{R}^{r|s,\mathbf{p},\mathbf{p}'} \to \mathcal{R}^{r+s+3}$ in which there are:\\
$\bullet$ two distinguished negative boundary punctures $z_0,z_0'$ on $S$ for the outputs of $\Delta$,\\
$\bullet$ one distinguished positive puncture $\underline{z}$ for the input of $\Delta$ from the module $CF^*(L_0',L_0;\underline{w}H)$.\\
We call $z_0$ the \emph{left output} and $z_0'$ the \emph{right output}.
The distinguished punctures $\underline{z},z_0,z_0'$ are fixed to lie at $1,e^{2\pi i/3},e^{4\pi i/3}$, after identifying $\overline{S}$ with the disc $\D$. This kills the reparametrization group of $\mathcal{R}^{r+s+3}$ and the local coordinates for $\mathcal{R}^{r+s+3}$ are given by the angular coordinates of the positive punctures $z_1,\ldots,z_r$, called \emph{left punctures}, lying on the arc $\underline{z},z_0$ and of the positive punctures $z_s',\ldots, z_1'$, called \emph{right punctures}, lying on the arc $z_0',\underline{z}$. The punctures are ordered as shown in Figure \ref{Figure coproduct}. The puncture $\underline{z}$ will be considered both a left and a right puncture.
%
%
The $\mathbf{p}$-labelling map now breaks into two maps 
%
%
%
\begin{equation}\label{Eqn p p' maps}
\mathbf{p}: F \to \{ r,\ldots,1,\infty\} \quad \textrm{ and } \quad \mathbf{p}': F' \to \{ \infty,1,\ldots,s\}
\end{equation}
where $F,F'$ are finite sets of indices $f,f'$ and where we use $\infty$ to denote $\underline{z}=z_{\infty}$. These give rise to labels $p_f,p_{f'}'$ (possibly not distinct). The map $\mathbf{p}$ encodes the presence of preferred points $\phi_f$ on geodesics connecting the left-output puncture $z_0$ with a left-input puncture $z_r,\ldots,z_1,z_{\infty}$, whereas the $\mathbf{p}'$ map encodes the presence of preferred points $\phi_{f'}$ on geodesics connecting 
the right-output $z_0'$ and the right-inputs $z_{\infty},z_1',\ldots,z_s'$. The module input puncture $\underline{z}=z_{\infty}$ can play both left/right roles: there can be geodesics connecting it to either $z_0$ or $z_0'$.

For the non-module inputs, the $\alpha$-forms are separated into two classes: $\alpha_k$ forms for the pairs $(z_0,z_k)$ and $\alpha_k'$ forms for the pairs $(z_0',z_k')$.
For the module input puncture $\underline{z}$, we have a form $\alpha_L$ associated to the left-pair $(z_0,\underline{z})$ and a form $\alpha_R$ associated to the right-pair $(z_0',\underline{z})$. 
The weight data is as follows, where all entries are integers $\geq 1$:
$$
\mathbf{w}=(w_r,\ldots,w_1) \quad \quad \mathbf{w}'=(w_1',\ldots,w_s')
 \quad \quad \underline{w} = \underline{w}_L + \underline{w}_R.
$$
%
%
%
%
%
%
%
The total one-form $\gamma$ on $S$ is:
$$
\gamma = \sum_{k=1}^r w_k \alpha_k + \underline{w}_L \alpha_L + \underline{w}_R \alpha_R + \sum_{k=1}^s w_k' \alpha_k' +  \sum_{f\in F} \beta_f + \sum_{f'\in F'} \beta_{f'},
$$
which satisfies $d\gamma \leq 0$; $d\gamma=0$ near $\partial S$; the pull-back of $\gamma$ to $\partial S$ is $0$. The output weights are:
$$
w_0 = \sum w_k + \underline{w}_L +  |F| \quad \textrm{ and } \quad 
w_0'=\sum w_k' + \underline{w}_R +  |F'|.
$$
Near the ends, in the relevant parametrization, $\gamma$ has the form: $w_k\, dt$, $\underline{w}\, dt$,  $w_k'\, dt$, $w_0\, dt$, $w_0'\, dt$ respectively near $z_k$, $\underline{z}$, $z_k'$, $z_0$, $z_0'$. This is consistent with Stokes' Theorem: $0\leq -\int d\gamma = w_0 + w_0' - \sum w_k - \sum w_k' - \underline{w}= |F|+|F'|$.

By mimicking \ref{Subsection moduli space of pseudoholo maps in wrapped case}, we obtain a moduli space $\mathcal{R}^{r|s,\mathbf{p},\mathbf{p}'}(\mathbf{x},\underline{x},\mathbf{x}')$ of solutions $u: S \to E$ of $(du-X\otimes \gamma)^{0,1}=0$ satisfying the relevant asymptotic and boundary conditions. Arguing as in \cite[Lemma 3.7]{Abouzaid-Seidel}, we can always assume that there is never more than one preferred point on a geodesic, since for symmetry reasons the contributions of isolated solutions will cancel when $\mathbf{p}$ or $\mathbf{p}'$ is not injective. When defining $\Delta$ as a sum of contributions $\Delta^{r|s,\mathbf{p},\mathbf{w},\underline{w}_L,\underline{w}_R,\mathbf{w}',\mathbf{p}'}$, we only sum over inclusions of subsets
\begin{equation}\label{Eqn p p' maps inclusions}
\mathbf{p}:F\subset \{r,\ldots,1,\infty\} \quad \textrm{ and } \quad 
\mathbf{p}':F'\subset \{\infty,1,\ldots,s\}.
\end{equation}
Because $\underline{z}$ can arise for both $\mathbf{p},\mathbf{p}'$ (i.e. there are two geodesics connecting to $\underline{z}$), we first need a chain model for $\mathcal{M}$ having two separate $\mathbf{q}$ variables.
\subsection{A chain level direct limit using a double telescope construction}
Let
$$
\boxed{\;\;\overline{CW}(X,X') \quad =\quad \bigoplus_{n\geq 2} \quad\bigoplus_{a + b = n,\, a\geq 1, b\geq 1}\quad CF^*(X,X';(a+b)H)[\mathbf{q}_L,\mathbf{q}_R]\;\;}
$$
for $X,X'\in \mathrm{Ob}(\mathcal{A})$. This vector space comes with two gradings $a,b \geq 1$ in $\N$, and it has two variables $\mathbf{q}_L,\mathbf{q}_R$ in degree $-1$ satisfying $\mathbf{q}_L^2=0$, $\mathbf{q}_R^2=0$, $\mathbf{q}_L\mathbf{q}_R=-\mathbf{q}_R \mathbf{q}_L$. 

The differential is defined linearly and equivariantly with respect to $\partial_{\mathbf{q}_L}$, $\partial_{\mathbf{q}_R}$. Summing along the row, the differential of $x + \mathbf{q}_L y  + \mathbf{q}_R z + \mathbf{q}_L\mathbf{q}_R w$ has the following contributions:
$$
\begin{array}{|l|r|r|r|r|}
\hline
\textrm{Input} & \multicolumn{4}{l}{\textrm{Contributions to } \mu^1(\textrm{Input})}\\
\hline
x  & (-1)^{|x|}\partial x & \strut & \strut & \strut \\
\hline
\mathbf{q}_L y   & (-1)^{|y|}(\mathfrak{K}_L \!-\! \mathrm{id})y &   (-1)^{|y|}\mathbf{q}_L\partial y & \strut & \strut \\
\hline
\mathbf{q}_R z    & (-1)^{|z|}(\mathfrak{K}_R \!-\! \mathrm{id})z &  & (-1)^{|z|} \mathbf{q}_R\partial z & \strut \\
\hline
\mathbf{q}_L\mathbf{q}_R w   & & 
-(-1)^{|w|} \mathbf{q}_R(\mathfrak{K}_L \!-\! \mathrm{id})w &   (-1)^{|w|}\mathbf{q}_L(\mathfrak{K}_R \!-\! \mathrm{id})w &   (-1)^{|w|}\mathbf{q}_L\mathbf{q}_R \partial w\\
\hline
\end{array}
$$
where we mimic Definition \ref{Definition wrapped floer complex}, in particular the $\mathfrak{K}$-maps are continuation maps: 
\begin{equation}\label{Eqn kappa L and R}
\begin{array}{rcl}
\mathfrak{K}_L:CF^*(X,X';(a+b)H)[\mathbf{q}_L,\mathbf{q}_R] & \to &
CF^*(X,X';(a+1+b)H)[\mathbf{q}_L,\mathbf{q}_R]\\
\mathfrak{K}_R:CF^*(X,X';(a+b)H)[\mathbf{q}_L,\mathbf{q}_R] & \to &
CF^*(X,X';(a+b+1)H)[\mathbf{q}_L,\mathbf{q}_R]
\end{array}
\end{equation}
where we stipulate that on $\overline{CW}^*(X,X')$ the map $\mathfrak{K}_L$ increases $a$ by $1$, and $\mathfrak{K}_R$ increases $b$ by $1$. At the analysis level (i.e. choices of auxiliary data) we construct the two maps in \eqref{Eqn kappa L and R} to be precisely the same if one ignores the grading by $(a,b)$:
$$\qquad \qquad \qquad \mathfrak{K}_L=\mathfrak{K}_R \qquad (\textrm{in particular } \mathfrak{K}_L\circ \mathfrak{K}_R=\mathfrak{K}_R\circ\mathfrak{K}_L).$$
An easy verification shows that $\mu^1\circ \mu^1=0$, so $\overline{CW}(X,X')$ is a chain complex.\\
{\bf Remark.} \emph{A double telescope at the topological level would usually result in an asymmetry in the chain complex. This would be present above if we had chosen more generally to consider the sum over $CF^*(X,X';aH_L + bH_R)[\mathbf{q}_L,\mathbf{q}_R]$ for two different Hamiltonians $H_L,H_R$, with $\mathfrak{K}_L,\mathfrak{K}_R$ increasing respectively $a,b$ by one. In this case $\mathfrak{K}_L,\mathfrak{K}_R$ are genuinely different, and to ensure $\mu^1\circ \mu^1=0$ an additional contribution $(-1)^{|w|}\mathbf{q}_L\mathbf{q}_R T(w)$ would be required in $\mu^1$ for a chain homotopy $T$ satisfying $\mathfrak{K}_L\circ \mathfrak{K}_R-\mathfrak{K}_R\circ\mathfrak{K}_L = \partial \circ T + T \circ \partial$.}

Define an ``inclusion'' $i_L: CW^*(X,X')\to \overline{CW}^*(X,X')$ by sending $\mathbf{q} \mapsto \mathbf{q}_L$ and
$$
CF^*(X,X';nH) \to CF^*(X,X';(n+1)H), \; x\mapsto \mathfrak{K}_R(x).
$$
Define the projection $\mathrm{pr}:\overline{CW}^*(X,X') \to CW^*(X,X')$ by sending $\mathbf{q}_L\mapsto \mathbf{q}$, $\mathbf{q}_R\mapsto \mathbf{q}$, and
$$
CF^*(X,X';(a+b)H) \to CF^*(X,X';nH), \; x\mapsto x \qquad \textrm{ where }n=a+b.
$$
Observe that $\mathrm{pr}\circ i_L = \mathfrak{K}_R|_{CW^*(X,X')}$.
\begin{lemma}\label{Lemma qis CW and overlineCW}
$\overline{CW}^*(X,X')$ is quasi-isomorphic to $CW^*(X,X')$, intertwining the action of $\partial_{\mathbf{q}}^L+\partial_{\mathbf{q}}^R$ on $\overline{CW}^*(X,X')$ with the action of $\partial_{\mathbf{q}}$ on $CW^*(X,X')$. Explicitly, $i_L$ and $\mathrm{pr}$ are quasi-isomorphisms, and  $\mathrm{pr}\circ i_L$ and $i_L \circ \mathrm{pr}$ are homotopic to identity maps.
\end{lemma}
\begin{proof}
Observe that, taking $x,y,z$ to be $\partial$-cycles, implies the identifications $[\mathfrak{K}_L y] = [y]$ and $[\mathfrak{K}_R z] = [z]$
on cohomology, and $[\mathfrak{K}_L \mathfrak{K}_R x] = [\mathfrak{K}_R\mathfrak{K}_L x]$ already at the chain level.
In particular, any $\partial$-cycle $x$ defines a class in the cohomology $\overline{HW}^*(X,X')$. Although many copies of this appear for given $n$, since $x\in CF^*(X,X';(a+b)H)$ for all $a,b\geq 1$ with $a+b=n$, these are all identified on cohomology since we can compare any two copies by applying $\mathfrak{K}_R,\mathfrak{K}_L$ appropriately many times until they both lie in the same summand of $CW^*(X,X')$. So there is a copy of each cohomology representative $[x]\in HF^*(L,L';nH)$ in grading $(a,b)=(n-1,1)$ (for $n\geq 2$), thanks to the cohomological identifications provided by the $\mathfrak{K}$-maps.

We now run an argument analogous to \cite[Sec.3.7]{Abouzaid-Seidel}.

Abbreviate $F_{x,y}=CF^*(X,X';(x+y)H)$ for integers $x,y\geq 1$. Abbreviate
$$
C_{k,\ell} = \bigoplus_{1 \leq x < k} \; \bigoplus_{1 \leq y < \ell} F_{x,y}[\mathbf{q}_L,\mathbf{q}_R]
\;\;\oplus\;\; \bigoplus_{1 \leq x < k} F_{x,\ell}[\mathbf{q}_L]
\;\;\oplus\;\; \bigoplus_{1 \leq y < \ell} F_{k,y}[\mathbf{q}_R]
\;\;\oplus\;\; F_{k,\ell}.
$$
Imagine these summands arranged as a grid of squres in the $xy$-plane, so $\mathfrak{K}_R$ maps upwards by one square, $\mathfrak{K}_L$ maps to the right by one square, and all non-$\mathfrak{K}$ terms in the differential map from a square to itself. Thus, $C_{k,\ell}$ are subcomplexes which exhaust $CW^*(X,X')$.
The key observation is that if we take any one of the squares (i.e. a summand above), and we consider the differential obtained by dropping $\mathfrak{K}$-terms, then this square becomes acyclic.

Consider the smaller subcomplex $C'$ obtained by removing the summand $F_{1,1}[\mathbf{q}_L,\mathbf{q}_R]$. Then $C'$ is quasi-isomorphic to $C_{k,\ell}$, because the quotient $C_{k,\ell}/C'$ is the acyclic square in position (1,1). More generally, up to quasi-isomorphism we can inductively remove one square at a time (always removing a square with minimal occurring $x$ value or minimal $y$ or both minimal). Thus $C_{k,\ell}$ is quasi-isomorphic to $F_{k,\ell}$. Just as in the original argument \cite[Sec.3.7]{Abouzaid-Seidel}, one checks that $\mathfrak{K}_R:F_{k,\ell}\to F_{k,\ell+1}$ (respectively $\mathfrak{K}_L:F_{k,\ell}\to F_{k+1,\ell}$) commutes up to chain homotopy with the inclusions $C_{k,\ell}\subset C_{k,\ell+1}$ (respectively $C_{k,\ell}\subset C_{k+1,\ell}$); the chain homotopy is $x\mapsto \mathbf{q}_R x$ (respectively $x\mapsto \mathbf{q}_L x$). On cohomology, the direct limit of the cohomologies $H^*(F_{k,\ell})$ over the maps $\mathfrak{K}_R$, $\mathfrak{K}_L$ is $HW^*(X,X')$ essentially by definition. Thus $H^*(CW^*(X,X');\mu^1)\cong HW^*(X,X')$ are canonically isomorphic.
\end{proof}
%
\subsection{A new construction of the diagonal bimodule}
\label{Subsection A new construction of the diagonal bimodule}
 We will define a new $\mathcal{A}$-bimodule $\overline{\mathcal{M}}$, which is quasi-isomorphic to $\mathcal{M}$, satisfying
$$
\overline{\mathcal{M}}(X,X') = \overline{CW}^*(X',X).
$$
The composition maps
$
\mu_{\overline{\mathcal{M}}}^{r|s}: \bA(L_r,\ldots,L_0)\otimes \overline{\mathcal{M}}(L_0,L_0')\otimes \bA(L_0',\ldots,L_s')\to \overline{\mathcal{M}}(L_r,L_s')
$
are defined by mimicking \ref{Subsection wrapped diagonal bimodule} using the same auxiliary data, so that we count the same Floer solutions as for $\mu_{\mathcal{M}}^{r|s}$, but we now ``separate'' the total output by distributing it according to weights. 
Explicitly, $\mu_{\overline{\mathcal{M}}}^{r|s}$ on a generator 
\begin{equation}\label{Eqn big module input}
\mathbf{q}^{i_r}x_r\otimes \cdots \otimes \mathbf{q}^{i_1}x_1 \otimes \mathbf{q}_L^{i_{\infty}}\mathbf{q}_R^{i_{\infty}'}\underline{m} \otimes \mathbf{q}^{i_1'}y_1 \otimes \cdots \otimes \mathbf{q}^{i_s'}y_s
\end{equation}
with input weights $w_r,\ldots,w_1,\underline{w}_L+\underline{w}_R,w_1',\ldots,w_s'$ (where the module input has grading $(a,b)=(\underline{w}_L,\underline{w}_R)$) will contribute to the output summand $$CF^*(L_s',L_r;(w_{0,L}+w_{0,R})H)\subset \overline{\mathcal{M}}(L_r,L_s')$$ in grading
$w_{0,L} = \sum w_k + \underline{w}_L + |F|$ and $w_{0,R} = \sum w_k' + \underline{w}_R + |F'|$
where $F,F'$ are defined as in \eqref{Eqn p p' maps inclusions}. Note we use two $\mathbf{p}$-maps as in \eqref{Eqn p p' maps}, one for the left punctures and one for the right punctures. So we place a preferred point on the geodesic connecting $(z_0,z_k)$, (respectively $(z_0,z_k')$) if $i_k=1$ (respectively $i_k'=1$).
To define the $\mathbf{q}_L$, $\mathbf{q}_R$,  $\mathbf{q}_L \mathbf{q}_R$ outputs of $\mu_{\overline{\mathcal{M}}}^{r|s}$, we impose $\partial_{\mathbf{q}}^L$ and $\partial_{\mathbf{q}}^R$ equivariance. Here, we define
the left and right operators $\partial_{\mathbf{q}}^L$, $\partial_{\mathbf{q}}^R$ to act respectively on the left and right inputs, associated to $\mathbf{p}$ and $\mathbf{p}'$. These should be viewed in grading $+1$ and as acting from the left, so:
$$
\begin{array}{rcl}
\partial_{\mathbf{q}}^L(x_r\otimes\cdots\otimes x_1\otimes\underline{x} \otimes x_1' \otimes\cdots \otimes x_s')&=&
\partial_{\mathbf{q}}(x_r\otimes \cdots \otimes x_1  \otimes \underline{x}) \otimes x_1' \otimes \cdots \otimes x_s'\\[1mm]
\partial_{\mathbf{q}}^R(x_r\otimes\cdots\otimes x_1\otimes\underline{x} \otimes x_1' \otimes\cdots \otimes x_s')&=&
(-1)^{\sigma(x)_{1}^r}\, x_r\otimes \cdots \otimes x_1 \otimes \partial_{\mathbf{q}}(\underline{x} \otimes x_1' \otimes \cdots \otimes x_s')
\end{array}
$$
and for the outputs we define $\partial_{\mathbf{q}}^L(x,y)=\partial_{\mathbf{q}}(x)\otimes y$ and $\partial_{\mathbf{q}}^R(x,y)=(-1)^{\mathrm{deg}(x)}x\otimes \partial_{\mathbf{q}}(y)$.

The $A_{\infty}$-relations required for $\overline{\mathcal{M}}$ to be an $\mathcal{A}$-bimodule hold because we know that these relations hold for the diagonal $\mathcal{A}$-bimodule $\mathcal{M}$ and because the gluing/breaking involved in compositions respects the grading on weights $(a,b)=(\underline{w}_L,\underline{w}_R)$ described above.

\noindent { \bf Example.} \emph{Given $x\in CF^*(X',X; (a+b)H)$, the following picture contributes to the $\mathbf{q}^0$-output of $\mu_{\overline{\mathcal{M}}}^{0|0}(\mathbf{q}_L \mathbf{q_R} x)$ lying in $CF^*(X',X; (a+1+b+1)H)$. Here $F=\{\infty\}$, $F'=\{\infty\}$.
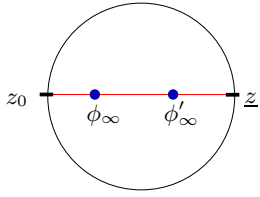
\begin{center}
\input{DoubleTelescope.tex}
\end{center}
Given such a Floer solution, there is another Floer solution contributing to the same output (in the same grading) with $\phi_{\infty},\phi_{\infty'}$ interchanged (analytically it is the same PDE solution, since our auxiliary data is constructed independently of the labelling of the preferred points). The usual symmetry argument (see the \emph{Remark about symmetry} in \ref{Subsection wrapped A infinity structure}) implies that these two Floer solutions are counted with opposite orientation signs, so they cancel.}
\begin{lemma}\label{Lemma qis M and overlineM}
The $\mathcal{A}$-bimodule $\overline{\mathcal{M}}$ is quasi-isomorphic to the diagonal $\mathcal{A}$-bimodule $\mathcal{M}$, intertwining the action of $\partial_{\mathbf{q}}^L+\partial_{\mathbf{q}}^R$ on $\overline{\mathcal{M}}$ with the action of $\partial_{\mathbf{q}}$ on $\mathcal{M}$. 

An explicit quasi-isomorphism $\mathbf{pr}:\overline{\mathcal{M}}\to \mathcal{M}$ of bimodules is given by:
$$
\begin{array}{l}
\mathbf{pr}^{r|s}:\mathcal{A}(X_r,\ldots,X_0)\otimes \overline{\mathcal{M}}(X_0,X_0')\otimes \mathcal{A}(X_0',\ldots,X_s') \to \mathcal{M}(X_r,X_s')
\\
\mathbf{pr}^{0|0}=\mathrm{pr}: \overline{\mathcal{M}}(X_0,X_0') \to \mathcal{M}(X_0,X_0').
\\
\mathbf{pr}^{r|s}=0 \textrm{ if } r+s \geq 1.
\end{array}
$$
where $\mathrm{pr}$ is the projection defined in Lemma \ref{Lemma qis CW and overlineCW}.
\end{lemma}
\begin{proof}
It remains to check that the map $\mathbf{pr}$ satisfies the $A_{\infty}$-relations described in \ref{Subsection Bimodules morphism}. This reduces to checking the identity
$$
\mathbf{pr}^{0|0}(\mu_{\overline{\mathcal{M}}}^{r|s}(x_r,\ldots,x_1,\underline{m},y_1,\ldots,y_s)) = \mu_{\mathcal{M}}^{r|s}(x_r,\ldots,x_1,\underline{\mathbf{pr}^{0|0}(\underline{m})},y_1,\ldots,y_s).
$$
This equation simply says that the solutions that $\mu_{\mathcal{M}}^{r|s}$ counts are the same as those counted by $\mu_{\overline{\mathcal{M}}}^{r|s}$ if we ignore the $(a,b)$-grading of the outputs. This holds by construction.
\end{proof}
\noindent {\bf Remark.} \emph{A priori this $A_{\infty}$-bimodule quasi-isomorphism has a quasi-inverse $\mathcal{M}\to \overline{\mathcal{M}}$, but it is probably difficult to describe it explicitly.}
%
\subsection{Definition of the coproduct}
\label{Subsection The coproduct}

The $\mathcal{A}$-bimodule map
$
\Delta: \overline{\mathcal{M}}\to \mathcal{L}\otimes \mathcal{R}
$
is determined by the $\partial_{\mathbf{q}}^L$- and $\partial_{\mathbf{q}}^R$-equivariant maps
$$
\begin{array}{ll} \mathcal{W}(L_r,\ldots,L_0;\mathbf{w}) \otimes 
CF^*(L_0',L_0;(\underline{w}_L+\underline{w}_R)H)[\mathbf{q}_L,\mathbf{q}_R] \otimes \mathcal{W}(L_0',\ldots,L_s';\mathbf{w}')
 \\  \longrightarrow CF^*(K,L_r;w_0H)[\mathbf{q}]\otimes 
CF^*(L_s',K;w_0'H)[\mathbf{q}],
\end{array}
$$
counting Floer solutions for the auxiliary data described in \ref{Subsection First steps in a preliminary definition of coproduct}.
Thus, given the input as in \eqref{Eqn big module input}, we require $k\in \mathbf{p}(F)$ iff $i_k=1$ (respectively $k\in \mathbf{p}'(F')$ iff $i_k'=1$). Equivalently: there is a $\mathbf{q}$ iff there is a preferred point on the geodesic connecting the relevant puncture with $z_0$ (respectively $z_0'$). For the module input we have a left/right  $\mathbf{q}$-variable which keeps track of the presence of a preferred point on the left/right geodesic connecting to $\underline{z}=z_{\infty}$.
The $x_0\otimes x_0'$ coefficient of the output is then the count with weight $\pm t^{E_{\mathrm{top}}(u)}$ of the isolated solutions $u\in \mathcal{R}^{r|s,\mathbf{p},\mathbf{p}'}(x_r,\ldots,x_1,\underline{x},x_1',\ldots,x_s')$.\\
\noindent {\bf Remark.} \emph{Solutions are counted with a sign $(-1)^{\ddagger}$ \cite[(4.22)]{Abouzaid} where
$$
\textstyle \ddagger = k+\sum_{j=1}^s (s-j+1)\mathrm{deg}(x_j') + s\, \mathrm{deg}(\underline{x}) + \sum_{j=1}^r (j+s)\mathrm{deg}(x_j)
$$
where $k$ is the Koszul sign obtained from viewing the $\mathbf{q}$ variables as operators of degree $-1$ acting from the left (compare \ref{Subsection wrapped A infinity structure} and \ref{Subsection The construction of psi}).
}

As an illustration, we show the decorated domains counted by $\Delta^{0|0}$, where we label an output with $\mathbf{q}$ if it contributes to the $\mathbf{q}$-part of that output:
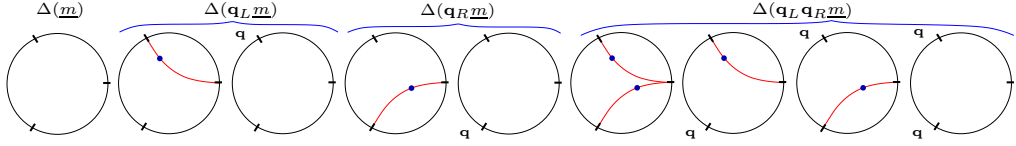
\begin{center}
\input{Coproduct-count.tex}
\end{center}
To illustrate the $A_{\infty}$-equations, we show the pictures involved in the equation $\Delta^{0|0}(\mu_{\overline{\mathcal{M}}}^{0|0}(\mathbf{q}_L \underline{m})) = \mu_{\mathcal{L}\otimes \mathcal{R}}^1(\Delta^{0|0}(\mathbf{q}_L \underline{m}))$. Ignoring signs, the left-hand side is $\Delta^{0|0}(\mathbf{q}_L\partial \underline{m}+(\mathfrak{K}_L-\mathrm{id})(\underline{m}))$ counting:
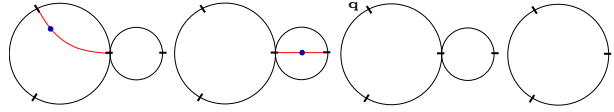
\begin{center}
\begin{picture}(0,0)%
\includegraphics{Coproduct-test1b.pstex}%
\end{picture}%
\setlength{\unitlength}{1119sp}%
\begingroup\makeatletter\ifx\SetFigFont\undefined%
\gdef\SetFigFont#1#2#3#4#5{%
  \reset@font\fontsize{#1}{#2pt}%
  \fontfamily{#3}\fontseries{#4}\fontshape{#5}%
  \selectfont}%
\fi\endgroup%
\begin{picture}(13387,2406)(1416,1745)
\put(8925,3920){\makebox(0,0)[lb]{\smash{{\SetFigFont{5}{6.0}{\rmdefault}{\mddefault}{\updefault}{\color[rgb]{0,0,0}$\mathbf{q}$}%
}}}}
\end{picture}%

\end{center}
and the right-hand side is $(\mu^1_{\mathcal{A}}\otimes \mathrm{id} + \mathrm{id}\otimes \mu^1_{\mathcal{A}})\circ \Delta^{0|0}(\mathbf{q}_L \underline{m})$ counting:
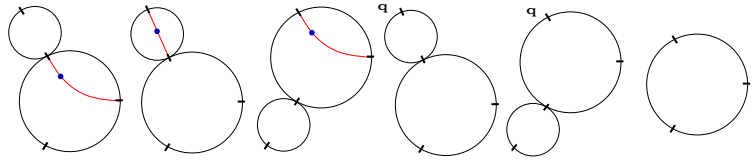
\begin{center}
\input{Coproduct-test2b.tex}
\end{center}
With the obvious labelling: pictures $1.1$, $1.2$, $2.1$, $2.2$, $2.3$ cancel as they are the boundary of a $1$-dimensional moduli space; similarly for $1.3$, $2.4$, $2.5$; pictures $1.4$, $2.6$ cancel due to signs.

That $\Delta$ satisfies the $A_{\infty}$-equations required from an $\mathcal{A}$-bimodule map $\overline{\mathcal{M}} \to \mathcal{L}\otimes \mathcal{R}$ is now an exercise in bookkeeping, compare Remark 4 in \ref{Subsection wrapped A infinity structure}. The key observation is that in codimension $1$, either a subset of the punctures (and possibly some preferred points) converges to a common point on the boundary. If the common point is not $z_0,z_0',\underline{z}$, it gives rise to a $\mu_{\mathcal{A}}^d$-term (part of the bar differential for the $\mathcal{A}$-bimodule $\overline{\mathcal{M}}$). If the common point is $z_0$ it gives rise to a $\mu^{R|0}_{\mathcal{L}\otimes \mathcal{R}}$-term; if $z_0'$, it gives rise to a $\mu^{0|S}_{\mathcal{L}\otimes \mathcal{R}}$-term;
if $\underline{z}$, it gives rise to a $\mu_{\overline{\mathcal{M}}}^{R|S}$-term.
%
%
%
\subsection{The $\mathrm{OC}$ map using the new diagonal bimodule}
\label{Subsection The OC map using new diag bimodule}

We define 
$$\mathrm{OC}:\mathrm{CC}_*(\mathcal{A},\overline{\mathcal{M}})
\to \mathrm{CC}_*(\mathcal{A},\mathcal{M})
\to SC^*(E)$$ 
by composing $\mathrm{CC}_*(\mathrm{pr})$ and the usual $\mathrm{OC}$ map,
where $\mathrm{pr}$ was defined in Lemma \ref{Lemma qis M and overlineM}.
The new $\mathrm{OC}$ map can also be defined directly, mimicking $\mu_{\overline{\mathcal{M}}}^{r|s}$, and all weights contribute to the output weight for $SC^*(E)$ (which only uses a $\mathbf{q}$-variable; we do not form a double complex).\\[1mm]
\noindent {\bf Remark.} \emph{We will not need a map $\mathrm{CO}$ which lands in $\mathrm{CC}^*(\mathcal{A},\overline{\mathcal{M}})$. For such a map, one would need to use a double telescope chain model for $SC^*(E)$.}
\section{The generation criterion in the monotone setting}
\label{Section Generation criterion}

\subsection{The generation criterion in the monotone setting}
\label{Subsection The generation criterion in the monotone setting}

\begin{definition}[Split-generating {\cite[Sec.4c,3j]{Seidel}}]
A full subcategory $\mathcal{S}$ \emph{split-generates} an $A_{\infty}$-category $\mathcal{A}$ if for each $A\in \mathrm{Ob}(\mathcal{A})$, there is a twisted complex built from $\mathrm{Ob}(\mathcal{S})$ admitting $A$ as a summand.
We sometimes say $\mathrm{Ob}(\mathcal{S})$ \emph{split-generates} $\mathcal{A}$.
\end{definition}

\begin{criterion*}[Abouzaid  \cite{Abouzaid}]\label{Theorem Abouzaid generating criterion}
Let $(M,d\theta)$ be an exact symplectic manifold conical at infinity. Suppose
$\mathcal{S}$ is a full subcategory of $\mathcal{W}(M)$ such that $1$ lies in the image of %
$$
\xymatrix@C=40pt{ 
\mathrm{HH}_*(\mathcal{S}) \ar@{->}^-{\mathrm{HH}_*(\mathrm{incl})}[r]  &  \mathrm{HH}_*(\mathcal{W}(M)) 
\ar@{->}[r]^-{\mathrm{OC}_M} & SH^*(M).
}
$$
Then $\mathcal{S}$ split-generates $\mathcal{W}(M)$.
\end{criterion*}

By Theorem \ref{Theorem OC is a module map}, it in fact suffices that $\mathrm{im}(\mathrm{OC}_M)$ contains an $SH^*(M)$-invertible:
 
\begin{lemma}\label{Lemma QH 1 is enough}
 If $s\in \mathrm{im}(\mathrm{OC}_M)$ is invertible in $SH^*(M)$ then $1 \in \mathrm{im}(\mathrm{OC}_M)$.
\end{lemma}
\begin{proof}
$\mathrm{OC}_M(x)=s$ for some $x \in \mathrm{HH}_*(\mathcal{W}(M))$, so since $\mathrm{OC}_M$ is an $SH^*(M)$-module map by Theorem \ref{Theorem OC is a module map} (which of course holds also in the exact setup), $\mathrm{OC}_M(s^{-1}\cdot x) = s^{-1}\cdot s = 1$.
\end{proof}

\begin{theorem}\label{Theorem generation in monotone setting}
Let $E$ be a monotone symplectic manifold conical at infinity. Let $\mathcal{S}$ be a full subcategory of $\mathcal{W}_{\lambda}(E)$. If 
$$
\xymatrix@C=40pt{ 
\mathrm{HH}_*(\mathcal{S}) \ar@{->}^-{\mathrm{HH}_*(\mathrm{incl})}[r]  &  \mathrm{HH}_*(\mathcal{W}_{\lambda}(E)) 
\ar@{->}[r]^-{\mathrm{OC}} & SH^*(E)_{\lambda},
}
$$
hits an invertible element (in the notation of \ref{Subsection Eigensummands}), then $\mathcal{S}$ split-generates $\mathcal{W}_{\lambda}(E)$.
\end{theorem}

\subsection{Outline of the proof of the generation criterion}
\label{Subsection Outline of the proof of generation}

Fix an object $K$ in $\mathcal{W} (E)$ (as usual, in the monotone setting, we always mean $\mathcal{W}_{\lambda}(E)$ for a fixed $m_0$-value $\lambda$). Recall $\mathcal{S}$ is a full subcategory of $\mathcal{W}(E)$. 
Let $\mathcal{L}$, $\mathcal{R}$ denote the Yoneda $\mathcal{S}$-modules
$$
\begin{array}{lll}
\mathcal{L}(X) = \mathrm{hom}_{\mathcal{S}}(K,X) = CW^*(K,X),\\ 
\mathcal{R}(X) = \mathrm{hom}_{\mathcal{S}}(X,K) = CW^*(X,K),
\end{array}
$$
where $X\in \mathrm{Ob}(\mathcal{S})$. Recall $\mathcal{R}$ represents the object $K$ (see \cite[Sec.(2g)]{Seidel}).
By \ref{Subsection tensor products},
\begin{equation}
\label{Eqn HH mathcal S LR is H R times S}
\mathrm{HH}_*(\mathcal{S},\mathcal{L}\otimes \mathcal{R}) \cong 
H^*(\mathcal{R}\otimes_{\mathcal{S}} \mathcal{L}).
\end{equation}
Recall from Section \ref{Section The coproduct} that the coproduct $\Delta$ is an $\mathcal{S}$-bimodule map 
$$\Delta:\mathcal{S} \to \mathcal{L}\otimes \mathcal{R},$$
where $\mathcal{S}$ is the diagonal $\mathcal{S}$-bimodule as in \ref{Subsection A new construction of the diagonal bimodule}. Consider the composites
$$
\xymatrix@C=35pt@R=6pt{
\mathrm{C}_1:\mathrm{HH}_*(\mathcal{S},\mathcal{S}) 
\ar@{->}^-{\mathrm{OC}}[r] &
 SH^*(E)
\ar@{->}^-{\mathrm{CO}^0}[r] &
 HW^*(K,K)\\
\mathrm{C}_2:\mathrm{HH}_*(\mathcal{S},\mathcal{S}) \ar@{->}^-{\mathrm{HH}_*(\Delta)}[r] &
 \mathrm{HH}_*(\mathcal{S},\mathcal{L}\otimes \mathcal{R}) \cong H^*(\mathcal{R}\otimes_{\mathcal{S}}\mathcal{L}) 
\ar@{->}^-{\mu_{\mathcal{S}}}[r] &
 HW^*(K,K)
}
$$
using the $\mu_{\mathcal{S}}^n$ to compose all morphisms of the chain complex $\mathcal{R}\otimes_{\mathcal{S}}\mathcal{L}$.

The key ingredient of the proof is to show that these composites are equal (up to a global sign $(-1)^{n(n+1)/2}$, where $\dim E=2n$), which we will prove in \ref{Subsection the key ingredient}. The hypothesis that $\mathrm{OC}$ hits $1\in SH^*(E)$ implies $\mathrm{C}_1$ hits the identity $e_K\in HW^*(K,K)$ since $\mathrm{CO}^0$ is a unital algebra homomorphism (see Theorem \ref{Theorem CO is unital alg hom}). So $\mathrm{C}_2$ hits $e_K$. For purely algebraic reasons \cite[Appendix A]{Abouzaid} this implies that $\mathcal{R}$ (which represents $K$) is a summand of a twisted complex built from objects in $\mathcal{S}$. Thus any object $K$ in $\mathcal{W}(E)$ is split-generated by objects in $\mathcal{S}$, as required. 


\begin{remark}\label{Remark idempotent summands}
 Suppose $\mathcal{S}$ is a full subcategory of $\mathcal{W}_{\lambda}(E)$, such that the composite
 $$
 \mathrm{CO}^0\circ \mathrm{OC}:\mathrm{HH}_*(\mathcal{S}) \to SH^*(E) \to HW^*(K,K)
 $$
 hits the identity $e_K$. Suppose $K$ is Floer cohomologically essential: $HW^*(K,K)\neq 0$. We claim that $m_0(K)=\lambda$, i.e. $K \in  \mathrm{Ob}(\mathcal{W}_{\lambda}(E))$. Indeed, by Theorem \ref{Theorem OC is a module map}, $\mathrm{OC}(\mathrm{HH}_*(\mathcal{S}))$ must land in the eigensummand $SH^*(E)_{\lambda}$, and the restricted map $\mathrm{CO}^0:SH^*(E)_{\lambda}\to HW^*(K,K)$ will vanish unless $m_0(K)=\lambda$ (see \ref{Subsection CO maps respect eigensummand decompositions}) and it is unital when $m_0(K)=\lambda$, by Corollary \ref{Corollary decomposition into esummands and diagram}.
So it suffices to consider Lagrangians with a single $m_0$-value $\lambda$ in the proof of the generation criterion. Also, in the Generation Criterion and in Lemma \ref{Lemma QH 1 is enough} we only care that $\mathrm{OC}|_{\mathcal{S}}$ hits the idempotent summand of $1$ in the generalized eigensummand $SH^*(E)_{\lambda}$.
\end{remark}

This outline is precisely the argument of Abouzaid \cite{Abouzaid}. To prove  Theorem \ref{Theorem Introduction 2} we will therefore only highlight the novelties of the proof in our monotone setup.
\subsection{The moduli space of marked annuli}
\label{Subsection the key ingredient}

We now prove that the composites $\mathrm{C}_1,\mathrm{C}_2$ from \ref{Subsection Outline of the proof of generation} agree. This involves studying a moduli space of annuli, which in Abouzaid's case is described in \cite[Appendix C4,C5]{Abouzaid}, and in our case will be endowed with additional decorations coming from geodesics and preferred points. Keeping track of this additional data may seem cumbersome at first, but it has the advantage that all gluings are natural: so in fact we will not need the first homotopy of Abouzaid \cite[Sec.6.2]{Abouzaid} which changes the weights so that gluing becomes possible, and we will only need to construct the analogue of his second homotopy \cite[Sec.6.3]{Abouzaid}. Moduli spaces of annuli with $n$ boundary marked points are given Deligne-Mumford compactifications by being viewed as moduli spaces of elliptic curves with real structures via doubling, and an \'{e}tale forgetful map to $\mathcal{M}_{1,n}$, see \cite[Ch.2]{FOOO}.

\begin{lemma}
Inside the Deligne-Mumford moduli space of annuli with two marked points, there is a real one-dimensional submanifold (not transverse to the boundary) whose boundary points are depicted in the first and last pictures of Figure \ref{Figure coproduct2} (ignoring decorations).
\end{lemma}
\begin{proof}[Sketch Proof.]
For $r\in (1,\infty)$, consider the subfamily defined by
\begin{equation}
\label{Eqn varepsilon annuli family}
\varepsilon_r = \{ z\in \C: 1 \leq |z| \leq r \} \setminus \{ -1,+r \}.
\end{equation}
(Recall the space of conformal structures on annuli in $\C$ centred at $0$ is parametrized by the ratio of the outer and inner radius.)
By doubling up the annulus, \eqref{Eqn varepsilon annuli family} determines a real 1-family in $\mathcal{M}_{1,2}$ parametrised by $e^r$.  Abouzaid shows in \cite[Appendix C.4]{Abouzaid} that the boundary strata carry two nodes (even though this would typically be a complex codimension two phenomenon).
When $r\to \infty$, one of the two nodes forms in the interior of the annulus, thus in the limit the annulus breaks into two discs joined at an interior point as in picture 1. When $r\to 1$, the two nodes arise as a limit of two pairs of boundary points of the annulus connected by two shrinking geodesic segments (the dotted lines in picture 3). In the limit this yields two discs joined at two nodes as in picture 5. 
\end{proof}

Analogous families of annuli can be constructed when there are several marked points on the outer circle $|z|=r$, which we view as families lying over the family \eqref{Eqn varepsilon annuli family} via the forgetful map. Thus the punctures at $-1,+r$ are distinguished (marked by dashes in Figure \ref{Figure coproduct2}).
The dotted lines in pictures 2 and 3 of Figure \ref{Figure coproduct2} are examples of hyperbolic geodesics which shrink in the limit to give respectively pictures 1 and 5. In picture 2, we work in the $(s,t)$ plane, where $z=e^{s+it}$ parametrizes the annulus, so we've passed to the cover $\R\to \R/\Z$ of the angular variable. The limit $r\to \infty$ corresponds to stretching the strip in picture 2 near the dotted line. Picture 4 is biholomorphic to 3, and shrinking the dotted lines in 3 corresponds to stretching the two handles in 4 (corresponding to the limit $r\to 1$). The dotted lines in pictures 2 and 3 play the same role as the ``cuts'' in Figure
\ref{Figure popsicle} of \ref{Subsection the auxiliary data for wrapped category} and Figure \ref{Figure cutting} of \ref{Subsection wrapped A infinity structure}.
%
%
\begin{figure}
\input{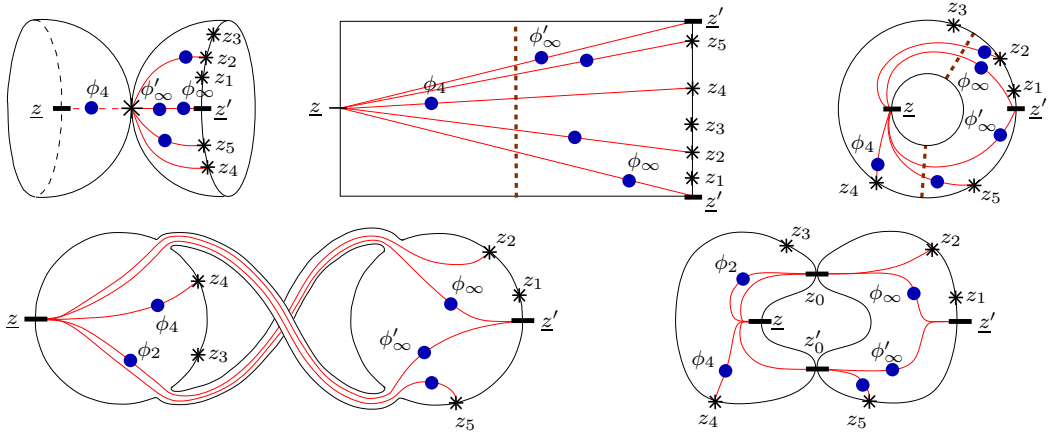}
\caption{Breaking analysis for the moduli space of annuli.} 
\label{Figure coproduct2}
\end{figure}

Observe our moduli space of annuli in picture 2 or 3 of Figure \ref{Figure coproduct2}. The distinguished boundary punctures are $\underline{z}=-1$ (which will be the output) and $\underline{z}'=\underline{z}_{\infty}=+r$ (which will be the bimodule input), and the other boundary punctures $z_1,\ldots,z_n$ are ordered counterclockwise on the outer boundary. Mimicking the construction of the coproduct, the auxiliary data involves two $\mathbf{p}$-maps as in \eqref{Eqn p p' maps} which keep track of the presence of preferred points on geodesics joining $\underline{z}$ with the other punctures, with the important feature that there are two types of (deformed) geodesics joining $\underline{z},\underline{z}'$ (in Figure \ref{Figure coproduct2}, the \emph{left geodesic} is the one carrying $\phi_{\infty}$, the \emph{right geodesic} carries $\phi_{\infty}'$). The geodesics are deformed so that they come in orthogonally to the boundary. We have weights $w_1,\ldots,w_n$ associated to the non-module inputs (coming from $\mathcal{A}$), and a graded weight $(a,b)=(\underline{w}_L',\underline{w}_R')$ for the module input (coming from the diagonal bimodule $\mathcal{S}\subset \overline{\mathcal{M}}$). There are closed forms $\underline{\alpha}_L, \underline{\alpha}_R,\alpha_1,\ldots,\alpha_n$; sub-closed forms $\beta_f$, $\beta_{f'}$. These define a total $1$-form $\gamma = \sum w_k \alpha_k + \underline{w}_L' \underline{\alpha}_L + \underline{w}_R' \underline{\alpha}_R + \sum \beta_f + \sum \beta_{f'}$, which pulls back to $w_k\, dt$, $(\underline{w}_L' + \underline{w}_R')\,dt$, $\underline{w}\, dt$ near $z_k,\underline{z}',\underline{z}$, where $\underline{w} = \sum w_k + \underline{w}_L' + \underline{w}_R'+|F|+|F'|$.

As usual, the auxiliary data is constructed consistently via \ref{Subsection A comment about the existence of universal choices of auxiliary data}, so that on the broken configurations we are using the auxiliary data previously constructed for other operations. So in picture 1, the left disc-bubble carries the data of the $\mathrm{CO}^0:SC^*(E)\to CW^*(K,K)$ map, the right disc-bubble carries the data of the $\mathrm{OC}:\mathrm{CC}_*(\mathcal{A},\overline{\mathcal{M}})\to SC^*(E)$ map constructed in \ref{Subsection The OC map using new diag bimodule}. In picture 5 the left bubble carries the data for $\mu_{\mathcal{A}}^*$ constructed in \ref{Subsection wrapped A infinity structure}, and the right bubble carries the data for $\Delta: \overline{\mathcal{M}}\to \mathcal{L}\otimes \mathcal{R}$ constructed in \ref{Subsection The coproduct}.

Consider a $1$-dimensional moduli space of solutions $u: S \to E$ of $(du-X\otimes \gamma)^{0,1}=0$ defined on a fixed annulus $S$, and satisfying the relevant boundary and asymptotic conditions. Pictures 1 and 5 of Figure \ref{Figure coproduct2} are two examples of boundary points of this moduli space. 
We conclude that picture 1 is a contribution to the coefficient of $\underline{x}$ in\footnote{By $\partial_{\mathbf{q}_R}$-equivariance, the $(\mathbf{q}^0,\mathbf{q}^0)$-output of $\mathrm{OC}(\mathbf{q}_L\mathbf{q}_R\underline{x}',\mathbf{q}x_5,x_4,x_3,\mathbf{q}x_2,x_1)$ contributes to the $\mathbf{q}^1$-coefficient of
$\mathrm{OC}(\mathbf{q}_L\mathbf{q}_R\underline{x}',\mathbf{q}x_5,\mathbf{q}x_4,x_3,\mathbf{q}x_2,x_1)$. Compare the Remarks in \ref{Subsection the auxiliary data for wrapped category} and \ref{Subsection wrapped A infinity structure}. 
}
$$
\mathrm{CO}^0 \circ \mathrm{OC}(\mathbf{q}_L\mathbf{q}_R\underline{x}',\mathbf{q}x_5,\mathbf{q}x_4,x_3,\mathbf{q}x_2,x_1),
$$
whereas picture 5 is a contribution to the coefficient of $\underline{x}$ in%
\footnote{By $\partial_{\mathbf{q}_L}$-equivariance, the $(\mathbf{q}^0,\mathbf{q}^0)$-output of $\Delta^{2|1}(x_2,x_1,\mathbf{q}_L\mathbf{q}_R\underline{x}',\mathbf{q}x_5)$ contributes to the $(\mathbf{q}^1,\mathbf{q}^0)$-coefficient of
$\Delta^{2|1}(\mathbf{q}x_2,x_1,\mathbf{q}_L\mathbf{q}_R\underline{x}',\mathbf{q}x_5)$. Compare the Remarks in \ref{Subsection the auxiliary data for wrapped category} and \ref{Subsection wrapped A infinity structure}. 
}
$$
\mu^{3}_{\mathcal{S}}\circ \mathcal{T}\left(\underline{\Delta^{2|1}(\mathbf{q}x_2,x_1,\mathbf{q}_L\mathbf{q}_R\underline{x}',\mathbf{q}x_5)}\otimes \mathbf{q}x_4 \otimes x_3\right), 
$$
where $\mathcal{T}$ is the reordering isomorphism from \ref{Subsection tensor products} which induces \eqref{Eqn HH mathcal S LR is H R times S} (notice we are inputting into $\mathcal{T}$ a typical summand of $\mathrm{CC}_*(\Delta)(\mathbf{q}x_2,x_1,\mathbf{q}_L\mathbf{q}_R\underline{x}',\mathbf{q}x_5,\mathbf{q}x_4, x_3)$; see Lemma \ref{Lemma functoriality of HH}).

This example now easily generalizes to show that the boundary points of such $1$-dimensional moduli spaces include the counts involved in the maps $\mathrm{C}_1=\mathrm{CO}^0\circ \mathrm{OC}$ and $\mathrm{C_2}=\mu_{\mathcal{S}}^{\mathrm{max}}\circ \mathcal{T} \circ \mathrm{CC}_*(\Delta)$. There are however also two other contributions which we now describe.

Let $\varphi: \mathrm{CC}_*(\mathcal{A},\mathcal{S})\to CW^*(K,K)$ denote the weighted count of isolated solutions $u$ defined on an annulus as above. In a $1$-family, when a subset of the outer boundary punctures converges to a common point, we obtain bubbling configurations counted by $\varphi\circ b$ where $b$ is the bar differential for $\mathrm{CC}_*(\mathcal{A},\mathcal{S})$. Secondly, when bubbling occurs on the inner boundary circle of the annulus, so at $\underline{z}=-1$, the broken configurations are counted by $\mu_{CW^*(K,K)}^1\circ \varphi$. 

We conclude that (up to signs) $\mathrm{C}_1$ and $\mathrm{C_2}$ differ at the chain level by $\varphi\circ b - \mu_{CW^*(K,K)}^1\circ \varphi$, therefore they agree on cohomology. This concludes the proof of Theorem \ref{Theorem generation in monotone setting}.\\ 
\noindent {\bf Technical Remark.} \emph{A more careful treatment of orientation signs \cite[Lemma 6.8]{Abouzaid} shows that
$C_1-(-1)^{n(n+1)/2}C_2=\varphi\circ b + (-1)^n \mu_{CW^*(K,K)}^1\circ \varphi$, so the maps $C_1$ and $C_2$ differ on homology by $(-1)^{n(n+1)/2}$, where $\dim E=2n$.}
\subsection{Generation for the compact category}
\label{Section Generation and functoriality for the compact category}
When working with $\mathcal{F}(E)$ or with $\mathcal{F}(B)$ for a closed monotone symplectic manifold $B$, the analogue of Theorem \ref{Theorem generation in monotone setting} holds, and the proof is analogous (indeed simpler, since one no longer needs a telescope chain model and no auxiliary forms $\alpha_k,\beta_f$ are in play, beyond using the Floer datum machinery of \ref{Subsection When Lagrangians do not intersect transversely}).

\begin{theorem}\label{Theorem generation for compact categories}
 Let $M$ be a closed monotone symplectic manifold, or a monotone symplectic manifold conical at infinity. Let $\mathcal{S}$ be a full subcategory of $\scrF_{\lambda}(M)$. If 
$$
\xymatrix@C=40pt{ 
\mathrm{HH}_*(\mathcal{S}) \ar@{->}^-{\mathrm{HH}_*(\mathrm{incl})}[r]  &  \mathrm{HH}_*(\scrF_{\lambda}(M)) 
\ar@{->}[r]^-{\mathrm{OC}} & QH^*(M)_{\lambda},
}
$$
hits an invertible element (in the notation of \ref{Subsection Eigensummands}), then $\mathcal{S}$ split-generates $\scrF_{\lambda}(M)$.
\end{theorem}

\section{Applications}
\label{Section Applications}

\subsection{Toric varieties and the Landau-Ginzburg superpotential}\label{Section Brief survey on Landau Ginzburg}
We assume the reader has some familiarity with the relevant literature (Cho-Oh \cite{Cho-Oh}, Auroux \cite{Auroux}, Fukaya-Oh-Ohta-Ono \cite{FOOOtoric}), and we refer to \cite[Appendix]{Ritter5} for details.
Let $(X,\omega_X)$ be a closed real $2n$-dimensional symplectic manifold together with an effective Hamiltonian action of the $n$-torus $U(1)^n$. This action determines a \emph{moment map} $\mu_X: X \to \R^n$
whose image $\Delta=\mu_X(X)\subset \R^n$ is a convex polytope, called the \emph{moment polytope}.
By Delzant's theorem $\Delta$ determines, up to isomorphism, $(X,\omega_X)$ together with the action. 
The moment polytope has the form
\begin{equation}\label{Eqn moment polytope}
\Delta = \{ y\in \R^n: \langle y,e_i \rangle \geq \lambda_i \textrm{ for } i=1,\ldots,r \},
\end{equation}
where $\lambda_i\in \R$ are parameters and $e_i\in \Z^n$ are the primitive inward-pointing normal vectors to the codimension 1 faces of $\Delta$.
%
The \emph{Landau-Ginzburg superpotential} is 
\begin{equation}\label{Eqn LG superpotential}
W:(\Lambda^*)^n \to \Lambda,\;\; W(z_1,\ldots,z_{n}) = \sum_{i=1}^r t^{-\lambda_i} z^{e_i}
\end{equation}
 where $\Lambda$ is the Novikov field defined over $\K=\C$ (see \ref{Subsection Novikov ring}), and $\Lambda^*=\Lambda\setminus \{0\}$.
%
%
The fibres $L_y=\mu_X^{-1}(y)$ of the moment map, for $y\in \mathrm{Int}(\Delta)\subset \R^n$, are Lagrangian tori.
Let $\mathrm{val}_t$ denote the valuation for the $t$-filtration on $\Lambda$, whose value is the lowest exponent of $t$ arising in the series. 
The conditions $\mathrm{val}_t(t^{-\lambda_i}z^{e_i}) > 0$ are equivalent to the equations
in \eqref{Eqn moment polytope} since
\begin{equation}\label{Eqvalt}
\mathrm{val}_t(t^{-\lambda_i}z^{e_i}) = \langle y,e_i \rangle - \lambda_i. 
\end{equation}
Thus $z$ determines a toric Lagrangian $L_y$, $y= \mathrm{val}_t(z)\in \mathrm{interior}(\Delta)\subset \R^n,$ together with a choice of holonomy around each generating circle of the torus $L_y$ given by 
 $t^{-\mathrm{val}_t(z)}z\in H^1(L,\Lambda_0^{\times}),$
using the ring $\Lambda_0^{\times}$ from \ref{Subsection Novikov rings for local systems}. 
Let $L_z$ be the Lagrangian $L_y$ with that holonomy.
\begin{lemma}[{\cite[Theorem 7.11]{FOOOtoric} using \cite[Theorem 8.1]{Cho-Oh}}]\label{Lemma crit pts give monotone torus at barycentre}
\strut

For a closed monotone toric manifold $X$, the $z\in \mathrm{Crit}(W)$ always 
have 
\begin{equation}\label{Eqn barycentre equation FOOO theorem}
y=\mathrm{val}_t(z)=\mathrm{barycentre}\in \mathrm{Int}(\Delta).
\end{equation}
and $L_y$ is a monotone Lagrangian. Thus $L_z$ for $z \in \mathrm{Crit}(W)$ is always the same Lagrangian torus $L_{\mathrm{crit}}$ (the fibre over the barycentre) but the choice of holonomy data depends on $z$.
\end{lemma}

%
\indent
In general, the barycentre $y_{\mathrm{bar}}\in \mathrm{Int}(\Delta)$ of $\Delta$ is determined by the equations
\begin{equation}\label{Eqvalt2}
\langle y_{\mathrm{bar}},e_i \rangle - \lambda_i = \tfrac{1}{\lambda_X}.
\end{equation}
In Lemma \ref{Lemma crit pts give monotone torus at barycentre}, given a critical point $z$, using \eqref{Eqvalt} and \eqref{Eqvalt2} we deduce
\begin{equation}
\label{Eqn ambiguity in z}
\boxed{z=(\zeta_1 t^{a_1},\ldots,\zeta_n t^{a_n}) \quad \textrm{ and } \quad
W(z) = c\,t^{1/\lambda_X}}
\end{equation}
where $\zeta_j \in \C^*$, $c \in \C$, $y_{\mathrm{bar}}=(a_1,\ldots,a_n)$ is the barycentre, and recall $t^{1/\lambda_X}$ has grading $2$. Also $L_z$, being monotone and orientable, defines an object of the (wrapped) Fukaya category in the sense of \ref{Subsection Lagrangians with local systems}.
We will only consider cases where $X$ is monotone (so for a Fano toric variety $X$ we always pick a monotone toric symplectic form), and in the non-compact case we require that $X$ is convex. We remark that the proof of the above Lemma still holds in the non-compact setting, after defining \emph{barycentre} as in the Appendix of \cite{Ritter5} to mean \eqref{Eqvalt2}. 
%
%
%
For $z\in\mathrm{Crit}(W)$, it follows \cite[Definition 3.3]{Auroux} that 
$$m_0(L_z) = W(z).$$
\begin{theorem}[Cho-Oh \cite{Cho-Oh}, see also Auroux \cite{Auroux}]\label{Theorem pt class is cycle for crit Lag}\strut\\
$HF^*(L_z,L_z) \neq 0$ if and only if $z \in \mathrm{Crit}(W)$, indeed this holds if and only if the point class $[\mathrm{pt}]\in C_*(L_z)$ is a Floer cycle in $HF^*(L_z,L_z)$ in the Morse-Bott model. In particular, if this holds, $L_z$ is not displaceable by a Hamiltonian isotopy.
\end{theorem}

\begin{corollary}\label{Corollary m0 is critval of W for toric}
$z\!\in\! \mathrm{Crit}(W) \Rightarrow \lambda\!=\!m_0(L_z)\!=\!W(z)$ is an eigenvalue of $c_1(TX)\!\in\! QH^*(X)$.
\end{corollary}
\begin{proof}(Mimicking Auroux \cite[Sec.6]{Auroux})
Let $c=c_1(TX)-\lambda\, \mathrm{Id}$. If $c$ were invertible in $QH^*(X)$, then since $\mathrm{CO}^0:QH^*(X)\to HF^*(L_z,L_z)$ is a unital algebra homomorphism (Theorem \ref{Theorem CO is unital alg hom}),
$$
1_{L_z}\equiv \mathrm{CO}^0(1)=\mathrm{CO}^0(c^{-1}*c)=\mu^2(\mathrm{CO}^0(c^{-1}),\mathrm{CO}^0(c))=0,
$$
as $\mathrm{CO}^0(c)=0$ by Lemma \ref{Lemma COc1 and m0}. But by Theorem \ref{Theorem pt class is cycle for crit Lag}, the unit $1_{L_z}\neq 0$ as $HF^*(L_z,L_z)\neq 0$.
\end{proof}

For closed monotone toric manifolds $B$, a classical result due to Batyrev and Givental is 
\begin{equation}\label{Eq QH is Jac}
QH^*(B) \cong \mathrm{Jac}(W) = \Lambda[z_1^{\pm 1},\ldots,z_n^{\pm 1}]/(\partial_{z_1}W,\ldots,\partial_{z_n}W), \; c_1(TB) \mapsto W.
\end{equation}
Since the eigenvalues of multiplication by $W$ acting on the Jacobian ring are precisely the critical values of $W$, one can strengthen Corollary \ref{Corollary m0 is critval of W for toric}:

\begin{corollary}\label{Corollary evalues of c1 and critvals of W}
For closed monotone toric $B$, the eigenvalues of $c_1(TB)\in QH^*(B)$ are precisely the critical values of $W$.
\end{corollary}

\begin{theorem}[{Galkin \cite{Galkin}}]\label{Theorem Galkin}
 For $B$ any closed monotone toric manifold, the superpotential $W: (\C^*)^n\to \C$ (working over $\C$ instead of $\Lambda$) always has a non-degenerate critical point $p\in (\C^*)^n$ with strictly positive real coordinates.
\end{theorem}
\begin{corollary}\label{Corollary c1TB is non nilpotent in monotone toric B}
 For $B$ closed monotone and toric, $c_1(TB)\in QH^*(B)$ is not nilpotent.
\end{corollary}
\begin{proof}
Observe that (putting $t=1$) $W=\sum z^{e_i}$ will take a strictly positive real value on Galkin's critical point $p$, so $c_1(TB)\in QH^*(B;\C)$ has a strictly positive real eigenvalue $\lambda_p$. The same holds true for the superpotential defined over $\Lambda$, since we can insert powers of $t$ as dictated by the $\Z$-grading of $QH^*(B)$ (the eigenvalue will be $\lambda_p t^{1/\lambda_B}$, in grading 2).
\end{proof}
%
%
\subsection{Toric negative line bundles}
A complex line bundle $\pi: E \to B$ over a closed symplectic manifold $(B,\omega_B)$ is a \emph{negative line bundle} if
$$
c_1(E) = -k[\omega_B] \; \textrm{ for some real } k>0.
$$
We refer to \cite{Ritter4,Ritter5} for a detailed exposition.
The negative line bundle $E$ is a symplectic manifold conical at infinity whose Reeb flow is the $S^1$-rotation in the fibres and whose symplectic form $\omega$ can be chosen \cite[Section 4]{Ritter5} to satisfy
$$
\begin{array}{lll}
 \omega|_{\mathrm{fibre}\equiv \C} & = & 
\omega_{\C}
\\
 \omega|_{\mathrm{base}=B} & = &  \omega_B,
\end{array}
$$
so $[\omega] \equiv [\omega_B]$ via $H^2(E)\cong H^2(B)$.
We require that $E$ is monotone, so $c_1(TE) = \lambda_E [\omega]$ for some $\lambda_E>0$.
Splitting $TE$ using a Hermitian connection, $c_1(TE)=\pi^*c_1(TB)+\pi^*c_1(E)$, where $c_1(E)$ is the first Chern class of the line bundle. Therefore, monotonicity of $E$ is equivalent to $B$ being monotone, so $c_1(TB) = \lambda_B [\omega_B]$, and that $\boxed{0< k < \lambda_B}$. In particular,
$$
\boxed{\lambda_E = \lambda_B - k}
$$
Our motivation for studying this particular class of non-compact manifolds, is because it is the only family of examples outside of exact cotangent bundles where $SH^*$ was computed:

\begin{theorem}[Ritter \cite{Ritter4}]\label{Theorem about SH of neg l b}
$c^*:QH^*(E) \to SH^*(E)$ induces a unital algebra isomorphism 
$$
SH^*(E) \cong QH^*(E)/QH^*(E)_0
$$
(in the notation of \ref{Subsection Eigensummands}), where we restrict\footnote{This is the full $SH^*(E)$ for toric negative line bundles, since toric varieties are simply connected.} $SH^*(E)$ 
to contractible loops (see \ref{Subsection Contractible vs noncontractible}). Moreover, $QH^*(E)_0$ is always non-trivial, and $c^*(c_1(TE))$ acts invertibly on $SH^*(E)$ (so $SH^*(E)_0=0$).
\end{theorem}

The Appendix in \cite{Ritter5} constructs the moment polytope $\Delta_E$ for the complex line bundle $E=\mathcal{O}(\sum n_i D_i)\to B$, where $D_i\subset B$ are the toric divisors, that is: $c_1(E)=\sum n_i \mathrm{PD}[D_i]$. 
If the moment polytope for $B$ is
$\Delta_B
=\{ y\in \R^{n}: \langle y, b_i \rangle \geq \lambda_i^B \textrm{ for } i=1,\ldots,r \}\subset \R^{n},
$
then take $e_1=(b_1,-n_1)$, $\ldots$, $e_r =(b_r,-n_r)$, $e_{r+1}=(0,\ldots,0,1) \in \Z^{n+1}$, $\lambda_i^E=\lambda_i^B$, and $\lambda_{r+1}^E=0$:
$$
\Delta_E  = \{ y\in \R^{n+1}: \langle y, e_i \rangle \geq \lambda_i^B \textrm{ for } i=1,\ldots,r \textrm{ and } y_{n+1}\geq 0 \}\subset \R^{n+1}.
$$
Although a closed Fano toric variety $B$ has a canonical polytope $\Delta$, with all $\lambda_i^{\Delta}=-1$, called the reflexive polytope, this will not be useful for our purposes. The reason is that $[\omega_{\Delta}]=c_1(TB)$, but we want a polytope $\Delta_B$ with $[\omega_B]\in H^2(B,\Z)$ integral and primitive so that $c_1(TB)=\lambda_B [\omega_B]$, $c_1(E)=-k[\omega_B]$ for integers $\lambda_B>0$, $k>0$, because then we can take
$$
\textstyle E = \mathcal{O}(\sum k\lambda_i^B D_i)
$$
since the toric symplectic form can be written in terms of the divisors as $[\omega_B]=\sum -\lambda_i^B \mathrm{PD}[D_i]$. If we used $\Delta$ instead, then $k$ would be fractional, and it would not be clear how to obtain $E$ as an integral combination $\sum n_i D_i$. Such a polytope $\Delta_B$ was constructed in \cite{Ritter5} (namely $\Delta_B=(\Delta - v)/\lambda_B$, for a vertex $v$ of $\Delta$). The construction ensures that $\lambda_i^B\leq 0$ are integers and the primitive inward normals $b_i$ are the same as for $\Delta$. 

\begin{lemma}
The barycentres $a\in \R^n$ of $\Delta_B$, and $a'\in \R^{n+1}$ of $\Delta_E$, are related by
\begin{equation}\label{Eqn comparison of barycentres}
\boxed{a_j' = \tfrac{\lambda_B}{\lambda_E}a_j \quad \textrm{ and }\quad a_{n+1}'=\tfrac{1}{\lambda_E}}
\end{equation}
\end{lemma}
\begin{proof}
This follows from \eqref{Eqvalt2} using that $e_j=(b_j,-k\lambda_i^B)$ and $e_{n+1}=(0,\ldots,0,1)$.
\end{proof}
%
%
%
Although one could a posteriori translate the polytope $\Delta_E$, so that at least $\Delta_B$ has barycentre at $0$ (this would make $\Delta_B=\Delta/\lambda_B$ and $\lambda_i^B=-1/\lambda_B$), this would not enable us to make the following substitution trick which relates the superpotentials $W_B,W_E$ of $\Delta_B,\Delta_E$, via the change of variables $\boxed{y=t w^k}$, abbreviating by $w=z_{n+1}$ the new variable associated to $e_{n+1}$.
$$
\begin{array}{rcl}
W_{E} & = & W_E(z_1,\ldots,z_n,w)\\
& = & \sum t^{-\lambda_i^B} z^{b_i}w^{-k\lambda_i^B} + w\\
& = & \sum y^{-\lambda_i^B} z^{b_i} + w\\
& = & W_B(y;z)+w
\end{array}
$$
where we view the superpotential $W_B$ of $B$ as a function not just of $z=(z_1,\ldots,z_n)\in (\Lambda_B^*)^n$ but also of the Novikov variable $T\in \Lambda_B$ (which above undergoes the substitution $T=y$), so
$$
W_B = W_B(T;z) = \sum T^{-\lambda_i^B} z^{b_i}.
$$
Notice there are two Novikov variables in play: $T,t$ respectively for $B,E$ so respectively in the Novikov rings $\Lambda_B,\Lambda$, and lying in degrees $2\lambda_B,2\lambda_E$. Our goal is to relate the critical points of $W_E,W_B$. This relationship is not straightforward because the resulting holonomies of the relevant line bundles on the barycentric torus fibres for $B$ and $E$ are quite different.

\subsection{Phases and roots of Novikov series}
\label{Subsection Phases and roots of Novikov series}

Given a solution of $dW_B=0$, to construct a solution of $dW_E=0$ we may need fractional powers of $y$ (e.g. the critical points for $B=\P^2$ involve $T^{1/3}$). For the purpose of proving the existence of a solution, it will be enough to prove that, given $y\in \Lambda$  involving only strictly positive powers of $t$, there exists \emph{some} $\K$-algebra homomorphism $\Lambda_B \to \Lambda$ sending $T\mapsto y$.
Our choice of substitution is governed by the following Definition/Claim.\\[1mm]
{\bf Definition.} $\theta=\theta(y)\in \R$  is a \emph{phase}
for $y\in \Lambda$ if $y=y_0 e^{i\theta}t^{r}+\mathcal{O}$, where we abusively write $\mathcal{O}$ to mean ``higher order $t$ terms'', with $y_0\in \R^*$. The shifts in phase, $\theta\mapsto \theta + 2\pi n$ for $n\in \Z$, preserve $y$ but they will affect our choice of $y^p$.
We call $r=\mathrm{val}_t(y)\in \R$ the \emph{valuation} of $y$, and $y_0e^{i\theta}=\mathrm{lead}_t(y)\in \C^*$ the \emph{leading term} of $y$.
\\[1mm]
{\bf Claim.} \emph{Given $y\in \Lambda$ with $\mathrm{val}_t(y)>0$, the group of $\K$-algebra homomorphisms $\Lambda_B \to \Lambda$, which send $T\mapsto y$, which are continuous in the $T$-adic and $t$-adic topologies, and for which $\mathrm{lead}_t(y^p)\in \C^*$ is continuous in $p\in \R$, is isomorphic to the group $\Z$ (more explicitly, we will construct such a homomorphism for each choice of the phase of $y$).}\\
\emph{Proof.}
Observe that $y_0^{-1} e^{-i\theta}t^{-r}y = 1+\mathcal{O}$. For any $p\in \R$, $1+\mathcal{O}$ has a unique $p$-th power $(1+\mathcal{O})^p$ of the form $1+\mathcal{O}$, obtained by applying $\exp(p\log(\cdot))$ using the explicit series expansion (Newton's generalised binomial theorem). We now define $y^p$ by the Novikov series
$$y^{p}= y_0^{p} e^{ip\theta}t^{pr}(1+\mathcal{O})^{p}$$
where $y_0^{p}\in \R^*$. The relations $y^py^q=y^{p+q}$ follow because $\exp(p\log(\cdot))$ is a group homomorphism on formal series of type $1+\mathcal{O}$. Any other choice of $p$-th power $y^p$ would be
$
y^{p}= f(p)\cdot y_0^{p} e^{ip\theta}t^{pr} (1+\mathcal{O})^{p}
$
where $f: \R \to S^1$ is a group homomorphism (in particular $f(0)=1$) satisfying the additional condition $f(1)=1$. 
To ensure that $\mathrm{lead}_t(y^p)$ is continuous in $p\in \R$, $f$ must be continuous. 
Such $f$ are classified: $f(p)= e^{2\pi i np}$ for $n\in \Z$, so they are encoded by the shifts in phase $\theta \mapsto \theta + 2\pi n$.
$\qed$
\\[1mm]\indent
Following \cite[Lemma 7.15]{Ritter5}, we now prove that critical points arise in families. Using the above phase formalism, changing the Novikov variable $t$ to $e^{2\pi i}t$ (viewed as a formal change in symbols) induces the following $\K$-algebra automorphism, which we call \emph{phase shift},
$$\textstyle \varphi:\Lambda\to \Lambda,\, \sum k_j t^{r_j} \mapsto \sum k_j e^{2\pi i r_j} t^{r_j}.$$
Observe that $\varphi$ is also $t$-adically continuous.\\
{\bf Remark.}\,\emph{We will show that a non-zero object in the (wrapped) Fukaya category with holonomy systems gives rise to a family of non-zero objects by taking the orbit under the subgroup of the Galois group of $\Lambda$ over the field   $\C[[t]][t^{-1}]$ of Laurent series, generated by the phase shift (the full Galois group is a larger, profinite group, as the function $f$ above need not in general be continuous).
In retrospect, our approach is similar to that of Fukaya \cite[Equation (3)]{FukayaGalois}.}

\begin{lemma}\label{Lemma phase shift automorphism}
Let $X$ be a monotone toric variety, whose polytope \eqref{Eqn moment polytope} involves integer parameters $\lambda_i$.
Then the superpotential \eqref{Eqn LG superpotential} is $\varphi$-equivariant, meaning $$W \circ \varphi^{\oplus n} = \varphi \circ W,$$ and $\varphi$ determines a free cyclic action $z\mapsto \varphi(z)$ on $\mathrm{Crit}(W)$ of period $\lambda_X$, given on \eqref{Eqn ambiguity in z} by
$$
(\zeta_1 t^{a_1},\ldots,\zeta_n t^{a_n})\mapsto 
(\zeta_1 e^{2\pi i a_1} t^{a_1},\ldots,\zeta_n e^{2\pi i a_n} t^{a_n})
 \quad \textrm{ and } \quad
W = c\,t^{1/\lambda_X}\mapsto c\,e^{2\pi i/\lambda_X}\,t^{1/\lambda_X}.
$$
\end{lemma}
\begin{proof}
$W(z)=\sum t^{-\lambda_j}z^{e_j}$ so
$\varphi(W(z))=\sum e^{-2\pi i \lambda_j}t^{-\lambda_j}\varphi(z)^{e_j} = W(\varphi(z_1),\ldots,\varphi(z_n))$ proves the first claim, using that $e^{-2\pi i \lambda_j}=1$ since $\lambda_j\in \Z$. Now differentiate, so $(dW)(\varphi^{\oplus n}(z))=(\varphi\circ dW)(z)$, proving the second claim.
\end{proof}

\indent 
We now clarify some notation we use later. Suppose we are given a critical point $z=z(T)$ of $W_B$, where we emphasize that this is a function of the Novikov variable $T\in \Lambda_B$. The critical value is $W_B(T;z(T))$. Notice that there is a difference between taking the derivative $\partial_T W_B(T;z)$ \emph{before} evaluating at the critical point $z=z(T)$, and the derivative $\partial_T (W_B(T;z(T)))$ of the critical value. The chain rule implies
\begin{equation}\label{Eqn dzW}
\left.\tfrac{\partial}{\partial T}\right|_{T=y} (W_B(T;z(T)) = 
\left.\tfrac{\partial}{\partial T}\right|_{T=y,z=z(y)} W_B(T;z) + \langle (d_z W_B)(y;z(y)), \left.\tfrac{\partial}{\partial T}\right|_{T=y} z(T)\rangle
\end{equation}
where the second term is the sum of $(\partial_{z_j} W_B)(y;z(y))\cdot \partial_T|_{T=y}z_j(T)$ over $j=1,\ldots,n$.\\[1mm]
{\bf Example.} Suppose $W_B(T;z)=T^3z^2$ and $z(T)=T^2$, in the $n=1$ case (so $z\in \Lambda_B^*$). Then $W_B(T;z(T))=T^7$ so \eqref{Eqn dzW} reads: $(7T^6)|_{T=y}= (3T^2z^2)|_{T=y,z=y^2} + (2T^3z)|_{T=y,z=y^2}(2T)|_{T=y}$.
%
%
%
%
\subsection{Generation results for toric negative line bundles}
By ``\emph{phase-shift invariant family}'' we mean an orbit of a critical point under the action of $\varphi$ as in Lemma \ref{Lemma phase shift automorphism}.

\begin{lemma}\label{Lemma finding crit pts for E from B}
 Each phase-shift invariant family in $\mathrm{Crit}(W_B)$, with non-zero critical values, gives rise via \eqref{Eqn solutions z and w}-\eqref{Eqn phase is consistent} to a unique phase-shift invariant family in $\mathrm{Crit}(W_E)$.
\end{lemma}
\begin{proof}
Recall $W_E=W_B(y;z)+w$ and $y=t w^k$. By the chain rule, using $\tfrac{\partial y}{\partial w} = t k w^{k-1}$,
$$
dW_E = (d_z W_B)(y;z) + (tkw^{k-1}(\partial_T W_B)(y;z) + 1)\,dw.
$$

It follows that the $(z,w)\in \mathrm{Crit}(W_E)$ are the solutions of the following three equations: 
$$
\begin{array}{rrcl}
(i) & tw^k &=& y\\
 (ii) & w &=& -k y\, (\partial_T W_B)(y;z) \\
(iii) & (d_z W_B)(y;z) &=&  0.
\end{array}
$$
%
{\bf Remark.}\,$\mathrm{val}_t(w)>0$ as $(z,w)\in \mathrm{Crit}(W_E)$ lies in the interior of the polytope, indeed at the barycentre \eqref{Eqvalt}-\eqref{Eqvalt2}. So $\mathrm{val}_t(y)>0$, so we may substitute $T\mapsto y$ as explained in Sec.\ref{Subsection Phases and roots of Novikov series}.\\
{\bf Clarification.}\,\emph{We seek a solution $z=z(y)\in (\Lambda^*)^n$, $w=w(y)\in \Lambda^*$ in terms of $y$, involving the Novikov variable $t\in \Lambda$, not $T\in \Lambda_B$. Observe that $$z=z(T)\in \mathrm{Crit}(W_B)\subset (\Lambda_B^*)^n \Rightarrow (d_z W_B)(T;z(T)) =0\Rightarrow (d_z W_B)(y;z(y))=0\Rightarrow \textrm{(iii)}.$$ 
(However, it is not clear whether the converse (iii) $\Rightarrow z\in \mathrm{Crit}(W_B)$ holds.)
%
}\\[1mm] 
%
%
%
%
\indent A consequence of \eqref{Eqn dzW} is that for any Novikov series $z(T)$ with $(z(y),w)\in\mathrm{Crit}(W_E)$ we have $W_B(T;z(T))\neq 0$, otherwise $(\partial_T W_B)(y;z(y))=0$ by \eqref{Eqn dzW} so $dW_E= dw \neq 0$.\\
\indent Denote by $z=z(T)\in \mathrm{Crit}(W_B)$ a generator of the given $\lambda_B$-family, so by assumption $W_B(z(T))\neq 0$. We abbreviate
$$
\boxed{\ell = \tfrac{1}{\lambda_B} \qquad \textrm{(in particular }1-k\ell = \tfrac{\lambda_E}{\lambda_B})}
$$
By \eqref{Eqn ambiguity in z}, $z=z(T)=(\zeta_1 T^{a_1},\ldots,\zeta_n T^{a_n})$ and $W_B(T;z(T)) = c\,T^{\ell}$ where $\zeta_j \in \C^*$, $c\neq 0 \in \C$, $a=(a_1,\ldots,a_n)$ is the barycentre of $\Delta_B$. Thus 
%
\begin{equation}\label{Eqn ambiguity in W_B(y)}
W_B(y;z(y)) = cy^{\ell}.
\end{equation}
%
%
By \eqref{Eqn dzW} and equation (iii), namely $(d_z W_B)(y;z(y))=0$, we obtain:
$$
(\partial_T W_B)(y;z(y)) = 
\partial_T|_{T=y} (W_B(T;z(T)))
= \partial_T|_{T=y} (c T^{\ell}) = c \ell y^{\ell -1}.
$$
Solving equation (ii), given $y$, we obtain $(z(y),w(y))\in \mathrm{Crit}(W_E)$ defined by\footnote{{\bf Remark.} \emph{We already know a priori that $(z(y),w(y))$ lies at the barycentre of $\Delta_E$ by \eqref{Eqvalt2} (for $X=E$), but one can also check by hand that $\mathrm{val}_t(z(y)) = \tfrac{\lambda_B}{\lambda_E}\, a$ and $\mathrm{val}_t(w(y)) = \tfrac{1}{\lambda_E}$ (confirming \eqref{Eqn comparison of barycentres} and \eqref{Eqvalt2}) and in particular $\mathrm{val}_t(z,w)>0$ in each entry, so $(z,w)\in \mathrm{Int}(\Delta_E)$.}}
\begin{equation}\label{Eqn solutions z and w}
\boxed{z(y)=(\zeta_1 y^{a_1},\ldots,\zeta_n y^{a_n})\quad\textrm{ and }\quad w(y)=-kc \ell y^{\ell}}
\end{equation}
and the critical value is
 \begin{equation}\label{Eqn Ambiguity in W_E crit}
 \boxed{W_E=W_E(z(y),w(y)) = W_B(y;z(y)) + w(y) = \tfrac{\lambda_E}{\lambda_B}\, cy^{\ell} \neq 0}
 \end{equation}
where $y\neq 0$ since $w\in \Lambda^*$ by definition.
It remains to solve for $y$. Combining (i) and \eqref{Eqn solutions z and w}
gives $y^{\lambda_E/\lambda_B} = (-kc\ell )^{k}\, t$, so:
\begin{equation}\label{Eqn ambiguity in y to the lambdaB-k}
\boxed{y = (-kc\ell )^{k\lambda_B / \lambda_E}t^{\lambda_B / \lambda_E}}
\end{equation}
In particular, phases also need to be consistent:
\begin{equation}\label{Eqn phase is consistent}
\boxed{\theta(y)=\tfrac{k\lambda_B}{\lambda_E}\theta(c)\qquad \textrm{ (in particular } \theta(cy^{\ell})=\tfrac{\lambda_B}{\lambda_E}\theta(c))}
\end{equation}
By Lemma \ref{Lemma phase shift automorphism}, the phase shift automorphism for $B$ acts by $\zeta_j\mapsto \zeta_j e^{2\pi i a_j}$ and $c\mapsto ce^{2\pi i/\lambda_B}$. This induces the change $y\mapsto e^{2\pi i k/\lambda_E} y$ and hence $cy^{\ell} \mapsto e^{2\pi i/\lambda_E} c y^{\ell}$. Thus \eqref{Eqn solutions z and w}
changes by $\zeta_j y^{a_j}\mapsto \zeta_j e^{2\pi i a_j'} y^{a_j}$ (using $a_j'=a_j\lambda_B/\lambda_E$ from \eqref{Eqn comparison of barycentres}) and $w(y)\mapsto e^{2\pi i/\lambda_E} w(y)$, and \eqref{Eqn Ambiguity in W_E crit} changes by $W_E\mapsto e^{2\pi i/\lambda_E} W_E$. This agrees with the action of the phase shift automorphism for $E$, by Lemma \ref{Lemma phase shift automorphism}. The claim follows.
\end{proof}
\noindent {\bf Remark.} \emph{Applying the phase shift $\lambda_B$ times preserves $z(T)\in \mathrm{Crit}(W_B)$, but it typically induces a change in $(z(y),w(y))\in \mathrm{Crit}(W_E)$ as the phase of $y$ undergoes a shift by \eqref{Eqn phase is consistent}.}
\begin{theorem}\label{Theorem neg toric lb have nonzero SH}
 Let $E \to B$ be a monotone negative line bundle over a closed monotone toric symplectic manifold $B$. Then the monotone Lagrangian torus $\mathcal{L}\subset E$, lying over $L_{\mathrm{crit}} \subset B$ in a sphere bundle of $E$, 
is non-displaceable and admits holonomy data such that $HF^*(\mathcal{L},\mathcal{L})\cong HW^*(\mathcal{L},\mathcal{L})\neq 0.$ Moreover, $c_1(E)\in QH^*(E)$ is non-nilpotent and $SH^*(E)\neq 0$.
\end{theorem}
\begin{proof}
By Corollaries \ref{Corollary evalues of c1 and critvals of W} and \ref{Corollary c1TB is non nilpotent in monotone toric B}, there exists a $z=z(T)\in \mathrm{Crit}(W_B)$ with non-zero critical value. Lemma \ref{Lemma finding crit pts for E from B} implies that $\mathrm{Crit}(W_E)\neq \emptyset$, so the first claim follows 
by Theorem \ref{Theorem pt class is cycle for crit Lag} and Lemma \ref{Lemma HF=HW for closed Lags}. In particular, it is proved in the Appendix of \cite{Ritter5} (by the same argument used for Lemma \ref{Lemma crit pts give monotone torus at barycentre}, using \cite[Theorem 8.1]{Cho-Oh} and \eqref{Eqvalt2}), that the toric fibre $\mathcal{L}$ of the moment map for $E$ over the barycentre of $\Delta_E$ is a monotone Lagrangian torus in a sphere bundle of $E$ lying over the monotone Lagrangian torus $L_{\mathrm{crit}}\subset B$ (corresponding to the barycentre of $\Delta_B$). The sphere bundle is defined using an appropriate Hermitian metric \cite[Section 7]{Ritter4}, and the radius is determined by the monotonicity condition. Concretely, $\mathcal{L}$ has an additional fibre circle factor, arising from the natural $U(1)$-toric rotation action on the fibres, and the natural filling disc $u$ of that circle in the $\C$-fibre must satisfy the monotonicity condition: $\mathrm{MaslovIndex}([u]) = 2\lambda_E\int_{\D} u^*\omega_E $.

That $c_1(TE)\in QH^*(E)$ is not nilpotent (so $SH^*(E)\neq 0$ by Theorem \ref{Theorem about SH of neg l b}) follows because \eqref{Eqn Ambiguity in W_E crit} is an eigenvalue of $c_1(TE)$ by Corollary \ref{Corollary m0 is critval of W for toric}.
\end{proof}

\begin{lemma}\label{Lemma we found all critical points}
All the critical points of $W_E$ arise via \eqref{Eqn solutions z and w}-\eqref{Eqn phase is consistent} from the critical points of $W_B$ which have non-zero critical value.  
\end{lemma}
\begin{proof}
We reverse the previous proof: if $(z',w')\in \mathrm{Crit}(W_E)$, we want to find $z=z(T)\in \mathrm{Crit}(W_B)$ such that $z'=z(y)$, $w'=w(y)$ for some solution $y$ of \eqref{Eqn ambiguity in y to the lambdaB-k}. By \eqref{Eqvalt2} for $X=E$,
$$
z' = (\zeta_1't^{a_1'},\ldots,\zeta_n't^{a_n'}) \qquad\quad w' = \xi' t^{\ell'} \qquad\quad \; W_E(z',w') = c't^{\ell'}
$$
where $\ell'=\frac{1}{\lambda_E} = \frac{1}{\lambda_B-k}$,
$a'\in \R^n$ is the barycentre of $\Delta_E$ so \eqref{Eqn comparison of barycentres} holds, $\zeta'\in (\C^*)^n$, $\xi'\in \C^*$, $c'\in \C$. 
By \eqref{Eqn ambiguity in z}
we want $z=(\zeta_1t^{a_1},\ldots,\zeta_nt^{a_n})$, with $\zeta\in (\C^*)^n$ and $a\in \R^n$ the barycentre of $\Delta_B$, such that $z'=z(y)$, $w'=w(y)$ where 
$$y=tw'^k= \xi'^k t^{\lambda_B/\lambda_E}$$ 
(using $1+k\ell'=\lambda_B/\lambda_E$). So $z'=z(y)$ forces the equality: $ \zeta_j't^{a_j'} = \xi'^{ka_j} \zeta_j t^{a_j\lambda_B/\lambda_E}.$
By \eqref{Eqn comparison of barycentres}, this equality holds
for $z(T)$ by taking 
\begin{equation}\label{Eqn from WE to WB}
\boxed{\zeta_j=\zeta_j'\xi'^{-ka_j}, \textrm{\,or equivalently: }z_j(T) = (z_j'w'^{-ka_j})|_{t=T}}
\end{equation}
By equation (iii) of the previous proof, $(d_z W_B)(y;z(y))=0$. We claim that $(d_z W_B)(T;z)=0$, so $z=z(T)\in \mathrm{Crit}(W_B)$ (compare the Clarification in the proof of Lemma \ref{Lemma finding crit pts for E from B}). Indeed, because $y=\xi'^k t^{\lambda_B/\lambda_E}$ is homogeneous in $t$, $y$ plays the same role as a formal variable in the equation (iii): $(d_zW_B)(y,z(y))=0$, so we can replace it with the formal variable $T$.
\end{proof}

\begin{corollary}\label{Corollary count of solutions}
 There are precisely $(\lambda_B-k)N/\lambda_B$ distinct choices of holonomy for $\mathcal{L}\subset E$ in Theorem \ref{Theorem neg toric lb have nonzero SH} for which $HF^*(\mathcal{L},\mathcal{L})\neq 0$, where $N>0$ is the number of critical points $z\in \mathrm{Crit}(W_B)$ with non-zero critical value $W_B(T;z)\neq 0$.
Moreover, the critical values of $W_E$ are $\mu\, t^{1/\lambda_E}$ where $\mu\in \C^*$ are the roots of
\begin{equation}\label{Eqn mu from c}
(\tfrac{\mu}{\lambda_E})^{\lambda_E} = 
(\tfrac{c}{\lambda_B})^{\lambda_B}\,(-k)^k 
\end{equation}
as $c\,T^{1/\lambda_B}$ varies over possible non-zero critical values of $W_B$, and recall $\lambda_E=\lambda_B-k$.
\end{corollary}
\begin{proof}
The first claim follows by the one-to-one correspondence between families of critical points established by Lemmas \ref{Lemma finding crit pts for E from B} and \ref{Lemma we found all critical points}.
That critical values of $W_E$ correspond to eigenvalues of $c_1(TE)$ follows by Corollary \ref{Corollary m0 is critval of W for toric}. 
By \eqref{Eqn Ambiguity in W_E crit}, the $W_E$-critical value is $\mu= c\, \frac{\lambda_E}{\lambda_B}\, y^{\ell}$, so by \eqref{Eqn ambiguity in y to the lambdaB-k},
$\mu^{\lambda_E}=(c\tfrac{\lambda_E}{\lambda_B})^{\lambda_E}(-kc\ell )^{k}\,t = (\tfrac{c}{\lambda_B})^{\lambda_B}\lambda_E^{\lambda_E}(-k)^k\, t.$ 
\end{proof}

\noindent {\bf Conjecture.} \emph{$\mathcal{W}(E)$ is split-generated by $\mathcal{L}$ taken with the holonomies from Corollary \ref{Corollary count of solutions}. This holds when $W_B$ is Morse by \cite{Ritter5}, but it should also hold in general.}\\
{\bf Remark 1.}\;\emph{In fact \cite{Ritter5}, $\pi^*[\omega_B]=\tfrac{c_1(TE)}{\lambda_E}$, $[\omega_B]=\tfrac{c_1(TB)}{\lambda_B}$ and $c_1(E)^k=(-k)^k[\omega_B]^k$ are the cause of the constant factors in \eqref{Eqn mu from c}.}\\
{\bf Remark 2.}\;\emph{By Theorem \ref{Theorem about SH of neg l b}, zero-eigenvalues for $c_1(TE)$ must exist, but in the proof of Lemma \ref{Lemma finding crit pts for E from B} we showed these cannot arise from critical values of $W_E$ (which are non-zero). In turns out \cite{Ritter5} that all the non-zero eigenvalues of $c_1(TE)$ arise as critical values of $W_E$.}

\subsection{Generation for $\mathbf{\mathcal{O}(-k) \to \P^m}$}
\label{Subsection Generation for O of -k}

We will consider the negative line bundles 
$E=\mathcal{O}(-k)$ over $B= \P^m,$
so $c_1(E)=-k[\omega_{\P^m}]$.
These are monotone for $1\leq k \leq m$. Moreover, $$\lambda_B = 1+m \qquad\quad\lambda_E = 1+m-k.$$

\begin{theorem}[Ritter {\cite{Ritter4,Ritter5}}]\label{Theorem QH and SH of O(-k)}
 Let $E=\mathrm{Tot}(\mathcal{O}(-k)\to \P^m)$ with $1\leq k\leq m$. Then $QH^*(E)$ and $SH^*(E)$ are generated over $\Lambda$ by $w=\pi^*[\omega_B]\in QH^*(E)$  (viewing $SH^*(E)$ as a quotient of $QH^*(E)$ by Theorem \ref{Theorem about SH of neg l b}) subject to the relations%
$$\begin{array}{rcl}
w^{k}(w^{1+m-k}-(-k)^k \,t)=0 &\mathrm{ in } & QH^*(E),\\ w^{1+m-k}-(-k)^k \,t=0 &\mathrm{ in } & SH^*(E).
  \end{array}
$$
\end{theorem}

Recall that $QH^*(\P^m) \cong \Lambda[w]/(w^{1+m} - T)$, $\omega_{\P^m} \mapsto w$, where $T$ is the Novikov variable for $B=\P^m$. So the eigenvalues of $c_1(TB)*\cdot$ on $QH^*(B)$ involve $(1+m)$-th roots of unity (up to a fixed constant factor). On the other hand, for $E=\mathcal{O}_{\P^m}(-k)$ the above calculation shows that the non-zero eigenvalues of $c_1(TE)*\cdot$ on $QH^*(E)$ involve $(1+m-k)$-th roots of unity (up to a constant factor). The moment polytope for $B=\P^m$ is determined by $b_1=(1,0,\ldots,0)$, $\ldots$, $b_m=(0,\ldots,0,1)$, $b_{m+1}=(-1,\ldots,-1)$, and $\lambda_i^B=0$ for $i\leq m$ and $\lambda_{m+1}^B=-1$, so: 
$$
\begin{array}{rcl}
\Delta_B&=&\{ y\in \R^{m}: y_1\geq 0, \ldots, y_m\geq 0,\sum y_j \leq 1\}
\\
W_B(T;z) &=& z_1 + \cdots + z_{m} + T z_1^{-1}\cdots z_{m}^{-1}.
\end{array}
$$
The critical points of $W_B$ are: for each $\lambda_B$-th root of unity $\zeta$, there is one critical point 
$$z_{\mathrm{crit}}=(1,\ldots,1)\cdot\zeta\, T^{1/\lambda_B}\quad\qquad W_B(z_{\mathrm{crit}}) = \lambda_B\, \zeta \,T^{1/\lambda_B}.$$  
The toric divisors $D_i\subset B$ are given by setting the $i$-th homogeneous coordinate of $\P^m$ to zero (using the $0$-th coordinate for $D_{m+1}$). So each $D_i$ is Poincar\'{e} dual to $\omega_B$. The moment polytope for $E=\mathcal{O}(-k)=\mathcal{O}(-kD_{m+1}) \to \P^m$ is determined by $e_1=(b_1,0),\ldots,e_m=(b_m,0),e_{m+1}=(b_{m+1},k)$, and $\lambda_i^E=0$ for all $i=1,\ldots,m+2$ except $\lambda_{m+1}^E=\lambda_{m+1}^B=-1$, so:
$$
\begin{array}{rcl}
\Delta_E &=&\{ y \in \R^{m+1}: y_1 \geq 0, \ldots, y_m \geq 0, \sum_{j=1}^m y_j - ky_{m+1} \leq 1,y_{m+1}\geq 0\}
\\
W_E &=& z_1 + \cdots + z_m + t z_1^{-1}\cdots z_{m}^{-1} w^k + w = W_B(tw^k;z)+w.
\end{array}
$$
The critical points of $W_E$ are: for each $\lambda_E$-th root of unity $\zeta'$,
the solution $x=\zeta'(-k)^{k/\lambda_E}t^{1/\lambda_E}$ of $x^{1+m-k}=(-k)^k t$ corresponds to a critical point $(x,\ldots,x,-kx)$ with $W_E=\lambda_E x$:
$$
z_{\mathrm{crit}}= 
(1,\ldots,1,-k)\cdot\zeta'\,(-k)^{k/\lambda_E}\,t^{1/\lambda_E}
\quad\qquad W_E(z_{\mathrm{crit}}) = \lambda_E\, \zeta'\,(-k)^{k/\lambda_E}\,t^{1/\lambda_E}.
$$
In particular, we obtain
 $$y=tw^k = t(-k)^k x^k =  (-k)^{k\lambda_B/\lambda_E} (\zeta')^k t^{\lambda_B/\lambda_E}.$$
Since \eqref{Eqn ambiguity in y to the lambdaB-k} becomes $y=(-k\zeta)^{k\lambda_B/\lambda_E} t^{\lambda_B/\lambda_E}$, the one-to-one correspondence between families of critical points for $B$ and for $E$ boils down to the relation
$(\zeta')^{\lambda_E} = \zeta^{\lambda_B}=1$, and one can confirm by hand that \eqref{Eqn solutions z and w}, \eqref{Eqn Ambiguity in W_E crit}, \eqref{Eqn from WE to WB}, \eqref{Eqn mu from c} indeed hold.

The Jacobian ring $\mathrm{Jac}(W_E) \equiv \Lambda[z_1^{\pm 1},\ldots,z_m^{\pm 1},w^{\pm 1}]/(\partial_{z_1}W_E, \ldots, \partial_{z_m} W_E,\partial_{w}W_E)$ is
$$
\mathrm{Jac}(W_E) \equiv \Lambda[x]/(x^{1+m-k}-(-k)^k t).
$$
\begin{corollary}
$SH^*(E) \cong \mathrm{Jac}(W_E)$ for $1\leq k \leq m$, and $W_E=(1+m-k)x \mapsto c_1(TE)$.
\end{corollary}

This Corollary is noteworthy since it is a proof of closed-string mirror symmetry in this example.
Notice that $SH^*(E)$ now plays the role that $QH^*(E)$ did in the closed case \eqref{Eq QH is Jac}.

\begin{theorem}\label{Theorem wrapped cat of O of -k}
 Let $E=\mathcal{O}_{\P^m}(-k)$, for $1 \leq k\leq m$, and let $\mathcal{L}$ be the lift of the Clifford torus in the base $\P^m$ to the sphere subbundle. Then
\begin{enumerate}
 \item  The wrapped category $\mathcal{W}(E)$ is compactly generated. Indeed, for each of the $1+m-k$ choices of non-zero $\lambda\in \mathrm{Spec}(c_1(TE))$, the Fukaya categories $\mathcal{W}_{\lambda}(E)$ and $\mathcal{F}_{\lambda}(E)$ are both split-generated by $\mathcal{L}$ with a choice of holonomy determined by $\lambda$.

 \item $\mathcal{W}(E)$ is proper (cohomologically finite), so $\mathrm{dim}\, HW^*(L_1,L_2) < \infty$ for any $L_1,L_2$.

 \item $HF^*(\mathcal{L},\mathcal{L})\cong HW^*(\mathcal{L},\mathcal{L})\neq 0$ (Theorem \ref{Theorem neg toric lb have nonzero SH}), so $\mathcal{L}$ is non-displaceable.
 
 \item $\mathrm{OC}^0:HW^*(\mathcal{L},\mathcal{L}) \to SH^*(E)$ is non-zero.
\end{enumerate}
\end{theorem}
\begin{proof}
 By Theorem \ref{Theorem QH and SH of O(-k)}, $QH^*(E)$ splits as a direct sum of $QH^*(E)_0$ and $1$-dimensional field summands corresponding to the non-zero eigenvalues, and by Theorem \ref{Theorem about SH of neg l b}, the map $c^*: QH^*(E) \to SH^*(E)$ is the quotient by $QH^*(E)_0$.

Let $\mathcal{L}=\mathcal{L}_z$ be a Lagrangian torus in $E$ with appropriate holonomy arising from $z\in \mathrm{Crit}(W_E)$.
Call $\lambda_{\mathcal{L}}=m_0(\mathcal{L})=W_E(z)$ the non-zero eigenvalue of $c_1(TE)$ corresponding to $\mathcal{L}$. By Theorem \ref{Theorem QH and SH of O(-k)} and the above calculation of the critical points of $W_E$, all except for the $0$-eigenvalue of $c_1(TE)$ arise as $\lambda_{\mathcal{L}}$ for some critical point.

Using the Morse-Bott model \cite{Auroux}, and using that $[\mathrm{pt}]$ is a cycle by Theorem \ref{Theorem pt class is cycle for crit Lag}, observe that
 $\mathrm{OC}: HF^*(\mathcal{L},\mathcal{L}) \to QH^*(E)$ satisfies
$$
\mathrm{OC}([\mathrm{pt}]) = -kx \cdot \pi^*\omega_{\P^m}^m + (\textrm{higher order }t\textrm{ terms}) \neq 0 \in QH^*(E),
$$
where the first term arises from the regular Maslov 2 disc in the fibre (the standard disc filling $S^1\subset \C$) and $-kx\neq 0 \in \Lambda$ is the holonomy for the $S^1$-fibre direction of $\mathcal{L}$ (here $-kx$ also includes the factor $t^{1/\lambda_E}$ lying in degree $2$). Observe that there is no order $t^0$ term because the constant disc would sweep the lf-cycle $[\mathrm{pt}]\in H_0^{lf}(E)\cong H^{2m+2}(E)=0$. The $\pi^*\omega_{\P^m}^m$ factor arises because the fibre disc intersects once transversely the cycle intersection condition $[\P^m]\in C_*(E)$, and the intersection dual of this is the lf-cycle $[\mathrm{fibre}]\in H_2^{lf}(E)$ which corresponds to $\pi^*\omega_{\P^m}^m\in H^{2m}(E)$ (the fibre is the preimage of a point in the base, and taking preimages is Poincar\'e dual to pull-back on cohomology). It is a general feature of the toric setting that the Maslov index of a disc bounding a toric Lagrangian is twice the number of intersections of the disc with the anticanonical divisor (see Auroux \cite[Lemma 3.1]{Auroux}). In our case, the anticanonical divisor viewed as an lf-cycle, consists of the base $\P^m$ together with the preimage under projection of the anticanonical divisor in the base $\P^m$. So the fibre disc is the only Maslov two disc that can intersect the cycle $[\P^m]$. By monotonicity, all other discs with Maslov index greater than $2$ will contribute higher order $t$ terms. Thus:\\[1mm]
\noindent\fbox{\begin{tabular}{l}
 \quad \emph{(i) By Theorem \ref{Theorem OC is a module map}, $\mathrm{OC}([\mathrm{pt}])$ can only be non-zero in the eigensummand $QH^*(E)_{\lambda_{\mathcal{L}}}$.}\\
 \quad \emph{(ii) That eigensummand is a field, $\mathrm{OC}([\mathrm{pt}])$ is non-zero, hence $\mathrm{OC}([\mathrm{pt}])$ is invertible }\\
 \quad\qquad \emph{in the eigensummand.}\\
\quad \emph{(iii) Thus, by Theorem \ref{Theorem generation for compact categories}, $\mathcal{L}$ split-generates $\mathcal{F}_{\lambda_{\mathcal{L}}}(E)$.}\\
\quad \emph{(iv) Moreover, by the acceleration diagram (Theorem \ref{Theorem acceleration commutative diagram}),  also the map $\mathrm{OC}^0:$} \\
\quad\qquad \emph{$HW^*(\mathcal{L},\mathcal{L})\to SH^*(E)$ hits an invertible element, since by Theorem \ref{Theorem QH and SH of O(-k)} all } \\ 
\quad\qquad \emph{eigensummands except the $0$-eigensummand survive via $c^*:QH^*(E)\to SH^*(E)$.}\\
\quad \emph{(v) By Theorem \ref{Theorem generation in monotone setting}, $\mathcal{L}$ split-generates $\mathcal{W}_{\lambda_{\mathcal{L}}}(E)$.}
      \end{tabular}
}\vspace{1mm}

The finite-dimensionality claim follows because the split-generators are closed Lagrangians. In particular, $HF^*(\mathcal{L},\mathcal{L})\cong HW^*(\mathcal{L},\mathcal{L})$ (Lemma \ref{Lemma HF=HW for closed Lags}) is generated at the chain level by the finitely many intersections of $\mathcal{L}$ with a Hamiltonianly deformed copy of $\mathcal{L}$.
\end{proof}

\noindent {\bf Remark 1.}\;\emph{It is not clear what the split-generators for $\mathcal{F}_0(E)$ might be. By Theorem \ref{Theorem pt class is cycle for crit Lag}, they cannot be Lagrangian torus fibres of the moment map, with holonomy data, as those would be detected by $W_E$. But $W_E$ does not detect the zero eigenvalues of $c_1(TE)$.}\\
\noindent {\bf Remark 2.}\;\emph{We emphasise that although the wrapped category is proper (cohomologically finite), it does contain objects which are non-compact Lagrangians -- such as real line bundles over suitable Lagrangians in the base or Lefschetz thimbles for appropriate Lefschetz fibrations -- whose wrapped Floer cochain complexes are infinitely generated.}
%

\subsection{Generation results for closed monotone manifolds}%
\label{Subsection Generation results for closed toric Fano varieties}
In Lemma \ref{Lemma OC nonzero for monotone B}, we use the Morse-Bott model $CF^*(L,L)=C_{n-*}(L)$ of \cite{Auroux}, so there is a well-defined element $[\mathrm{pt}]\in CF^n(L,L)$. Note this need not a priori be a Floer cocycle (if it is not, the Lemma gives no information, as $[\mathrm{pt}]$ will not survive to cohomology).

\begin{lemma}\label{Lemma OC nonzero for monotone B}
For any closed monotone manifold $B$, and any monotone Lagrangian $L\in \mathrm{Ob}(\mathcal{F}_{\lambda}(B))$, the map $\mathrm{OC}^0: CF^*(L,L) \to QC^*(B)$ satisfies
$$
\mathrm{OC}([\mathrm{pt}]) = \mathrm{vol}_B + (\textrm{higher order }t\textrm{ terms})  \neq 0 \in QH^*(B).
$$
\end{lemma}
\begin{proof}
Since $L\subset B$ is monotone, Maslov index zero discs are constant. So the first term above corresponds to the constant disc, which is the lowest Maslov index disc, using that the volume form $\mathrm{vol}_B$ is the Poincar\'e dual of $[\mathrm{pt}]\in H_0(B)$.
All other terms will involve strictly positive powers of $t$ since they involve non-constant discs.
\end{proof}

\begin{theorem}
 Let $B$ be a closed monotone toric manifold, such that the eigenvalues of $c_1(TB)$ acting on $QH^*(B)$ are all distinct and have multiplicity one. Then $\mathcal{F}(B)$ is split-generated by $L_{\mathrm{crit}}$ with the holonomies described in Lemma \ref{Lemma crit pts give monotone torus at barycentre}.
\end{theorem}
 \begin{proof}
  The assumption on the eigenvalues of $c_1(TB)$ ensures that $QH^*(B)$ splits into $1$-dimensional eigensummands. By \eqref{Eq QH is Jac}, each $1$-dimensional eigensummand corresponds to a critical point $z\in \mathrm{Crit}(W_B)$ which determines a toric Lagrangian $L_z$ with suitable holonomy (see \ref{Section Brief survey on Landau Ginzburg}). By Lemma \ref{Lemma OC nonzero for monotone B}, $\mathrm{OC}^0: CF^*(L_z,L_z)\to QH^*(B)$ is non-zero on $[\mathrm{pt}]$. By Theorem \ref{Theorem pt class is cycle for crit Lag}, $[\mathrm{pt}]\in CF^*(L_z,L_z)$ is a cycle, so $\mathrm{OC}^0:HF^*(L_z,L_z)\to QH^*(B)$ is non-zero.

By the $QH^*(B)$-module structure of $\mathrm{OC}$, the element $\mathrm{OC}([\mathrm{pt}])$ can only be non-zero in the eigensummand $QH^*(B)_{\lambda_{L_z}}$ corresponding to the eigenvalue $\lambda_{L_z} = W_B(z)=m_0(L_z)$ of $c_1(TB)$. Being non-zero in a field implies invertibility, so $\mathrm{OC}$ hits an invertible element in the relevant eigensummand. By Theorem \ref{Theorem generation for compact categories}, the Lagrangians $L_z$ split-generate $\mathcal{F}(B)$.
 \end{proof}
%
\subsection{Generation results for convex symplectic manifolds}%
\label{Subsection Generation results for convex symplectic manifolds}

Let $(E,\omega)$ be a monotone symplectic manifold conical at infinity.
Let $L\in \mathrm{Ob}(\scrF(E)_{\lambda})$, so $L\subset E$ is a monotone closed orientable Lagrangian with $m_0(L)=\lambda$.

\begin{theorem} If $\lambda \neq 0$ and $[\mathrm{pt}]$ is a Floer cycle in $HF^*(L,L)$, then 
\begin{enumerate}
 \item $\mathrm{OC}([\mathrm{pt}]) \neq 0 \in QH^*(E)$;
 \item if the generalized eigensummand $QH^*(E)_{\lambda}$ is $1$-dimensional, then $L$ split-generates $\scrF(E)_{\lambda}$ and $\mathcal{W}(E)_{\lambda}$, in particular $\mathcal{W}(E)_{\lambda}$ is proper (cohomologically finite).
\end{enumerate}
\end{theorem}
\begin{proof}
By Lemma \ref{Lemma COc1 and m0},
the image under $\mathrm{CO}^0: QH^*(E) \to HF^*(L,L)$ of $c_1(TE)$ is $m_0(L)\, e_L$.
Now $c_1(TE)$ is a multiple of $[\omega]$, and $\omega$ is exact at infinity. Therefore $c_1(TE)$ can be represented by a closed de Rham form which is compactly supported, which in turn is Poincar\'e dual to a cycle $C\in H_2(E)$. Thus we can represent $c_1(TE)$ by a \emph{compact} lf-cycle $C$. Since the disc count above does not depend on the choice of representative, we can use $C$. 

Denote by $\langle \cdot,\cdot \rangle_E$ the intersection pairing in $E$ between lf-cycles and cycles, and by $\langle \cdot,\cdot \rangle_L$ the intersection pairing in $L$ between cycles. Then observe that, by definition, the following two disc counts are equal:
\begin{equation}\label{EqnDualityOCandCO}
\langle \mathrm{OC}([\mathrm{pt}]),C \rangle_E = \langle [\mathrm{pt}],\mathrm{CO}^0(C) \rangle_L = m_0(L).
\end{equation}
This proves the first claim since $m_0(L) = \lambda \neq 0$ by assumption, so $\mathrm{OC}([\mathrm{pt}])\neq 0$. The second claim follows by the same argument as in the proof of Theorem \ref{Theorem wrapped cat of O of -k} and by the acceleration diagram (Theorem \ref{Theorem acceleration commutative diagram}).
\end{proof}

%
%
%

\begin{corollary}
 Let $E$ be any (non-compact) toric Fano variety convex at infinity. Let $z\in \mathrm{Crit}(W)$ be a critical point of the superpotential $W$ with non-zero critical value $\lambda=W(z)\neq 0$. If the generalized eigensummand $QH^*(E)_{W(z)}$ is $1$-dimensional then the Lagrangian torus $L_z$ split-generates $\scrF(E)_{\lambda}$ and $\mathcal{W}(E)_{\lambda}$, in particular $\mathcal{W}(E)_{\lambda}$ is proper (cohomologically finite).
\end{corollary}
\begin{proof}
 The fact that $[\mathrm{pt}]$ is a Floer cycle in $HF^*(L_z,L_z)$ follows by Theorem \ref{Theorem pt class is cycle for crit Lag}. The Lagrangian $L_z$ (with holonomy) associated to $z$ has  eigenvalue $\lambda = m_0(L_z) = W(z) \neq 0$. So the claim follows by the previous Theorem.
\end{proof}

%

\end{document}